%% file: bpcy.tex
\documentclass[11pt,reqno]{amsart}
\usepackage[letterpaper,margin=1in,footskip=0.25in]{geometry}
\usepackage {amssymb}
\usepackage {amsmath}
\usepackage {bbm}
\usepackage{amsthm}
\usepackage{mathtools}
\usepackage{graphicx}
\usepackage {amscd}
\usepackage {epic}
\usepackage[mathscr]{euscript}
\usepackage{multirow}
\usepackage{stmaryrd}

\usepackage[dvipsnames]{xcolor}
\usepackage[colorlinks,citecolor=OliveGreen,linkcolor=Mahogany,urlcolor=Plum,pagebackref]{hyperref}
\usepackage{cleveref}
\usepackage{enumerate}
\usepackage{tikz}
\usepackage{tikz-cd}
\usepackage{verbatim,color,geometry}
\usepackage[all]{xy}
\usepackage{enumitem}
\usepackage{bm}
\newcommand{\A}{\mathbb{A}}
\newcommand{\G}{\mathbb{G}}

\newcommand{\N}{\mathbb{N}}

\newcommand{\Q}{\mathbb{Q}}

\newcommand{\T}{\mathbb{T}}
\newcommand{\Z}{\mathbb{Z}}

\newcommand{\cL}{{\mathcal{L}}}

\newcommand{\cB}{{\mathcal{B}}}
\newcommand{\cC}{{\mathcal{C}}}
\newcommand{\cF}{{\mathcal{F}}}
\newcommand{\cH}{{\mathcal{H}}}
\newcommand{\cI}{{\mathcal{I}}}

\newcommand{\cM}{{\mathcal{M}}}
\newcommand{\cR}{{\mathcal{R}}}

\newcommand{\cX}{{\mathcal{X}}}
\newcommand{\cW}{{\mathcal{W}}}
\newcommand{\cY}{{\mathcal{Y}}}
\newcommand{\cZ}{{\mathcal{Z}}}

\newcommand{\sM}{{\mathcal{M}}}

\newcommand{\sK}{{\mathscr{K}}}

\newcommand{\fm}{{\mathfrak{m}}}

\newcommand{\bP}{\mathbb{P}}

\newcommand{\bA}{\mathbb{A}}
\newcommand{\bQ}{\mathbb{Q}}
\newcommand{\bZ}{\mathbb{Z}}

\newcommand{\bG}{\mathbb{G}}
\newcommand{\bF}{\mathbb{F}}

\newcommand{\bN}{\mathbb{N}}
\newcommand{\bk}{\mathbbm{k}}
\newcommand{\bK}{\mathbb{K}}
\newcommand{\bT}{\mathbb{T}}

\newcommand{\Diff}[0]{\operatorname{Diff}}

\newcommand{\red}[0]{\operatorname{red}}

\newcommand{\dlt}[0]{\operatorname{dlt}}
\newcommand{\Aut}[0]{\operatorname{Aut}}

\newcommand{\Pic}[0]{\operatorname{Pic}}

\newcommand{\mld}{{\rm mld}}
\newcommand{\mult}{{\rm mult}}
\newcommand{\Supp}{{\rm Supp}}
\newcommand{\Hom}{{\rm Hom}}
\newcommand{\sm}{{\rm sm}}

\newcommand{\ord}{{\rm ord}}
\newcommand{\gr}{{\rm gr}}

\newcommand{\orb}{\mathrm{orb}}

\numberwithin{equation}{section}
\newtheorem{thm}{Theorem}[section]

\newtheorem{lem}[thm]{Lemma}
\newtheorem{cor}[thm]{Corollary}

\newtheorem{prop}[thm]{Proposition}

\theoremstyle{definition}
\newtheorem{defn}[thm]{Definition}

\newtheorem{example}[thm]{Example}

\newtheorem{rem}[thm]{Remark}

\newtheorem{defn-thm}[thm]{Definition--Theorem}  
\newtheorem{defn-prop}[thm]{Definition--Proposition}  
\newtheorem{defn-lem}[thm]{Definition--Lemma}  

\theoremstyle{remark}

\newcommand{\Proj}{{\rm Proj}}

\newcommand{\DivVal}{{\rm DivVal}}

\newcommand{\cD}{{\mathcal{D}}}
\newcommand{\cE}{{\mathcal{E}}}
\newcommand{\cG}{{\mathcal{G}}}
\newcommand{\cO}{{\mathcal{O}}}
\newcommand{\cV}{{\mathcal{V}}}
\newcommand{\cP}{{\mathcal{P}}}

\newcommand{\fa}{{\mathfrak{a}}}

\newcommand{\Spec}{{\rm Spec}}

\newcommand{\la}{\lambda}

\newcommand{\bmu}{\bm{\mu}}

\newcommand{\Div}{{\rm Div}}
\newcommand{\Exc}{{\rm Exc}}

\newcommand{\Gal}{{ \rm Gal}}

\newcommand{\PGL}{{ \rm PGL}}

\newcommand{\LC}{{ \rm LC}}

\newcommand{\Src}{{\rm Src}}
\newcommand{\cbir}{{ \overset{\rm cbir}{\sim}}}
\newcommand{\STR}{\overline{\rm ST}_R}

\newcommand{\tX}{\widetilde{X}}
\newcommand{\tD}{\widetilde{D}}
\newcommand{\tcX}{\widetilde{\cX}}
\newcommand{\tcD}{\widetilde{\cD}}
\newcommand{\tcDe}{\Delta_{\widetilde{\cX}}}
\newcommand{\ocX}{\overline{\cX}}
\newcommand{\ocD}{\overline{\cD}}
\newcommand{\ocG}{\overline{\cG}}

\newcommand{\hcD}{\widehat{\cD}}

\newcommand{\vep}{\varepsilon}

\newcommand{\Hodge}{\mathrm{Hodge}}
\newcommand{\CY}{\mathrm{CY}}

\newcommand{\sF}{\mathscr{F}}
\newcommand{\mybar}[1]{\makebox[0pt]{$\phantom{#1}\overline{\phantom{#1}}$}#1}

\makeatletter
\providecommand{\leftsquigarrow}{%
  \mathrel{\mathpalette\reflect@squig\relax}%
}
\newcommand{\reflect@squig}[2]{%
  \reflectbox{$\m@th#1\rightsquigarrow$}%
}
\makeatother

\newcommand{\sS}{\mathscr{S}}

\newcommand{\reg}{\mathrm{reg}}
\newcommand{\cU}{\mathcal{U}}
\newcommand{\tC}{\widetilde{C}}
\newcommand{\tQ}{\widetilde{Q}}

\newcommand{\oX}{\overline{X}}
\newcommand{\oDe}{\overline{\Delta}}
\newcommand{\oD}{\overline{D}}
\newcommand{\oG}{\overline{G}}
\newcommand{\oK}{\overline{K}}
\newcommand{\oE}{\overline{E}}
\newcommand{\oF}{\overline{F}}

\newcommand{\oeta}{\overline{\eta}}
\newcommand{\coreg}{\mathrm{coreg}}
\newcommand{\WDiv}{\mathrm{WDiv}}
\newcommand{\GIT}{\mathrm{GIT}}

\newcommand{\thickslash}{\mathbin{\!\!\pmb{\fatslash}}}
\newcommand{\et}{\mathrm{\acute{e}t}}

\begin{document}

\title{Moduli of boundary polarized Calabi--Yau pairs}

\author[K.\ Ascher]{Kenneth Ascher}
\address{Department of Mathematics, University of California, Irvine, CA, 92697, USA}
\email{kascher@uci.edu}
\author[D.\ Bejleri]{Dori Bejleri}
\address{Harvard University, One Oxford Street, Cambridge, MA 02138, USA}
\email{bejleri@math.harvard.edu}
\author[H.\ Blum]{Harold Blum}
\address{Department of Mathematics,  University of Utah, Salt Lake City, UT 84112, USA}
\email{blum@math.utah.edu}
\author[K.\ DeVleming]{Kristin DeVleming}
\address{Department of Mathematics and Statistics,
University of Massachusetts, Amherst, MA 01003, USA}
\email{kdevleming@umass.edu}
\author[G.\ Inchiostro]{Giovanni Inchiostro}
\address{University of Washington, Department of Mathematics, Box 354350, Seattle, WA 98195, USA}
\email{ginchios@uw.edu}
\author[Y.\ Liu]{Yuchen Liu}
\address{Department of Mathematics, Northwestern University, Evanston, IL 60208, USA}
\email{yuchenl@northwestern.edu}
\author[X.\ Wang]{Xiaowei Wang}
\address{Department of Mathematics and Computer Science, Rutgers University, Newark, NJ 07102, USA}
\email{xiaowwan@rutgers.edu}


\begin{abstract}
We develop the moduli theory of boundary polarized CY pairs, which are slc Calabi--Yau pairs $(X,D)$ such that $D$ is ample.
The motivation for studying this moduli problem is to construct a moduli space at the Calabi--Yau wall interpolating between certain K-moduli and KSBA moduli spaces.    
We prove that the moduli stack of boundary polarized CY pairs is S-complete, $\Theta$-reductive, and satisfies the existence part of the valuative criterion for properness, which are steps towards constructing a proper moduli space.
A key obstacle in this theory is that the irreducible components of the moduli stack are not in general of finite type. Despite this issue, in the case of pairs $(X,D)$ where $X$ is a degeneration of $\bP^2$, we construct a projective moduli space on which the Hodge line bundle is ample.
As a consequence, we complete the proof of 
a conjecture of Prokhorov and Shokurov in relative dimension two.
\end{abstract}

\maketitle

\setcounter{tocdepth}{1}

\tableofcontents

\newpage

\section{Introduction}

Recently, there has been tremendous progress in constructing compact moduli spaces of higher dimensional varieties. 
The main classes of varieties to consider are canonically polarized varieties, 
Calabi--Yau varieties, and Fano varieties,
which, from the viewpoint of birational geometry, are the three  building blocks of algebraic varieties.
For canonically polarized varieties, the  Koll\'ar, Shepherd-Barron, and Alexeev (KSBA)  approach  produces
a projective moduli space parametrizing canonically polarized varieties with slc singularities \cite{KSB88, Ale94, Kol08,BCHM10,AH11, HX13, KP17, HMX18, Fuj18, KolKflat}.
For Fano varieties, an approach using K-stability  produces a projective moduli space parametrizing K-polystable Fano varieties \cite{Jia20, CP21, LWX21, BX19, ABHLX20, Xu20, BLX19, XZ20, XZ21, BHLLX21, LXZ21}.
See \cite{KolNewBook,Xu21,Xu23} for surveys on these topics.

In the remaining case, the problem is to construct compact moduli spaces of Calabi--Yau varieties polarized by an ample line bundle.
While there is no unified solution to this problem,  
 techniques including  Hodge theory, KSBA moduli, Fano K-moduli,  mirror symmetry,  geometric invariant theory, and differential geometry provide various approaches to the problem and are successful in specific examples. 
In this paper, we introduce a new approach that bridges the KSBA  and Fano K-moduli theories and apply it to certain examples.

The main objects we consider  are \emph{boundary polarized CY pairs},
which are semi-log-canonical (slc) pairs $(X,D)$ such that $X$ is projective, $D$ is an ample $\bQ$-Cartier $\bQ$-divisor, and $K_{X}+D\sim_{\bQ}0$.\footnote{More generally, we will also consider pairs of the form $(X,\Delta+D)$; see Definition \ref{d:bpcy}.}
Note that  $-K_X$ is ample and so $X$ is an slc Fano variety. 
Many important classes of Calabi--Yau varieties are naturally associated to such pairs.
\begin{enumerate}
	\item A smooth hypersurface $H \subset \bP^n$ of degree $n+1$ gives such a pair $(\bP^n,H)$.
	\item A general  K3 surface of degree two is the double cover of such a pair $(\bP^2, \frac{1}{2} C_6)$, where $C_6$ is a sextic plane curve. More generally, a hyperelliptic K3 surface is a double cover of a pair of the form $(X,\tfrac{1}{2}C)$, where $X$ is a del Pezzo surface.
 	\item The projective cone $X:= C_p(Y,L)$ over a Calabi--Yau variety  $Y$ with respect to an ample line bundle $L$  with divisor $Y \cong D\subset X$ at infinity is such a pair.
\end{enumerate}
Via the construction in (3), the moduli theory of Calabi--Yau varieties polarized by an ample line bundle is contained in the moduli theory of boundary polarized CY pairs. This relationship will be studied in detail in future work.

In certain situations, both KSBA and K-moduli theories can be used to construct compactifications of the moduli of boundary polarized CY pairs.
Indeed, if $(X,D)$ is a boundary polarized CY pair that is klt (for example, the klt condition occurs in (2)), then
\begin{itemize}
	\item   $(X,(1-\vep)D)$ is a K-semistable log Fano pair, 
	\item $(X,D)$ is a klt CY pair, and
	\item $(X,(1+\vep)D)$ is a KSBA stable canonically polarized pair, 
\end{itemize}
for $0<\vep \ll1$. 
Therefore,  Fano K-moduli  and KSBA machinery combined with boundedness results (needed to deal with an $\vep $ in the coefficient)
produce projective  moduli spaces of such pairs
 and  we denote them by $M^{\rm K}$ and $M^{\rm KSBA}$.
Such moduli spaces were constructed in \cite{ADL19,Zhou} and \cite{KX20,Bir20,Bir22}, respectively.

In the special case when $X:=\bP^2$ and $D:= \frac{3}{d} C$, where $C\subset \bP^2$ is a smooth degree $d$ curve, the above K-moduli  and KSBA compactifications were studied by Ascher, DeVleming, and Liu   in \cite{ADL19} and by Hacking in \cite{Hac04}, respectively. 
Furthermore, in this case, the first three authors conjectured that at $\vep=0$
there exists a moduli space of  CY pairs that is the common ample model of the Hodge line bundles on $M^{\rm K}$ and $M^{\rm KSBA}$.

One reason to believe the existence of this  moduli space of CY pairs  comes from K-stability for polarized varieties.
For Fano varieties, K-stability produces the  K-moduli space of Fano varieties.
For canonically polarized varieties, KSBA stability agrees with K-stability by \cite{Oda12,Oda13b}.
Hence, the moduli spaces $M^{\rm K}$ and $M^{\rm KSBA}$ are both K-moduli spaces and there ideally would be a K-moduli space $M^{\CY}$ 
at $\vep=0$ parametrizing certain CY pairs.
Such a space would interpolate between two very different geometries: log general type and Fano.

To determine what the Calabi--Yau moduli  space should parametrize,  note that the  Fano K-moduli space $M^{\rm K}$ is a good moduli space as defined in \cite{Alp13} and  parametrizes S-equivalence classes of K-semistable log Fano pairs.
 Here, two K-semistable log Fano pairs   are {S-equivalent} if they degenerate to a common K-semistable log Fano pair via  test configurations.
By \cite{Oda13Slope,Oda13b}, a  polarized K-trivial variety is K-semistable if and only if it is slc. 
Thus, when the Calabi--Yau moduli space exists, it  should  parametrize S-equivalence classes of (slc) boundary polarized CY pairs.

The moduli theory of boundary polarized CY pairs differs from both the KSBA and Fano K-moduli theories:
\begin{itemize}
\item  Boundary polarized CY pairs can have non-discrete automorphism groups and can admit isotrivial degenerations to other such pairs. In particular, the moduli stack will be algebraic, but not in general Deligne--Mumford. 
\item The irreducible components of the moduli stack are often not quasi-compact and hence not of finite type.
See Example \ref{e:unbounded} for simple examples.
\end{itemize}
While the first issue does not occur for KSBA stable varieties, it appears for K-semistable Fano varieties and requires  technical machinery to understand and construct a good moduli space in the sense of \cite{Alp13}.
The second issue differs significantly from both the KSBA and Fano K-moduli theories and will be a  crucial obstacle moving forward.

\subsection{Higher dimensional results}
Before constructing the moduli space of CY pairs in specific examples, we will need to build up considerable  theory.  As a first step, we construct a moduli stack parametrizing boundary polarized CY pairs with fixed invariants (see Definition \ref{defn:modulistack}).
This is carried out in Section \ref{s:stack} and uses many of the same techniques in the construction of the moduli stack of KSBA stable pairs in \cite{KolNewBook}.
As mentioned above, the irreducible components of the stack are not in general of finite type.

In hopes of constructing a good moduli space for the stack, 
we recall the foundational result of \cite{AHLH18}, which states 
that a finite type algebraic stack with affine diagonal admits a separated good moduli space if and only if it  is S-complete and $\Theta$-reductive \cite{AHLH18}. 
The latter two notions are defined in Sections \ref{ss:defScom} and \ref{ss:deftheta}
and may be thought of as valuative criteria over two-dimensional bases and require extensions of families over punctured surfaces.  The first main theorem of this paper shows  that almost all of these conditions are satisfied for the moduli stack of boundary polarized CY pairs.

\begin{thm}\label{t:main1}
The  moduli stack of boundary polarized CY pairs  is algebraic, locally of finite type, and has affine diagonal. 
Moreover, the stack is
\begin{enumerate}			
\item S-complete,
\item $\Theta$-reductive, and 
\item satisfies the existence part of the valuative criterion for properness
\end{enumerate}
all with respect to DVRs essentially of finite type over the base field.
\end{thm}

In a future work we will prove an analogue of this result for the moduli stack of Calabi--Yau varieties polarized by an ample $\bQ$-line bundle.

While Theorem \ref{t:main1} may seem abstract, it implies the following concrete  results regarding degenerations of  boundary polarized CY pairs over the germ of a smooth curve $0\in C$.
\begin{itemize}
	\item {\sc Properness}:
 If
	$
	(X^\circ,D^\circ)\to C^\circ
	$
is a family of boundary polarized CY pairs over $C^\circ := C\setminus 0$, then it admits an extension to a family over $C$, possibly after a finite ramified base change.

\item {\sc $\Theta$-reductivity}: 
If $(X,D)\to C$ is a family of boundary polarized CY pairs, then any weakly special test configuration of the generic fiber extends to the special fiber. 

\item {\sc S-completeness}: S-completeness implies the next corollary, which is a Calabi--Yau analogue
  of the main theorem of  \cite{BX19}.
\end{itemize}

\begin{cor}\label{c:sequiv}
Let 
$(X,D)\to C$
and
$(X',D')\to C$ 
be families of boundary polarized CY pairs over the germ of a smooth pointed curve $0 \in C$.
If there is an isomorphism 
\[
(X,D)\vert_{C^\circ} \cong (X',D')\vert_{C^\circ}
,\]
 over $C^\circ:=C\setminus 0$, then $(X_0,D_0)$ and $(X'_0,D'_0)$ are S-equivalent.
\end{cor}

In the above corollary, \emph{S-equivalence} means that 
 $(X_0,D_0)$ and $(X'_0,D'_0)$  admit degenerations via test configurations to a common boundary polarized CY pair; see Definition \ref{d:Sequiv}. 
 Roughly, the corollary states that any moduli space parametrizing
 S-equivalence classes of boundary polarized CY pairs is separated.

The proofs  of Theorem \ref{t:main1}.1 and \ref{t:main1}.2 
are inspired by the proofs of S-completeness and $\Theta$-reductivity for the moduli stack of  K-semistable Fano varieties in \cite{ABHLX20}, which builds on  \cite{BX19,LWX21}. 
Additional subtleties arise in the proof in the boundary polarized CY case due to the presence of slc, rather than klt, singularities. 
To deal with this obstacle, the proof uses the MMP for lc pairs \cite{HX13} and Koll\'ar's gluing theory \cite{Kol13}.
Theorem \ref{t:main1}.3 is related to, but different from, extension results for CY pairs in the literature \cite{Fuj11, KLSV, KX20}. 
Its proof combines \cite{KX20} with a perturbation argument. In particular, we obtain a stronger version of the existence part of the valuative criterion than one expects for general Artin stacks (see Remark \ref{r:rootstack}).

Note that Theorem \ref{t:main1}
may not be immediately applied to construct a moduli space of boundary polarized CY pairs, 
since the irreducible components of the moduli stack  will not always be of finite type.
Significant additional work is needed to construct a moduli space in specific examples.

\subsection{Moduli of plane curve pairs}

Let $\cP_d$ denote the stack parametrizing families of pairs $(\bP^2, \tfrac{3}{d}C)$, where $C\subset\bP^2$ is a smooth degree $d\geq 3$  curve. We consider the following three compactifications of the stack.

\begin{enumerate}

	\item[(i)] {\sc Calabi--Yau}: 
	 $\cP_{d}^{\rm CY}$ denotes the stack theoretic closure of $\cP_d$ in the moduli  stack of boundary polarized CY pairs.
	
	\item[(ii)] {\sc K-moduli}:
	 $\cP_d^{\rm K}\subset \cP_d^{\rm CY}$ denotes the open locus consisting of pairs $(X,D)$ such that $(X,(1-\vep )D)$ is K-semistable for $0<\vep  \ll1$. 
	
	\item[(iii)] {\sc KSBA}:
		$\cP_d^{\rm H}\subset \cP_d^{\rm CY}$ denotes the open locus  consisting of pairs  $(X,D)$ such that $(X,(1+\vep )D)$ is KSBA stable (i.e. slc)  for $0<\vep  \ll1$. 	 The notation is a reference to Hacking \cite{Hac04}.
\end{enumerate}
Concretely, the points of $\cP_d^{\CY}$ correspond to boundary polarized CY pairs that are degenerations of families of pairs $(\bP^2, \tfrac{3}{d}C)$, where $C$ is a smooth degree $d$ plane curve.
By \cite{ADL19} and \cite{Hac04}, the K-moduli and KSBA stacks  have associated good  moduli spaces $\phi_{\rm K}:\cP_d^{\rm K} \to P_d^{\rm K}$ and $\phi_{\rm H}:\cP_d^{\rm H}\to P_d^{\rm H}$.
The following theorem produces a moduli space of Calabi--Yau pairs as   conjectured in  \cite[Conj. 1.8]{ADL19}.

	\begin{thm} \label{t:main2}
For any $d \ge 3$, there exists a seminormal 
 projective variety $P_{d}^{\CY}$ and a surjective morphism 
\[
\phi:
(\cP_{d}^{\CY})^{\rm sn} \to P_{d}^{\CY}
,\]
where  $\rm{sn}$ denotes seminormalization, such that the following hold: 
\begin{enumerate}
	\item The map $\phi$ is universal among maps from $(\cP_{d}^{\CY})^{\rm sn}$ to algebraic spaces.
	\item  The moduli space $P_d^{\CY}$ parametrizes S-equivalence classes of pairs in $\cP_{d}^{\CY}$.
	\item There exists a wall crossing diagram 
	\[
	\begin{tikzcd}
\cP_{d}^{\rm K} \arrow[r,hook]\arrow[d, "\phi_{\rm K}",swap] & (\cP_{d}^{\CY})^{\rm sn}\arrow[d,"\phi"] & (\cP_d^{\rm H})^{\rm sn}\arrow[l,hook'] \arrow[d, "\phi_{\rm H}^{\rm sn}"]\\
		P_{d}^{\rm K} \arrow[r] & P_{d}^{\CY} & (P_d^{\rm H})^{\rm sn}\arrow[l] 
	\end{tikzcd}
	\]
    where the arrows in the first row are open immersions and in the second are projective birational morphisms for $d > 3$.\footnote{When $d = 3$, $\cP_3^{\rm H}$ and $P_3^{\rm H}$ are empty and the bottom left arrow is an isomorphism.}

	\item The  Hodge line  bundle on $P_{d}^{\CY}$ is ample.
\end{enumerate}
	\end{thm}

The geometry of $\phi$ is governed by the divisibility of the degree. 
Indeed,  when $3\nmid d$, the geometry is well behaved: $\cP_d^{\rm CY}$ is a smooth algebraic stack of finite type, $P_d^{\CY}$ is normal, and $\phi$ is a good moduli space morphism.
When $3 \mid d$,  the moduli problem is unbounded: $\cP_d^{\CY}$  is not finite type and $\phi$ fails to be a good moduli space morphism by Remark \ref{r:notgms}. 
To remedy this failure, we introduce the notion of an \emph{asymptotically good moduli space} (Definition \ref{d:asgm}), which allows for unboundedness.

The moduli space $P_d^{\rm CY}$ has similarities to the Baily--Borel (BB) compactification of the moduli space of polarized K3 surfaces. 
Indeed the Hodge line bundle is ample on $P_d^{\CY}$, the boundary parametrizes lower dimensional CY pairs, and the compactification is minimal in that both $P^{\rm K}_d$ and $P^{\rm H}_d$ map down to it. 
In fact, when $d=4,6$,  this heuristic can be made precise: $P_4^{\rm \CY}$ is  isomorphic to the BB compactification of the moduli space of ADE K3 surfaces of degree 4 with a $\mathbb{Z}/4\mathbb{Z}$ symmetry constructed by Kond\=o \cite{Kondo}, and  $P_6^{\CY}$ is   isomorphic to the BB compactification  of the  moduli space of primitively polarized K3 surfaces of degree two; see Section \ref{sec:K3}.
Thus Theorem \ref{t:main2} gives a moduli interpretation to the previous two BB compactifications.
It would be interesting  to find  explicit descriptions of $P_d^{\CY}$ for additional degrees and relate them to other moduli spaces of polarized Calabi--Yau varieties.

\subsubsection{Sketch of proof}
The proof of Theorem \ref{t:main2} combines the abstract theory developed in Theorem \ref{t:main1}  and an in depth analysis of degenerations of plane curve pairs.

To construct the moduli space,  we first analyze the case when $3\nmid d$. 
In this case, $\cP_d^{\rm CY}$ is of finite type 
and we can use   Theorem \ref{t:main1} to construct $P_{d}^{\CY}$.
The case when $3 \mid d$ is significantly more involved since  $\cP_d^{\rm CY}$ will not be of finite type. 
In this case, we write $\cP_d^{\CY}$
as a union of an ascending chain of finite type open substacks
\[
\cP_{d,1}^{\CY}\subset \cP_{d,2}^{\CY}\subset
\cP_{d,3}^{\CY}\subset \cdots 
,\]
where  $\cP_{d,m}^{\CY}\subset
\cP_d^{\rm CY}$ is the open substack consisting of pairs $(X,D)$ such that the Cartier index of $K_X$ at each point is at most $m$.
Next, we prove that  $\cP_{d,m}^{\CY}$ is  S-complete and $\Theta$-reductive by an involved argument using Theorem \ref{t:main1}, a reduction argument to the case when $d=3$, and the geometry of twisted stacky curves. 
Thus  $\cP_{d,m}^{\CY}$ admits a separated good moduli space  $ P_{d,m}^{\CY}$.
The latter result is somewhat miraculous, since an arbitrary finite type open substack of $\cP_{d}^{\rm CY}$ will not admit a good moduli space!

Next, the inclusions  of the stacks induce a sequence of morphisms  
\[
P_{d,1}^{\rm \CY }\to P_{d,2}^{\rm CY} \to P_{d,3}^{\rm CY} \to \cdots
.\]
We prove that the morphisms are eventually isomorphisms after seminormalization
and deduce that 
Theorem \ref{t:main2}.1-.3 holds with   $P_d^{\CY}:=(P_{d,m}^{\CY})^{\rm sn}$ for $m$ sufficiently large.

The remaining step is to prove the ampleness of  the Hodge line bundle on  $P_{d}^{\rm CY}$. 
Results that rely on Hodge theory  \cite{Amb05,Kol07,FG14moduli} imply the Hodge line bundle is nef and, additionally, big  on the locus of $P_{d}^{\CY}$ where the sources of the pairs (see Definition \ref{d:source})  have maximal variation.

We perform  a detailed analysis of the S-equivalence classes of pairs in $\cP_d^{\rm CY}$ to prove that the sources have maximal variation on all of $P_d^{\CY}$. 
With this done, the proof of Theorem \ref{t:main2} is complete.

As part of the proof of the ampleness result, 
we prove a result of independent interest on the existence of generically finite covers of certain stacks by schemes (Theorem \ref{p:finitecoverbyscheme}). This generalizes a well-known result about Deligne-Mumford stacks \cite[Thm. 6.6]{LMB00} to a larger class of stacks.

\medskip

\subsection{Application to the b-semiampleness conjecture}

The canonical bundle formula is a technical, but powerful, tool that has many applications in the MMP.
The formula, which was developed in the work of \cite{Kod,Kaw95,FM,Amb04}, associates to an lc-trivial fibration $(X,D)\to Y$ a $\bQ$-linear equivalence
\[
K_X+ D \sim_{\bQ}  f^*(K_Y+D_Y+M_Y),
\]
where $D_Y$ is the discriminant divisor, which measures the singularities of the fibers, and $M_Y$ is the moduli divisor, 
which measures the variation of the fibers.
Furthermore, the moduli divisor gives rise to a b-divisor ${\bf M}$;
see Section \ref{s:canonicalbundle} for details.

The main open problem on this topic is a conjecture of Prokhorov and Shokurov stating that the moduli divisor is b-semiample \cite[Conj. 7.13]{PS09}. 
The conjecture is known in  only a few cases, 
which include
when the base is a curve \cite{Amb04,Amb05,Flo14a}
and when the geometric generic fiber is a curve \cite{PS09}
or a surface not isomorphic to $\bP^2$  \cite{Fuj03,Fil18}. 

 Shokurov has suggested the following strategy to prove the conjecture: construct a compact moduli space parametrizing certain equivalence classes of  CY pairs such that the Hodge line bundle is ample on the moduli space and verify that the moduli b-divisor is the pullback of this line bundle.
In the case when the geometric generic fiber of $X\to Y$ is $\bP^2$, the moduli space $P_{d}^{\rm CY}$  with its ample Hodge line bundle realizes this vision. 
Combining the $\bP^2$ case with \cite{PS09,Fuj03,Fil18} leads to the following result.

\begin{thm}\label{t:effectiveb}
If $(X,D)\to Y$ is an lc-trivial fibration such that the generic fiber $(X_\eta,D_\eta)$ is a projective lc pair of dimension $\leq 2$, then  ${\bf M}$ is b-semiample.

 If additionally $D_\eta\neq 0$, then there exists an integer $I>0$ depending only on the coefficient set of $D_\eta$ such that $I{\bf M}$ is b-free.
\end{thm}

\subsection{Additional examples}
It is natural to ask whether or not Theorem \ref{t:main2} generalizes to other moduli problems, not just the case of degenerations of pairs $(\bP^2,\tfrac{3}{d}C)$, where $C$ is a smooth degree $d$ plane curve.
In Section \ref{sec:K3},  we give a few examples where boundedness holds (as in the case of $\cP_d^{\CY}$ with $3\nmid d$) and use Theorem \ref{t:main1}  to construct the moduli space of boundary polarized CY pairs. These examples are degenerations of pairs $(\bP^1, \tfrac{2}{d} \sum_{i=1}^d p_i)$, where the $p_i$ are distinct points, $ (\bP^1\times \bP^1, \tfrac{2}{d}C)$, where $d\geq 3$ is an odd integer and $C$ is a smooth bidegree $(d,d)$ curve, and $(\bP^3, \frac{4}{5}S)$, where $S$ is a smooth quintic surface.
We expect that Theorem \ref{t:main2} likely extends to compactifying boundary polarized CY pairs $(X,D)$, where $X$ is a smooth del Pezzo surface.

In higher dimensions,  such a moduli space  will not always exist. 
Indeed, in a future paper, we will explain
that there does not exist a moduli space parametrizing S-equivalence classes of boundary polarized CY pairs that are degenerations of $(C_p(A),H)$,
where $C_p(A)$ is a projective cone over a principally polarized abelian surface and $H$ is the divisor at infinity.
It would be interesting to find geometric conditions that guarantee the existence of such a moduli space.

\subsection{Appendices}
In two appendices to the paper, we prove additional results that are independent of the main theorems stated above. 

In  Appendix \ref{a:specialdegenerations}, we characterize all special degenerations of $\bP^2$. 
This gives a complete answer to a problem in the Fano K-stability literature \cite[Prob. 3.4]{AIM20}. The proof relies on a technical result from Section \ref{ss:1comp} to show that all  special degenerations of $\bP^2$ are induced by lc places of a 1-complement. 
This reduces the classification to \cite[\S 6]{LXZ21} and an additional computation.

Appendix \ref{a:coreg0}  concerns the index of CY pairs and the boundedness of complements. 
A celebrated result of Birkar states that a klt Fano variety $X$ of dimension $n$ admits an  $N(n)$-complement, i.e. a $\bQ$-divisor $D$ on $X$ such that $(X,D)$ is lc and $N(n)(K_X+D)\sim0$ \cite{Bir19}. 
Here,  $N(n)$ is a positive integer depending only on $n$.

In \cite{FFMP22},  the authors prove that
 if $X$ additionally admits a complement of coregularity 0, then $N(n)$ can be replaced with 2.
 Here, the coregularity of a Fano variety with complement  is the dimension of the minimal lc center on a dlt modfication; see Definition \ref{d:coregularity}.
 Their  argument uses  ideas from \cite{Bir19} with significant upgrades to utilize the coregularity assumption.
 Using a degeneration argument, we give a new proof of their result, which also holds for lc Fano varieties. 

Using a similar degeneration argument, we also give a new proof of a result on the index of coregularity 0  CY pairs \cite[Thm. 1]{FMM22}.

\subsection{Relation to other work}\label{sec:related}
In constructing compact moduli spaces of Calabi--Yau varieties, a fundamental issue is that a  family of polarized smooth Calabi--Yau varieties over a punctured curve can admit infinitely  many slc Calabi--Yau  degenerations after finite base changes.
In order to construct canonical limits of Calabi--Yau varieties, an effective technique is to use K-moduli or KSBA theory. 
As mentioned above, if $(X,D)$ is a klt boundary polarized CY pair, then $(X, (1-\varepsilon)D)$ is a K-semistable log Fano pair, and K-stability machinery can be applied to construct a projective moduli space \cite{ADL19, Zhou}. 
In the KSBA setting, if  $(X,D)$ is a klt CY pair, then the standard technique is to consider the pair $(X,D + \varepsilon H)$, where $H$ is an ample divisor, possibly equal to $D$.
Since the perturbed pair is KSBA stable, KSBA theory gives an approach to compactifying the moduli space of such pairs.
In \cite{KX20}, Koll\'ar and Xu construct the irreducible components of the moduli space and prove that the components are projective schemes. 
In \cite{Bir20,Bir22}, Birkar proves a powerful boundedness result for such pairs and uses the boundedness  to give a more direct construction of the moduli space using KSBA theory.
This type of moduli space was notably constructed and described for abelian varieties in \cite{Ale02} and K3 surfaces in \cite{AE21}.
See 
 \cite{Hac04, Laz16,  AB19, AET, Inc20, AE21, DH21, GKS21, Gol21, MS21, DeV19, AB22, ABE22, Sch23}
 and
 \cite{GMGS18, ADL20, ADL21, Zha22, Pap22, PSW23} for a partial list of additional explicit examples of these KSBA and K-moduli spaces.
 The approach of this paper is different in that we do not use KSBA or K-moduli theory to construct the moduli space.

In this paper, we do not aim to choose a canonical limit of a family of boundary polarized CY pairs over a punctured disc, but instead parametrize S-equivalence classes of all (slc) boundary polarized CY pairs.
The approach is philosophically similar to the notion of a \emph{galaxy} in \cite{Oda22}, which takes a family of Calabi--Yau varieties over a punctured disc $X^\circ \to \mathbb{D}^\circ$ and outputs the projective limit of all extensions $X\to \mathbb{D}$ to a flat family of slc Calabi--Yau varieties after finite base change. 
Since the aim of this paper is to construct moduli spaces, we consider S-equivalence classes rather than galaxies.

As mentioned earlier, the Hodge line bundle on $P_d^{\CY}$ is ample and, hence, the compactification is an analogue of the BB compactification of the moduli space of K3 surfaces.
A number of authors, including Shokurov, Odaka \cite{Oda22}, and Laza \cite{Laz23}, have conjectured that there exists a compactification of the moduli space of smooth Calabi--Yau varieties on which the Hodge line bundle is ample and the boundary parametrizes lower dimensional Calabi--Yau varieties. 
This conjecture is closely related to the b-semiampleness conjecture of Prokhorov and Shokurov \cite{PS09}. There are also various approaches in the literature for constructing the ample model of the (augmented) Hodge line bundle, see e.g. \cite{BBT, GGLR}.

Finally, note that $P_d^{\CY}$ sits   between $P_d^{\rm K}$ and $P_d^{\rm H}$, which are given by perturbing the coefficient of the boundary divisor, and gives a ``wall crossing'' between the two moduli spaces. Constructing this wall crossing is one motivation for the conjectural description of $\cP_{d}^{\CY}$ present in \cite{ADL19}. The idea of wall-crossing for moduli spaces of pairs originated in the case of weighted stable pointed curves \cite{Has03} and has been generalized to higher dimensions for both the moduli of K-stable and KSBA stable varieties in \cite{ADL19, Zhou} and \cite{ABIP}, respectively.

\subsection{Structure of the paper}
The first half of the paper builds general theory on the moduli of boundary polarized CY pairs.
 In Section \ref{s:prelim}, we introduce necessary preliminaries and the definition of boundary polarized CY pairs.
 In Section \ref{s:stack}, we construct the moduli stack of boundary polarized CY pairs. 
 Section \ref{s:tc} discusses relevant results on test configurations of such pairs.
 In Sections \ref{s:Scompleteness}, \ref{s:Thetared}, and \ref{s:properness}, we prove that the moduli stack of boundary polarized CY pairs is S-complete, $\Theta$-reductive, and satisfies the existence part of the valuative
criterion of properness, respectively. 
Section \ref{s:sources} discusses the source and regularity of a CY pair and constructs a good moduli space for pairs of regularity at most 0.

The  second half of the paper studies compactifications of the moduli of plane curve pairs.
After introducing the relevant stacks in Section \ref{s:planecurves1}, 
we explicitly describe the S-equivalence classes of pairs in $\cP_d^{\CY}$ in Section \ref{s:sequivcurves}.
In Section \ref{s:twistedvarieties}, we apply the theory of twisted varieties to study local properties of the moduli stack.
We construct the  good moduli  spaces $P_{d,m}^{\CY}$  in Section \ref{s:gmsforcurves} and the asymptotically good moduli space $P_d^{\CY}$ in Section \ref{s:agm}.
In Section \ref{s:hodgelinebundle}, we prove the ampleness of the Hodge line bundle on $P_d^{\CY}$
and deduce the b-semiampleness conjecture in relative dimension 2. 
In Section \ref{s:proofofmainthms}, we combine the previous results to obtain the theorems stated in the introduction.
 In Section \ref{sec:K3}, we give explicit descriptions of $P_d^{\CY}$ when $3 \leq d \leq 6$.
 
The paper has two appendices. 
In Appendix \ref{a:specialdegenerations}, we characterize special
degenerations of $\bP^2$. 
In Appendix \ref{a:coreg0}, we prove results on coregularity 0 pairs.

\subsection*{Acknowledgments}

We thank the American Institute of Mathematics (AIM) for hosting the 2021 workshop \textit{Moduli problems beyond geometric invariant theory} at which this project began
and Jarod Alper, Daniel Halpern-Leistner, and Filippo Viviani for organizing the workshop.
We are grateful to Valery Alexeev, Jarod Alper, Elden Elmanto, Philip Engel, Stefano Filipazzi, Samuel Grushevsky, Daniel Halpern-Leistner, Mattias Jonsson, J\'anos Koll\'ar, S\'andor Kov\'acs, Aaron Landesman, Radu Laza, Jihao Liu, Siddharth Mathur, Joaqu\'{i}n Moraga, Yuji Odaka, Vyacheslav Shokurov, Chenyang Xu, Chuyu Zhou, and Ziquan Zhuang for helpful conversations and comments.

KA was partially supported by NSF grant DMS-2140781 (formerly DMS-2001408).
DB was partially supported by NSF Grant DMS-1803124.
HB is partially supported by NSF Grant DMS-2200690.  
KD is partially supported by NSF Grant DMS-2302163.
YL is partially supported by NSF Grant DMS-2148266 (formerly DMS-2001317), NSF CAREER Grant DMS-2237139, and an Alfred P. Sloan research fellowship. 
XW is partially supported by NSF Grant DMS-1609335 and a Collaboration Grants for Mathematicians from Simons Foundation: award 631318. KA, KD, YL, and XW were partially supported by AIM as part of the AIM SQuaREs program.

\section{Preliminaries}\label{s:prelim}

\subsection{Conventions}
Throughout, all schemes are defined over an algebraically closed field $\bk$ of characteristic 0, which is our ground field. 
A scheme  $X$ is \emph{demi-normal} if it is $S_2$ and its codimension 1 points are either regular or nodes \cite[Def. 5.1]{Kol13}. We denote by $\bN$ the set of all non-negative integers.

\subsection{Pairs}
A \emph{pair} $(X,\Delta)$ consists of a demi-normal scheme $X$ that is essentially of finite type over a field $\bK$ containing $\bk$ and an effective $\bQ$-divisor $\Delta$ on $X$ such that  $\Supp(\Delta)$ does not contain a codimension 1 singular point of $X$ and
 $K_{X}+\Delta$ is $\bQ$-Cartier.
The pair $(X,\Delta)$ is  \emph{projective} if $X$ is projective over $\bK$.
For the notions of singularities of pairs that assume $X$ is normal, including \emph{lc}, \emph{klt}, \emph{plt}, and \emph{dlt},  see \cite[Def. 2.8]{Kol13}. Unless stated otherwise, we assume that $X$ is geometrically connected.

The \emph{normalization} of a pair $(X,\Delta)$ is the possibly disconnected pair
$(\overline{X},\overline{\Delta}+\overline{G})$,
where 
$\pi:\overline{X}\to X$ is the normalization morphism, $\overline{G} \subset \overline{X}$ the conductor divisor on $\overline{X}$, and $\overline{\Delta}$  the divisorial part of $\pi^{-1}(\Delta)$, 
and satisfies
$K_{\overline{X}}+\overline{G}+\overline{\Delta}=\pi^*(K_X+\Delta)$ \cite[5.7]{Kol13}.
This construction induces a generically fixed point free involution $\tau:\oG^n\to \oG^n$  that fixes ${\rm Diff}_{\oG^n}(\oDe)$; see  \cite[5.2 \& 5.11]{Kol13} for details.
A pair $(X,\Delta)$ is \emph{slc} if its normalization $(\overline{X},\overline{\Delta}+\overline{G})$ is \emph{lc} \cite[Def. 5.10]{Kol13}.

\begin{defn}\label{d:FanoCYKSBA}
A projective slc\footnote{Note that this definition  differs from common conventions, which  assume a log Fano pair is klt and a CY pair is lc.} pair  $(X,\Delta)$ is called
\begin{enumerate}
	\item \emph{log Fano} if $-K_X-\Delta$ is ample,
	\item  \emph{CY} if $K_{X}+\Delta \sim_{\bQ} 0$, and 
	\item \emph{canonically polarized} if $K_{X}+\Delta$ is ample.
\end{enumerate}
The \emph{index} of a CY pair $(X,\Delta)$ is the smallest positive integer $N$ such that $N(K_X+\Delta)\sim0$.
\end{defn}

Throughout this paper, we focus on the following special class of CY pairs.

\begin{defn}\label{d:bpcy}
A \emph{boundary polarized CY pair} is a projective  slc pair $(X,\Delta+D)$ such that 
\begin{enumerate}
	\item $\Delta$ and $D$ are effective $\bQ$-divisors,
	\item $K_{X}+\Delta+D\sim_{\bQ} 0$, and
	\item $D$ is $\mathbb{Q}$-Cartier and ample.
\end{enumerate}
Note that (3) is equivalent to the condition that $(X,\Delta)$ is log Fano. 
Hence we call $\Delta$ the  \emph{log Fano boundary} and $D$ the \emph{polarizing boundary} and write them in this order to avoid confusion.
\end{defn}

\begin{rem}
The above definition can be alternatively phrased using the language of \textit{complements}. Let $(X,\Delta)$ be a log Fano pair.  A \emph{complement} $D$ of $(X,\Delta) $ is an effective $\bQ$-divisor such that $(X,\Delta+D)$ is a CY pair or equivalently, a boundary polarized CY pair. A complement is an
\emph{$N$-complement} if $N(K_X+\Delta+D)\sim 0$ or, equivalently, the index of $(X,\Delta+D)$ divides $N$.
\end{rem}

While we are primarily interested in the case when $\Delta=0$, the general case arises when normalizing.
Indeed, if $(X,\Delta+D)$ is a boundary polarized CY pair, 
then
$(\overline{X},\overline{G}+\overline{\Delta}+\overline{D})$
is a possible disconnected lc boundary polarized CY pair, 
where $\overline{G}+\overline{\Delta}$ is the log Fano boundary.
The latter holds, since $\pi$ is finite, and
\[
K_{\overline{X}}+\overline{G}+\overline{\Delta}=\pi^*(K_X+\Delta) \quad \text{ and } \quad K_{\overline{X}}+\overline{G}+\overline{\Delta}+\overline{D}=\pi^*(K_X+\Delta+D).\]

\subsection{Families of pairs}

To define a family of pairs, we  first need the following notion of a family of divisors \cite[Def. 4.68]{KolNewBook}.

\begin{defn}[Relative Mumford divisor]
Let $f:X\to B$ be a flat finite type morphism with $S_2$ fibers of pure dimension $n$. 
A subscheme $D\subset X$ is a \emph{relative Mumford divisor} if there is an open set $U\subset X$ such that 
\begin{enumerate}
\item ${\rm codim}_{X_b}(X_b\setminus U_b) \geq 2$ for each $s\in S$, 
\item $D\vert_U$ is a relative Cartier divisor,
\item $D$ is the closure of $D\vert_U$, and
\item $X_b$ is smooth at the generic points of $D_b$ for every $b\in B$.
\end{enumerate}
\end{defn}
	
If $D\subset X$ is a relative Mumford divisor for $f:X\to B$ and $B'\to B$ is a morphism, then the \emph{divisorial pullback} $D_{B'}$  on $X_{B'}: = X\times_B B'$ is the relative Mumford divisor defined to be the closure of the pullback of $D\vert_U$ to $U_{B'}$. Note that $D_b$ as in (4) always denotes the divisorial pullback.

\begin{defn}\label{def:mumford}	
A \emph{family of slc (resp., lc, klt) pairs} $(X,\Delta) \to B$  over a reduced Noetherian scheme $B$
is a flat finite type morphism $X\to B$ with $S_2$ fibers and a $\bQ$-divisor $\Delta$ on $X$   satisfying
\begin{enumerate}
\item each prime component of $\Delta$ is a relative Mumford divisor,
\item $K_{X/B} +\Delta$ is $\bQ$-Cartier, and
\item  $(X_b,\Delta_b)$ is an slc (resp. lc, klt) pair for all points $b\in B$.
\end{enumerate}
\end{defn}

Using this definition, we can define a family of the classes of pairs appearing in Definition \ref{d:FanoCYKSBA}.

\begin{defn}
A  family of slc pairs $(X,\Delta)\to B$ over a reduced Noetherian scheme $B$ with $X\to B$ projetive is called 
\begin{enumerate}
 \item \emph{a family of KSBA stable pairs} if $K_{X/B} + \Delta$ is relatively ample,
    \item \emph{a family of CY pairs} if $K_{X/B} + \Delta \sim_{\bQ,B} 0$, and
    \item \emph{a family of log Fano pairs} if $-(K_{X/B}+\Delta)$ is relatively ample.
\end{enumerate}
\end{defn}

The following  characterization of a family of slc pairs  follows from \cite[Proof of Thm. 4.54]{KolNewBook}.

\begin{lem}\label{l:slcadj}
Let $f:X \to B$ be a 
finite type morphism to a regular, local, essentially of finite type scheme $0 \in B$.
Fix a $\Q$-divisor $\Delta$ on $X$ and a snc divisor $D_1+ \cdots + D_r$  on $B$ so that $D_1 \cap \cdots \cap D_r = \{0\}$. The following are equivalent:
\begin{enumerate}
\item $(X,\Delta) \to B$ is a family of (possibly disconnected) slc pairs.
\item $(X,\Delta+f^*D_1+ \cdots+ f^*D_r)$ is a (possibly disconnected) slc pair.
\end{enumerate}
\end{lem}

The previous lemma implies that the normalization and restriction of a family of slc pairs over a regular scheme is a family of slc pairs.
	
\begin{lem}\label{l:familyslcnormadj}
Let $(X,\Delta)\to B$ be a family of slc pairs over a regular, essentially of finite type, Noetherian scheme $B$.

\begin{enumerate}
\item 
If $f:(\overline{X},\overline{G}+\overline{\Delta})$ is the normalization of $(X,\Delta)$, 
then  $\overline{f}:(\overline{X},\overline{G}+\overline{\Delta})\to B$ is a family of (possibly disconnected) slc pairs.
	
\item If $\Delta:= \Delta'+\Delta''$, where $\Delta'$ is a reduced divisor and $\Delta''\geq0$, then $(\Delta', {\rm Diff}_{\Delta'}(\Delta'')) \to B$ is a family of (possibly disconnected) slc pairs.
\end{enumerate}
	
	\end{lem}

\begin{proof}
Since both statements are local on $B$, we may assume $B$ is a regular local scheme.
Fix an snc divisor $D:=D_1+ \cdots +D_r$ on $B$ so that $D_1 \cap \cdots \cap D_r $ is the unique closed point.
Since $(X,\Delta)\to B$ is a family of slc pairs, 
  Lemma \ref{l:slcadj} implies $(X,\Delta+f^*D)$ is slc.
Hence $(\overline{X},\overline{G}+\overline{\Delta}+ \overline{f}^*D)$ is lc.  Thus  Lemma \ref{l:slcadj} implies (1) holds.

For (2), note that  
$(\Delta', {\rm Diff}_{\Delta'}(\Delta'' +f^*D))$ is slc  by adjunction  \cite[Thm. 11.17]{KolNewBook} and ${\rm Diff}_{\Delta'}(\Delta'' +f^*D)= {\rm Diff}_{\Delta'}(\Delta'') +f^*D\vert_{\Delta''}$.
Thus (2) holds by Lemma \ref{l:slcadj}.
\end{proof}

\subsubsection{Families of boundary polarized CY pairs}

We now define a family of boundary polarized CY pairs over a reduced Noetherian scheme. A generalization of this  definition to arbitrary base schemes will be given in Section \ref{s:stack}. 
	
\begin{defn}\label{d:Familybasenormal}
A family of \emph{boundary polarized CY pairs} $f:(X,\Delta+D)\to B$ over a reduced  Noetherian scheme $B$ is a family of slc pairs $f:(X,\Delta+D) \to B$ such that 
\begin{enumerate}
\item $f: X\to B$ is projective,
\item $K_{X/B}+\Delta +D\sim_{\bQ,B}0$, and
\item $D$ is $\bQ$-Cartier and ample over $B$.
\end{enumerate}
\end{defn}

We now prove the following uniqueness results for extensions. 
See  \cite{Bou14} and \cite[Cor. 4.3]{Oda12moduli} for similar statements already in the literature.

	\begin{lem}\label{l:CYbirational}
Fix two families of  boundary polarized CY pairs 	
\[(X^1,\Delta^1+D^1) \to B \quad \text{ and } \quad (X^2,\Delta^2+D^2)\to B
\]
over a regular local Noetherian scheme $0 \in B$ with an isomorphism
		$\phi:(X^1 ,\Delta^1+D^1)\vert_{B\setminus 0} \cong (X^2, \Delta^2+D^2)\vert_{B\setminus 0}$ over $B\setminus 0$.
		\begin{enumerate}
			\item If $\dim B=1$ and both $X^i$ are normal, then  $(X^1,\Delta^1 +D^1+X^1_0) $ and $(X^2,\Delta^2 +D^2+X^2_0)$ are crepant birational.
			\item If $\dim B=1$ and $(X^1_0,\Delta^1_0+D^1_0)$ is klt, then  $\phi$ extends to an isomorphism over $B$.
			\item  If $\dim B=2$, then $\phi$ extends to an isomorphism over $B$.
		\end{enumerate}
	\end{lem}

 Recall, two birational pairs $(Z,\Delta_Z)$ and $(Z',\Delta'_Z)$ are \emph{crepant birational} 
 if the discrepancies $a(E,Z,\Delta_Z) = a(E, Z', \Delta_{Z'})$ for every prime divisor $E$ over $Z$ \cite[Def. 2.23.2]{Kol13}.
	
	\begin{proof}
 Let $Y$ denote the main component of the normalization of the graph of $X^1\dashrightarrow X^2$. 
Write $g_i:Y\to X^i$  for the natural map.
We define $\bQ$-divisors $F^1$ and $F^2$ on $Y$ by the  
\[
K_Y+ {g_i}_*^{-1}(\Delta^i+D^i) +Y_0^{\red}  +F^i =  {g_i}^*(K_{X^i}+\Delta^i+D^i +X^i_0)
\]
Note that
$F^i$ is anti-effective, since $(X^i,\Delta^i+D^i+X^i_0)$ is lc.
Therefore both $(X^i,\Delta^i+D^i +X^i_0)$ are weak minimal models of
$(Y, (g_i)_*^{-1}(\Delta^1 +D^1) +Y_0^{\red})$ \cite[Def. 1.19]{Kol13}.
Hence the pairs are crepant birational  by \cite[Prop. 1.21]{Kol13}.

The proof of (2) reduces to the separatedness of the KSBA-moduli space.
Indeed,  since $D^2$ is relatively ample, there exists a $\bQ$-divisor   $0 \leq H^2 \sim_{\bQ}D^2$ such that 
$(X^2_0,\Delta^2_0+D^2_0+H^2_0)$
is slc. 
Let $H^1$ denote the birational transform of $G^2$ on $X^1$. 
Since $H^1_{B\setminus 0 } \sim_{\bQ} D^1_{B\setminus 0}$,
$H^1\sim_{\bQ} D^1+ E^1$ for some $\bQ$-divisor $E^1$ with $\Supp(E^1) \subset X^1_0$.
Using that $X^1_0$ is irreducible by the klt assumption and $X^1_0 \sim 0$, we see $H^1\sim_{\bQ}D^1$ as well.
By our assumption in (2),  there exists $\vep>0$, such that 
$(X^1_0,\Delta^1+D^1_0 +\vep H^1_0 )$ is klt.
Hence
\[
(X^1,\Delta^1+D^2+\vep H^1) \to B
\quad \text{ and } \quad  (X^2,\Delta^2+D^2+\vep H^2)\to B
\]
are families of slc pairs and are isomorphic  over $B\setminus 0$.
Since 
\[
K_{X^i/B}+\Delta^i+D^i+\vep H^i \sim_{\bQ}\vep D^i
,
\]
 which is ample,
 \cite[Prop. 2.50]{KolNewBook} implies $\phi$ extends to an isomorphism.

Finally, to verify (3), fix  a sufficiently divisible  integer $r$ such that
$L^i := -r(K_{X^i/B} +\Delta^i)$ is a relative ample Cartier divisor over $B$ for $i=1,2$.
Since  $(X^1,L^1) \to B$ and $(X^2,L^2)\to B$ are isomorphic over $B\setminus 0$,  \cite[Lem. 2.16]{ABHLX20} implies (3) holds.
\end{proof}

\begin{lem}\label{l:Ncompspecialize}
Let $(X,\Delta+D) \to B$ be a family of  boundary polarized CY pairs over a regular  local scheme $0\in B$ with generic point $\eta\in B$.

If $N(K_{X_\eta}+\Delta_\eta+D_\eta) \sim 0$ for some integer $N$, then $N(K_{X/B}+\Delta+D) \sim 0$.
\end{lem}
	
\begin{proof}
First, assume $B$ is the spectrum of a DVR.
Let $X_{0,1},\ldots, X_{0,m}$ denote the irreducible components of $X_0$. 
Since $(K_{X/B}+\Delta+ D) \vert_{X_\eta} \sim 0$, 
$K_{X/B}+\Delta+D \sim \sum a_i X_{0,i}$ for some $a_i \in \bZ$.
After twisting by the principal divisor $X_0 = \sum_i X_{0,i}$, 
we may assume 
\[
0 = a_1 = \cdots = a_{r} < a_{r+1}  \leq \cdots \leq a_m
.\]
We seek to show $r=m$. If not, then, using that $X_0$ is connected, we can choose a curve  $C\subset X_0$ such that $C\not \subset X_{i,0}$ for all $r<i \leq m$ and $C \cap X_{0,j} \neq  \emptyset $ for some $r<j \leq m$.  
Hence, 
		\[
		C \cdot N(K_{X/B}+\Delta+D) 
		= 
		C \cdot \sum_{i=r+1}^m a_i D_{0,i} \geq  a_j  X_{0,j} \cdot C_j
		>0
		.\]
		The latter is not possible, since $K_{X/B}+\Delta+D \sim_{\bQ,S} 0$ by assumption.
		Therefore, $r=m$, which implies $N(K_{X/B}+\Delta+D)\sim0$ when $B$ is the spectrum of a DVR. 
	
We now deduce the full result from the above special case. 
By \cite[Prop. 9.42]{KolNewBook}, there is locally closed subscheme $B' \hookrightarrow B$ such that for any $T\to B$, $N(K_{X_T/T} +\Delta_T+ D_T) \sim_B 0$ if and only if $T\to B$ factors through $B'$. 
By the above case, every spectrum of a DVR $T$ with a dominant map $T\to B$ must factor through $B'$. 
Thus \cite[Lem. 4.31]{KolNewBook} implies $B'=B$ and so $N(K_{X/B}+\Delta+D)\sim_B 0$.
\end{proof}

Next, we prove a result, which is a consequence of Koll\'ar's gluing theory for slc pairs.

\begin{prop}\label{p:bpcygluing}
Let $B$ be a regular  scheme essentially of finite type over $\bk$ and $B^\circ $ a dense open subset. 
Let $(X^\circ, \Delta^\circ+D^\circ)\to B^\circ$  be a family of boundary polarized CY pairs and write 
 $(\oX^\circ, \oG^\circ+\oDe^\circ+\oD^\circ) $ for the normalization of $(X^\circ, \Delta^\circ+D^\circ)$.

If $\overline{f}^\circ: (\oX^\circ, \oG^\circ+\oDe^\circ+\oD^\circ) \to B^\circ$ extends to a family of boundary polarized CY pairs $\overline{f}: (\oX,\oG+\oDe +\oD) \to B$  and the involution $\tau^\circ$ of $(\overline{G}^\circ)^n$ extends to an involution $\tau$ of $\overline{G}^n$, 
then $f^\circ$ extends to a family of boundary polarized CY pairs $f:(X,\Delta+D)\to B$ such that there is a commutative diagram 
\[
\begin{tikzcd}
(\oX^\circ, \oG^\circ+\oDe^\circ +\oD^\circ ) \arrow[r,hook] \arrow[d,"\pi^\circ"] & (\oX,\oG+\oDe+\oD )\arrow[d,"\pi"]\\
(X^\circ, \Delta^\circ +D^\circ ) \arrow[r,hook] & (X,\Delta+D)
\end{tikzcd},
\]
where the horizontal maps are open embeddings and the vertical maps  normalizations.
\end{prop}

\begin{proof}
To begin, note that $(X^\circ, \Delta^\circ+D^\circ)$ and  $(\oX,\oG+\oDe+\oD)$  are slc pairs by Lemma \ref{l:familyslcnormadj}.1
Since $\tau^\circ$ fixes ${\rm Diff}_{(\oG^\circ)^{n}}(\oDe^\circ+\oD^\circ)$ by \cite[Prop. 5.12]{Kol13} and ${\rm Diff}_{\oG^n}(\oDe+\oD)$ has no components supported over $B\setminus B^\circ$ by Lemma \ref{l:familyslcnormadj}.2, $\tau$ fixes ${\rm Diff}_{\oG^{n}}(\oDe+\oD)$.

We will  now construct $X$ as a geometric quotient \cite[Def. 9.4]{Kol13} by the gluing relation  $(n,n\circ \tau):\oG^n  \rightrightarrows  \oX$ \cite[\S 5.31]{Kol13}.
Note that 
$(\oX, \oG+\oDe +\oD)$ is lc and has no lc centers contained in $\oX\setminus \oX^\circ$  by Lemma \ref{l:slcadj}.
Additionally, over $\oX^\circ$, the equivalence relation is finite and has quotient $X^\circ$.
Therefore \cite[Lem. 9.55]{Kol13} implies $(n,n\circ \tau)$ generates a finite equivalence relation.
Thus, by \cite[Cor. 5.33]{Kol13}, 
there exist  a deminormal pair $(X,\Delta +D)$ and a diagram as above.
Since $\tau$ fixes ${\rm Diff}_{\oG^n} ( \Delta+D)$,
\cite[Thm. 5.38]{Kol13} implies $(X,\Delta +D)$  is slc. 
By the universal property of  categorical quotients \cite[Def. 9.4]{Kol13},  there exists a map $f:X\to B$ such that $\overline{f}= f \circ\pi$.
By  \cite[9.31]{Kol13},  $f$ is finite type. Using that $\overline{f}$ is proper and $\pi$ is proper and surjective, $f$ is proper.

It remains to show that $(X,\Delta+D)\to B$ is a family of boundary polarized CY pairs. 
First, fix any $b\in B$ and regular system of parameters $x_1\ldots, x_r\in \cO_{B,0}$.
Since
$
(\oX,\oG+\oDe+\oD+ \overline{f}^*\{ x_1 \cdots x_r =0\})
$
is  slc  over a neighborhood of $b$ by Lemma \ref{l:slcadj}, 
$
(X,G+\Delta+D+ {f}^*\{ x_1 \cdots x_r =0\})
$
is slc over a neighborhood of $b$. 
Thus $(X,\Delta+D)\to B$  is a family of slc pairs in a neighborhood of $b\in B$ by Lemma \ref{l:slcadj}.
Since the Definition \ref{def:mumford} is local on the base, it follows that $(X,\Delta+D)\to B$ is a family of slc pairs. 
Next, note that 
\[
K_{\oX}+\oG+ \oDe+\oD = \pi^* (K_{X}+\Delta+D) 
\quad \text{ and } \quad 
K_{\oX}+\oG+ \oDe = \pi^* (K_{X}+\Delta) 
\]
Since $\pi$ is finite, $-K_X-\Delta$ is ample over $B$
and $K_{X}+\Delta+D\equiv_{B} 0$. 
The latter  implies $K_{X}+\Delta+D\sim_{B,\bQ}0$ by \cite[Cor. 1.6]{HX16}.
Therefore $(X,\Delta+D)\to B$ is a family of boundary polarized CY pairs.
\end{proof}	

\subsection{Valuations}
Let $X$ be a normal scheme. 
If $\mu: Y \to X$ is a proper birational morphism with $Y$ normal and $E\subset Y$ is a prime divisor (called a  \emph{divisor over} $X$), then $E$ defines a valuation $\ord_{E}:K(X)^\times \to \bZ$. 
A valuation $v:K(X)^\times \to \bZ$ of the form $v:= c \, \ord_E$, where $c\in \bZ_{\geq 0}$ and $E$ a divisor over $X$, is called \emph{divisorial}
and we write $\DivVal_X$ for the set of divisorial valuations on $X$. 

Let $(X,\Delta)$ be a pair. 
The  \emph{log discrepancy} of a divisorial valuation $v=c\ord_E$ on $X$ is
\[
A_{X,\Delta}(v) := c(1+{\rm coeff}_E(K_{Y}- \mu^*(K_X+\Delta))
.\]
If $v = \ord_E$ for a divisor $E$ over $X$, we will often write $A_{X,\Delta}(E)$ for $A_{X,\Delta}(v)$.  Note that $(X,\Delta)$ is lc if and only if $A_{X,\Delta}(v) \geq 0$ for all divisorial valuations $v$ on $X$.

\subsection{Koll\'ar components}

\begin{defn}[\cite{LX20}]
Let $(X,\Delta)$ be a klt pair of finite type over $\bk$. Let $x\in X$ be a closed point. 
A proper birational morphism $\mu: Y\to X$ is said to provide a \emph{Koll\'ar component} $E$ over $x\in (X,\Delta)$ if the following statements hold.
\begin{enumerate}
    \item $\mu$ is an isomorphism over $X\setminus\{x\}$;
    \item $\mu^{-1}(x)= E$ is a prime divisor on $Y$;
    \item $-E$ is $\bQ$-Cartier and $\mu$-ample;
    \item $(Y, E+\mu_*^{-1}\Delta)$ is plt.
\end{enumerate}
Sometimes we also say that $(E,\Delta_E)$ is a Koll\'ar component over $x\in (X,\Delta)$ where $\Delta_E$ is the different of $(Y, E+\mu_*^{-1}\Delta)$ on $E$.
\end{defn}

\begin{lem}\label{l:pltcyclicquot}
Let $x\in X$ be a cyclic quotient surface singularity. Let $(E, \Delta_E)$ be a Koll\'ar component over $x\in X$. Then $(E,\Delta_E)$ is toric.
\end{lem}

\begin{proof}
As we will explain, the result follows from 
\cite[Proof of Theorem 4.5]{LX21}.
Let $f: Z\to X$ be the minimal resolution of the singularity $x\in X$. By \cite{Kol13}*{\S 3.40.1}, we know that the dual graph of $\Exc(f)$ is a chain, i.e. $x\in X$ is an $A$-type singularity in terms of \cite{LX21}. Since $-(K_Y+E)\sim_{\bQ,\mu} A_{X}(E) E$ is $\bQ$-Cartier and $\mu$-ample, there exists an integer $m>1$ such that $-m(K_Y+E)$ is Cartier and $\mu$-very ample. By Bertini's theorem \cite{KM98}*{Lem. 5.17}, there exists an effective Cartier divisor $\Gamma\sim_f -m(K_Y+E)$ on $Y$ such that $(Y, E+\frac{1}{m}\Gamma)$ is plt. Since $K_Y + E+\frac{1}{m}\Gamma \sim_{\bQ,\mu} 0$, we know that $K_Y + E+\frac{1}{m}\Gamma = \mu^*(K_X+ \frac{1}{m}\mu_*\Gamma)$. Thus $(X, \frac{1}{m}\mu_*\Gamma)$ is log canonical with only one lc place $E$. Hence, for $0<\varepsilon\ll 1$, we have that $x\in (X, \frac{1-\varepsilon}{m}\mu_*\Gamma)$ is a klt surface germ whose $\mld$ is computed by $E$. Thus the assumptions of \cite[Thm. 4.5]{LX21} are satisfied. 
By \cite[Proof of Theorem 4.5]{LX21}, there exists a sequence of smooth blowups $h: W\to Y$ such  the dual graph of $\Exc(g)$ is a chain where $g=f\circ h: W\to X$, and $E$ is a $g$-exceptional prime divisor on $W$.
Therefore, we have a birational morphism $p: W\to Y$ such that $g=\mu\circ p$, and the dual complex of $\Exc(g)$ being a chain implies that $p(\Exc(p))$ is a subset of $E$ containing at most two points. Since $\Delta_E$ is supported in $p(\Exc(p))$, we know that $(E, \Delta_E)$ is toric. 
\end{proof}

\subsection{Seifert $\bG_m$-bundles and orbifold cones}\label{s:seifert}

\subsubsection{Seifert $\bG_m$-bundles}
\begin{defn}[\cite{Kol04}]\label{def:seifert}
Let $X$ be a normal variety over $\bk$. A \emph{Seifert $\bG_m$-bundle} $Y$ over $X$ is a normal variety $Y$ equipped with a $\bG_m$-action together with a morphism $\pi: Y\to X$ such that the following hold.
\begin{enumerate}
    \item $\pi$ is affine and $\bG_m$-equivariant with respect to the trivial $\bG_m$-action on $X$.
    \item For every $x\in X$, the $\bG_m$-action on the reduced fiber $Y_x^{\red} := \red f^{-1}(x)$ is $\bG_m$-equivariantly isomorphic to the natural left action of $\bG_m$ on $\bG_m/ \bmu_{m(x)}$ for some $m(x)\in \bZ_{>0}$.
    \item $m(x)=1$ for $x$ in a dense open subset of $X$.
\end{enumerate}
\end{defn}

\begin{defn}[cf. \cite{Kol04}]\label{def:compact-G_m-bundle}
Let $X$ be a normal variety. Let $L$ be a $\bQ$-Cartier $\bQ$-divisor on $X$. We define a normal variety $Y_L$ together with an affine morphism $\pi_L: Y_L\to X$ as
\[
Y_L: = \Spec_X \bigoplus_{i\in \bZ} \cO_X(\lfloor iL\rfloor).
\]
There is a natural $\bG_m$-action on $Y_L$ where $\cO_X(\lfloor i L \rfloor)$ has weight $i$. 
We define another normal variety $\mybar{Y}_L$ together with a projective morphism $\mybar{\pi}_L: \mybar{Y}_L \to X$ as 
\[
    \mybar{Y}_L: =\Proj_X \bigoplus_{i=0}^\infty \bigoplus_{j=0}^\infty \cO_X(\lfloor iL\rfloor) \cdot  t^j,
\]
where  $\cO_X(\lfloor iL\rfloor)$ and $t$ have degree $i$ and $1$ respectively. Then $\mybar{Y}_L$ admits a $\bG_m$-action where $\cO_X(\lfloor iL\rfloor)$ and $t$ have weight $i$ and $0$ respectively. There are natural $\bG_m$-equivariant open immersions $Y_L\hookrightarrow Y_L^{a} \hookrightarrow \mybar{Y}_L$, where 
\[
Y_L^a:=\Spec_X\bigoplus_{i=0}^\infty \cO_X(\lfloor iL\rfloor).
\]
In fact, $Y_L^a$ is $\bG_m$-equivariantly isomorphic to the open subset $(t\neq 0)$ of $\mybar{Y}_L$. We call $\mybar{Y}_L$ a \emph{compactified Seifert $\bG_m$-bundle over $X$.} From \cite[\S 14]{Kol04} we know that $\mybar{Y}_L\setminus Y_L $ is the disjoint union of two sections $X_0$ and $X_\infty$, where $X_0 = Y_L^a\setminus Y_L$ and $X_\infty = \mybar{Y}_L\setminus Y_L^a$.
\end{defn}

\begin{thm}[\cite{Kol04}*{Thm. 7}]\label{thm:Kollar-Seifert}
Let $X$ be a normal variety.  Then the following hold.
\begin{enumerate}
    \item Let $L$ be a $\bQ$-Cartier $\bQ$-divisor on $X$. Then $Y_L$ is a Seifert $\bG_m$-bundle over $X$.
    \item Every Seifert $\bG_m$-bundle $Y$ over $X$ is $\bG_m$-equivariantly isomorphic to $Y_L$ for some $\bQ$-Cartier $\bQ$-divisor $L$ on $X$.
    Moreover, such $L$ is uniquely determined by $Y$ up to $\bZ$-linear equivalence.
\end{enumerate}
\end{thm}

\subsubsection{Orbifold cones}
\begin{defn}[cf. \cite{Kol04}]
Let $X$ be a normal projective variety. Let $L$ be an ample $\bQ$-Cartier $\bQ$-divisor on $X$. 
\begin{enumerate}
    \item The \emph{affine orbifold cone} $C_a(X,L)$ is defined as 
    \[
        C_a(X,L):=\Spec\bigoplus_{i=0}^\infty H^0(X, \cO_X(\lfloor iL\rfloor)).
    \]
    It admits a $\bG_m$-action where $H^0(X, \cO_X(\lfloor iL\rfloor))$ has weight $i$.
    \item The \emph{projective orbifold cone} $C_p(X,L)$ is defined as 
    \[
        C_p(X,L):=\Proj\bigoplus_{i=0}^\infty \bigoplus_{j=0}^\infty H^0(X, \cO_X(\lfloor iL\rfloor))\cdot t^j,
    \]
    where $H^0(X, \cO_X(\lfloor iL\rfloor))$ and $t$ have degree $i$ and $1$ respectively. It admits a $\bG_m$-action where $H^0(X, \cO_X(\lfloor iL\rfloor))$ and $t$ have weight $i$ and $0$ respectively.
\end{enumerate} 

Note that $C_a(X,L)$ is $\bG_m$-equivariantly isomorphic to the open subset $(t\neq 0)$ of $C_p(X,L)$. Moreover, both $C_a(X,L)$ and $C_p(X,L)$ can be obtained by contracting the zero section $X_0$ in $Y_L^a$ and $\mybar{Y}_L$ to a point respectively, see \cite[\S 14]{Kol04}.
\end{defn}

\begin{defn}\label{def:orb-div}
Let $X$ be a normal projective variety. Let $L$ be an ample $\bQ$-Cartier $\bQ$-divisor on $X$. Assume $\{L\}:=L- \lfloor L \rfloor =\sum_{i=1}^l \frac{a_i}{b_i} L_i$ where $a_i$ and $b_i$ are coprime positive integers, and the $L_i$ are distinct prime divisors.  We define the \emph{orbifold divisor} $L_{\orb}$ of $L$ as the effective $\bQ$-divisor on $X$ given by
\[
L_{\orb}:= \sum_{i=1}^l \frac{b_i-1}{b_i}L_i.
\]
We often say that $L_{\orb}$ is the \emph{orbifold divisor} of the projective orbifold cone $C_p(X,L)$.

Let $D$ be an effective $\bQ$-divisor on $X$ such that $D\geq L_{\orb}$. Let $\mybar{\pi}_L: \mybar{Y}_L\to X$ be the compactified Seifert $\bG_m$-bundle. Suppose $D = \sum_{i} \frac{b_i-1+c_i}{b_i} L_i + \sum_{j} c_j' D_j$ where  the $D_j$'s are irreducible components of $D$ not contained in $\Supp(L_{\orb})$. We define 
\begin{equation}\label{eq:Seifert-adj-0}
\mybar{D}_L:= \sum_{i} c_i \red(\mybar{\pi}_L^{-1}(L_i))  + \sum_j c_j' \red(\mybar{\pi}_L^{-1}(D_j)).
\end{equation}
From \eqref{eq:Seifert-adj-0} we see that if $D=L_{\orb}$ then $\mybar{D}_L = 0$.
Define the $\bQ$-divisor $D_{L,p}$ on $C_p(X,L)$ to be the pushforward of $\mybar{D}_L$ under the contraction $\mu_L:\mybar{Y}_L \to C_p(X,L)$. As an abuse of notation, we denote the section at infinity of $C_p(X,L)$ by $X_\infty$ which is just $(\mu_L)_* X_\infty$. It is clear that $X_\infty =(t=0)$ is $\bQ$-Cartier ample on $C_p(X,L)$. 
\end{defn}

\begin{prop}[cf. \cite{Kol04}]\label{prop:orbcone-adjunction}
    Let $X$ be a normal projective variety. Let $L$ be an ample $\bQ$-Cartier $\bQ$-divisor on $X$. Let $D$ be an effective $\bQ$-divisor on $X$ such that $D\geq L_{\orb}$ and $K_X+D$ is $\bQ$-Cartier.

    \begin{enumerate}
        \item    If $(X, D)$ is lc (resp. klt), then $(\mybar{Y}_L, \mybar{D}_L+X_0+X_\infty)$ is an lc (resp. plt) pair whose different divisors on $X_0$ and $X_\infty$ are the pullbacks of $D$ under the isomorphisms $\mybar{\pi}_L|_{X_0}$ and $\mybar{\pi}_L|_{X_\infty}$.
        \item If $(X, L_{\orb})$ is an lc (resp. klt) log Fano pair satisfying $rL\sim_{\bQ} -(K_X+L_{\orb})$ for some $r\in \bQ_{>0}$, then $(C_p(X,L), X_\infty)$ is an lc (resp. plt) log Fano pair whose different divisor on $X_\infty$ is the pullback of $L_{\orb}$.
    \item If $(X,D)$ is an lc CY pair, then $(C_p(X,L), D_{L,p} + X_\infty)$ is a boundary polarized CY pair whose different divisor on the polarizing boundary $X_\infty$ is the pullback of  $D$. 
    \end{enumerate}

\end{prop}

\begin{proof}
    (1) Since $\mybar{\pi}_L:\mybar{Y}_L\to X$ is a pure relative dimension $1$ morphism between normal varieties, the pullback map $\mybar{\pi}_L^*: \WDiv(X)\to \WDiv(Y)$ is well-defined, where $\WDiv(\cdot)$ denotes the free abelian group of Weil divisors. The following equalities on ($\bQ$-)Weil divisors are often obtained by first restricting to the smooth locus $X_{\sm}$ of $X$ and its preimage $\mybar{Y}_L\times_X X_{\sm}$, and then taking Zariski closure. By \cite[\S 40 and \S 41]{Kol04}, we have 
    \[
    K_{Y_L} = \pi_L^* (K_{X} + L_{\orb}).
    \]
    By similar arguments to \cite[Proof of Lemma 3.11]{LL19}, we can extend the above equality to 
    \begin{equation}\label{eq:seifert-adj-1}
    K_{\overline{Y}_{\!L}} + X_{0} + X_{\infty} = \mybar{\pi}_L^* (K_{X} + L_{\orb}).
    \end{equation}
    By \cite[Proof of Prop. 16]{Kol04}, we know that 
    \begin{equation}\label{eq:seifert-adj-2}
    \mybar{\pi}_L^* L_i = b_i \mathrm{red}(\mybar{\pi}_L^{-1}(L_i)) \quad \textrm{and}\quad
    \mybar{\pi}_L^* D_j = \mathrm{red}(\mybar{\pi}_L^{-1}(D_j)).
    \end{equation}
    Combining \eqref{eq:seifert-adj-1} and \eqref{eq:seifert-adj-2} together with the definition of $\mybar{D}_L$ from \eqref{eq:Seifert-adj-0}, we have
    \begin{equation}\label{eq:seifert-adj-3}
    K_{\overline{Y}_{\!L}} + \mybar{D}_L+ X_{0} + X_{\infty} = \mybar{\pi}_L^* (K_{X} + D).
    \end{equation}

    Since $K_X+D$ is $\bQ$-Cartier, \eqref{eq:seifert-adj-3} implies that the different divisors of $(\mybar{Y}_L, \mybar{D}_L+X_0+X_\infty)$ along $X_0$ and $X_\infty$ are pullbacks of $D$.
    By inversion of adjunction, if $(X,D)$ is lc (resp. klt), then $(\mybar{Y}_L, \mybar{D}_L+X_0+X_\infty)$ is lc (resp. plt) in a neighborhood of $X_0\cup X_\infty$. This implies that $(\mybar{Y}_L, \mybar{D}_L+X_0+X_\infty)$ is lc (resp. plt) as the closure of every $\bG_m$-orbit of  $\mybar{Y}_L$ intersects either $X_0$ or $X_\infty$. Thus (1) is proved.

    (2) For simplicity, denote by $\pi_0:=\mybar{\pi}_L|_{X_0}$ (resp. $\pi_\infty:=\mybar{\pi}_L|_{X_\infty}$) the isomorphism between $X_0$ (resp. $X_\infty$) and $X$. Thus $X_0|_{X_0}\sim_{\bQ} \pi_0^* (-L)$. Hence by (1) we have
    \[
    (K_{\overline{Y}_{\!L}} + (1-r)X_0 + X_\infty)|_{X_0} = K_{X_0} + L_{\orb} -r X_0|_{X_0} \sim_{\bQ} \pi_0^*(-rL) - \pi_0^*(-rL) = 0. 
    \]
    As a result,  we have that 
    \[
    K_{\overline{Y}_{\!L}} + (1-r)X_0 + X_\infty\sim_{\bQ} \mu_L^* (K_{C_p(X,L)} + X_\infty).
    \]
    If $(X,L_{\orb})$ is lc (resp. klt), then by (1) we know that $(\mybar{Y}_L, X_0 + X_\infty)$ is lc (resp. plt), which implies that $(C_p(X,L), X_\infty)$ is lc (resp. plt) as $A_{C_p(X,L), X_\infty}(X_0) = r>0$.

    Next, we show that $K_{C_p(X,L)}+X_\infty\sim_{\bQ} -r X_\infty$. By \eqref{eq:Seifert-adj-0} and \eqref{eq:seifert-adj-3}, we have
    \begin{equation}\label{eq:seifert-adj-4}
    K_{\overline{Y}_{\!L}} + X_0 + X_\infty = \mybar{\pi}_L^*(K_X+L_{\orb}) \sim_{\bQ} -r\mybar{\pi}^* L.
    \end{equation}
    Since $X_\infty = (t=0)$, we know that $X_\infty \sim_{\bQ} (\mu_L)_*\mybar{\pi}_L^* L$. Thus taking pushforward of \eqref{eq:seifert-adj-4} under $\mu_L$ yields the desired $\bQ$-linear equivalence.    
    As a result, $-K_{C_p(X,L)}-X_\infty\sim_{\bQ} r X_\infty$ is ample, which implies that $(C_p(X,L), X_\infty)$ is a log Fano pair. The statement on the different divisor follows from (1).

    (3) By (1) and \eqref{eq:seifert-adj-3} we know that $(\mybar{Y}_L, \mybar{D}_L + X_0+X_\infty)$ is an lc CY pair. Thus $(C_p(X,L), D_{L,p}+X_\infty)$ is also  an lc CY pair. Clearly $X_\infty$ is ample on $C_p(X,L)$. The statement on the different divisor follows from (1).
\end{proof}

\begin{prop}\label{prop:orbcone-vol}
Under the assumptions of Proposition \ref{prop:orbcone-adjunction}(2) with $n=\dim X$, we have 
\[
(-K_{C_p(X,L)})^{n+1} = (1+r)^{n+1}(L^n) \quad \textrm{and}\quad (-K_{C_p(X,L)}-X_\infty)^{n+1} = r^{n+1}(L^n).
\]
\end{prop}
\begin{proof}
By the proof of Proposition \ref{prop:orbcone-adjunction}(2), we have $-K_{C_p(X,L)}\sim_{\bQ} (1+r) X_\infty$ and $-K_{C_p(X,L)}-X_\infty\sim_{\bQ} r X_\infty$. Thus it suffices to show $(X_\infty^{n+1}) = (L^n)$. Since $X_\infty = (t=0)$, we know that $X_\infty \sim_{\bQ} (\mu_L)_*\mybar{\pi}_L^* L$ which implies that $X_\infty|_{X_\infty}\sim_{\bQ} \pi_\infty^* L$. Thus we have $(X_\infty^{n+1}) = (L^n)$.
\end{proof}

\subsubsection{Surfaces with $\bG_m$-actions}

\begin{defn}\label{d:T-var-divisors}
Let $Y$ be a normal surface with an effective $\bG_m$-action. Let $E$ be a  $\bG_m$-invariant prime divisor on $Y$. We say $E$ is \emph{vertical} if it is the closure of a  $\bG_m$-orbit. We say $E$ is \emph{horizontal} if every point in $E$ is a $\bG_m$-fixed point.
\end{defn}

\begin{prop}\label{p:horizontaldiv}
Under the assumption of Definition \ref{d:T-var-divisors}, there exist at most two horizontal divisors on $Y$.
\end{prop}

\begin{proof}
This follows from \cite[Cor. 3.17]{PS11} as the divisorial fan of $Y$ with the $\bG_m$-action is $1$-dimensional, hence there are at most two rays in the tail fan. Note that a horizontal divisor is called a prime divisor of type 2 in loc. cit.  
\end{proof}

\begin{defn}[\cite{Kol13}*{Def. 4.36}]
A \emph{standard $\bP^1$-link} is a pair $(Y, \Delta + D_1+D_2)$ where $Y$ is a normal variety, $D_1$ and $D_2$ are prime divisors on $Y$, together with a proper morphism $\pi: Y\to X$ to a normal variety $X$ such that the following hold.
\begin{enumerate}
    \item $K_Y + \Delta + D_1+D_2\sim_{\bQ,\pi} 0$;
    \item $(Y, \Delta+ D_1+D_2)$ is plt (in particular, $D_1$ and $D_2$ are disjoint);
    \item $\pi|_{D_i}: D_i \to X$ is an isomorphism for every $i\in \{1,2\}$;
    \item every reduced fiber $Y_x^{\red} = \red \pi^{-1}(x)$ is isomorphic to $\bP^1$.
\end{enumerate}
\end{defn}

\begin{prop}\label{p:Seifert-P^1link}
Let  $\pi:(Y, D_1+D_2)\to B$ be a standard $\bP^1$-link over a smooth projective curve $B$. Assume that $(Y,D_1+D_2)$ admits an effective $\bG_m$-action such that every fiber of $\pi$ is $\bG_m$-invariant. Then $Y$ is  $\bG_m$-equivariantly isomorphic to a compactified Seifert $\bG_m$-bundle over $B$.
\end{prop}

\begin{proof}
Denote by $Y^\circ := Y\setminus(D_1\cup D_2)$ and $\pi^\circ:=\pi|_{Y^\circ}$. Since the reduced structure of every fiber of $\pi$ is isomorphic to $\bP^1$ and $D_1$ and $D_2$ are two disjoint sections of $\pi$, we know that the reduced structure of $(\pi^{\circ})^{-1}(b)$ is isomorphic to $\bA^1\setminus \{0\}$ for every $b\in B$. Since $D_1$ and $D_2$ are two horizontal divisors on $Y$, 
we know that every fiber of $\pi$ is vertical by Proposition \ref{p:horizontaldiv}. In particular,  $(\pi^{\circ})^{-1}(b)$ consists of a single $\bG_m$-orbit for every $b\in B$. This implies that $Y^\circ$ is a Seifert $\bG_m$-bundle by Definition \ref{def:seifert}. By Theorem \ref{thm:Kollar-Seifert}, there exists a $\bQ$-divisor $L$ on $B$ such that $Y^\circ$ is $\bG_m$-equivariantly isomorphic to $Y_L$ in terms of Definition \ref{def:compact-G_m-bundle}. Denote by $Y':=\mybar{Y}_{L}$ the compactified Seifert $\bG_m$-bundle over $B$  from Definition \ref{def:compact-G_m-bundle}. Denote by $Y'^\circ$ the open subset of $Y'$ that is isomorphic to $Y_L$.

Denote by $D_1'$ and $D_2'$ the two horizontal divisors (which are exactly the two sections) of $Y'$ such that $D_1'\cup D_2' = Y'\setminus Y'^\circ$. Thus we get a $\bG_m$-equivariant birational map $Y\dashrightarrow Y'$ that is isomorphic between $Y^\circ$ and $Y'^\circ$.  Since $D_1$ and $D_2$ both dominate $B$, they are not exceptional over $Y'$; similarly, $D_1'$ and $D_2'$ are not exceptional over $Y$. This implies that $Y\dashrightarrow Y'$ is an isomorphism in codimension $1$, which implies that it is an isomorphism as both $Y$ and $Y'$ are normal projective surfaces.
\end{proof}

\subsection{Good moduli spaces}

The definition of a good moduli space was introduced by Alper to generalize a good quotient in GIT to arbitrary algebraic stacks \cite{Alp13}. 
Below, we recall its definition and some results that will be used later on.
Throughout, we use conventions for algebraic stacks in \cite[\S2]{Alp13}.

\begin{defn}[\cite{Alp13}]\label{d:gm}
A quasi-compact morphism $\phi:\cM\to M$ from an algebraic stack to an algebraic space is a \emph{good moduli space} if 
\begin{enumerate}
	\item  The push-forward functor on quasi-coherent sheaves is exact.
	\item  The induced morphism $\cO_{M} \to \phi_{*} \cO_{\cM}$ is an isomorphism.
\end{enumerate}
\end{defn}

Good moduli space morphisms satisfy a number of natural properties.

\begin{prop}[{\cite[Main Prop.]{Alp13}}]\label{p:Alpergmprops}
If $\phi: \cM\to M$ is a good moduli space, then:
\begin{enumerate}
	\item  The morphism $\phi$ is surjective and universally closed.
	\item Two geometric points $x,x' \in \cM(\bK)$ are identified in $M$ if and only if their closures $\overline{ \{x\}}$ and $\overline{\{x'\}}$  in $\cX\times_{\bk} \bK$  intersect.
	\item Assuming $\cM$ is locally Noetherian,  $\phi$ is universal for  maps from $\cM$ to algebraic spaces.
\end{enumerate}
\end{prop}

\begin{prop}[{\cite[Rem. 6.2]{Alp13}}]\label{p:saturated}
Let $\phi:\sM\to M$ be a good moduli space  and $\cU\subset \cM$ be an open substack.
If $\cU$ is saturated (which means $\phi^{-1} (\phi (\cU)) = \cU$), then
\begin{enumerate}
\item  $\phi(\cU)\subset M$ is an open subspace and
\item  $\cU \to \phi(\cU)$ is a good moduli space morphism.
\end{enumerate} 
\end{prop}

In order to construct good moduli spaces, we will use the following existence result due to Alper, Halpern-Leistner, and Heinloth.
The terms S-completeness and $\Theta$-reductivity, which appear in the result, will be defined in Sections \ref{ss:defScom} and \ref{ss:deftheta}.

\begin{thm}[{\cite[Thm. A]{AHLH18}}]\label{t:AHLH}
Let $\cM$ be a finite type algebraic stack with affine diagonal. The following are equivalent:
\begin{enumerate}
	\item There exists a good moduli space morphism $\cM\to M$ with $M$ separated.
	\item  $\cM$ is S-complete and $\Theta$-reductive. 
\end{enumerate}
Moreover, when (1) holds, $M$ is proper if $\cM$ satisfies the existence part of the valuative criterion for properness.
\end{thm}

\begin{rem}\label{r:AHLHeftDVRs}
In the above theorem, it suffices to check S-completeness, $\Theta$-reductivity, and the existence part of the valuative criterion with respective to DVRs that are essentially of finite type over $\bk$ \cite[Thm. 5.4 \& Lem. A.11]{AHLH18}. 
\end{rem}

\section{Moduli stack}\label{s:stack}
In this section, we define a family of boundary polarized CY pairs over an arbitrary base scheme.
Using the definition, we introduce a moduli stack parametrizing families of boundary polarized CY pairs  and prove it is a locally finite type algebraic stack with affine diagonal.

\medskip

Throughout this section, we fix the following numerical invariants: an index $N\in \bZ_{>0}$, a coefficient set 
${\bf r}=(r,s) \in \bQ_{>0}$ satisfying that $N {\bf r} \in \bZ^2$, and a function $\chi:\bN\to \bZ$.

\subsection{Families of boundary polarized CY pairs}
We first define a family of boundary polarized CY pairs over an arbitrary Noetherian scheme.

\begin{defn}\label{d:familyNoetherian}
A \emph{family of boundary polarized CY pairs}
$(X,\Delta+D)\to B$ over a Noetherian scheme $B$ and with coefficients ${\bf r}$ and index dividing $N$ consists of 
\begin{enumerate}
\item a flat projective morphism $X\to B$, and 
\item relative K-flat Mumford divisors $\Delta^{\Div}$ and $D^{\Div}$,
where
$\Delta:=r\Delta^{\Div}$ and $D:=sD^{\Div}$
\end{enumerate}
satisfying the following conditions\begin{enumerate}
\item[(3)] $(X_b,\Delta_b+ D_b)$ is a boundary polarized CY pair for each $b\in B$,
\item[(4)] $\omega_{X/B}^{[N]}(N(\Delta + D))\cong_B \cO_{X}$,
\item[(5)] $\omega_{X/B}^{[m]}(m\Delta + n D^{\rm Div})$ commutes with base change for all $m, n\in \bZ$ with  $mr$ integral.
\end{enumerate}

\end{defn}

\begin{rem}
We now explain some of the terminology in the above definition.
\begin{enumerate}

\item[(i)] By writing $\Delta$ and $D$ as rational multiples of relative Mumford divisors in (2), we are viewing $(X,\Delta+D)\to B$ as a family of \emph{marked} pairs. See \cite[Sec. 8.1]{KolNewBook} for a discussion on markings and why they are necessary when defining a moduli functor.

\item[(ii)] 
The term K-flatness is defined in \cite[Sec. 7.1]{KolNewBook} and is needed to have a moduli theory for divisors over non-reduced schemes.

\item[(iii)]
The notation $\cong_B$ means 
$\omega_{X/B}^{[N]}(N(\Delta + D))\cong \cO_{X} \otimes L_B$ 
for some line bundle $L_B$ pulled back from $B$.
\item[(iv)] The sheaves  in (4) and (5) are defined as follows. Let $U\subset X$ be the largest open subset where the fibers of $U\to B$ are either smooth or nodal  and   $D^{\Div}_U$ and $\Delta^{\Div}_U$ are relative Cartier divisors. 
Note that ${\rm codim}_{X_b}(X_b\setminus U_b)\geq 2$ for each $b\in B$. 
We set 
\[
\omega_{X/B}^{[m]}(m\Delta + n D^{\Div}):=i_* \big( \omega_{U/B}^{\otimes m}(m\Delta_U + nD_U^\Div)\big)
,\]
where $i:U\hookrightarrow X$ is the inclusion.	
The sheaves are reflexive by \cite[Cor. 3.7]{HK}.

\item[(v)] 
Condition (5) means that for any morphism of schemes $g: B'\to B$, the natural map 
\[
g'^* \omega_{X/B}^{[m]}(m\Delta + nD^{\Div})  \to \omega_{X'/B'}^{[m]}(m\Delta_{B'} + nD^{\Div}_{B'})
\]
is an isomorphism, where $X':= X\times_{B}B'$ and $g':X'\to X$ the first projection.
This is often referred to as \emph{Koll\'ar's condition} in the literature; see \cite[Ch. 9]{KolNewBook}.

\item[(vi)] Conditions (3) and (5) imply $\omega_{X/B}^{[-m]}(-m\Delta)$ and  $\mathcal{O}_X(mD)$ are relatively ample line bundles for any sufficiently divisible positive integer $m$.

\item[(vii)]  Condition (5) with $m=0$ and $n = 1$ implies that $D^\Div$ is flat with pure dimensional scheme theoretic fibers by \cite[4.19 and 4.33]{KolNewBook}.

\end{enumerate}
\end{rem}

It is clear that, over a reduced Noetherian scheme, Definition  \ref{d:familyNoetherian} implies Definition \ref{d:Familybasenormal}. 
The following proposition shows the notions are in fact equivalent in this setting.

\begin{lem}\label{lem:flatdiv}
If $(X,\Delta+D)\to B$ is a family of boundary polarized CY pairs over a reduced base scheme $B$ as defined in Definition \ref{d:Familybasenormal} such that
\begin{enumerate}
 \item $N(K_{X/B}+\Delta+D)\sim 0$ and
 \item $\Delta:=r \Delta^{\Div}$ and  $D:= s D^{\rm Div}$, where $\Delta^{\Div}$ and $D^{\Div}$  are  relative Mumford divisors,
\end{enumerate}
then the family satisfies Definition \ref{d:familyNoetherian}.
\end{lem}

\begin{proof}
It suffices to show that $\Delta^{\rm Div}$ and $D^{\rm Div}$  are K-flat and  	
Definition \ref{d:familyNoetherian}.5 is satisfied.
Since $B$ is reduced and the divisors $K_{X/B}+\Delta$ and $D^{\rm div}$ are $\bQ$-Cartier, 
the statements follow from  \cite[7.4.2]{KolNewBook} and
\cite[4.19, 4.33, 9.17]{KolNewBook}, respectively.
\end{proof}

In Definition \ref{d:familyNoetherian}, we assumed $B$ is  a Noetherian scheme, since the definition of K-flatness and its properties are only developed in \cite{Kol13} for families over Noetherian schemes. 
To define a family over an arbitrary scheme, we bootstrap from the Noetherian case.

\begin{defn}\label{d:nonNoeth}
A  \emph{family of boundary polarized CY pairs}
$(X,\Delta+D)\to B$  over a scheme $B$ and with coefficients ${\bf r}$ and index dividing $N$ consists of 
\begin{enumerate}
\item a flat projective morphism of schemes $X\to B$, and 
\item relative Mumford divisors $\Delta^{\Div}$ and $D^{\Div}$, where $\Delta:= r\Delta^{\Div}$ and $D:= s D^{\Div}$
\end{enumerate}
satisfying the condition 
\begin{enumerate}
\item[(3)] for each affine open set $ U\subset X$, there exists a Noetherian subring $A^{\rm N}\subset A:= \cO_X(U)$ such that $(X,\Delta+D)_U \to U$ is the pullback of a family  
$(X^{\rm N},\Delta^{\rm N}+D^{\rm N})\to
U^{\rm N}:= \Spec(A^{\rm N})$ 
satisfying Definition \ref{d:familyNoetherian}.
\end{enumerate}
\end{defn}

We now collect a number of results about the behavior of various invariants of boundary polarized CY pairs in families. 

\begin{prop}\label{p:anticanonicalsheaf}
Let $(X,\Delta+D)\to B$ be a family of boundary polarized CY pairs with coefficients ${\bf r}$ and index dividing $N$. Fix $m \in \bN$ with $m r\in \bZ$. The following hold.
\begin{enumerate}
\item $H^i(X_b,\omega_{X_b}^{[-m]}(-m\Delta))=0$ for all $b\in B$, and $i\geq 1$.
\item The sheaf $f_* \omega_{X/B}^{[-m]}(-m\Delta)$ is locally free and commutes with  base change.
\item The  functions  $B\to \bZ$ sending
\[
b\mapsto H^0(X,\omega_{X_b}^{[-m]}(-m \Delta_b)),
\quad 
b\mapsto \deg(\Delta_b^{\Div}),
\quad \text{ and } \quad
b\mapsto \deg(D_b^{\Div})
\]
are locally constant, 
where $\deg(G_b):=G_b \cdot (-K_{X_b}- \Delta_b)^{d-1}$ for a divisor $G_b$ on $X_b$.
\item The subset of $B$ is open
\[
U:=\{ b\in B\, \vert \, \omega_{X_b}^{[-m]}(-m\Delta) \text{ is a very ample line bundle} \}
.\]

\end{enumerate}
\end{prop}

\begin{proof}

By Definition \ref{d:familyNoetherian}.5, 
\[
\omega_{X/B}^{[-m]}(-mr\Delta)\vert_{X_b} \cong \omega_{X_b}^{[-m]}(-m\Delta),
\] which will be used throughout.
Statement (1) follows  from	\cite[Thm. 11.34]{KolNewBook}.
Statement (2)  follows from (1) and the extension of  \cite[Thm. 12.11]{Har77} to finitely presented morphisms using standard Noetherian approximation arguments. 
Statement (3) follows from (2) and \cite[Thm. 4.3.5]{KolNewBook}.
Next  note that  
$
V:= \{ b\in B\, \vert \,  \omega_{X/B}^{[-m]}(-m\Delta)\vert_{X_b} \text{ is a line bundle}\}
$ 
is open by Nakayama's Lemma. 
Since (2) holds, 
$U$ consists of $b\in V$ such that the restriction of
\[
X_{V} \dashrightarrow \bP_{V}({f}_*\omega_{X/B}^{[-m]}(-m\Delta)^* \vert_V)
\]
is a closed embedding. Since the latter set is open, (4) holds.	
\end{proof}

\subsection{Moduli stack}

We now define a moduli stack of boundary polarized CY pairs.  

\begin{defn}\label{defn:modulistack}
Let $\cM(\chi,N,{\bf r})$ denote the category fibered in groupoids over ${\rm Sch}_{\bk}$ where:
\begin{itemize}

\item The objects are families of boundary polarized CY pairs 
\[
(X,\Delta+D)\to B
\]
with coefficients ${\bf r}$, index dividing $N$, and 
$\chi(X_s,\omega_{X_b}^{[-m]}(-m\Delta))=\chi(m)$ for each $b\in B$ and  $m \in \bN$ such that $mr \in \bN$.

\item the morphisms $[(X',\Delta'+D')\to B']\to [(X,\Delta+D)\to B]$ consist of morphisms of schemes $X'\to X$ and $B'\to B$ such that the natural map $X'\to X\times_{B}B'$ is an isomorphism and $\Delta'^{\Div}$ and $D'^{\Div}$ are the divisorial pullbacks of $\Delta^{\Div}$ and $D^{\Div}$.
\end{itemize}
\end{defn}

The main result of this section is the following statement. 

\begin{thm}\label{t:stack}
 $\cM(\chi,N,{\bf r})$ is a locally  finite type algebraic stack with affine diagonal.
\end{thm}

\begin{rem}\label{r:Delta=0}
In explicit examples, we will be primarily interested in the substack of $\cM(\chi,N,{\bf r}) $ consisting of families $[(X,\Delta+D)\to B]$ such that $\Delta^{\Div} = \emptyset$.
This is an open and closed substack by Proposition \ref{p:anticanonicalsheaf}.3.
\end{rem}

To prove the theorem, we will show the stack has an open cover by finite type quotient stacks.
Fix ${\bf c} = (c,c') \in \bQ^2_{>0}$ and $k \in \bN$ such that $kr \in \bN$. Let 
\[
\cM(\chi,N,{\bf r},{\bf c},k)
\subset 
\cM(\chi,N,{\bf r})
\]
denote the full subcategory consisting of families $[(X,\Delta+D)\to B]$ such that $-k(K_{X_b}+\Delta_b)$ is a very ample Cartier divisor and ${\bf c}=(\deg(\Delta_b^{\Div}),\deg(D_b^{\Div}))$.

\begin{prop}\label{prop:localstack}
$\cM(\chi,N,{\bf r},{\bf c},k)$ is a finite type algebraic stack with affine diagonal.
\end{prop}

\begin{proof}
If $[f:(X,\Delta+D) \to B]$ is an object of $\cM(\chi,N,{\bf r},{\bf c},k)$, then
$\omega_{X/B}^{[-k]}(-k\Delta)$ is a $f$-very ample line bundle of rank $\chi(k)$ and there is an induced embedding 
\[
X\hookrightarrow \bP_{B}\big(f_* \omega_{X/B}^{[-k]}(-k\Delta)\big).
\]
We proceed to parametrize such embeddings. 
	
Let $M:= \chi(k)-1$, and consider the Hilbert scheme $H:={\rm Hilb}_{\chi(k\,\cdot)}(\bP^M)$
parametrizing subschemes $X\subset \bP^M$ with $\chi(X,\cO_X(m)) = \chi(mk )$ for all $m\in \bN$.
Write $X_H \subset \bP^M_H$ for the corresponding universal family. Note that the $\PGL_{M+1}$ action on $\bP^M$ induces a  $\PGL_{M+1}$ action on $H$.
	
We claim that the locus $h \in H$ such that 
\begin{enumerate}
\item[(a)] $X_h$ is geometrically connected and reduced, equidimensional, and deminormal,
\item[(b)] $H^i(X_h, \cO_{X_h}(1))= 0$ for all $i>0$, and
\item[(c)] the embedding $X_h \hookrightarrow \bP_{\kappa(h)}^M$ is by the complete linear series $|\cO_{X_h}(1)|$,
\end{enumerate}
is an open subscheme $H_1\subset H$. 
Indeed, (a) is open   by
\cite[12.2.1(ii) and 12.2.4(vi)]{Gro66} and \cite[Cor. 10.42]{KolNewBook}, while (b) is open by \cite[Thm. III.12.8]{Har77}. 
Since condition (b) implies 
\[
\dim H^0(X_h, \cO_{X_h}(1))= \chi(k)=\dim H^0(\bP_{k(h)}^M, \cO_{\bP_{k(h)}}(1))
,\]
and $H^i(\bP_{k(h)}^M, \cI_{X_h}(1)) = 0$ for $i \geq 2$, condition (c) holds if and only if $H^1(\bP_{k(h)}^M,\cI_{X_h}(1))=0$, 
which is again an open condition. Hence $H_1\subset H$ is open. 
	
Next, we parametrize the boundary divisors in the moduli problem.
Let 
\[
H_2: = {\rm KDiv}_{k^{n-1}c} (X_{H_1}/ H_1 ) \times_{H_1} {\rm KDiv}_{k^{n-1}c'} (X_{H_1}/H_1)
,\]
where $n:= \dim (X_{H_1}) - \dim H_1$ and ${\rm KDiv}_{k^{n-1}c}(\,)$ denotes the parameter space of K-flat,  relative Mumford divisors of degree $k^{n-1}c$ with respect to $\cO_{X_{H_1}}(1)$ as constructed in \cite[Thm. 7.3]{KolNewBook}. 
Write $\Delta_{H_2}:= r \Delta_{H_2}^{\Div}$ and $D_{H_2}:= sD_{H_2}^{\rm Div}$, where
$\Delta_{H_2}^{\rm Div}$ and $D_{H_2}^{\rm Div}$ are the universal families of relative Mumford divisors on $X_{H_2}$ of degree $c$ and $c'$, respectively.
Again, the $\PGL_{M+1}$-action on $\bP^M$ induces a natural $\PGL_{M+1}$-action on $H_2$.
	
Now, we shrink $H_2$ to the locus where Definition \ref{d:familyNoetherian}.3--5 is satisfied.
By  \cite[Prop. 9.42 \& Thm. 3.29]{KolNewBook}, there is a locally closed partial decomposition $H_3\hookrightarrow H_2$ such that $B\to H_2$ factors through $H_3$ if and only if 
\[
\omega_{X_B/B}^{[-k]}(-k\Delta_B) \cong_{B} \cO_{X_B}(1) \quad \text{ and } \quad  \omega_{X_B/B}^{[N]}(N(\Delta_B+D_B)) \cong_{B} \cO_{X_B}.
\]
By our choice of $H_3$, $\omega_{X_{H_2}/H_2}^{[m]}(m\Delta_{H_2})$
and $\cO_{X_{H_3}}(mD_{H_3})$ are line bundles for  $m>0$ sufficiently divisible.
Hence \cite[Prop. 3.31]{KolNewBook} implies there is a locally closed partial decomposition $H_4 \hookrightarrow H_3$ such that $B\to H_3$ factors through $H_4$ if and only if the formation of 
$\omega_{X_B/B}^{[m]}(m\Delta_B + nD^{\Div}_B)$
commutes with base change for all $ m,n\in \bZ$ and $ mr \in \bZ$.
 
Finally, by \cite[Cor. 4.45]{KolNewBook}, the locus  
\[
H_5:= \{ h \in H_4 \, \vert\, (X_h, \Delta_h+D_h)\text{ is slc}\}\subset H_4
\]
is open. 
Now, note that 
$(X_{H_5}\subset \bP^M_{H_5}, r\Delta_{H_5}+sD_{H_5}) \to H_5$
is an object in $\cM(\chi,N,{\bf r},{\bf c},k)$.
Furthermore, by our construction, $H_5$ represents the category  fibered over ${\rm Sch}_{\bk}$ whose objects consist of  a family 
$[f:(X,\Delta+D) \to B]$ in $ \cM(\chi,N,{\bf r},{\bf c},k)$ with an isomorphism $\bP(f_*\omega_{X/B}^{[-k]}(-k\Delta))\cong \bP^M_B$.
	
We claim $\cM(\chi,N,{\bf r},{\bf c},k)\cong [H_5/\PGL_{M+1}]$. 
Indeed, the universal family \[(X_{H_5}\subset \bP^M_{H_5}, \Delta_{H_5}+D_{H_5}) \to H_5,\] which is $\PGL_{M+1}$-equivariant, induces a morphism $H_5 \to \cM(\chi,N,{\bf r},{\bf c},k)$.
The latter induces a morphism
$
[H_5/ \PGL_{M+1}]^{\rm pre} \to \cM(\chi,N,{\bf r},{\bf c},k)$, which by Lemma \ref{l:descent}, induces a morphism 
$[H_5/ \PGL_{M+1}] \to \cM(\chi,N,{\bf r},{\bf c},k)$.
To construct the inverse, consider $[f:(X, \Delta + D) \to B]$ in $ \cM(\chi,N,{\bf r},{\bf c},k)$.
Since $f_*\omega_{X/B}^{[-k]}(-k\Delta)$ is locally free, there exists an open cover $B = \cup_i B_i$ over which its restriction is free.
Choosing trivializations induce embeddings  $X_{B_i} \subset \bP^M_{B_i}$ and, hence,  morphisms $\phi_i: B_i \to H_5$. 
Over the intersections $B_i \cap B_j$, the trivializations differ by a section $s_{ij} \in H^0(B_i \cap B_j, \PGL_{M+1})$.
Therefore the $\phi_i$ glue to a morphism $\phi: B \to [H_5/\PGL_{M+1}]$, which induces a morphism $ \cM(\chi,N,{\bf r},{\bf c},k) \to [H_5/ \PGL_{M+1}]$ that is the inverse.
	
To finish the proof, note that $H_5$ is a finite type scheme.
Thus $[H_5/\PGL_{M+1}]$, which is isomorphic to $\cM(\chi,N,{\bf r},{\bf c},k)$, is a finite type algebraic stack with affine diagonal.
\end{proof} 

The following lemma was used in the above proof.

\begin{lem}\label{l:descent}
The fibered category $\cM(\chi,N,{\bf r})$ satisfies \'etale descent.
\end{lem}

\begin{proof}
For an \'etale cover of a Noetherian scheme, standard arguments show  $\cM(\chi,N,{\bf r})$  satisfies descent (see for example \cite[Proof of Prop. 4.2]{FI22}).
Using
\cite[\href{https://stacks.math.columbia.edu/tag/049N}{Tag 049N}]{stacks-project} and \cite[\href{https://stacks.math.columbia.edu/tag/00QN}{Tag 00QN}]{stacks-project}, the general case can be deduced from the Noetherian case.
We leave the details to the reader.
\end{proof}

\begin{proof}[Proof of Theorem \ref{t:stack}]
For each ${\bf c} \in \bQ_{>0}^2$ and $k \in \bN$, the natural morphism
\[
\cM(\chi,N,{\bf r},{\bf c},k) \longrightarrow \cM(\chi,N,{\bf r})
\]
is representable and an open immersion by Proposition \ref{p:anticanonicalsheaf}.
Additionally, the morphism
\[
\bigsqcup_{{\bf c},k} \cM(\chi,N,{\bf r},{\bf c},k) \to \cM(\chi,N,{\bf r})
\]
is surjective by construction.
Since each $\cM(\chi,N,{\bf r},{\bf c},k)$ is a finite type algebraic stack with affine diagonal by Proposition \ref{prop:localstack}, $\cM(\chi,N,{\bf r})$ is a locally finite type algebraic stack with affine diagonal.
\end{proof}

\subsection{Boundedness}

In this section, we state an elementary criterion for when a set of boundary polarized CY pairs is bounded. 

\begin{defn}
A subset $\mathfrak{F} \subset \cM(\chi, N, {\bf r}) (\bk)$ is \emph{bounded} if there exists a family of boundary polarized CY pairs $(X,\Delta+D)\to B$ over a finite type scheme $B$ such that every element of $\mathfrak{F}$ appears as a fiber of the family.
\end{defn}

In the definition we may assume $B$ is smooth.
Indeed, if $B$ is not smooth, we can choose a finite collection of smooth locally closed subschemes $B_i \subset B$ such that $B= \cup_{i=1}B_i$ and  replace $(X,\Delta+D)\to B$ by the base change by $\sqcup_i B_i \to B$.

\begin{prop}\label{p:boundedness}
A subset  $\mathfrak{F} \subset \cM(\chi, N, {\bf r}) (\bk)$
is bounded if and only if there positive integers $k$ and $c$ such that $-k(K_X+\Delta)$ is very ample Cartier divisor and $\deg(\Delta^{\Div}):= (-K_X-\Delta)^{\dim X-1} \cdot \Delta^{\Div} \leq c$ for every $(X,\Delta+D)\in \mathfrak{F}$ . 
\end{prop}

\begin{proof}
The forward implication follows immediately from Proposition \ref{p:anticanonicalsheaf}. 
To prove the reverse implication, 
assume the existence of $k$ and $c$.
Note that  $(\deg(\Delta^{\Div}),\deg(D^{\Div}))$ takes finitely many values for $(X,\Delta+D) \in \mathfrak{F}$, since
\[
\deg(D^{\Div}) = s^{-1} (-K_X-\Delta)^{\dim X-1} \cdot D =s^{-1} (-K_X-\Delta)^{\dim  X} 
,\]
which is determined by ${\bf r}$ and $\chi$,
and $ \deg(\Delta^{\Div}) \in  k^{-\dim X+1}\bZ \cap[0,c]$.
Thus the proof of Proposition \ref{prop:localstack} shows that $\mathfrak{F}$ is bounded.
\end{proof}

\section{Test configurations}\label{s:tc}

In this section, we discuss properties of test configurations of boundary polarized CY pairs. 
These  degenerations will play a key role in the moduli theory of such pairs.

\subsection{Test configurations of polarized schemes}\label{ss:testconfigps}
We recall content from \cite{BHJ17} on  test configurations.
 Throughout,    $X$ is a projective scheme, and $L$ is an ample line bundle on $X$.

\begin{defn}\label{d:testconfig}
A \emph{test configuration} $(\cX, \cL)$ of $(X,L)$ is the data of 
\begin{itemize}
\item a $\bG_m$-equivariant flat proper morphism $\pi:\cX\to \bA^1$,
\item a relatively ample  $\bG_m$-linearized line bundle $\cL$ on $\cX$, and 
\item a $\G_m$-equivariant isomorphism 
\[
(\cX, \cL)\vert_{\A^1\setminus 0} \cong (X, L) \times (\A^1\setminus 0),
\]
where $\bG_m$ acts on the right side as the product of the trivial and standard actions.
\end{itemize}
\end{defn}

A test configuration induces a multiplicative $\Z$-filtration $\cF^\bullet$ of $H^0(X,mL)$
defined by  
\begin{equation*}
	\cF^\la H^0(X,mL) = \{ s\in H^0(X,mL) \, \vert \, t^{-\la} \overline{s} \in H^0(\cX, m\cL) \} 
\end{equation*}
for each $\la \in \Z$ and $m \in \N$, where  $\overline{s} \in H^0(\cX \setminus \cX_0, m\cL|_{\cX\setminus \cX_0})$ is the $\bG_m$-invariant section such that $\overline{s}\vert_{X_1}=s$. By the Rees construction, there are $\G_m$-equivariant isomorphisms of graded rings 
\begin{equation}\label{e:tc}
	\bigoplus_{m\in \N} H^0(\cX,m\cL)  \cong \bigoplus_{m \in \N} \bigoplus_{\la \in \Z} \cF^\la H^0(X, mL) t^{-\la}
\end{equation}
and
\begin{equation}\label{e:gr}
	\bigoplus_{m\in \N} H^0(\cX_0,m\cL_0)  \cong \bigoplus_{m \in \N} \bigoplus_{\la \in \Z} F^\la H^0(X,mL)/ F^{\la+1}H^0(X,mL),
\end{equation}
where the
$\G_m$-actions on the right  sides are induced by the $\Z$-grading \cite[Section 2.5]{BHJ17}. 

When $X$ is a normal variety and $\cX$ is normal,  $\cF^\bullet $ has an interpretation in terms of valuations
\cite[Section 4]{BHJ17}.
Each irreducible component $E\subset \cX_0$ induces a $\G_m$-equivariant valuation $\ord_E$ of $K(\cX)$.
Set
\[
b_E := \ord_E(\cX_0) = \ord_E(t)
\quad \text{ and } \quad 
v_E: = b_E^{-1} r(\ord_E),\]
where  $r(\ord_E)$ is the restriction of $\ord_E$ to $K(X) \subset K(X)(t) \cong K(\cX)$. 
By \cite[Lem. 4.1]{BHJ17}, $r(\ord_E)= c\,  \ord_F$, where $c\in \bZ_{\geq0}$ and $F$ is a prime divisor over $X$.
Let $\cY$ denote the normalization of the graph  of  $\cX\dashrightarrow X_{\A^1} :=X\times\bA^1 $ with natural morphisms 

\begin{equation*}
\begin{tikzcd}[row sep= 1.25 em]
& \cY   \arrow[rd,"g"] \arrow[ld,swap,"f"]  &  \\
\cX \arrow[rr,dashed] & & X_{\A^1}
\end{tikzcd}
	.\end{equation*}
Let $L_{\bA^1}:=p_1^* L$ where $p_1: X_{\bA^1}\to X$ is the projection to the first factor.
Since $f^*\cL $ and $g^*L_{\A^1}$ are canonically isomorphic over $\cY\setminus \cY_0$, there exists  a divisor $G$ supported on $\cY_0$  such that $ f^* \cL = g^*L_{\A^1} (G) $.
By the proof of \cite[Lem. 5.17]{BHJ17},
\begin{equation}\label{e:filttestconfig}
\cF^\la H^0(X,mL) =\bigcap_{E \subset \cX_0} \left\{ s\in H^0(X, mL) \, \vert \, v_E(s) + m b_E^{-1} \ord_E(G) \geq \la \right\}  
.\end{equation}

The following lemma says that the valuations $v_E$ for $E\subset \cX_0$ are the minimal set of valuations that defining the filtration.

\begin{lem}\label{l:tcReesvals}
Let $(\cX,\cL)$ be a test configuration of $(X,L)$ with $\cX$ normal and irreducible.
If $v_1,\ldots, v_\ell \in\DivVal_X$ and  $c_1,\ldots, c_\ell \in \bZ$
satisfy
\[
\cF^\la H^0(X,mL) = \bigcap_{i=1,\ldots, \ell} \{ s\in H^0(X,mL)\, \vert\, v_i(s)  + mc_i \geq \la  \}
\]
for each $m \in \bN$ and $\la \in \bZ$,
then for each irreducible component $E\subset \cX_0$ there exists $ i \in \{ 1,\ldots, \ell\}$ such that $v_i =v_E$ and $c_i = b_E^{-1} \ord_E(G)$.  
\end{lem}

\begin{proof}
Let $S: = \bigoplus_{m \in \bN}S_m := \bigoplus_{m \in \bN}  H^0(X,mL)$ and $I_\la := \cF^\la S := \bigoplus_{m \in \bN} \cF^\la S_m$ for each $\la \in \bZ$.
Note that replacing $\cL$ with $\cL(d\cX_0)$ results in replacing
$G$ with $G+ d \cX_0$, 
 $\cF^\la H^0(X,mL)$ with $\cF^{\la -md} H^0(X,mL)$,
and $c_i$ with $c_i+d$.   
Thus, after performing such a replacement with $d \gg0$,
we may assume $G$ is effective, $\cF^0 S = S$, and  $c_i \geq 0$ for all $i$.
Using that
 \[
 S \cdot \cF^\la S = \cF^0 S \cdot \cF^\la S \subset \cF^\la S
 ,\]
 we see
  $I_{\la}:= \cF^\la S$ is an ideal of $S$.
Since the algebra in \eqref{e:tc} is finitely generated, $\bigoplus_{ \la \in \bN} I_\la $ is  a finitely generated $\bk$-algebra.
 Thus there exists $\mu>0$ such that $I_{\mu}^\la = I_{\mu \la}$ for all $\la \in \bN$.

We now interpret the ideals $I_{\la}$ valuatively. 
Consider the valuation $w_i$ of $S$ defined by 
\[
w_i ( f) = \min \{ v_i(f_m) +m c_i \, \vert\, f_m\neq 0\}
,\]
where $f= \sum_{m \geq 0} f_m \in \bigoplus_{m \in \bN}S_m$. 
Similarly, for each irreducible component $E\subset \cX_0$, define
\[
w_E ( f) := \min \{ r(\ord_E)(f_m) +m b_E^{-1} \ord_E(G) \, \vert\, f_m\neq 0\}
,\]
which is also a valuation of $S$.
Note that
\[
I_{\la}=
\cF^\la S 
=   \bigcap_{i=1,\ldots, \ell} \{ f\in S \, \vert\, w_i(s) \geq \la\} 
= \bigcap_{E\subset \cX_0} \{ f\in S \, \vert\, w_E(s) \geq \la\} 
\]
for each $\la \in \bN$.
Since  $I_{\la}$ is the intersection of valuation ideals,
 $I_{\la}$ is integrally closed by \cite[Rem. 1.1.3.6 \& Ex. 6.10.7]{HS06}.

By \cite[Thm. 1.10]{BHJ17}, there exists a finite set of Rees valuations  $\{ u_j \, \vert\, j\in J\}$ of $I_\mu$. This is the unique minimal finite set of valuations of $S$  with the property that 
\[
\overline{I_{\mu}^\la}   = \bigcap_{j\in J}\{ f\in S \, \vert\, u_j(f) \geq  \la \mu \}
\]
for all $\la \in \bN$.
Since ${\overline{I^\la_\mu}}= I_{\mu}^\la = I_{\la \mu}$, $\{u_j \, \vert\, j\in J \}$ is a subset of both $\{w_i \, \vert\, 1\leq i \leq \ell\}$ and $\{w_E \, \vert\, E\subset \cX_0\}$.

To compare  $\{u_j\, \vert\, j\in J \}$ and $\{v_E \, \vert\, E\subset \cX_0\}$,
consider the map
\[
Z:= \Proj (\bigoplus_{\la \in \bN} I_\la ) \to \Spec( S) =: Y
.\]
Note that  $\Proj (\bigoplus_{\la \in \bN} I_\la )  = \Proj (\bigoplus_{\la \in \bN} I_{\mu\la} )  = \Proj (\bigoplus_{\la \in \bN} I_\mu^\la )$, which is the blowup of $Y$ along $I_\la$.  
Since $I_{\mu}^\la = I_{\mu \la}$, which is integrally closed, $Z$ is normal. 
By \cite[Thm 1.10]{BHJ17}, the  Rees valuations of $I_\mu$ are in bijection with the irreducible components of $V(I_\mu \cdot \cO_Z)$.
Observe that
\[
 V(I_\mu  \cdot \cO_Z) = \Proj( \bigoplus_{ \la \in \bN} I_{\la}/I_{\la+\mu}))
\]
On the other hand, the effective Cartier divisor $\mu \cX_0 \subset \cX$ is the $\Proj$ of 
\[
\bigoplus_{m \in \bN}  \bigoplus_{\la  \in \bZ} \cF^\la S_m / \cF^{\la +\mu } S_m 
\cong
\bigoplus_{\la \in \bZ} I_{\la}/I_{\la+\mu}
,\]
where the $\Proj$ is taken with respect to the $m$-grading.
Therefore,  $V(I_{\mu}\cdot \cO_Z)$ and $\mu\cX_0$ are $\Proj$'s of the same ring with respect to different gradings and, hence, have the same number of irreducible components. 
Therefore  $\{u_j \, \vert\, j \in J\} $ and  $\{w_E\, \vert\, E\subset \cX_0\}$  have the same cardinality  and and so  $\{w_E\, \vert\, E\subset \cX_0\} =\{u_j \, \vert j \in J\} \subset \{w_i\, \vert\, 1\leq i \leq \ell\}$. Thus, for each $E\subset \cX_0$, there exists $i$ such that $w_E=w_i$, which gives $v_E=v_i$ and $b_E^{-1} \ord_E(G)=c_i$.
\end{proof}

\subsection{Test configurations of boundary polarized CY pairs}

\begin{defn}
A \emph{weakly special test configuration} $(\cX,\Delta_{\cX}+\cD)$ of a boundary polarized CY pair $(X,\Delta+D)$ is the data of 
\begin{itemize}
\item a $\G_m$-equivariant family of boundary polarized CY pairs $\pi:(\cX,\Delta_{\cX}+\cD)\to \A^1$ and 
\item a $\bG_m$-equivariant isomorphism 
\[
(\cX,\Delta_{\cX}+\cD)_{\bA^1\setminus 0} \cong (X,\Delta+D)\times(\bA^1\setminus 0),\]
where $\bG_m$ acts on the right side as the product of the trivial and standard actions.
\end{itemize}
A test configuration  is a \emph{product} if there exists a $\bG_m$-equivariant isomorphism 
\[
(\cX,\Delta_{\cX}+\cD)\cong (X,\Delta+D)\times \bA^1,
\]
where $\bG_m$-acts on the right side as a product of a $\bG_m$-action on $(X,\Delta+D)$ and the standard action on $\bA^1$.
It is \emph{trivial} if the isomorphism above can be chosen so that the $\bG_m$-action on $(X,\Delta+D)$ is trivial.
\end{defn}

\begin{defn}
We say that there is a \emph{weakly special degeneration} 
\[
(X,\Delta+D)\rightsquigarrow(X_0,\Delta_0+D_0)
\]
if there exists a weakly special test configuration  $(\cX,\Delta_{\cX}+\cD)$ of $(X,\Delta+D)$ such that $(X_0, \Delta_0+D_0)\cong (\cX_0,\Delta_{\cX_0}+\cD_0)$.
\end{defn}

\begin{rem}[Terminology]
The term \emph{weakly special} is a reference to  \cite[Def. 2.16]{LWX21} and highlights the assumptions that the special fiber of the test configuration is slc. 
While one could allow test configurations of boundary polarized CY pairs with worse singularities, these will not appear in the paper. 
Hence, we will often refer to a \emph{weakly special test configuration}  as simply a \emph{test configuration}.
\end{rem}

\begin{rem}[Stacky intepretation]
Test configurations can be understood in terms of   maps into   the moduli stack of boundary polarized CY pairs. 
Indeed, a test configuration $(\cX,\Delta_{\cX}+\cD)\to \bA^1$ of a pair $(X,\Delta+D)$ in $	\cM(\chi,N,{\bf r})(\bk)$ is equivalent to a map $f: [\bA_{\bk}^1/\G_m]\to \cM(\chi,N,{\bf r})$ with an isomorphism $f(1) \cong [(X,\Delta+D)]$.
\end{rem}

Let $(\cX,\Delta_{\cX}+\cD)$ be a weakly special test configuration of $(X,\Delta+D)$.
If we set  
\[
L:= \omega_{X}^{[-r]}(-r\Delta)\quad \text{ and } \quad \cL:=\omega_{\cX/\bA^1}^{[-r]}(-r\Delta_{\cX}),
\] where $r\in \bZ_{>0}$ is sufficiently divisible, and endow $\cL$ with the canonical $\bG_m$-linearization, then $(\cX,\cL)$ is a test configuration of $(X,L)$ in the sense of Definition \ref{d:testconfig}.
The following lemma describes the filtration $\cF^\bullet$ of $H^0(X,mL)$.  

\begin{lem}\label{l:lctestconfig}
Let $(\cX, \Delta_{\cX}+\cD)$ be a weakly special test configuration of a boundary polarized CY pair $(X,\Delta+D)$. If $\cX$ is normal, then 
\begin{itemize}
\item[(1)] 
$\cF^\la H^0(X,mL)
=
\bigcap_{E \subset \cX_0} \{ s\in H^0(X,mL) \, \vert \, v_E(s) -mr A_{X,\Delta}(v_E)\geq \la \}$ and
\item[(2)] $A_{X,\Delta+D}(v_E) = 0$ for each irreducible component $E \subset \cX_0$. 
\end{itemize}
\end{lem}

\begin{proof}
Fix an irreducible component $E\subset \cX_0$. 
Since $(\cX,\Delta_{\cX}+\cD+ \cX_0)$ is lc, 
\[
A_{\cX,\Delta_{\cX}+\cD+ \cX_0}(E) =0=A_{\cX,\Delta_{\cX}+ \cX_0}(E) \]
and 
$\ord_{E}(\cX_0)=1$.
Now, set
$G: =- r f^*(K_{\cX/\bA^1} +\Delta_{\cX})  +r g^*p_1^*(K_{X} +\Delta)$.
Note that
\begin{align*}
G&= -rf^*(K_{\cX} +\Delta_{\cX}) +r g^*(K_{X\times\bA^1} +\Delta\times \bA^1)\\
&= r \left( K_{\cY}- f^*(K_{\cX} +\Delta_{\cX} +\cX_0)\right) - 
r \left( K_{\cY}- g^*(K_{X\times\bA^1} +\Delta-X\times 0)\right)
,
\end{align*}
where the first equality uses   $K_{\cX/\bA^1} = K_{\cX}-\pi^*K_{\bA^1}$ and $K_{X\times\bA^1} = p_1^* K_X+p_2^*K_{\bA^1}$.
Thus
\[
r^{-1} \ord_{E}(G) 
= 
A_{\cX,\Delta_{\cX}+\cX_0}(E) - A_{X \times \A^1, \Delta\times \A^1+X\times 0}(E) = -A_{X\times \A^1, \Delta\times \A^1+X\times 0}(E) 
=-A_{X,\Delta}(v_E),
\]
where the third  equality is \cite[Prop. 4.11]{BHJ17}.
Thus (1)  follows from \eqref{e:filttestconfig}.
For (2),
\[
A_{X,\Delta+ D}(v_E) 
=
A_{X\times \A^1, \Delta\times\A^1+D\times \A^1 +X\times 0}(E)  
=
A_{\cX,\Delta_{\cX}+\cD+\cX_0} (E) 
=
0
,\]
where the first equality is  \cite[Prop. 4.11]{BHJ17}
and the second is by Lemma \ref{l:CYbirational}.1.
\end{proof}

\subsection{Lc places and test configurations}
Now, we discuss the relation between lc places of complements and test configurations of slc log Fano pairs proven in \cite{BLX19,CZ21}.

Recall, the  \emph{lc places} of an lc pair $(X,\Delta)$ is 
\[
{\rm LC}(X,\Delta):= \{ v\in \DivVal_X \, \vert \, A_{X,\Delta}(v)=0 \}.
\]
If $(X,\Delta)$ is an slc pair, we write 
$
\LC(X,\Delta) := \sqcup_{i=1}^{r}\LC(\overline{G}_i+\overline{X}_i, \overline{\Delta}_i)
$,
which is the disjoint union of the lc places of the irreducible components of the normalization.

\begin{thm}[\cite{CZ21}]\label{thm:CZ-lcplace}
If $(X,\Delta+D)$ is an lc boundary polarized CY pair, then there is a bijection between the sets of: 
\begin{enumerate}
	\item Weakly special test configurations $(\cX,\Delta_{\cX}+\cD)$ of $(X,\Delta+D)$ with
	 $\cX_0$ integral.
	\item $\Z$-valued divisorial valuations $v\in \LC(X,\Delta+D)$.
\end{enumerate}
The map is given by sending $(\cX,\Delta_{\cX}+\cD)$ to the valuation $v_{\cX_0}$.
\end{thm}

While the theorem follows from \cite{CZ21}, we will instead deduce it from the following more general result, which produces a weakly special test configuration from the data of a collection of lc places.

\begin{prop}\label{p:lcplacestestconfig}
Let $(X,\Delta+D)$ be an lc boundary polarized CY pair and $v_1,\ldots, v_\ell\in \LC(X,\Delta+D)$. 
\begin{enumerate}
\item Then the $\bk[t]$-algebra
		${\rm Rees}(\cF^\bullet):= \bigoplus_{m\in \bN} \bigoplus_{\la \in \bZ}\cF^\la H^0(X,mL)t^{-\la}$
		is finitely generated,
where 
\[
\cF^\la H^0(X,mL):= \{ s\in \cF^\la H^0(X,m L) \, \vert\, v_i(s) -mr A_{X,\Delta}(v_i)\geq \la \}
\]
\item 
If we set 
$\cX:=  \Proj ({\rm Rees}(\cF^\bullet))$, then 
$(\cX,\Delta_{\cX}+\cD)$ is   a weakly special test configuration of $(X,\Delta+D)$ 
and each irreducible component $E\subset \cX_0$ satisfies $v_E= v_i$ for some $i$. 
\end{enumerate}
\end{prop}

In statement (2), $(\cX,\Delta_{\cX}+\cD)$ is defined as follows. 
The $\bZ$-grading on ${\rm Rees}(\cF^\bullet)$ induces a $\bG_m$-action on $\cX$ and the isomorphism 
\[
{\rm Rees}(\cF^\bullet)\bigotimes_{k[t]} k[t^\pm] \cong \bigoplus_{m \in \bN }\bigoplus_{\la \in \bZ} H^0(X,mL) t^{-\la}
\]
induces a $\bG_m$-equivariant isomorphism $\cX\vert_{\bA^1\setminus 0} \cong X\times (\bA^1\setminus 0)$. 
The $\bQ$-divisors $\Delta_{\cX}$ and $\cD$ then  denote the closures of $\Delta\times (\bA^1\setminus 0)$
and $D\times (\bA^1\setminus 0)$ in $\cX$.

\begin{proof}
Statement (1) holds by Theorem \ref{t:fg}, which appears in a later section, or by  \cite[pg. 14]{CZ22}.
Next, set $\cX:= \Proj ({\rm Rees}(\cF^\bullet))$ and $\cL:= \cO_{\cX}(1)$.
After replacing $r>0$ with a positive multiple, we may assume ${\rm Rees}(\cF^\bullet)$ is generated in degree $1$ and, hence, $\cL$ is a line bundle. 
By  \cite[\S 2.5]{BHJ17},  $(\cX,\cL)$ is a test configuration of $(X,L)$.
Using the valuative definition of $\cF^\bullet$ and the fact that $X$ is integral, 
we see $\cX_0$ is reduced.
Since $X$ is normal and $\cX_0$ is reduced, $\cX$ is normal by  \cite[Prop. 2.6.iv]{BHJ17}.

 Let $\cY$ denote the normalization of the graph of $\cX \dashrightarrow X_{\bA^1}$
with maps $f:\cY\to\cX$ and $g: \cY \to X_{\bA^1}$.
Let $G$ be the divisor supported on $\cY_0$ such that $f^* \cL =g^* L_{\bA^1}(G)$.
By Lemma \ref{l:tcReesvals} and the definition of $\cF^\bullet$, each irreducible component $E\subset \cX_0$ satisfies
\[
A_{X,\Delta+D}(v_E)= 0
\quad \text{ and } \quad
{\rm coeff}_E(G) = -rA_{X,\Delta} (r(\ord_E))
.\]
In addition, $v_E= v_i$ for some $i$.

To prove $(\cX,\Delta_{\cX}+\cD)$ is a weakly special test configuration of $(X,\Delta+D)$, it remains to show $(\cX,\Delta_{\cX}+\cD)\to \bA^1$ is a family of boundary polarized CY pairs.
Using Lemma \ref{l:slcadj}, it suffices to to prove (i) $(\cX,\Delta_{\cX}+\cD+\cX_0)$ is lc, (ii) $K_{\cX}+\Delta_{\cX}+\cD+\cX_0 \sim_{\bQ}0$, and (iii) $-K_{\cX}-\Delta_{\cX}$ is ample over $\bA^1$.
To verify (i) and (ii), note that
	\[
	A_{X_{\bA^1},\Delta_{\bA^1}+D_{\bA^1}+ X\times 0} (E) 
	= 
	A_{X,\Delta+D}(v_E)
	= 
	0
	,\]
	where the first equality is \cite[Prop. 4.11]{BHJ17}.
	Thus, if we define $\Gamma_{\cY}$ by the formula
	\[
	K_{\cY}+\Gamma_{\cY}=g^* (K_{X_{\bA^1}}+\Delta_{\bA^1}+D_{\bA^1}+X\times 0), 
	\]
	then $f_*\Gamma_{\cY} = \Delta_{\cX}+\cD+ \cX_0$.
	Since $(X_{\bA^1},\Delta_{\bA^1}+D_{\bA^1}+X\times 0)$ is lc and $K_{X_{\bA^1}}+
	\Delta_{\bA^1}+D_{\bA^1}+X\times 0\sim_{\bQ}0$, it follows that 
	(i) and (ii) hold.
For (iii), note that  
	\[
	{\rm coeff}_E(G)
	=
	-rA_{X,\Delta} (v_E)
	=
	-r A_{X_{\bA^1}+\Delta_{\bA^1}+X_0}(E)
	=
	-r {\rm coeff}_E( K_{\cY/X\times \bA^1}-g^*\Delta_{\bA^1})
	,\]
	where the second equality is \cite[Prop. 4.11]{BHJ17}.
	Next, fix a divisor $\cH$ on $\cX$ such that $\cL\cong \cO_{\cX}(\cH)$. 
	Since $\cH$ is ample over $\bA^1$ and
	\[
	\cH \sim  -r g_* f^*( K_{X_\bA^1}+\Delta_{\bA^1})+g_*G \sim 
	-rg_* ( K_{\cY}+ f_*^{-1} \Delta_{\bA^1}) 
	= 
	-r(K_\cX+ \Delta_{\cX})
	,\]
	(iii) holds. Therefore $(\cX,\Delta_{\cX}+\cD)$ is a weakly special test configuration.
\end{proof}

\begin{proof}[Proof of Theorem \ref{thm:CZ-lcplace}]
The map is well defined by Lemma \ref{l:lctestconfig}.2.
The map is injective, 
since a weakly special test configuration $(\cX,\Delta_{\cX}+\cD)$  of $(X,\Delta+D)$ is uniquely determined by its 
induced filtration of the section ring  \cite[Prop. 2.15]{BHJ17} and the filtration is uniquely determined by $v_{\cX_0}$ (Lemma \ref{l:lctestconfig}.1).
The map is  surjective by Proposition \ref{p:lcplacestestconfig}.
\end{proof}

We now prove two additional statements concerning test configurations of boundary polarized CY pairs and modifying the complement.

\begin{lem}\label{l:tclcplace}
Let $(X,\Delta+D)$ be an lc boundary polarized CY pair. 
If  $(\cX,\Delta_{\cX})$ is a weakly special test configuration of $(X,\Delta)$ and  $v_E \in \LC(X,\Delta+D)$ for each  irreducible component $E\subset \cX_0$, then 
$(\cX,\Delta_{\cX}+\cD)$ is a weakly special test configuration of $(X,\Delta+D)$.
\end{lem}

By a \emph{weakly special test configuration} of an lc log Fano pair as above, we mean that
$(\cX,\Delta_{\cX})\to \bA^1$ is a $\bG_m$-equivariant family of log Fano pairs,  and there is a  $\bG_m$-equivariant isomorphism 
\[
(\cX,\Delta_{\cX})\vert_{\bA^1 \setminus 0} \cong (X,\Delta)\times (\bA^1\setminus 0),
\]
where $\bG_m$-acts on the right as the product of the trivial action and the standard action.

\begin{proof}
Let $\cY$ denote the normalization of the graph of $\cX\dashrightarrow X_{\bA^1}$ with maps $f:\cY \to \cX$ and $g:\cY\to X_{\bA^1}$. 
Define a $\bQ$-divisor on $\Gamma$ by 
\[
K_{\cY}+\Gamma= g^*(K_{X_{ \bA^1}}+\Delta_{\bA^1}+D_{\bA^1}+X\times 0)
.\]
For each irreducible component $E\subset \cX_0$,
\[
{\rm coeff}_{E}(f_*\Gamma)
=
1- A_{X_{\bA^1},\Delta_{\bA^1}+D_{\bA^1} +X\times 0}(\ord_E)
=
1 - A_{X,\Delta+D}(v_E)
=1
,\]
where the second equality is \cite[Prop. 4.11]{BHJ17}. 
Thus $f_* \Gamma= K_{\cX}+\Delta_{\cX}+\cD+\cX_0$.
Since $(X_{\bA^1},\Delta_{\bA^1}+D_{\bA^1} +X\times 0)$ is lc and $K_{X_{\bA^1}}+\Delta_{\bA^1}+D_{\bA^1} +X\times 0\sim_{\bQ} 0$,
$(\cX,\Delta_{\cX}+\cD+\cX_0)$ lc and $K_{\cX}+\Delta_{\cX}+\cD+\cX_0\sim_{\bQ}0$. 
Thus $(\cX,\Delta_{\cX}+\cD)$ is a weakly special test configuration.
\end{proof}

\begin{lem}\label{l:lcDlcB}
Let $(\cX,\Delta_{\cX}+\cD)$ be a weakly special test configuration of a boundary polarized CY pair $(X,\Delta+D)$.

If $B$ is a complement of $(X,\Delta)$ satisfying $\LC(X,\Delta+D) \subset \LC(X,\Delta+B)$, then 
\begin{enumerate}
\item $(\cX,\Delta_{\cX}+\cB)$ is a weakly special test configuration of $(X,\Delta+B)$, where $\cB$ is the closure of $B\times(\bA^1\setminus 0)$ in $ \cX$ and
\item $\LC(\cX_0,\Delta_{\cX_0}+\cD_0) \subset \LC (\cX_0,\Delta_{\cX_0}+\cB_0)$.
\end{enumerate}
\end{lem}

\begin{proof}
We first prove the result when $X$ is normal.
 By Lemma \ref{l:lctestconfig} and our assumption,
 $v_E\in \LC(X,\Delta+B)$ for each $E\subset \cX_0$. 
Thus  Lemma \ref{l:tclcplace} implies (1) holds. 

To prove  (2) when $X$ is normal, fix a  log resolution  $f:\cY\to \cX$ of both  $(\cX,\Delta_{\cX}+\cD+\cX_0)$ 
and $(\cX,\Delta_{\cX}+\cB+\cX_0)$. 
Define $\bQ$-divisors $\Gamma_B$ and $\Gamma_D$ by 
\begin{align*}
K_\cY + \Gamma_{D}=
f^*(K_{\cX}+\Delta_{\cX}+\cD+\cX_0) \quad \text{ and} \quad
K_\cY + \Gamma_{B}=
f^*(K_{\cX} +\Delta_{\cX}+\cB+\cX_0).
\end{align*}
Since $\LC(X,\Delta+D) \subset \LC(X,\Delta+B)$, it follows that 
\[
\LC(X_{\bA^1},\Delta_{\bA^1}+D_{\bA^1}+X\times 0)
\subset 
\LC(X_{\bA^1},\Delta_{\bA^1}+B_{\bA^1}+X\times 0)
.\]
Thus Lemma \ref{l:CYbirational}.1 implies
$\Gamma_{D}^{=1} \leq \Gamma_{B}^{=1}$.
Now, fix an irreducible component $E\subset \cX_0$. 
Write $\overline{E}$ for its normalization and $\widetilde{E}$ for its birational transform on $\cY$.
By \cite[Prop 4.6]{Kol13},
\[
(\widetilde{E}, (\Gamma_D- \widetilde{E})\vert_{\widetilde{E}})
\to 
(\overline{E},{\rm Diff}_{\overline{E}}(\Delta_{\cX}+\cD+\cX_0))
\]
is crepant birational.
Also, the same holds with  the $D$ and $\cD$ replaced  by $B$ and $\cB$.
Hence
\begin{multline*}
\LC(\overline{E},{\rm Diff}_{\overline{E}}(\Delta_{\cX}+\cD+\cX_0))
=
\LC(\widetilde{E},(\Gamma_D^{=1}-\widetilde{E} )\vert_{\widetilde{E}})\\
\subset 
\LC(\widetilde{E},(\Gamma_B^{=1}-\widetilde{E} )\vert_{\widetilde{E}})
=
\LC(\overline{E},{\rm Diff}_{\overline{E}}(\Delta_{\cX}+\cB+\cX_0)).
\end{multline*}
Therefore (2) holds when $X$ is normal.

If $X$ is not normal, then the normalization $(\ocX,\ocG+\Delta_{\ocX}+\ocD)$ of $(\cX,\Delta_{\cX}+\cD)$ is a test configuration of $(\oX,\oG+\oDe+\oD)$. By (1) in the normal case, $(\ocX,\ocG+\Delta_{\ocX}+\overline{\cB}+\ocX_0)$ is a possibly disconnected lc CY pair. 
Hence $(\cX,\Delta_{\cX}+\cB+\cX_0)$ is an slc CY pair, which implies  (1) holds. 
For (2), note that $(\overline{\cX}_0, 
\overline{\cG}_0+\overline{\Delta}_{\cX_0}+ \overline{\cD}_0) \to (\cX_0,\Delta_{\cX_0}+\cD_0)$ is crepant and the same holds with each $\cD$ replaced by $\cB$. Thus the case when $X$ is normal implies the result.
\end{proof}

\subsection{Families of test configurations}

\begin{defn}
A \emph{family of test configurations} of a family of boundary polarized CY pairs $(X,\Delta+D)\to S$ is the data of a
\begin{enumerate}
\item $\bG_m$-equivariant family of boundary polarized CY pairs
\[
(\cX,\Delta_{\cX} \times \cD)\to \bA^1_S  
,\]
where $\bG_m$-acts on $\bA^1_S:= \bA^1\times S$ by the product of the standard action on $\bA^1$ and the trivial action on $S$ and
\item an isomorphism $(\cX,\Delta_{\cX} \times \cD)\vert_{\{1 \} \times S} \cong (X,\Delta+D)$ over $S$.
\end{enumerate}
Note that the restriction to  any $s\in S$ is naturally a test configuration of $(X,\Delta+D)$.
\end{defn}

\begin{lem}\label{l:tcextend}
Let $(X,\Delta+D)\to S$ be a family of boundary polarized CY pairs over a variety $S$, and let $K:= k(S)$ denote the function field of $S$.

If $(\cX_K,\Delta_{\cX_K}+ \cD_K)\to \bA^1_K$ is a test configuration of $(X_K,\Delta_K+D_K)$,
then there exists an open set $U\subset S$ and a family of test configurations
\[
(\cX_U,\Delta_{\cX_U} + \cD_U) \to \bA^1_U
\]
of $(X_U,\Delta_U+\Delta_U)\to U$ that extends $(\cX_K,\Delta_K +\cD_K) \to \bA^1_K$.
Additionally, if  $(\cX_K,\Delta_K +\cD_K) \to \bA^1_K$ is non-trivial, then so is $(\cX_s,\Delta_{\cX_s} + \cD_s) \to \bA^1_{k(s)}$ for each $s\in U$.
\end{lem}

\begin{proof}
Over a non-empty open set $U\subset X$, $(\cX_K,\Delta_{\cX_K}+ \cD_K)\to \bA^1_K$
extends to a projective $\bG_m$-equivariant family of pairs $(\cX_U,\Delta_{\cX_U}+\cD_U)\to \bA^1_U$. 
Since the slc condition is constructible \cite[Lem. 4.44]{KolNewBook}, after shrinking $U$, we may assume $(\cX_U,\Delta_{\cX_U}+\cD_U) \to \bA^1_U$ is a family of slc pairs. 
Further shrinking $U$, we may assume $-K_{\cX_U/\bA^1_U}-\Delta_{\cX_U}$ is relatively ample and $
K_{\cX_U/\bA^1_U}+\Delta_{\cD_U}+\cD_U\sim_{\bA^1_U,\bQ}0$. 
Hence
$(\cX_U,\Delta_{\cX_U}+\cD_U) \to \bA^1_U$ 
is a $\bG_m$-equivariant  family of boundary polarized CY pairs.
Further shrinking $U$, the $\bG_m$-equivariant isomorphism 
\[
(\cX_K,\Delta_{\cX_K}+\cD_K) \cong  (X_K,\Delta_K+D_K)\times(\bA^1\setminus 0)
\]
extends to a $\bG_m$-equivariant isomorphism 
\[
(\cX_U,\Delta_{\cX_U}+\cD_U) \cong  (X_U,\Delta_U+D_U)\times(\bA^1\setminus 0)
.\]
Hence, $(\cX_U,\Delta_{\cX_U}+\cD_U) \to \bA^1_U$ is a family of test configurations of $(X_U,\Delta_U+D_U)$. 
Additionally, if $(\cX_K,\Delta_K +\cD_K) \to \bA^1_K$ is non-trivial, then the $\bG_m$-action on the fiber over $0 \in \bA^1_K$ is non-trivial. 
Hence, after shrinking $U$, the $\bG_m$-action on the fiber of $(\cX_s,\Delta_{\cX_s} + \cD_s) \to \bA^1_{k(s)}$ over $0 \in \bA^1_{k(s)}$ will also be non-trivial.
\end{proof}

\subsection{Torus equivariant test configurations}

The following proposition shows that all test configurations of boundary polarized CY pairs with torus action are equivariant. 
See \cite[Lem. 2.18 \& 2.20]{Oda20} for related results.

\begin{prop}\label{p:TequivariantTC}
Let  $(X,\Delta+D)$ be a boundary polarized CY pair admitting an action by a torus $\T:=\bG_m^r$ for some $r\geq0$.

If $(\cX,\Delta_{\cX}+\cD)$ is a test configuration of $(X,\Delta+D)$, then it is $\T$-equivariant.
\end{prop}

In the above proposition, the term $\bT$-\emph{equivariant test configuration}
means that the $\T\times \bG_m$-action on $\cX_{\A^1\setminus 0} \cong X\times(\A^1\setminus 0)$, 
which is given by the product of the $\T$-action on $X$ and the standard $\bG_m$-action on $\A^1\setminus0$, extends to a $\bT \times \bG_m$-action on $\cX$.

\begin{proof}
We first prove the result when $X$ is normal. Fix an integer $r>0$ such that $L:=-r(K_X+\Delta)$ is a Cartier divisor.
Note that the $\bT$-action on $\cX$ induces a $\bT$-action on $H^0(X,mL)$.
We claim that the filtration  $\cF^\bullet$ of $H^0(X,mL)$ induced by $(\cX,\Delta_{\cX}+\cD)$ is $\bT$-invariant.
By the formula for $\cF^\bullet $ in  Lemma \ref{l:lctestconfig}.1, 
it suffices to show that $v_E$ is $\bT$-invariant for each  irreducible component $E\subset \cX_0$.
To verify the latter, fix a $\bT$-equivariant log resolution $f:Y\to X$ of $(X,\Delta+D)$ and a define a $\Q$-divisor $B$  by 
\[
K_{Y}+B= f^*(K_{X}+\Delta+D)
.\]
Since
$A_{Y,B}(v_E)
=
A_{X,\Delta+D}(v_E)=0$,
where the second equality is Lemma \ref{l:lctestconfig}.2, \cite[Lem. 2.3]{BLX19} implies $v_E $ is quasimonomial with respect to $(Y,\Supp(B))$.
Since $\Supp(B)$ is a $\bT$-invariant divisor on $Y$,
 $v_E$ is a $\bT$-invariant valuation.
Thus the claim holds.

Since $\cF^\bullet$ is $\bT$-invariant, there is a decomposition
$
\cF^\la H^0(X,mL) = \bigoplus_{ \mu \in \bZ^r} \cF^\la H^0(X,mL)_{\mu} $,
into  weight spaces. 
Using that
\[
\cX \cong  \Proj_{\bA^1}  \Big( \bigoplus_{m \in \bN} \bigoplus_{(\mu,\la) \in  \Z^r\times  \bZ} \cF^\la H^0(X,mL)_{\mu} t^{-\la} \Big)
,
\]
\hspace{-.2 cm}
the $\bZ^r\times\bZ$-grading  inside the $\Proj$ induces a $\bT\times\bG_m$-action on $\cX$ extending the $\bT\times\bG_m$-action on $\cX_{\bA^1\setminus 0}$.
Therefore the proposition holds when $X$ is normal.

To deduce the full result from the normal case, 
note that the $\bT$-action on $(X,\Delta+D)$ induces a $\bT$-action on $(\overline{X},\oG+\overline{\Delta}+\overline{D})$.
Additionally, the normalization $(\overline{\cX},\ocG+\overline{\Delta}_{\overline{\cX}}+\overline{\cD})$ 
 of $(\cX,\Delta_{\cX}+\cD)$  is naturally a test configuration of 
$(\overline{X},\oG+\overline{\Delta}+\overline{D})$.  By the result in the normal case,  
the $\bT\times\bG_m$-action on $\overline{\cX}\vert_{\bA^1\setminus 0}$
extends to a $\bT\times\bG_m$-action on $\overline{\cX}$.
Hence, the $\bT\times\bG_m$-action on $\cX \vert_{\bA^1\setminus 0}$ extends to a $\bT\times\bG_m$-action on $\cU:= \cX \setminus (\cX_0 \cap \cG) $, where $\cG$ is the conductor divisor on $\cX$.
Since ${\rm codim}_{\cX}(\cX\setminus \cU)\geq2$, the  natural map
\[
\Proj\Big( \bigoplus_{m\in \bN}
H^0\big(\cU, \omega_{\cX/\bA^1}^{[-m]}(-m\Delta_{\cX})\vert_\cU \big)\Big)
\to
\Proj\Big( \bigoplus_{m\in \bN} H^0\big(\cX , \omega_{\cX/\bA^1}^{[-m]}(-m\Delta_{\cX}\big)\Big)
=
\cX
\]
is an isomorphism.
Since the $\bT\times\bG_m$-action on $\cU$ induces a $\bT\times\bG_m$-action on the left hand side,   $\cX$ admits a $\bT\times\bG_m$-action extending the $\bT\times\bG_m$-action on $\cX_{\bA^1\setminus 0} $.
\end{proof}

\section{S-completeness}\label{s:Scompleteness}
The notion of S-completeness for an  algebraic stack was introduced in \cite{AHLH18} 
and is closely related to the separatedness of the stacks good moduli space when it exists. 
The goal of this section is to prove that  the moduli stack of boundary polarized CY pairs is S-complete.
\medskip 

Throughout, $R$ denotes a DVR with uniformizer $\pi$, fraction field $K$, and residue field $\kappa$.

\subsection{Defintion}\label{ss:defScom}
S-completeness is defined in terms of the stacky surface
\begin{equation*}
	\STR
	: =
	[ \Spec \left( R[s,t]/(st-\pi)\right) / \G_m  ], 
\end{equation*}
where $\G_m$ acts on $\Spec\left( R[s,t]/(st-\pi) \right)$ with weights $1$ and $-1$ on $s$ and $t$.
We write $0 \in \STR$ for the closed point defined by $s=t=0$.

To understand the geometry of $\STR$, observe that 
\[\STR \vert_{s\neq 0} 
\cong
[\Spec (R[s^{\pm1}] ) /\bG_m]
\cong
\Spec(R) \quad \text{ and } \quad	\STR \vert_{t\neq 0} \cong [\Spec (R[t^{\pm1}] ) /\bG_m]
\cong
\Spec(R).\]
Therefore 
\begin{equation}\label{e:STR-0}
\STR \setminus 0
\cong 
\Spec(R) \bigcup_{\Spec(K)} \Spec(R).
\end{equation}

Hence a morphism  $\STR  \setminus 0 \to \sM$ to a stack $\sM$ is equivalent to the data of two morphisms  $\Spec(R) \to \sM$ and $\Spec(R) \to \sM$ that agree over $\Spec(K)$. 

The loci where $s$ and $t$ vanish have the description:
\[
\STR \vert_{s= 0} 
\cong [\Spec (\kappa[t] )/ \bG_m] \cong [\bA^1_\kappa/\bG_m] 
\quad\text{ and } \quad 
\STR \vert_{t= 0} 
\cong [\Spec (\kappa[s] )/ \bG_m] \cong [\bA^1_\kappa/\bG_m] 
.\]
Thus the closed points  of   $   \STR \vert_{s\neq 0}$ and $ \STR\vert_{t\neq 0}$ specialize to $0$.

\begin{defn}[{\cite[Def. 3.7]{AHLH18}}]
A stack $\sM$ is \emph{S-complete} if for any DVR $R$ and  morphism $\STR\setminus 0 \to \sM$, there exists a unique extension $\STR \to \sM$. 
\end{defn}

To reduce notation throughout this section, we write
\[
	S: = \Spec (R[s,t]/(st-\pi) ) \quad \text{ and }\quad   S^\circ := S \setminus 0.
\]
If $\cM:= \cM(\chi,N,{\bf r})$ is the moduli stack defined in Section \ref{s:stack}, then a map $\STR\to \cM$, is equivalent to a $\bG_m$-equivariant family in $\sM(S)$. 
Hence proving that $\cM$ is S-complete is equivalent to showing that any $\bG_m$-equivariant family of boundary polarized CY pairs
\[
g^\circ: (\cX^\circ, \Delta_{\cX^\circ}+\cD^\circ)\to S^\circ
\] 
extends uniquely to a $\bG_m$-equivariant family of boundary polarized CY pairs
\[
g:(\cX,\Delta_{\cX}+\cD)\to S
.\]

\subsection{Polarized schemes }\label{ss:Scompletepol}
Following \cite[\S 3.2]{ABHLX20}, we describe the problem of extending a $\bG_m$-equivariant family of polarized schemes over $S^\circ$ to a family over $S$.

Fix a $\G_m$-equivariant family of polarized schemes
\[g^\circ :(\cX^\circ, \cL^\circ ) \to S^\circ
.
\]
By \eqref{e:STR-0}, such a family  is equivalent to the data of two families of polarized schemes
\[
(X,L) \to \Spec(R)\quad \text{and } \quad (X',L') \to \Spec(R)
\]
with an isomorphism $(X_K,L_K) \cong (X'_K,L'_K)$.\footnote{The family $(\cX^\circ,\cL^\circ) \to S^\circ$ can be recovered from the previous data by gluing $(X,L)\times \G_m \to \Spec(R)\times \G_m$  
	and $(X',L')\times \G_m \to \Spec(R) \times \G_m$ along $\Spec(K)\times \G_m$.
}
To describe when $g^\circ$ extends to a family over $S$, we consider the filtration of $H^0(X,mL)$ by $R$-modules defined by
\begin{equation}\label{e:Gfilt}
	\cF^\la H^0 (X, mL) = H^0(X, mL) \cap \pi^\la H^0(X',mL')
\end{equation}
for each $\la \in \Z$ and $m \in \N$. 
In the previous intersection, we use the isomorphism $H^0(X_K, mL_K) \cong H^0(X'_K, mL'_K)$ to view  $ H^0(X',mL')$ as an  $R$-submodule  of $H^0(X_K,mL_K)$.
By  \cite[\S3.2]{ABHLX20}, 
there is a  $\G_m$-equivariant isomorphism of $R[s,t]/(st-\pi)$-modules
\begin{equation*}
	H^0(\cX^\circ, m\cL^\circ)
	\cong 
	\bigoplus_{\la \in \Z} \cF^\la H^0(X,mL) t^{-\la}, 
\end{equation*}
where  the $\G_m$-action on the right hand side is induced by the $\Z$-grading 
and the $R[s,t]/(st-\pi)$-module structure is induced by having 
$s$ and $t$ act on $f t^{-\la} \in \cG^\la H^0(X,mL) t^{-\la}$ by 
\begin{equation*}
	s \cdot ft^{-\la} = (\pi f) t^{-\la -1} \quad \text{ and }\quad  t \cdot ft^{-\la} = f t^{-\la+1}.
\end{equation*}
The  module isomorphism induces a $\G_m$-equivariant isomorphism of $R[s,t]/(st-\pi)$-algebras
\begin{equation}\label{e:algebra}
	\bigoplus_{m \in \N} \bigoplus_{\la \in \Z} H^0(\cX^\circ , m\cL^\circ) 
	\cong 
	\bigoplus_{m \in\N} \bigoplus_{\la \in \Z} \cF^\la H^0(X,mL) t^{-\la}
\end{equation}
If the algebras are finitely generated, we set  
$ \cX: = {\Proj} \left( \bigoplus_{m \in\N} \bigoplus_{\la \in \Z} \cF^\la H^0(X,mL) t^{-\la} \right) $.

\begin{prop}\label{p:extendS}\cite[\S 3.2]{ABHLX20}
	If the algebra in \eqref{e:algebra} is finitely generated, 
	then  
	\begin{equation*}
		{g}:
		\big({\cX}, \cO_{{\cX}}(m)\big) 
		\to
		S
	\end{equation*} 
	is the unique extension of $(\cX^\circ,m\cL^\circ) \to S^\circ$ to a $\G_m$-equivariant family of polarized schemes for all $m>0$ sufficiently divisible.
\end{prop}

In the case when $\cX^\circ$  is normal, the filtration $\cF^\bullet$ has a valuative interpretation.
Let $Y$ denote  the normalization of the graph  of $X\dashrightarrow X'$ with induced maps
\begin{equation*}
\begin{tikzcd}[row sep= 1.25 em]
 & Y   \arrow[rd,"h'"] \arrow[ld,swap,"h"]  &  \\
 X \arrow[rr,dashed] & & X'
     \end{tikzcd}.
\end{equation*}
Since $(X,L)$ and $(X',L')$ are isomorphic over $\Spec(K)$, we may write 
$h'^*L' = h^*L(G)$,
where $G$ is a divisor supported on the special fiber $Y_\kappa$. 

\begin{lem}\label{l:filtval}
The filtration $\cF^\bullet$ satisfies
\begin{equation*}
	\cF^\la H^0(X,mL) = \bigcap_E \left\{ f \in H^0(X,mL) \, \vert \, \ord_E(f)+m \cdot \ord_E(G) \geq \ord_{E}(\pi) \la \right\}
	,\end{equation*}
where $E$ runs through all irreducible components $ Y_\kappa$ such that $h'(E)$ is a divisor.
\end{lem}

\begin{proof}
By definition,  $f\in \cG^\la H^0(X,mL)$ if and only if $f \pi^{-\la} \in H^0(X',mL')$. 
Since $X'$ is normal and ${\rm codim}_{X}(h'(\Exc(h'))\geq 2$, the latter holds if and only if $f \pi^{-\la} \in H^0(U, L'\vert_U)$, where $U':= Y\setminus \Exc(h')$.
Using that $h'^*L' = h^*L(G)$, the latter is equivalent to 
\[
\ord_E(f \pi^{-\la}) + m \cdot  \ord_E(G) \geq 0
\]
for each irreducible component $E\subset Y_\kappa$ such that $h'(E)$ is a divisor. 
Since $\ord_E( f \pi^{-\la})= \ord_{E}(f)- \la \ord_{E}(\pi)$, the result follows.
\end{proof}

\subsection{Boundary polarized CY pairs}\label{ss:ScompleteFano}
Fix a $\bG_m$-equivariant family of  boundary polarized CY pairs
\[
g^\circ :(\cX^\circ,\Delta_{\cX^\circ }+\cD^\circ) \to S^\circ.
\] 
By \eqref{e:STR-0}, $g^\circ$ is equivalent to the data of
 two families of boundary polarized CY pairs
\[
(X,\Delta+D) \to \Spec(R) 
\quad \text{and } \quad
(X',\Delta'+D') \to \Spec(R)
\] 
with an isomorphism $(X_K,\Delta_K+D_K) \cong (X'_K,\Delta'_K+D'_K)$.

\begin{thm}\label{t:Scomplete}
If $R$ is essentially of finite type over $\bk$, then $g^\circ$ extends  uniquely to a $\G_m$-equivariant family of   boundary polarized CY pairs
${g}: ({\cX} ,\Delta_{{\cX}}+{\cD})  \to S$.
\end{thm}

The essentially of finite type assumption is needed to apply Theorem \ref{t:fg}, which relies on results from the MMP. 
To prove the theorem, we first verify the case when both $X$ and $X'$ are normal  by showing that the filtration in Section \ref{ss:Scompletepol}  is finitely generated.
We  then use Koll\'ar's gluing theory for slc pairs to prove the full result.

To proceed, fix a positive integer $r$ such that
$
L:=  -r(K_{X}+\Delta)$ and 
$L':=-r(K_{X'}+\Delta')$
are Cartier divisors. 
Let $\cL^\circ:=-r(K_{\cX^\circ/S^\circ} +\Delta_{\cX^\circ})$. Write $\cF^\bullet$ for the filtration of $H^0(X,mL)$ defined in \eqref{e:Gfilt}.

\begin{prop}\label{p:finitegen}
If $X$ and $X'$ are normal, then  $\bigoplus_{(m,\la)\in \N\times \Z}  \cF^\la H^0(X,mL)t^{-\la}$ is a finitely generated $R[s,t]/(st-\pi)$-algebra.
\end{prop}

\begin{proof}
Let $Y$ denote the normalization of the graph  of $X\dashrightarrow X'$, and consider the induced morphisms $h:Y\to X$ and $h':Y\to X'$.
Let $E_1,\ldots, E_\ell$ denote the birational transforms on $Y$ of the irreducible components of $X'_\kappa$. 
We claim that 
\begin{enumerate}
	\item[(i)] $\cF^\la H^0(X,mL) =\bigcap_{i=1}^\ell \left\{ s\in H^0(X,mL) \, \vert \, \ord_{E_i}(s) \geq \la + m rA_{X,\Delta+X_\kappa}(E_i) \right\}$
	\item[(ii)] $A_{X,\Delta+D+X_\kappa}(E_i)=0$ for  $i=1,\ldots, \ell$. 
\end{enumerate}
Assuming the claim holds, then Theorem \ref{t:fg}, which is proven in the next section, immediately implies the relevant  algebra is finitely generated.

To verify the claim, note that $(X',\Delta'+D'+X'_\kappa)$ is lc. 
Thus
\[
A_{X',\Delta'+D'+X'_\kappa}(E_i) =0=  A_{X',\Delta'+X'_\kappa}(E_i) 
\]
and $\ord_{E_i}(\pi) = \ord_{E_i} (X'_\kappa) =1$. 
Set $G:= h'^*L'- h^*L$. 
Since
\[
\ord_{E_i}(G) =
rA_{X',\Delta'+X'_\kappa}(E_i)- rA_{X,\Delta+X_\kappa} (E_i) = 0 - rA_{X,\Delta+X_\kappa} (E_i),
\]
Lemma \ref{l:filtval} implies (i) holds.
For (ii), observe that   
\[
A_{X,\Delta+D+X_\kappa}(E_i) =A_{X',\Delta'+D'+X'_\kappa}(E_i)= 0
\]
where the first equality holds by Lemma \ref{l:CYbirational}.1.
\end{proof}

\begin{proof}[Proof of Theorem \ref{t:Scomplete}  when $\cX^\circ$ is normal]
Consider the morphism
\begin{equation*}
	{g}: \cX : = \Proj  \Big( \bigoplus_{m\in \N} \bigoplus_{\la \in \Z}  \cG^\la H^0(X,mL)t^{-\la}  \Big) \to {S}.
\end{equation*}
By  Propositions \ref{p:extendS} and \ref{p:finitegen},  
$(\cX, \cO_{\cX}(m)) \to {S}$ is the unique extension of $(\cX^\circ,m\cL^\circ)\to S^\circ$ to a flat family of polarized schemes for $m>0$ sufficiently divisible. 
 Let $\Delta_{\cX}$ and $\cD$ denote the  closures of $\Delta_{\cX^\circ}$ and $\cD^\circ$ under the embedding $j: \cX^\circ \hookrightarrow \cX$.

We will now show
$(\cX, \Delta_{\cX} +\cD)\to S$ is a family of  boundary polarized CY pairs.
First, we verify that ${\cX}$ is normal. For $m>0$ sufficiently divisible,  the natural map $H^0(\cX, \cO_{\cX}(m)) \to H^0(\cX, j_* j^* \cO_{\cX}(m))$ is an isomorphism, since
\[
H^0(\cX, \cO_{\cX}(m)) = \bigoplus_{\la \in \mathbb{Z}}  \cG^\la H^0(X,mL)t^{-\la} \cong 
H^0(\cX^\circ, m\cL^\circ) 
\cong H^0(\cX, j_* j^* \cO_{\cX}(m))
,\]
where the second isomorphism is \eqref{e:algebra}. 
Therefore the natural map $ \cO_{\cX} \to j_* j^*\cO_{\cX}$ is an isomorphism as well. Since $ \cO_{\cX} \cong j_* j^*\cO_{\cX}$ and $\cX^\circ$ is normal, $\cX$ is normal. 
Next, recall that 
\[
\cO_{\cX^\circ}(-mr(K_{\cX^\circ/S^\circ}+\Delta_{\cX^\circ}))
\cong
\cO_{\cX^\circ }(m)
\quad \text{ and }
\quad
K_{\cX^\circ} +\Delta_{\cX^\circ} + \cD^\circ \sim_{S^\circ,\bQ}0
.\]
Since $\cX$ is normal
and  ${\rm codim}_{\cX}(\cX_0) =2$, 
\[
\cO_{\cX}(-mr(K_{\cX/S}+\Delta_{\cX}))
\cong
\cO_{\cX}(m)
\quad \text{ and } \quad
K_{\cX} +\Delta_{\cX} + \cD \sim_{S,\bQ}0
.\]
Finally, we show $(\cX,\Delta_{\cX}+\cD)\to S$ is a family of slc pairs. 
Consider the family 
\begin{equation*}
	(\tcX,  \Delta_{\tcX} + \tcD) := (X,\Delta+D) \times_{\Spec(R)} S
\end{equation*}
The isomorphisms $\tcX_{t\neq 0} \cong X\times \bG_m \cong \cX_{t\neq0}$
induce a birational map $\tcX \dashrightarrow \cX$. 
Over $t=0$ there may be indeterminacy, since
\begin{equation}\label{e:snot01}
\tcX_{s\neq 0} \cong X\times \bG_m \quad \text{ and  } \quad 
 \cX_{s\neq 0}\cong X' \times \G_m
	.\end{equation}
Let $\cY$ denote the normalization of the  graph of $\tcX \dashrightarrow \cX$ and write $\tilde{f}:\cY\to \tcX$ and $f:\cY\to \cX$ for the induced morphisms.
Set
\[
\Gamma := \tilde{f}^*(K_{\tcX} + \tcDe + \tcD + \tcX_{st=0})  - {f}^*(K_{\cX} + \Delta_{\cX} + \cD + \cX_{st=0}) 
.\] 
We claim  (i) $f_*\Gamma=0$ and (ii) $\Gamma \equiv_{f}0$. 
Assuming the claim holds,  the Negativity Lemma implies $\Gamma=0$. 
For (i), note that $\Supp ( \Gamma )\subset \cY_{t=0}$, since  $\tcX \dashrightarrow \cX$ is an isomorphism over $t\neq0$. 
Additionally $\Supp(\Gamma) \subset \cY_{s=0}$ by  \eqref{e:snot01} and Lemma \ref{l:CYbirational}. 
Therefore $\Supp (f_*\Gamma) \subset \cX_0$. 
Since ${\rm codim}_{\cX}( \cX_0)=2$, (i) holds. 
Since $K_{\tcX} + \tcDe + \tcD\sim_{\Q,S}0$, (ii) also holds. 
Using that $(\tcX, \tcDe +\hcD+ \tcX_{st=0} )$ is lc and $\Gamma=0$, the pair $(\cX, \Delta_{\cX} +\cD+ \cX_{st=0})$ is lc. Therefore
$(\cX, \Delta_{\cX} + \cD) \to S$ is a family of boundary polarized CY pairs by Lemma \ref{l:slcadj}. 
The family is $\bG_m$-equivariant by \cite[Lem. 2.16]{ABHLX20} and the  unique extension 
by Lemma \ref{l:CYbirational}.3.
\end{proof}

We will now deduce  Theorem \ref{t:Scomplete} from  the case when $\cX^\circ$ is normal.

\begin{proof}[Proof of Theorem \ref{t:Scomplete}]
Let $(\overline{\cX^\circ}, \overline{\cG^\circ}+\Delta_{\overline{\cX^\circ}} +\overline{\cD^\circ})$ denote the normalization of $(\cX^\circ, \Delta_{\cX^\circ} +\cD^\circ)$ and $\tau^\circ: \overline{\cG^\circ}^n \to \overline{\cG^\circ}^n$  the induced involution.
We claim   that
\begin{enumerate}
	\item 
	$\overline{g}^\circ:(\overline{\cX^\circ}, \overline{\cG^\circ}+\Delta_{\overline{\cX^\circ}} +\overline{\cD^\circ})\to S^\circ$ 
	extends to a $\bG_m$-equivariant family of  boundary polarized CY pairs 
	$\overline{g}:(\overline{\cX}, \overline{\cG}+\Delta_{\overline{\cX}} +\overline{\cD})  \to S$, and

	\item   $\tau^\circ$ extends to an involution  $\tau:\overline{\cG}^n \to \overline{\cG }^n$.
\end{enumerate}
Since each connected component of 
$\overline{g}^\circ:(\overline{\cX^\circ}, \overline{\cG^\circ}+\Delta_{\overline{\cX^\circ}} +\overline{\cD^\circ}) \to S$
is a $\bG_m$-equivariant family of boundary polarized CY pairs by Lemma \ref{l:familyslcnormadj}.1,  
the normal case of Theorem \ref{t:Scomplete} implies (1).
Next observe that Lemma \ref{l:familyslcnormadj}.2 and adjunction imply
 $( \overline{\cG}^n,  {\rm Diff}_{\overline{\cG}^n} ( \Delta_{\overline{\cX}} +\overline{\cD})) \to S$ is a family of  boundary polarized CY pairs.
Thus   (2) holds by Lemma \ref{l:CYbirational}.2. 

Since (1) and (2) hold, Proposition \ref{p:bpcygluing} implies that $g^\circ: (\cX^\circ, \Delta_{\cX^\circ}+\cD^\circ)\to S^\circ$ extends to a family of boundary polarized CY pairs $(\cX,\Delta_{\cX}+\cD) \to S$. 
The extension is $\bG_m$-equivariant by \cite[Lem. 2.16]{ABHLX20} and unique 
by Lemma \ref{l:CYbirational}.3.
\end{proof}

\subsection{Finite generation}
In this section, we prove the finite generation of certain bigraded algebras that appear when verifying the moduli stack of boundary polarized CY pairs is S-complete and $\Theta$-reductive.

To state the result, we fix the data of
a family of  boundary polarized CY pairs
\[
(X,\Delta+D) \to \Spec(R)
\] 
such that $X$ is normal,
valuations $v_1,\ldots, v_\ell$, where $v_i:= b_i \ord_{E_i}$,    $b_i\in \bQ_{\geq 0}$, and $E_i$ is a  divisor over $X$, and a positive integer $r>0$ such that  $L:=-r(K_{X}+\Delta)$ is a Cartier divisor.

The above data  induces a filtration $\cF^\bullet$ of $H^0(X,mL)$  by $R$-submodules defined by 
\begin{equation*}
	\cF^\la H^0(X,mL) : =  \bigcap_{i=1,\ldots, \ell} \{ s\in H^0(X,mL) \, \big\vert \,   v_i(s) \geq  \la + mr A_{X,\Delta+X_\kappa}(v_i)\}
\end{equation*}
for each $\la\in \Z$ and $m \in \N$. 
The \emph{Rees algebra} of $\cF^\bullet$ is
\begin{equation*}
	{\rm Rees} (\cF): = 
	\bigoplus_{(m,\la) \in \N\times \Z} \cF^\la H^0(X,mL) \subset \bigoplus_{(m,\la)\in \N\times \Z} H^0(X,mL),
\end{equation*}
which has the structure of a graded $R$-algebra.
The following theorem gives a criterion for when the algebra is finitely generated.

\begin{thm}\label{t:fg}
	For $\cF^\bullet$ as above, if  $R$ is essentially of finite type over $\bk$ and $A_{X,\Delta+D+X_\kappa}(E_i)=0$ for all $i=1,\ldots,\ell$,
	then ${\rm Rees}(\cF)$
	is a finitely generated $R$-algebra.
\end{thm}

A less general finite generation result was shown in \cite[Section 5]{BX19}
to prove the uniqueness of K-polystable degenerations  and the S-completeness of the moduli stack of K-semistable Fano varieties \cite{ABHLX20}.
Theorem \ref{t:fg} is different in that the special fiber is lc, but not necessarily klt,  and we allow multiple divisorial valuations.
To prove  Theorem \ref{t:fg}, we first relate the  finite generation of ${\rm Rees}(\cF)$ to a local finite generation result over the relative cone over $(X,\Delta+D)$ with respect to the polarization $L$ and then apply results on the Minimal Model Program for lc pairs.

\medskip

\subsubsection{Cone construction}\label{ss:coneconst}
Let $(Y,\Gamma+G) \to \Spec(R)$ denote the relative cone over $(X,\Delta+D)$ with respect to the polarization $L$. 
By this we mean, $Y := \Spec(V)$, where 
\begin{equation*}
	V 
	:= 
	\bigoplus_{m \in \N} V_m 
	:= 
	\bigoplus_{m \in \N} H^0(X,mL)
	.
\end{equation*}
There is a natural morphism  $Y\to \Spec(R)$ induced by the $R$-algebra structure on $V$ and it admits a section $\sigma: \Spec(R) \to Y$ given by the  cone points. 
The $\Q$-divisor  $\Gamma+G$ is defined as the closure of the pullback of $\Delta+D$ under the morphism 
\begin{equation*}
	Y \setminus \sigma\left(\Spec(R)\right) 
	\cong 
	X\times (\A^1\setminus 0) 
	\overset{{\rm pr}_1}{\longrightarrow}
	X
	.
\end{equation*}
Since $(X,\Delta+D+X_\kappa)$ is lc and $K_{X}+\Delta+D +X_{\kappa} \sim_{\bQ} 0$, a relative version of \cite[Lem. 3.1]{Kol13} implies
$(Y,\Gamma+G+Y_\kappa)$ is lc. 
Next, fix   an integer  $N>  r A_{X,\Delta+X_\kappa}(v_i)$ for each $i=1,\ldots, \ell$.

\begin{lem}\label{l:coneconst}
	For each $i=1,\ldots,\ell$, there is a valuation $w_i$ on $Y$ such that 
	\[w_{i} ( f) := 
	\min \left\{  v_i( f_m) + \left(N- r A_{X,\Delta+X_\kappa}(v_i)\right) m \,\, \vert \,\, f_m \neq 0 \right\},
	\]
	for all  $f \in V$, where $f= \sum_{m \geq 0} f_m$ and $f_m \in V_m$.
	Furthermore, 
	\begin{enumerate}
		\item $w_{i}= c_i \ord_{F_i}$ for some $c_i\in \Q_{>0}$ and prime divisor $F_i$ over $Y$, and
		\item  $A_{Y,\Gamma+G+ Y_\kappa}(w_{i}) =0$.
	\end{enumerate}
\end{lem}

\begin{proof}
This result follows from arguments similar to \cite[Section 2.5.1]{BX19}.
Let $\mu:W\to X$ be a log resolution of $(X,\Delta+D+X_\kappa)$ such that each $E_i$ appears as a divisor on $W$.  Consider the natural birational morphisms
\[
W_L \to X_L \to Y
,\]
where $W_L := \Spec_W( \oplus_{m \in\bN} \cO_{W}(m\mu^*L))$ and $X_L : = \Spec_X( \oplus_{m\in \bN} \cO_X(mL))$, which are the total spaces of the line bundles corresponding to the invertible sheaves $\cO_W(-\mu^*L)$ and  $\cO_X(-L)$. 
Let  $\widehat{E}_i$ denote the preimage of $E_i$ under the morphism $W_L\to W$ and $\hat{W}\subset W_L$ denote the zero section. 
By a computation similar to \cite[\S 2.5.1]{BX19}, $w_i$ is the quasi-monomial valuation with weights $1$ and $N- r A_{X,\Delta+X_\kappa}(v_i)$ along $\widehat{E}_i$ and $\widehat{W}$. 
Thus $w_i = c_i \ord_{F_i}$, where $c_i \in \bQ_{>0}$ and  $F_i$ is a prime divisor on a  weighted blowup of $Y_L$ along $\widehat{E}_i$ and $\widehat{Y}$.
For (2), note that
\[
A_{Y,\Gamma+G+Y_\kappa}(\widehat{E}_i) = A_{ (X,\Delta+D+X_\kappa)\times \bA^1}(E\times \bA^1) =A_{X,\Delta+D+X_\kappa}(E_i)=0
.
\] 
Additionally, if we let $\widehat{X} \subset X_L$ denote the zero section, then
\[
A_{Y,\Gamma+G+Y_{\kappa}}(\widehat{Y}) = A_{Y,\Gamma+G+Y_{\kappa}}(\widehat{X}) =0
,\]
where the final equality can be computed by restricting to the fiber over $\Spec(K)$
and applying the computation in the proof of \cite[Lem. 3.1]{Kol13}.
Since $w_i$ is a quasi-monomial  combination of $\widehat{E_i}$ and $\widehat{X}$, 
we conclude (2) holds.
\end{proof}

\subsubsection{Finite generation}

\begin{proof}[Proof of Theorem \ref{t:fg}]
We use the notation of Lemma \ref{l:coneconst}.  Since $A_{Y,\Gamma+G+Y_\kappa}(F_i) =0$ for each $i$,   \cite[Cor. 1.38]{Kol13}
implies that there exists a proper birational morphism $g: Z\to Y$
such that $Z$ is $\Q$-factorial,  $F_1,\ldots, F_\ell$ appear as prime divisors on $Z$, $(Z,B)$ is dlt where $B:= g_*^{-1}(\Gamma+G+Y_\kappa)+ \sum_{E\subset \mathrm{Exc}(g)}   E $, and  $K_{Z} + B = g^*(K_Y+ \Gamma+G+Y_\kappa)$.
Now, set
$F: = \sum_{c_i \neq 0}  c_i^{-1} F_i$.
	
\medskip 
	
\noindent \emph{Claim 1}: 	The $\cO_Y$-algebra
$\bigoplus_{p \in \N} g_* \cO_Z\left( \lfloor -pF \rfloor \right)$
is finitely generated. 
	
\noindent \emph{Proof:} Since $A_{Y,\Gamma+G+Y_\kappa}(F_i)=0$,  ${\rm coeff}_{F_i} (B) =1$ for $i=1,\ldots, \ell$. 
Therefore we may find  an integer $d>0$ such that $B- \frac{1}{d} F$ is effective. 
Since $(Z,B)$ is $\Q$-factorial dlt and $K_{Z}+B \sim_{\Q,g} 0$, \cite[Cor. 1.6]{HX13} 
implies that $(Z, B- \tfrac{1}{d} F)$ admits a good minimal model. 
Therefore 
\[
\bigoplus_{p \in \N} g_* \cO_{Z} (\lfloor  p d( K_Z + B- \tfrac{1}{d} F) \rfloor ) 
\]
is  a finitely generated $\cO_Y$-algebra.
Since $d(K_Z + B- \tfrac{1}{d} F)\sim_{\Q,g} - F$,
we have the $\cO_Y$-algebra $\bigoplus_{p \in \N} g_* \cO_{Z} (\lfloor  - p F \rfloor )$ is finitely generated as well.
\qed
\medskip
	
\noindent \emph{Claim 2}: For each $p \in \N$,  
$
g_* \cO_Z\left( \lfloor -pF \rfloor \right)(Y)   
\cong 
\bigoplus_{m \in \N} \cF^{p -Nm} H^0(X,mL) 
$. 
	
\noindent \emph{Proof}: First, observe that
$g_*\cO_Z\left( \lfloor -pF \rfloor \right)(Y) \subset \cO_Y(Y)$, since $F$ is effective and $Y$ is normal,
and $\cO_Y(Y) =  \bigoplus_{m \in \N} H^0(X,mL)$. 
Now, for $f = \sum_{m \geq 0} f_m \in V$,
$f  \in 
g_* \cO_Z\left( \lfloor -pF \rfloor \right)(Y) $ if and only if 
${\rm div}_{Y}(g^* f) + \lfloor - p F \rfloor $ is effective.
The latter is equivalent to the condition that 
$c_i \ord_{F_i} (f)  \geq p  $ for each $i=1,\ldots, \ell$, 
which holds if and only if
\[
{v_i}(f_m) +m rA_{X,\Delta+D +X_\kappa }(v_i ) \geq p -Nm  
\]
for each  $i=1,\ldots, \ell$.
The result now follows. \qed
\medskip
	
The theorem follows easily from the previous two claims.
Indeed, consider the monoids 
\[
S_1 : = \{ (m, \la ) \in \N\times \bZ \, \vert \, \la \leq -Nm \}
\quad \text{ and } \quad 
S_2 : = \{ (m, \la ) \in \N\times \bZ \, \vert \, \la \geq -Nm \} .
\]
Since ${\rm Rees}(\cF) = S_1 \cup S_2$, it suffices to show 
$\bigoplus_{(m,\la) \in S_i} \cF^\la H^0(X,mL)$ is finitely generated
for  $i=1,2$.
To verify the latter statement, first note that 
\[
\bigoplus_{  (m,\la) \in S_2 } \cF^{\la } H^0(X,mL)
\cong 
\bigoplus_{(m,\la) \in \N\times \N} \cF^{\la - Nm} H^0(X,mL)
\]
and the latter $R$-algebra is finitely generated by Claims 1 and 2. Next, note that 
\[
\bigoplus_{  (m,\la) \in S_1 } \cF^{\la } H^0(X,mL)
\cong 
\bigoplus_{(m,\la) \in S_1 } H^0(X,mL)
.\]
Since $\bigoplus_{ (\la,m) \in \N\times \Z} H^0(X,mL)$ is a finitely generated $R$-algebra by the assumption that $L$ is  ample over $\Spec(R)$, and $S_1 \subset \N\times \Z$ is a finitely generated sub-monoid,  \cite[Lem. 4.8]{ELMNP06} implies  
$
\bigoplus_{ (m,\la)\in S_1} H^0(X,mL)
$ is finitely generated.  Therefore ${\rm Rees}(\cF)$ is finitely generated.
\end{proof}

\subsubsection{Additional result}
The above argument also implies the following finite generation result for boundary polarized CY pairs, rather than families of such pairs over $\Spec(R)$.

\begin{thm}
Let $(X,\Delta+D)$ be an lc boundary polarized CY pair and  $r$ a positive integer such that  $L:= -r(K_X+\Delta)$ is a Cartier divisor.
Fix  $v_1,\ldots, v_{\ell }\in \DivVal_X$ and set
\[
\cG^\la H^0(X,mL):= \{ s\in H^0(X,mL)\, \vert\,  v_i(s) - m r A_{X,\Delta}(v_i) \geq \la \}
\]

If $A_{X,\Delta+D}(v_i)=0$ for $i =1,\ldots, \ell$, then the $\bk[t]$-algebra
$
\bigoplus_{(m,\la)\in \bZ } \cG^\la H^0(X,mL)t^{-\la}$ 
is finitely generated.
\end{thm}

\begin{proof}
The proof is  essentially the same  as the proof of Theorem \ref{t:fg}, but  simpler since we work on the cone over $X$ with respect to $L$, rather than the relative cone over $\Spec(R)$.  
\end{proof}

\section{Theta-reductivity}\label{s:Thetared}

The notion of $\Theta$-reductivity for an algebraic stack was introduced in \cite{HL14}
and is a necessary condition for the existence of a good moduli space \cite{AHLH18}. 
The goal of this section is to prove that the moduli stack of boundary polarized CY pairs is $\Theta$-reductive.
After proving the result, we discuss S-equivalence classes and stability.
\medskip

Throughout, 
 $R$ denotes a DVR with uniformizer $\pi$, fraction field $K$,  and residue field $\kappa$. 

\subsection{Definition}\label{ss:deftheta}
The definition of $\Theta$-reductivity  involves the stacky surface
\[
\Theta_R = \Theta\times \Spec \,R\cong  [\Spec(R[t]) /\bG_m],
\]
where $\Theta:= [\bA^1/\bG_m]$.
We write $0_\kappa \in \Theta_R$ for the unique closed point, which is defined by the vanishing of $t$  and a uniformizing parameter $\pi \in R$. 
Note that 
\[
\{ t\neq 0 \}\cong [\Spec(R[t^{\pm 1}])/\bG_m]  \cong \Spec(R)
\quad \text{ and } \quad
\{ \pi \neq 0\} \cong [\Spec(K[t]) / \bG_m]  \cong \Theta_K.
\]
Therefore a map $\Theta_{R}\setminus 0_\kappa \to \cM$, where $\cM$ is a stack, is equivalent to the data of maps $\Spec(R) \to \cM$ and $\Theta_K \to \cM$ with an isomorphism of there restrictions to $\Spec(K)$. 
\begin{defn}
An algebraic stack $\sM$ is \emph{$\Theta$-reductive} if for any DVR $R$ and 
map $\Theta_{R} \setminus 0_\kappa \to \sM$, there is a unique extension $\Theta_{R} \to \sM$.
\end{defn}

If $\cM:= \cM(\chi,N,{\bf r})$ is the moduli stack in Section \ref{s:stack}, then a map $\Theta_R \to \cM$ is equivalent to a $\bG_m$-equivariant family in $\cM(\bA^1_R)$. Hence, proving that $\cM$ is $\Theta$-reductive is equivalent to showing that any $\bG_m$-equivariant family of boundary polarized CY pairs
\[
g^\circ : (\cX^\circ, \Delta_{\cX^\circ}+\cD^\circ) \to \bA^1_R \setminus 0_\kappa 
\]
extends uniquely to a 
$\bG_m$-equivariant family of boundary polarized CY pairs
\[
g: (\cX, \Delta_{\cX}+\cD) \to \bA^1_R. 
\]

\subsection{Polarized schemes}\label{ss:Thetapol}
Following  \cite[\S 5.2]{ABHLX20}, we describe a criterion for when a $\bG_m$-equivariant family of polarized schemes over $\bA^1_R\setminus 0_\kappa$ extends to a family over $\bA^1_R$. 
Fix a family of $\G_m$-equivariant family of polarized schemes 
\[
g^\circ :(\cX^\circ, \cL^\circ)\to \A^1_R \setminus 0_\kappa. 
\]
By restricting over $1 \times \Spec(R)$ and $\A^1_K$, we get 
\begin{itemize}
\item[(1)] a family of polarized schemes $(X,L)\to \Spec(R)$ and 
\item[(2)] a test configuration $(\cX_K,\cL_K) \to \bA^1_K$ of $(X_K,L_K)$. 
\end{itemize}

Let $\cF_K^\bullet $  denote the filtration of $H^0(X_K, mL_K)$ induced by the test configuration $(\cX_K,\cL_K)$.
We define a filtration $\cF^\bullet $ of $H^0(X, mL) $ by $R$-submodules  by
\[
\cF^\la H^0(X,mL) : = H^0(X,mL) \cap  \cF_K^\la  H^0(X_K,mL_K)
\]
for each $m \in \N$ and $\la \in \Z$. 
As explained in \cite[\S 5.2]{ABHLX20}, there is a natural $\G_m$-equivariant isomorphism of graded $R[t]$-algebras 
\begin{equation}\label{e:thetaalgebra}
\bigoplus_{m \in \N} H^0(\cX^\circ, m\cL^\circ) \cong \bigoplus_{m \in \N} \bigoplus_{\la \in \Z} \cF^p H^0(X,mL) t^{-\la},
\end{equation}
where the $\G_m$-action on the right is induced by the $\Z$-grading.
If the algebra is finitely generated, we set $\cX: = \Proj( \bigoplus_{m \in \N} \bigoplus_{\la \in \Z} \cF^\la H^0(X,mL) t^{-\la})$.

\begin{prop}\label{p:extendtheta}
If the $R$-algebra in \eqref{e:thetaalgebra} is finitely generated, then 
\[
g: (\cX, \cO_{\cX}(m) ) \to \A^1_R
\]
is the unique extension of $(\cX^\circ, m\cL^\circ)\to \A^1_R \setminus 0_\kappa$ to a $\G_m$-equivariant family of polarized schemes over $\A^1_R$ for all $m>0$ sufficiently divisible. 
\end{prop}

\subsection{Boundary polarized CY pairs}\label{ss:ThetaFanoComp}
Fix a $\G_m$-equivariant family of  boundary polarized CY pairs
\[
g^\circ: (\cX^\circ, \Delta_{\cX^\circ}+ \cD^\circ) \to \A^1_R \setminus 0_\kappa. 
\]
By restricting over $1 \times \Spec(R)$ and $\A^1 \times \Spec(K)$, we get 
\begin{itemize}
\item a family of  boundary polarized CY pairs $(X,\Delta+D) \to \Spec(R)$ and 
\item a test configuration $(\cX_K,\Delta_{\cX_K}+\cD_K)$ of $(X_X,\Delta_K+D_K)$. 
\end{itemize}
The following theorem is the main result of this section.

\begin{thm}\label{t:Thetared}
If $R$ is essentially of finite type over $\bk$, then  $g^\circ$ admits a unique extension 
\[
g: (\cX, \Delta_{\cX}+ \cD) \to \A^1_R 
\]
to  a $\G_m$-equivariant family of  boundary polarized CY pairs over $\bA^1_R$.
\end{thm}

The strategy to prove the Theorem \ref{t:Thetared} is similar to that of Theorem \ref{t:Scomplete}. 
We first prove the theorem when $\cX^\circ$ is normal by using  Theorem \ref{t:fg} to prove the filtration in  Section \ref{ss:Thetapol} is finitely generated. 
We  then use Koll\'ar's gluing theory for slc pairs to deduce the full result.

To proceed, fix a positive integer $r$ such that 
\[
L: =-r(K_{X} +\Delta) \quad \text{ and } \quad \cL_K: = -r(K_{\cX_K} +\Delta_{\cX_K}), 
\] 
are Cartier divisors. 
The test configuration $(\cX_K,\Delta_{\cX_K}+\cD_K)$ induces a filtration $\cF_K^\bullet $ of $H^0(X_K,mL_K)$. Following Section \ref{ss:Thetapol}, we define  a filtration $\cF^\bullet $ of $H^0(X,mL)$ by 
\[
\cF^\la H^0(X,mL) := H^0(X,mL) \cap  \cF_K^{\la} H^0(X_K,mL_K).
\]

\begin{prop}\label{p:fgTheta}
If $R$ is essentially of finite type over $\bk$ and $\cX^\circ$ is normal, then the $R[t]$-algebra $\bigoplus_{(m,\la) \in \N\times \Z} \cF^\la H^0(X,mL)$ is finitely generated.
\end{prop}

\begin{proof}
Observe that $X$ is normal, since 
 $\cX^\circ  \vert_{t\neq 0} \cong X\times (\A^1\setminus 0)$ and $\cX^\circ$ is normal by assumption.
Let $E_{1,K},\ldots, E_{q,K}$  denote the irreducible components of the fiber of $\cX_K \to \bA^1_K$ over $0_K$.
Note that 
\[
v_{E_{i,K}} := r(\ord_{E_i,K})) = c_i \ord_{F_{i,K}}
\]
for some  $c_i \in \bZ_{\geq 0}$ and divisor $F_{i,K}$ over $X_K$.
For each $i$, there exists a divisor $F_i$ over $X$ such that its restriction to $X_K$ is $F_{i,K}$.
With this notation, we claim 
\begin{enumerate}
\item[(i)] $\cF^\la H^0(X, mL) = \bigcap_{i=1}^q \{ s\in H^0(X,mL) \, \vert \,   c_i \ord_{F_i}(s)  \geq \la + A_{X,\Delta+D+X_\kappa}(c_i \ord_{F_i})\}$
\item[(ii)] $A_{X,\Delta+D +X_\kappa}(c_i \ord_{F_i}) = 0$. 
\end{enumerate}
Assuming the claim holds, then Theorem \ref{t:fg} implies $\bigoplus_{(m,\la) \in \N\times \Z} \cF^\la H^0(X,mL)$  is finitely generated.
To verify the claim note that each $F_i$ dominates $\Spec(R)$.
Therefore
\[
A_{X,\Delta+X_\kappa}( F_i) = A_{X_K, \Delta_K}( F_{i,K}), \quad  A_{X,\Delta+D+X_\kappa}(F_i) = A_{X_K, \Delta_K +D_K}(F_{i,K}), \]
and 
$\ord_{F_{i,K}}(s\vert_K)$ for all $s\in H^0(X,mL)$
Thus the claim holds by  Lemma \ref{l:lctestconfig} and the definition of $\cF^\bullet$.
\end{proof}

We are now ready to prove Theorem \ref{t:Thetared} in the case when $\cX^\circ$ is  normal. 

\begin{proof}[Proof of Theorem \ref{t:Thetared}  when $\cX^\circ$ is normal]
Consider the morphism 
\[
g: \cX: = \Proj \Big( \bigoplus_{m \in \N} \bigoplus_{\la \in \Z} \cG^\la H^0(X,mL) t^{-\la} \Big) \to \A^1_R
.
\]
By Propositions \ref{p:extendtheta} and \ref{p:fgTheta}, $(\cX,\cO_{\cX}(m))\to \A^1_R$ is a flat family of polarized schemes and is the unique extension of $(\cX^\circ, m\cL^\circ) \to \A^1_R$ for all $m >0$ sufficiently divisible. 
Let $\Delta_{\cX}$ and $\cD$ be the closures of $\Delta_{\cX^\circ}$ and $\cD^\circ$  in $\cX$.

To  show $(\cX,\Delta_{\cX}+\cD) \to \A^1_R$  is a family of boundary polarized CY pairs, we follow the  proof of Theorem \ref{t:Scomplete}.
First, arguing as in the proof of Theorem \ref{t:Scomplete} shows $\cX$ is normal,
\[
\cO_{\cX}(-mr(K_{\cX/\bA^1}+\Delta_{\cX}+\cD))
\cong \cO_{\cX}(m)
\quad \text{ and } \quad  
K_{\cX/\bA^1}+\Delta_{\cX}+\cD \sim_{\bQ} 0
\]
for $m>0$ sufficiently divisble. 
Hence, it remains to show $(\cX,\Delta_{\cX}+\cD)\to \bA^1$ is a family of slc pairs.
To proceed, consider the family of  boundary polarized CY pairs
\[
(X_{\bA^1},\Delta_{\bA^1}+D_{\bA^1} ):= (X,\Delta+D)\times \bA^1 \to \bA^1_R.
\]
The isomorphism $\cX\vert_{t\neq 0} \cong (X,\Delta+D)\times (\bA^1\setminus 0)$ induces a birational map 
$\cX \dashrightarrow X_{\bA^1}$. 
Let $\cY$ denote the normalization of the graph with  maps 
$f:\cY\to \cX$ and $f': \cY\to X_{\bA^1}$. 
Set 
\[
\Gamma:=f'^*(K_{X_\bA^1}+\Delta_{\bA^1}+D_{\bA^1}+X\times 0 + (X_{\bA^1})_{\pi t=0})
-
f^*(K_{\cX}+\Delta_{\cX}+\cD+\cX_{\pi t=0 })
.\]
Since $(X_{\bA^1},\Delta_{\bA^1}+D_{\bA^1} + (X_{\bA^1})_{\pi t=0})$ is lc by Lemma \ref{l:slcadj}
and $\Gamma=0$ by arguing as in the proof of Theorem \ref{t:Scomplete}, 
$(\cX, \Delta_{\cX}+\cD+\cX_{\pi t=0 })$ is lc. 
Thus $(\cX,\Delta_{\cX}+\cD)\to \bA^1_R$ is a family of slc pairs by Lemma \ref{l:slcadj} and so a family of boundary polarized CY pairs.
The extension is $\bG_m$-equivariant by \cite[Lem. 2.16]{ABHLX20} and unique by Lemma \ref{l:CYbirational}.3.
\end{proof}

We will now  deduce Theorem \ref{t:Thetared} from the  case when $\cX^\circ$ is normal.

\begin{proof}[Proof of Theorem \ref{t:Thetared}] 
Let 
$(\overline{\cX}^\circ, \overline{\cG}^\circ+\Delta_{\overline{\cX}^\circ}+\overline{\cD}^\circ)
$ 
denote the normalization of $(\cX^\circ, \Delta_{\cX^\circ}+\cD^\circ)$ and $\tau^\circ: \overline{\cG^\circ}^n \to 
\overline{\cG^\circ}^n$
denote the  involution.
By arguing as in the proof of  Theorem \ref{t:Scomplete}, 
\begin{enumerate}
	\item 
$\overline{g}^\circ:(\overline{\cX^\circ}, \overline{\cG^\circ}+\Delta_{\overline{\cX^\circ}} +\overline{\cD^\circ})\to \bA^1_R \setminus 0_\kappa$ 
extends to a $\bG_m$-equivariant family of  boundary polarized CY pairs 
$\overline{g}:(\overline{\cX}, \overline{\cG}+\Delta_{\overline{\cX}} +\overline{\cD})  \to \bA^1_R \setminus 0_\kappa$, and

\item   $\tau^\circ$ extends to an involution  $\tau:\overline{\cG}^n \to \overline{\cG }^n$.
\end{enumerate}
By Proposition \ref{p:bpcygluing},
$g^\circ$ extends to a family of  boundary polarized CY pairs $g:(\cX,\Delta_{\cX}+\cD) \to \A^1_R$.  
The extension is $\bG_m$-equivariant by \cite[Lem. 2.16]{ABHLX20} and unique by Lemma \ref{l:CYbirational}.3.
\end{proof}

\subsection{Automorphism groups}
For a boundary polarized CY pair $(X,\Delta+D)$,  set 
\[
\Aut(X,\Delta+D):= \{ g\in \Aut(X) \, \vert \, g^*D=D \text{ and }  g^*\Delta=\Delta\}
.\]
Note  that $\Aut(X,\Delta+D)$ has the structure of a linear algebraic group.
To see this, fix an integer $r>0$ sufficiently divisible  such that $L:=\omega_{X}^{[r]}(-r \Delta)$ is a  very ample line bundle.
Since $g^*L \cong L$ for any $g\in \Aut(X,\Delta+D)$,
we can identify $\Aut(X,\Delta+D)$ with the closed subgroup of 
${\rm PGL}\left(H^0(X, L)\right)$ composed of elements fixing $\Delta$ and $D$. We let $\Aut^0 (X, \Delta + D)$ denote the connected component of $\Aut(X,\Delta+D)$ containing the identity.

\begin{thm}\label{t:AutTorus}
If $(X,\Delta+D)$ is a boundary polarized CY pair,
then  $\Aut^0 (X,\Delta+D)$ is isomorphic to an algebraic torus (i.e. $\G_m^r$ for some $r\geq 0$).
\end{thm}

The statement is known to experts and follows easily from
\cite[Lem. 2.1]{Hu18}, which more generally implies that any connected algebraic subgroup of the automorphism group of an  slc CY pair is a torus. 
We thank Joaqu\'in Moraga for providing us with the reference.

\begin{proof}
It suffices to prove the result when $X$ is normal. 
Indeed, there is an embedding
\[
\Aut^0 (X,\Delta+D) \hookrightarrow \Aut^0 (\overline{X},\overline{\Delta}+\overline{G}+\overline{D})
=
\times_{i=1}^r
\Aut^0 (\overline{X}_i,\overline{G}_i+\overline{\Delta}_i+\overline{D}_i)
,\]
where $(\overline{X},\overline{G}+\overline{\Delta}+\overline{D})
:=
\sqcup_{i=1}^r(\overline{X}_i,\overline{G}_i+\overline{\Delta}_i+\overline{D}_i)$
is the normalization of $(X,\Delta+D)$. 
Since the  connected algebraic subgroups of a torus are tori, the normal case implies the full result. 
	
Now, assume $X$ is normal.
Fix an $\Aut(X,\Delta+D)$-equivariant log resolution $\mu:Y\to X$ of $(X,\Delta+D)$
and set $B:=\Supp(\mu_*^{-1}(\Delta+D))+\Exc(\mu)$.
Since  $K_{X}+\Delta+D\sim_{\bQ}0$ 
and
\[
K_{Y}+B = f^*(K_X+\Delta+D)+ \sum_{E} A_{X,\Delta+D}(E)E,
\]
where $E$ runs through prime divisors in $\Supp(\mu_*^{-1}(\Delta+D))+\Exc(\mu)$,
$K_{Y}+B$ is psuedoeffective. 
Thus \cite[Lem. 2.1]{Hu18} states that any connected affine algebraic subgroup of $\Aut^0(Y,B)$ is a torus.
Since there is an embedding $\Aut^0(X,\Delta+D)  \hookrightarrow \Aut(Y,B)$,   $\Aut^0(X,\Delta+D)$ is a torus.
\end{proof}

The next proposition shows that the dimension of the automorphism group increases under weakly special degeneration. 
While the statement is well known, we include a proof for the lack of a suitable reference.

\begin{prop}\label{p:torusrk+1}
If $(X,\Delta+D)\rightsquigarrow(X_0,\Delta_0+D_0)$ is a weakly special degeneration of  boundary polarized CY pair and  $(X,\Delta+D)\not\cong (X_0,\Delta_0+D_0)$,
then 
\[
\dim (\Aut (X,\Delta+D) )< \dim( \Aut(X_0,\Delta_0+D_0)).
\]
\end{prop}

\begin{proof}
We may assume 	$x=[(X,\Delta+D)]$ and $x_0=[(X_0,\Delta_0+D_0)]$ are $\bk$-points of 
the stack $\sM(\chi,N,{\bf r})$ for some choice of $\chi$, $N$, and ${\bf r}$.
By the proof of Theorem \ref{t:stack}, there is an open substack 
$ \cU\subset \sM(\chi,N,{\bf r})$ such that  $x,x_0\in\cU$ and $\cU \cong [Z/ G ]$, 
where $Z$ is a finite type  $\bk$-scheme with an action by an affine algebraic group $G$.
Choose $z,z_0\in Z$ that map to $x,x_0\in \cU$, respectively.
Observe that
\[
	\dim {\rm Stab}(z) = \dim G - \dim (G\cdot z) \\
	< \dim G - \dim (G\cdot z_0)   = \dim {\rm Stab}(z_0 ) 
,\]
where the inequality uses that $G\cdot z_0 \in \overline{G\cdot z}$.
Since $\Aut(X,\Delta+D)\cong {\rm Stab}(z)$ and $\Aut(X_0,\Delta_0+D_0)\cong {\rm Stab}(z_0)$, the result follows. 
\end{proof}

\subsection{S-equivalence}\label{ss:sequiv}

We now discuss S-equivalence classes of boundary polarized CY pairs.  

\begin{defn}[S-equivalence]\label{d:Sequiv}
Two  boundary polarized CY pairs
$(X,\Delta+D)$ and $(X',\Delta'+D')$ are \emph{S-equivalent} if there exist weakly special test configurations degenerating 
\[
(X,\Delta+D) \rightsquigarrow (X_0,\Delta_0+D_0) \leftsquigarrow (X',\Delta'+D')
\]
to a common  boundary polarized CY pair $(X_0,\Delta_0+D_0)$. To denote S-equivalence, we write $(X,\Delta+D)\sim_S (X',\Delta'+D')$.
\end{defn}

\begin{rem}[Terminology]
The term S-equivalence is a reference to work of Seshadri on vector bundles.
Recall, two semistable vector bundles  $E$ and $E'$ on a smooth curve $C$ to be \emph{S-equivalent} if the graded bundles associated to the  Jordan-Holder filtrations satisfy
\[
\gr_{\rm JH}E \cong \gr_{\rm JH} E', 
\]
The latter is related to the above definition by using the Rees construction to give a degeneration $E \rightsquigarrow \gr_{\rm JH}E$ via an equivariant family over $\bA^1$.
\end{rem}

\begin{prop}\label{p:sequiv}
S-equivalence for boundary polarized CY pairs is an equivalence relation. 
\end{prop}

The proposition is a consequence of the following lemma that interprets S-equivalence stack theoretically. Statement (2) of the lemma relies  on Theorem \ref{t:Thetared}.

\begin{lem}\label{l:sequiv}
Let $x=[(X,\Delta+D)]$ and $x'=[(X',\Delta'+D')]$ be $\bk$-points of $\cM(\chi, N, {\bf r})$. 
\begin{enumerate}
\item $(X,\Delta+D) \sim_{S} (X',\Delta'+D')$ if and only if $\overline{ \{x\}} \cap  \overline{\{x'\}} \neq \emptyset$. 
\item if $x_1, x_2 \in \overline{\{ x\}}$ are $\bk$-points, then $\overline{\{x_1\}}\cap \overline{\{x_2 \}} \neq \emptyset$.
\end{enumerate}
\end{lem}

\begin{proof}
The forward implication of (1) is clear since a common degeneration will be contained in $ \overline{ \{x\}} \cap  \overline{\{x'\}}$.
For the converse, assume  $\overline{ \{x\} } \cap \overline{\{x'\}}\neq \emptyset$. Thus there is a $\bk$-point $x_0 \in \overline{ \{x\} } \cap \overline{\{x'\}}$.
Since $x_0$ has linearly reductive stabilizer group by Theorem \ref{t:AutTorus}, \cite[Proof of Lem. 3.24]{AHLH18} implies that the specializations  $x \rightsquigarrow x_0$ and $x'\rightsquigarrow x_0$ are realized by maps $\Theta \to \sM(\chi,N,{\bf r})$. Thus the reverse implication of (1) holds.
	
For (2), note that the specializations $x\rightsquigarrow x_1$ and $x \rightsquigarrow x_2$ are given by maps $f_1,f_2:\Theta\to \cM(\chi,N,{\bf r})$ as in the above paragraph.
Following \cite[Proof of Lem. 3.25]{AHLH18},
the maps $f_1$ and $f_2$ glue to give a map 
$[\bA^2 \setminus \{(0,0)\} / (\bG_m)^2] \to \sM(\chi,N,{\bf r})$.
By choosing one of the two $\bG_m$ factors, $\Theta$-reductivity (Theorem \ref{t:Thetared})
implies that there exists an extension $g:[\bA^2/\bG_m] \to \sM(\chi,N,{\bf r})$.
Since $\gamma(0,0)$ is a specialization of both $x_1$ and $x_2$, (2) holds. 
\end{proof}

\begin{proof}[Proof of Proposition \ref{p:sequiv}]
The proposition follows immediately from Lemma \ref{l:sequiv}.
\end{proof}

The next proposition shows that klt pairs are only S-equivalent to themselves.

\begin{prop}\label{p:sequivklt}
If  $(X,\Delta+D)$ is a klt boundary polarized CY pair and  $(X,\Delta+D)\sim_S (X',\Delta'+D')$, then $(X,\Delta+D)\cong (X',\Delta'+D')$.
\end{prop}

\begin{proof}
The proposition is a consequence of the following claim: If $(\cX,\Delta_{\cX}+\cD)$ is a test configuration of a boundary polarized CY pair such that   the fiber  over   $0\in \bA^1$ or  $1\in \bA^1$ is klt,  then the test configuration is trivial.

To verify the claim, note that if the fiber over $0$ is klt, then so is the fiber over $1$ by the openness of klt singularities.
If the fiber over $1$  is klt, then Lemma \ref{l:CYbirational}.3 applied to $(\cX,\Delta_{\cX}+\cD)\dashrightarrow (\cX_1,\Delta_{\cX_1}+\cD_1) \times \bA^1$ implies the test configuration is trivial.
\end{proof}

We now  prove Corollary \ref{c:sequiv}, which  states that any
two degenerations of a family of boundary polarized CY pairs over a punctured curve are S-equivalent.

\begin{proof}[Proof of Corollary \ref{c:sequiv}]
Let $R:= \cO_{C,0}$ and consider the $\bG_m$-equivariant family of  boundary polarized CY pairs
$(\cX^\circ,\Delta_{\cX^\circ}+\cD^\circ)\to \Spec(R[s,t]/(st-\pi))\setminus 0$ constructed by gluing 
\[
(X,\Delta+D)\times \bG_m \to \Spec(R)\times \bG_m
\quad \text{ and } \quad 
(X',\Delta'+D')\times \bG_m \to \Spec(R)\times \bG_m
\] 
along $\Spec(K)\times \bG_m$.
By Theorem \ref{t:Scomplete}, the family extends to a $\bG_m$-equivariant family of  boundary polarized CY 
pairs  $(\cX,\Delta_{\cX}+\cD)\to \Spec(R[s,t]/(st-\pi))$.
The restrictions of the extension to $\{s=0\}$ and $\{t=0\}$ are test configurations degenerating
\[
(X_0,\Delta_0+D_0) \rightsquigarrow (\cX,\Delta_{\cX}+\cD)_{s=t=0} \leftsquigarrow (X',\Delta'_0+D'_0)
\]
Thus the pairs are S-equivalent.
\end{proof}

\subsection{Stability}

Motivated by the definition of K-stability, it is natural to introduce the following definition.

\begin{defn}\label{def:polystable}
A boundary polarized CY pair  $(X,\Delta+D)$ is called
\begin{enumerate}
	\item  \emph{polystable} if every weakly special test configuration $(\cX,\Delta_{\cX}+\cD)$ of $(X,\Delta+D)$ is a product, i.e. it satisfies $(\cX_0,\Delta_{\cX_0}+\cD_0) \cong(X,\Delta+D)$, and
	\item \emph{stable} if every weakly special test configuration  $(\cX,\Delta_{\cX}+\cD)$ of $(X,\Delta+D)$ 
	is trivial.
\end{enumerate}	
\end{defn}

Semistability is not defined above, since it is natural to call every boundary polarized CY pair \emph{semistable}.
The above two notions have the following stacky interpretation.

\begin{prop}
A pair $x=[(X,\Delta+D)]$ in $ \cM(\chi,N,{\bf r})(\bk)$ is
\begin{enumerate}
	\item  polystable if and only if the point $x$ is closed, and
	\item  stable if and only if the point $x$ is closed and has finite stabilizer group.
\end{enumerate}
\end{prop}

\begin{proof}
Since the stabilizers of points in $\cM(\chi, N,{\bf r})(\bk)$ are reductive by Theorem \ref{t:AutTorus},
\cite[Lem. 3.24]{AHLH18} implies that any specialization $x\rightsquigarrow x_0$ of $\bk$-points is given by a morphism $\Theta\to \cM(\chi,N,{\bf r})$.
Therefore (1) holds. 

Statement (2) follows from (1) and the following observation.  By Theorem \ref{t:AutTorus}, the group $\Aut^0(X,\Delta+D)$ is trivial if and only if there are no non-trivial maps to $\bG_m\to \Aut(X,\Delta+D)$.
In addition, the latter holds if and only if there are no non-trivial weakly special test configurations $(\cX,\Delta_\cX+\cD)$ of $(X,\Delta+D)$ with $(X_0,\Delta_0+D_0) \cong (X,\Delta+D)$. 
\end{proof}

\begin{rem}
An unusual feature of our moduli problem is that there can be S-equivalence classes that are unbounded (see Example \ref{e:unbounded} and Theorem \ref{t:TypeIIISequiv}). Such S-equivalence classes will not admit polystable representatives.
\end{rem}

\begin{thm}
A boundary polarized CY pair $(X,\Delta+D)$ is stable if and only if it is klt. 
\end{thm}

\begin{proof}
If $X$ is normal, then the statement follows immediately from Theorem \ref{thm:CZ-lcplace}.
It remains to show that a non-normal boundary polarized CY pair is not stable. 

Let $(\overline{X},\overline{G}+\overline{\Delta}+\overline{D})$ denote the normalization of a non-normal boundary polarized CY pair $(X,\Delta+D)$.
Fix an integer $r>0$ such that   $\overline{L}:=-r(K_{\overline{X}}+\overline{G}+\overline{\Delta})$ is a Cartier divisor.
Consider the filtration $\cF^\bullet$ of 
$
S:= \bigoplus_{m \in \bN} S_m : = \bigoplus_{m \in \bN} H^0(\overline{X}, m \overline{L})
$
defined by  
\[
\cF^\la S_m := \{s \in S_m \, \vert\, {\rm div}_{\overline{X}}(s ) \geq \la \overline{G}\}.
\]
By Proposition \ref{p:lcplacestestconfig}, there exists a test configuration $(\ocX, \ocG+\oDe_{\ocX}+\ocD)$ of $(\oX,\oG+\oDe+\oD)$ such that  
$
\cX:= \Proj( \bigoplus_{m \in \bN} \bigoplus_{\la \in \bZ} \cF^\la S_m)
$
and the isomorphism $\cX\vert_{\bA^1\setminus 0} \cong X\times(\bA^1\setminus 0)$ is induced by the isomorphism of graded rings
$
\Big( \bigoplus_{m \in \bN} \bigoplus_{\la \in \bZ} \cF^\la S_m \Big) \bigotimes_{k[t]} k[t^{\pm 1}]\cong \bigoplus_{m \in \bN} \bigoplus_{\la \in \bZ} S_m t^{-m}
.
$

We claim that the $\bG_m$-equivariant birational map 
$
\oG \times \bA^1 \dashrightarrow  \ocG
$
is an isomorphism.
To verify the claim, note that  the closed subscheme $\ocG\subset \ocX$ is defined by the homogeneous ideal 
\[
\bigoplus_{m\in \bN} \bigoplus_{\la \in \bZ} (\cF^\la S_m \cap I_{m} )t^{-\la}
\subset 
\bigoplus_{m\in \bN} \bigoplus_{\la \in \bZ} \cF^\la S_m t^{-\la},
\] 
where $I_{m} := H^0(\oX,m\overline{L} \otimes \cI_{\oG})\subset S_m$; see \cite[2.3.1]{BX19}.
Using that  $I_{m} = \cF^1 S_m$, we see
\[
\ocG \cong  \Proj\Big( \bigoplus_{m \in \bN} \bigoplus_{\la \leq 0}  (S_m / I_{m}) t^{-\la}\Big)
\cong 
\oG\times \bA^1,
\]
where the isomorphisms are $\bG_m$-equivariant. Thus the claim holds.

We will now construct a non-trivial test configuration of $(X,\Delta+D)$ by gluing.
Consider the involution $\tau:\oG^n\to \oG^n$.
By the previous claim, 
the involution $\tau \times {\rm id} :\oG^n\times (\bA^1\setminus 0)\to \oG^n\times(\bA^1\setminus 0)$
extends to an involution $\tau: \ocG^n\to \ocG^n$.
Thus Proposition \ref{p:bpcygluing} implies that 
$(X,\Delta+D)\times (\bA^1 \setminus 0)$
extends to a family of boundary polarized CY pairs $(\cX,\Delta_{\cX}+\cD)$ whose normalization is the extension $(\ocX,\ocG+\oDe_{\ocX}+\ocD)$.
Arguing as in the proof of Proposition \ref{p:TequivariantTC}, the $\bG_m$-action on $X\times(\bA^1\setminus 0)$ extends to a $\bG_m$-action on $\cX$.
Thus $(\cX,\Delta_{\cX}+\cD)$  is a test configuration of $(X,\Delta+D)$. 
Since $(\ocX, \ocG+\oDe_{\ocX}+\ocD)$ is non-trivial,  $(\cX,\Delta_{\cX}+\cD)$ is non-trivial and so $(X,\Delta+D)$ is not stable. 
\end{proof}

\begin{rem}
One can also consider the K-(poly/semi)stability of a boundary polarized CY pair $(X,\Delta+D)$ with respect to the polarization $L:= -K_{X}-\Delta$.
\begin{itemize}
	\item By \cite{Oda13Slope,OdakaSun15}, the above pair $(X,\Delta+D)$ is always K-semistable with respect to $L$, since the singularities are slc and $K_{X}+\Delta+D\equiv 0$. 
	\item The notion of K-polystability  differs  from polystability as defined above, since  the definition of K-(poly/semi)stability takes into account test where where the degeneration of $(X,\Delta+D)$ is no long a boundary polarized CY pair.
	For example, the pairs in Theorem \ref{t:deg6ps} are polystable, but not klt and so not K-polystable by \cite[Thm. 1.3]{Oda20}.
\end{itemize}	
\end{rem}

\section{Properness}\label{s:properness}

In this section, we show that the moduli stack of boundary polarized CY pairs satisfies the existence part of the valuative criterion for properness.

\medskip
Throughout, let $R$ be a DVR essentially of finite type over $\bk$ with fraction field $K$ and residue field $\kappa$.

\begin{thm}[Properness]\label{thm:properness}
If  $(X_K,\Delta_K+D_K)$ is a boundary polarized CY pair over $K$, then 
there is a finite extension of DVRs $R'/R$ 
 such that the base change by $K' = \mathrm{Frac}(R')$ 
\[
(X_K,\Delta_K +D_K)\times_K K',\]
extends to a family of boundary polarized CY pairs $(X',\Delta'+D') \to \Spec(R')$.
\end{thm}

\begin{rem}\label{r:rootstack}
In fact the proof shows we can take $R' = R[\sqrt[d]{\pi}]$, where $\pi$ is a uniformizer of $R$. Phrased another way, $(X_K, \Delta_K + D_K)$ extends to a family over a root stack of $\operatorname{Spec} R$. This stronger version of the existence part of the valuative criterion of properness holds in particular for proper Deligne--Mumford stacks and has had several applications in the literature \cite{AB192, BV}. 
\end{rem}

To prove the theorem, 
we use a perturbation argument to reduce to a result of Koll\'ar and Xu, 
which states that the moduli stack of slc pairs $(X,\Delta+\vep H)$, where $(X,\Delta)$ is a  CY pair and $H$ is an ample Weil divisor, is proper \cite{KX20}.
More precisely, we use the following technical statement shown in the proof of \cite[Lem. 4]{KX20}.

\begin{prop}\label{prop:KXD}
Let  $(X_K, \Gamma_K)$ be an lc CY pair over $K$ and $H_K$ an ample $\Q$-Cartier $\Q$-divisor on $X_K$. Assume the following hold:
\begin{enumerate}
\item $(X_K, \Gamma_K+\delta H_K)$ is lc for $0<\delta\ll 1$. 
\item There exists a log resolution  $f:Y_K\to X_K$  of $(X_K,\Gamma_K+H_K)$ and an extension of $Y_K$ to a flat projective family $Y \to \Spec(R)$
such that 
$(Y,B+Y_\kappa)$
is dlt, where $B$ is the closure of  $\Supp(f_{*}^{-1}(\Gamma_K+H_K)) +\Exc(f)$ in $Y$.
\end{enumerate}

Then  $(X_K, \Gamma_K)$  extends to   a family of   CY pairs $(X,\Gamma)\to \Spec(R)$  such that 
\begin{enumerate}
\item[(3)] $(X,\Gamma+\delta H)\to \Spec(R)$ is a family of KSBA stable pairs for $0<\delta\ll1$.
\item[(4)] $Y\dashrightarrow X$ is a birational contraction. 
\end{enumerate}
\end{prop}

Before proving Theorem  \ref{thm:properness}, we prove the following auxiliary statement, which we will deduce from Proposition \ref{prop:KXD}.

\begin{prop}\label{prop:propernessaux}
Let  $(X_K,\Delta_K+D_K)$ be an lc boundary polarized CY pair over $K$ and $0 \leq H_K  \sim_{\bQ} -K_{X_K}-\Delta_K$.

If conditions (1) and (2) of  Proposition \ref{prop:KXD} are  satisfied with 
$\Gamma_K:= \Delta_K+D_K $, then   
\[
(X_K, \Delta_K+ D_K)
\] extends to  a family of boundary polarized CY pairs
$(X,\Delta+D)\to \Spec(R)
$
 such that 
$
(X,\Delta+(1-\vep) D+ (\vep+\delta) H)\to \Spec(R)
$
is a family of KSBA stable pairs for $0<\vep,\delta\ll1$.
\end{prop}

\begin{proof}
For $0<\vep \ll 1$, we consider the pair 
\begin{equation}\label{eq:pairepsilon}
(X_K,\Delta_K + (1-\varepsilon) D_K + \varepsilon H_K ).
\end{equation}
Since  $H_K  \sim_{\Q} -K_{X_K}-\Delta_K$ and condition (1) holds, \eqref{eq:pairepsilon} is an lc CY pair.
Therefore, we may apply Proposition \ref{prop:KXD} to produce an extension of \eqref{eq:pairepsilon}
to a family of  slc CY pairs 
\[
(X^{\vep},\Delta^{\vep}+ (1-\vep) D^{\vep}+ \vep H^{\vep} ) \to \Spec(R)
\]
such that 
$(X^{\vep},\Delta^{\vep}+ (1-\vep) D^{\vep}+ (\vep+\delta) H^{\vep} ) \to \Spec(R)$
is a family of KSBA stable pairs for $0<\delta\ll1$ and $Y\dashrightarrow X^{\vep}$ is a birational contraction.
Note that the previous conditions imply $H^{\vep}$ is $\bQ$-Cartier and ample.

\medskip

\noindent \emph{Claim:} There exists a sequence of positive rational numbers $(\vep_i)_i $ such that 
$\lim \vep_i=0$ and the birational map $X^{\epsilon_i} \dashrightarrow  X^{\epsilon_j}$ is an isomorphism for all $i,j\geq1$.
\medskip

\noindent \emph{Proof of Claim:} 
Since  $Y_\kappa$ is a union of finitely many prime divisors,
there are finitely many  combinations of prime divisors contained in $Y_\kappa$ that can get contracted by $g^\epsilon:Y \dashrightarrow X^\epsilon$ for some $\vep$.
Therefore we may find a sequence $(\vep_i)_i$ such that $\lim \vep_i=0$ and  
the set of components of $Y_\kappa$ that get contracted by $g^{\vep_i}$  is independent of $i$.

Now note that  $X^{\vep_i} \dashrightarrow X^{\vep_j}$ is an isomorphism over $\Spec(K)$, since both varieties are  extensions of $X_{K}$.
Using that $g^{\vep_i}$ and $g^{\vep_j}$ contract the same components of $Y_\kappa$, we see $X^{\vep_i} \dashrightarrow X^{\vep_j}$ is an isomorphism in codimension one.  We now observe
\[
X^{\vep_i} \cong \Proj
\Big(
\bigoplus_{m \in \N}   H^0(X^{\vep_i},\cO_{X^{\vep_i}}(\lfloor mH^{\vep_i} \rfloor )) \Big) 
\cong 
\Proj
\Big(\bigoplus_{m \in \N}
H^0(X^{\vep_j},\cO_{X^{\vep_j}}(\lfloor mH^{\vep_j} \rfloor )) \Big)
\cong 
X^{\vep_j}
,\]
where the first and third isomorphisms use that $H^{\vep_i}$ and $H^{\vep_j}$ are ample
and the second  uses Hartog's Lemma.\qed

Let $(X,\Delta+ D+ H):= (X^{\vep_i}, \Delta^{\vep_i}+ D^{\vep_i}+H^{\vep_i})$,
which is independent of the choice of $i$ by the previous claim. 
Note that
\begin{enumerate}
\item[(a)]$(X,\Delta+ (1-\vep_i) D+ \vep_i H) \to \Spec(R)$ is a family of projective  CY pairs and 
\item[(b)] $(X,\Delta+ (1-\vep_i) D+ (\vep_i+\delta) H) \to \Spec(R)$ is a family of KSBA pairs for $0<\delta\ll1$
\end{enumerate} 
for each $i \geq 1$. 
Using (a), we see that
\[
0  \sim_{\bQ}  (K_X+\Delta+ (1-\vep_i) D+ \vep_i H) - (K_{X}+ \Delta+ (1-\vep_j) D+ \vep_j H )
\sim_{\bQ} (\vep_i-\vep_j) (D-H),
\]
for $i, j \geq1$, which implies    $D\sim_{\bQ} H$.
Thus $D$ is ample, $\bQ$-Cartier,
and satisfies  
\[
K_X+\Delta+D \sim_{\bQ}  K_{X}+\Delta+(1-\vep_i)D+\vep_i H\sim_{\bQ}0
.\]
Using (a) and that $\lim \vep_i =0$, we see $(X,\Delta+D)\to \Spec(R)$ is a family of slc pairs. 
Therefore $(X,\Delta+D)\to \Spec(R)$ is a family of boundary polarized CY pairs.
Furthermore, (b) implies the last statement of the proposition holds.
\end{proof}

We will now deduce Theorem \ref{thm:properness} from Proposition \ref{prop:propernessaux}.

\begin{proof}[Proof of Theorem \ref{thm:properness}]
Let $\nu:\overline{X}_K \to X_K$ denote the normalization morphism and $\overline{G}_K \subset \overline{X}_K$  the conductor divisor. 
We seek to apply Proposition \ref{prop:propernessaux} to the connected components of $(\overline{X}_K,  \overline{\Delta}_K+ \overline{G}_K+ \overline{D}_K) $.
Since the linear system $\nu^*|-m(K_{X_K}+\Delta_K)|$
is basepoint free for $m>0$ sufficiently divisible, there exists $H_K \sim_{\bQ} -K_{X_K}-\Delta_K$ such that 
\[
(\overline{X}_K,  \overline{\Delta}_K+ \overline{G}_K+ \overline{D}_K +\overline{H}_K)  
\]
is lc.
Fix a log resolution $f : Y_K\to \overline{X}_K$.
By \cite[Sec. IV.3, pg 198-199]{KKMSD}, 
there exists a finite extension $R' = R[\sqrt[d]{\pi}]$, where $\pi$ is the uniformizer of $R$, and an extension of the pullback $Y_{K'}$ to a flat projective morphism
$g:Y \to  \Spec(R')$
such that 
$(Y,B+Y_{\kappa})$
is dlt, where  $B$ denotes the closure of $(f_{K'})_{*}^{-1}(\overline{G}_{K'} + 
\overline{\Delta}_{K'}+ \overline{D}_{K'})+\Exc(f_{K'})$.
By Proposition \ref{prop:propernessaux}, there exists an extension of $(\overline{X}_{K'},\overline{G}_{K'}+\overline{\Delta}_{K'}+ \overline{D}_{K'})$ to a family of boundary polarized CY pairs
\begin{equation} \label{eq:X'Fano}
(\overline{X}',\overline{G}'+ \overline{\Delta}' + \overline{D}') \to \Spec(R')
\end{equation}
such that 
\begin{equation}
\label{eq:X'KSB}
(\overline{X}',\overline{G}'+ \overline{\Delta}' + (1-\vep) \overline{D}' + (\vep +\delta) \overline{H}') \to \Spec(R')
\end{equation}
is a family of  KSBA-stable pairs for $0<\vep,\delta\ll1$.
By \cite[Thm. 11.38]{KolNewBook}
$
(X_{K'},\Delta_{K'}+(1-\vep)D_{K'} +(\vep+\delta) H_{K'})
$
extends to a family of KSBA-stable pairs 
\[
({X}',{\Delta}' + (1-\vep){D}' + (\vep +\delta) {H}') \to \Spec(R')
\]
with normalization \eqref{eq:X'KSB}.
Since the normalization of $(X',\Delta'+D')$ is \eqref{eq:X'Fano}, which is a family of boundary polarized CY pairs,  $(X',\Delta'+D')\to \Spec(R')$ is a proper family of slc pairs, 
$D'$ is relatively ample, and $K_{X'}+\Delta'+D'\equiv 0$. 
By abundance for numerically trivially slc pairs \cite[Cor. 1.4]{HX16},
 $K_{X'}+\Delta'+D' \sim_{\bQ}0$.
Thus $(X',\Delta'+D')\to \Spec(R')$ is a family of boundary polarized CY pairs. 

By choosing $f$ to be a $\Gal(K'/K) = \bmu_d$-equivariant resolution, the resulting family $(X', \Delta' + D') \to \operatorname{Spec} R'$ is $\bmu_d$-equivariant. Indeed the stack of KSBA-stable pairs has finite, and in particular proper and representable, inertia. Thus by the valuative criterion of properness for representable morphisms, the $K'$-point of the inertia stack corresponding to the generator of the $\bmu_d$ action extends uniquely to an order $d$ automorphism of the KSBA-stable extension over $R'$. Thus the resulting family descends to a family of boundary polarized CY pairs over the $d$-th root stack $[\operatorname{Spec}R'/\bmu_d]$ of $\operatorname{Spec} R$.
\end{proof}

\section{Sources and regularity}\label{s:sources}

In this section, we discuss the source and regularity of  CY pairs.
After recalling their definitions,
we study their behavior in families and then we construct a good moduli space parametrizing regularity at most 0 boundary polarized CY pairs.

\subsection{Definition}\label{ss:sourceregdef}
The source \cite{Kol16} and regularity \cite{Sho00}  of a  CY pair are defined in terms of minimal lc centers on a dlt modification.

\subsubsection{Dlt modification}
A \emph{dlt modification} of an lc pair $(X,\Delta)$ is the data of a pair $(Y,\Delta_Y)$ and a proper birational morphism $f:Y\to X$ such that 
\begin{enumerate}
\item[(i)] $(Y,\Delta_Y)$ is  dlt,
\item[(ii)] $\Delta_Y= E+ f_*^{-1} \Delta$, where $E$ is the sum of the exceptional divisors of $f$, and
\item[(iii)] $K_{Y}+\Delta_Y= f^*(K_X+\Delta)$. 
\end{enumerate}
Every lc pair $(X,\Delta)$ admits a dlt modification by \cite[Thm. 1.34]{Kol13}.

If $(Y,\Delta_Y)$ is a dlt pair and $S\subset Y$ an lc center, 
then $(S,\Delta_S)$ is a dlt pair, 
where $\Delta_S:={\rm Diff}^*_S(\Delta)$ is defined via adjunction, 
and $K_S+ \Delta_S \sim_{\bQ} (K_Y+\Delta_Y) \vert_S$; see \cite[Sec 4.2]{Kol13}.

\subsubsection{Sources}

The source was  introduced in \cite{Kol16} for crepant log structures, which are roughly  CY fibrations. 
We are primarily interested in  sources of CY pairs, defined below. 
See also \cite[Defn. 2]{Sho13} for a related notion.

\begin{defn}[Source]\label{d:source}
Let $(X,\Delta)$ be a CY pair.

\begin{enumerate}
\item If $(X,\Delta)$ is klt, we set ${\rm Src}(X,\Delta):=(X,\Delta)$.

\item If $(X,\Delta)$ is not klt,  then let $(\overline{X},\oG+\overline{\Delta}) := \sqcup_{i=1}^r (\overline{X}_i,\oG_i+\overline{\Delta}_i)$ denote its normalization and fix a dlt modification 
\[
(Y,\Delta_{Y}) \to (\overline{X}_i,\oG_i+\overline{\Delta}_i)
\]
for some $i$.  We set ${\rm Src}(X,\Delta)$  equal to the crepant birational equivalence class of $(S,\Delta_S)$, where $S\subset Y$ is a minimal lc center of $(Y,\Delta_Y)$.
\end{enumerate}
\end{defn}

Note that a boundary polarized CY pair $(X,\Delta+D)$ is also a CY pair and so $\Src(X,\Delta+D)$ is defined.

\begin{lem}
Definition \ref{d:source}.2 is independent of the choice of dlt modification, minimal lc center $S$, and index $i$.
\end{lem}

\begin{proof}
The result follows  from \cite{Kol16}. 
Indeed, if $(X,\Delta)$ is lc, then ${\rm Src}(X,\Delta)$ is independent of the choice of dlt modification and minimal lc center by \cite[Thm. 1]{Kol16}. 
If $(X,\Delta)$ is slc, but not lc, it remains to show  independence of the choice of  $i$. 
To proceed, note that $X$ is connected.
Thus, it suffices to show that 
if $\pi(\overline{X}_i)$ and $\pi(\overline{X}_j)$ intersect, then 
\begin{equation}\label{eq:Srcij}
{\rm Src}(\overline{X}_i,\oG_i+\overline{\Delta}_i) \overset{cbir}{\sim} {\rm Src}(\overline{X}_j,\oG_j+\overline{\Delta}_j).
\end{equation}
Now, if the images intersect, then there exist  prime divisors $\overline{F}_i \subset \overline{G}_i$ and $\overline{F}_{j}\subset \overline{G}_j$ 
such that the involution $\tau:\overline{G}^n\to \overline{G}^n$ interchanges $\overline{F}_i^n$ and $\overline{F}_j^n$.
By  \cite[Thm. 1.6]{Kol16}, 
\[
{\rm Src}( \overline{X}_i, \oG_i+\overline{\Delta}_i ) 
\cong
{\rm Src}(\overline{F}_i^n, {\rm Diff}_{\overline{F}_i^n} (\oG_i+ \overline{\Delta}_i - \overline{F}_i))
\]
and the same holds for $i$ replaced by $j$. 
Since 
\[(\overline{F}_i^n, {\rm Diff}_{F_i^n} ( \oG_i+\overline{\Delta}_i - \overline{F}_i)) \cong (\overline{F}_j^n, {\rm Diff}_{F_j^n} (\oG_j +\overline{\Delta}_j - \overline{F}_j))
\]
by  \cite[Prop. 5.12]{Kol13}, we conclude that 
\eqref{eq:Srcij} holds.
\end{proof}

\begin{rem}[Structure]
The source of a CY pair is either a klt  CY pair or an equivalence class containing such a pair. 
Indeed, if $(X,\Delta)$ is klt, this is clear. 
If not, then the assumption that $S$ is minimal and $K_{S}+\Delta_S\sim_{\bQ} f^*(K_X+\Delta)\vert_S $ implies $(S,\Delta_S)$ is a klt  CY pair. 
\end{rem}

\begin{rem}[Surfaces]
If $(X,\Delta)$ is a  CY surface pair, then the source is either a point, the crepant birational equivalence class of a klt  CY curve pair, or $(X,\Delta)$.
In each case, the source is a single log pair.  
\end{rem}

\begin{defn}[Regularity]\label{d:coregularity}
The \emph{regularity}  and \emph{coregularity }of an slc CY pair is
\begin{align*}
\reg(X,D) &:= \dim X - \dim \Src(X,D) - 1\\
{\rm coreg}(X,D)&:= \dim \Src(X,D)
\end{align*}
For notational purposes, it will be convenient to have both definitions.
\end{defn}

When $(X,\Delta)$ is an lc CY pair that is not klt, the above definition of coregularity agrees with Shokurov's original definition, which is the dimension of the dual complex of the dlt modification \cite{Sho00}. 
When $(X,\Delta)$ is klt, $\reg(X,\Delta)=-1$, while it is $-\infty$ in \emph{loc. cit}.
The term \emph{coregularity} arose in the work of Moraga and his coauthors \cite{Mor22}.

\subsection{Behavior in families}\label{ss:sourceregfamilies}
We now analyze the behavior of the source and regularity in families. We begin with the following helpful lemma. 

\begin{lem}\label{l:sourceovercurve}
Let $(X,\Delta+D)\to C$ be a family of boundary polarized CY pairs over the germ of a smooth pointed curve $0 \in C$. If $X$ is normal, then 
\[
{\rm Src} (0,X,\Delta+D+X_0) \cbir \Src(X_0,\Delta_0+D_0)
.\]
\end{lem}

In the above statement, ${\rm Src} (0,X,\Delta+D_0+X_0)$ denotes the source of $(X,\Delta+D_0+X_0)\to C$ over $0$, which is defined in \cite[Thm. 1]{Kol16} as the crepant birational equivalence class of a pair $(S,\Gamma_S)$, where $(Y,\Gamma_Y)\to (X,\Delta+D+X_0)$ is a dlt modification, $S\subset Y$ a minimal lc center dominating $0$,  and $\Gamma_S : = {\rm Diff}^*_S(\Gamma_Y)$. By \cite{Kol16}, it is well defined.

\begin{proof}
Fix an irreducible component $F \subset  X_0$ .
Let
$f:(Y,\Gamma_Y) \to (X,\Delta+D+X_0)$ be a dlt modification and $S\subset Y$ be a minimal lc center of $(Y,\Gamma_Y)$ contained in $F_Y := f_*^{-1}(F)$. 
By definition, $(S,\Gamma_S) \in {\rm Src} (0,X,\Delta+D)$.
Note that 
\begin{equation}\label{e:F_Y}
(F_Y, {\rm Diff}_{F_Y} (\Gamma_Y- F_Y)) \to (  F^n, {\rm Diff}_{F^n}(\Delta+D+X_0 - F))
\end{equation}
is a dlt modification, since the morphism is crepant by \cite[Prop. 4.6]{Kol13} and 
adjunction \cite[Thm. 4.19]{Kol13} implies 
$(F_Y, {\rm Diff}_{F_Y} (\Gamma_Y- F_Y))$ is dlt.
Since $(F^n, {\rm Diff}_{F^n} (\Delta+D+X_0- F))$ is a component of the normalization of $(X_0,\Delta_0+D_0)$,
$(S,\Gamma_S)$ is a source of $(X_0,\Delta_0+D_0)$.
\end{proof}

\begin{prop}\label{p:SrcSequiv}
If $(X,\Delta+D)$ and $(X',\Delta'+D')$ are S-equivalent boundary polarized CY pairs,
then $\Src(X,\Delta+D) \cbir \Src(X',\Delta'+D')$.
\end{prop}

\begin{proof}	
It suffices to show that
if $(\cX,\Delta+\cD) \to \bA^1$  is a test configuration of $(X,\Delta+D)$, then  
$\Src (X,\Delta+D) \cbir \Src (\cX_0,\Delta_{\cX_0}+\cD_0)$.
If $X$ is normal, then
\begin{multline*}
\Src(X,\Delta+D)  
\cbir
\Src(0,X_{\bA^1},\Delta_{\bA^1}+D_{\bA^1}+X\times 0)\\
\cbir
\Src(0,\cX,\Delta_{\cX}+\cD+\cX_0)
\cbir
\Src(\cX_0,\Delta_{\cX_0}+\cD_0),
\end{multline*}
Indeed, the first and third equivalence hold by Lemma \ref{l:sourceovercurve}. 
For the second, note that $(X_{\bA^1},\Delta_{\bA^1}+D_{\bA^1}+X\times 0)$ and $( \cX,\Delta_{\cX}+\cD+\cX_0)$ are crepant birational by Lemma \ref{l:CYbirational}. 
Thus \cite[Thm. 4.45.7]{Kol13} implies their sources over $0$ are equal.

If $X$ is not normal, then the normalization $(\overline{\cX},
 \overline{\cG}+\Delta_{\overline{\cX}} + \overline{\cD})$ of $(\cX,\Delta_{\cX}+\cD)$ is a test configuration of $(\overline{X},\overline{G}+\overline{\Delta}+\overline{D})$.
Note that  
\[
{\rm Src}(X,\Delta+D) \overset{cbir}{\sim} {\rm Src} (\overline{X},\overline{G}+\overline{\Delta}+\overline{D})
\quad \text{ and } \quad 
{\rm Src}(\cX_0,\Delta_{\cX_0}+\cD_0) \overset{cbir}{\sim} {\rm Src} 
(\overline{\cX}_{0},  \overline{\cG}_0+\Delta_{\overline{\cX}_{0} }+ \overline{\cD}_{0})
.
\]
Indeed, the first relation follows from the definition of source, and the second  from the fact that 
$(\cX_0,\Delta_{\cX_0}+\cD_0)$ and 
$(\overline{\cX}_{0}, \overline{\cG}_0+\Delta_{\overline{\cX}_{0} }+ \overline{\cD}_{0})$ have the same normalization.
Hence, the case when $X$ is normal implies the non-normal case.
\end{proof}

\begin{prop}\label{p:regularityusc}
If $(X,\Delta+D) \to B$ is a family of boundary polarized CY pairs over a reduced finite type scheme $B$, then the function 
\[
B\ni 
b\mapsto \reg(X_b,\Delta_b +D_b)  \in \bZ
\]
is constructible and upper semi-continuous.
\end{prop}

\begin{proof}
To prove the proposition, it suffices to verify the following statements. 
\begin{enumerate}
\item[(1)] There exists a  non-empty open set $ B^\circ \subset B$ such that $\reg(X_b,\Delta_b+D_b)$ is independent of $b\in B^\circ$.
\item[(2)] If $B$ is a curve and $b\in B$, then 
$\reg(X_b,\Delta_b+D_b)\geq \reg(X_{K(B)},\Delta_{K(B)})$.
\end{enumerate}
Indeed, (1) implies  constructibility 
and (2) then implies upper-semicontinuity.

We first  verify (1) in the case when $X$ is normal.
By replacing $B$ with a non-empty open subset, we may assume $B$ is smooth. 
The latter implies $(X,\Delta+D)$ is an lc pair by Lemma \ref{l:slcadj}.
Thus there exists a dlt modification $(Y,\Gamma_Y)\to (X,\Delta+D)$.
Now fix an lc center $S\subset Y$ of $(Y,\Gamma_Y)$, which is minimal among those dominating $B$.
There exists an open set $B^\circ \subset B$, such that 
$(Y_b,\Gamma_{Y_b}:= \Gamma_{Y} \vert_{Y_b})\to (X_b,\Delta_b+\Delta_b)$ is a dlt modifcation for each $b\in B^\circ$.
Furthermore, after shrinking $B^\circ$, $S_b$ is a  union of minimal lc centers of $(Y_b,\Delta_b)$
of dimension $\dim(S)- \dim (B)$. 
Hence $\reg(X_b)=\dim(X) - \dim (S)-1$ for $b\in B^\circ$, 
which proves (1) when $X$ is normal.

If $X$ is not normal, let $(\overline{X},\oG+\overline{\Delta}+\overline{D})$ denote the normalization of $(X,\Delta+D)$.
There exists an open set $U\subset B$ such  $(\overline{X}_U,\oG_U+\overline{\Delta}_U+\overline{D}_U) \to U$ is a family of boundary polarized CY pairs and $(\overline{X}_b,\oG_b+\overline{\Delta}_b+\overline{\Delta}_b)$ is the normalization  of $(X_b,\Delta_b+D_b)$ for each $b\in U$.
Hence applying the normal case to the family of over $U$ completes the proof of (1).

For  (2), it suffices to prove the result when $X$ is normal. 
Indeed, the reduction step to the normal case is similar to the last paragraph of  the proof of Proposition \ref{p:SrcSequiv}. 

Assume $X$ is normal. 
If $(X_\eta,\Delta_\eta+D_\eta)$ is klt, then 
there is nothing to prove, since ${\rm reg} (X_0,\Delta_0+D_0) = -1$.
If not, then choose a dlt modification $(Y,\Gamma_Y) \to  (X,\Delta+D+X_0)$
and a minimal lc center $S\subset Y$ of $(Y,\Gamma_Y)$ dominating $C$. 
Since $S\cap Y_0 \neq \emptyset$,
there exists an irreducible component $F\subset Y_0$ with $S\cap F \neq \emptyset$. 
Since $F$ is an lc center of $(Y,\Gamma_Y)$, 
$S\cap F$ is an lc center of $(Y,\Gamma_Y)$ as well, 
and, hence,  there is  a minimal lc center $T$ of $(Y,\Delta_Y)$ with $T\subset S\cap F$.  
Now observe
\begin{multline*}
	{\rm coreg}(X_0,\Delta_0+D_0) = \dim (\Src(X_0,\Delta_0+D_0) ) 
	= \dim (\Src(0,X,\Delta+D+X_0)) \\
	= \dim T 
	\leq \dim S -1 = \dim  S_\eta ={\rm coreg}(X_\eta,\Delta_\eta+D_\eta),
\end{multline*}
where the second equality holds by Lemma \ref{l:sourceovercurve}.
\end{proof}

\subsection{Regularity 0 pairs}\label{ss:reg0pairs}

In this section, we study regularity $0$ pairs.
Under mild hypotheses, we prove that they  are parametrized by a good moduli space (Theorem \ref{t:modulireg0}).

\subsubsection{Properties}

\begin{lem}\label{l:typeIIlcplace}
If  $(X,\Delta+D)$ is an lc regularity 0  boundary polarized CY pair, then
\begin{enumerate}
\item  $(X,\Delta+D)$ has at most two lc places and

\item $(X,\Delta)$ has at most one lc place.
\end{enumerate}
If $(X,\Delta+(1+\vep)D)$ is lc for $0<\vep\ll1$ as well, then $(X,\Delta+D)$ has exactly one lc place.

\end{lem}

\begin{proof}
Let  $(Y,\Gamma_Y) \to (X,\Delta+D)$ be a dlt modification  and $\Delta_Y\leq \Gamma_Y$   the  $\bQ$-divisor  on $Y$ such that $(Y,\Delta_Y) \to (X,\Delta)$ is also a dlt modification.
Note that
$\Gamma_{Y}^{=1}$ is a union of non-intersecting prime divisors, since $\reg(X,\Delta+D)=0$.

Since $K_Y+\Gamma_Y\sim_{\bQ}0$, \cite[Prop. 4.37]{Kol13} implies $\Gamma_Y^{=1}$
has at most two connected components and so (1) holds.
Next, note that $-K_{Y}-\Delta_Y$ is big and nef, since it is the pullback $-K_X-\Delta$, which is ample. 
Thus  $\Delta_Y^{=1}$ is connected by \cite[Thm. 17.4]{FlipAbund} and so (2) holds.  
If  $(X,\Delta+(1+\vep)D)$ is lc for $0<\vep\ll1$, then the lc places of $(X,\Delta+D)$ are the same as those of $(X,\Delta)$. Thus (2) implies the last statement.
\end{proof}

\begin{prop}\label{p:typeIIslc}
If $(X,\Delta+D)$ is a regularity 0 boundary polarized CY pair, then
\begin{enumerate}
\item $X$ is normal or
\item $X$ is non-normal, has at most two irreducible components, and the normalization $(\overline{X},\overline{\Delta}+ \overline{G})$ of $(X,\Delta)$ is plt.
\end{enumerate}
\end{prop}

\begin{proof}
If $X$ is non-normal, then the normalization 
$ (\overline{X},\oG+\overline{\Delta} +\overline{D})= \sqcup_{i=1}^r  (\overline{X}_i,\oG_i+\overline{\Delta}_i +\overline{D}_i )$
is a union of regularity 0 boundary polarized CY pairs. 
By Lemma \ref{l:typeIIlcplace}.2,
each pair $(\overline{X}_i,\oG_i+\overline{\Delta}_i)$ is plt and $\overline{G_i}$ is prime. 
Since $X$ is the
gluing of the $\overline{X}_i$ via an involution $\tau:\overline{G}\to \overline{G}$, the normalization $\overline{X}$ has at most two components.
\end{proof}

\begin{prop}\label{prop:reg0polystable}
If $(X,\Delta+D)$ is a regularity 0 boundary polarized CY pair, then
\[
\Aut^0(X,\Delta+D) \cong \bG_m^r 
\]
for some $0 \leq r \leq 2$.
Furthermore, when $r=2$, $(X,\Delta+D)$ is polystable.
\end{prop}

\begin{proof} 
First, assume $X$ is normal.
By Theorem \ref{t:AutTorus}, $\Aut^0(X,\Delta+D) \cong \bG_m^r$ for some $r\geq 0$. 
Now, each  $\rho \in N:= \Hom(\bG_m , \Aut^0(X,\Delta+D))$ induces a distinct product test configuration 
\[
(\cX_{\rho}, \Delta_{\cX_\rho}+\cD_{\rho})
:= 
(X,\Delta+D)\times \bA^1, 
\]
and, hence,  a $\bZ$-valued divisorial valuation $v_\rho $ satisfying 
$A_{X,\Delta+D}(v_\rho)= 0$ by Lemma \ref{l:tclcplace}.
Thus $v_{\rho}= b_\rho \cdot \ord_{E_{\rho}}$ for some $b_\rho\in \bN$ and a divisor $E_{\rho}$ over $X$  that is an lc place of $(X,\Delta+D)$.
Since the correspondence respects the scaling of test configurations and valuations,
\[
\ord_{E_\rho} = \ord_{E_{\rho'}}\quad \quad \iff \quad \quad m \rho = m'\rho' \ \text{ for some } m,m'\in \bZ_{>0}
.\]
Therefore, if $r\geq 2$, then there are infinitely many divisors $E$ over $X$ with $A_{X,\Delta+D}(E)=0$. Since the  latter cannot occur by Lemma \ref{l:typeIIlcplace}.2,
 $ r \leq 1$ when $X$ is normal.

Next, assume  $X$ is not normal.
By the normal case and Proposition \ref{p:typeIIslc}, $\Aut(\overline{X}, \oG+\overline{\Delta}+\overline{\Delta})$ is a torus of rank at most two.
Since there is an injection
\[
\Aut^0(X,\Delta+D)\hookrightarrow \Aut^0 (\overline{X}, \oG+\overline{\Delta}+\overline{\Delta}) 
,\]
$\Aut^0(X,\Delta+D) $ is also a torus of rank at most two.

Finally, we analyze the case when  $\Aut^0(X,\Delta+D)\cong \bG_m^2$.
If
 $(X,\Delta+D)\rightsquigarrow(X_0,\Delta_0+D_0)$
is a weakly special degeneration, then 
$\reg(X_0,\Delta_0+D_0) =0$ by Proposition \ref{p:SrcSequiv} and, hence, $\Aut^0(X_0,\Delta_0+D_0)$ is a torus of rank at most two. 
Thus Proposition \ref{p:torusrk+1} implies $(X,\Delta+D)\cong (X_0,\Delta_0+D_0)$.
Therefore $(X, \Delta+D)$ is polystable as in Definition \ref{def:polystable}.
\end{proof}

\begin{prop}\label{p:tcHstableII}
If $(X,\Delta+D)$ is a regularity 0 boundary polarized CY pair
such that  $(X,\Delta+(1+\vep)D)$ is lc for $0<\vep \ll1$, then the  test configurations of $(X,\Delta+D)$ are
\begin{enumerate}
	\item the trivial test configuration and
	\item $
	(\cX_F,\Delta_{\cX_F}+\cD_F)
	$,
	where $\cX_F$ denotes the   test configuration induced by the unique lc place $F$ of $(X,\Delta)$, as well as its base changes by  $\bA^1\to \bA^1$ sending $t\mapsto t^c$.
\end{enumerate}
\end{prop}

\begin{proof}
Fix a test configuration $(\cX,\Delta_{\cX}+\cD)$ of $(X,\Delta+D)$ and $r>0$ such that $L:=-r(K_X+\Delta)$ is a Cartier divisor.
By Lemma \ref{l:lctestconfig},
the induced $\bZ$-filtration  $\cF^\bullet$ of $H^0(X,mL)$ satisfies
\[
\cF^\la H^0(X,mL) : = \bigcap_{E\subset \cX_0} \{ s\in H^0(X,mL)\, \vert\, v_E(s) \geq  \la +A_{X,\Delta}(v_E) \}
,\]
and 
$A_{X,\Delta+D}(v_E)=0$ for each $E\subset \cX_0$
Thus
$v_E = c_E \cdot \ord_F$ for some integer $c_E \geq 0$  by Lemma \ref{l:typeIIlcplace}.
Hence, if we set $c:= \max \{ c_E \, \vert\, E\subset \cX_0\}$, then
\[
\cF^\la H^0(X,mL)
=
\{ s\in H^0(X,mL)\, \vert\, c  \cdot \ord_F(s) \geq  \la  \}
\]
If $c=0$, then (1) holds.
If $c>0$, then (2) holds using the bijection  in Theorem \ref{thm:CZ-lcplace}.
\end{proof}

\subsubsection{Boundedness of degenerations}

\begin{prop}\label{p:Hstabledegbound}
Let $\mathfrak{F}$ be a set of regularity 0 boundary polarized CY pairs $(X,\Delta+D)$ satisfying $(X,\Delta+(1+\vep)D)$ is slc for $0<\vep \ll1 $.
Consider the set  of weakly special degenerations of elements in $\mathfrak{F}$:
\[
\overline{\mathfrak{F}}
:= 
\{(X_0,\Delta_0+D_0) \, \vert\, \exists (X,\Delta+D)\in \mathfrak{F} \, \text{ and w.s. degeneration } (X,\Delta+D) \rightsquigarrow (X_0,\Delta_0+D_0) 
\}
\]
If $\mathfrak{F}$ is bounded, then $\overline{\mathfrak{F}}$ is bounded.

\end{prop}

Before proving the statement, we verify the following lemma, which allows us to reduce to the case when the pairs are normal.

\begin{lem}\label{l:reducetonormalcase}
Let $\mathfrak{F}$ be a set of regularity 0 boundary polarized CY pairs and set 
\[
\mathfrak{F}^{\nu}: = \{
(\overline{X},\overline{\Delta}+\overline{G}+\overline{D}) \, \vert\, (X,\Delta+D) \in \mathfrak{F} \}.
\]
If $\mathfrak{F}^{\nu}$ is bounded, then $\mathfrak{F}$ is bounded.
\end{lem}

\begin{proof}
The subset of $\mathfrak{F}$ consisting of the lc pairs is bounded since it is contained in $\mathfrak{F}^\nu$. 
Thus  we may assume each pair in  $\mathfrak{F}$ is not lc.

For any pair $(X,\Delta+D) \in \mathfrak{F}$, the normalization   $(\overline{X},\oG+\overline{\Delta}+\overline{D})$ is a union of at most two lc boundary polarized CY pairs and 
the pair $(\overline{X},\oG+\overline{\Delta})$ is plt by Proposition \ref{p:typeIIslc}. Thus $G$ is normal by \cite[Prop. 5.51]{KM98}.
Let $\tau:\overline{G}\to  \overline{G}$ denote the induced generically fixed point free involution  fixing ${\rm Diff}_{\overline{G}}(\overline{\Delta})$ and ${\rm Diff}_{\overline{G}}(\overline{\Delta}+\overline{D})$. 
By \cite[Prop. 5.3]{Kol13},  $(\overline{X},\oG+\overline{\Delta}+\overline{D},\tau)$ uniquely determines $(X,\Delta+D)$. We will now parametrize such triples.

\noindent \emph{Claim}: There exists a triple $(Y,P+\Gamma+B, \Phi) \to S$ such that
\begin{enumerate}
\item[(i)] $(Y,P+\Gamma+B)\to S$ is a family of not necessarily connected lc boundary polarized CY pairs over a smooth finite type scheme $S$,
\item[(ii)] $P$ is a $\bZ$-divisor and is normal,
\item[(iii)] $\Phi:P\to P$ is an involution over $S$ that is generically fixed point free  and fixes ${\rm Diff}_{P}(\Gamma)$ and ${\rm Diff}_P(\Gamma+B)$, and 
\item[(iv)] if $(X,\Delta+D)\in \mathfrak{F}$, then there is  $s\in S$ such that $(Y_s,P_s+\Gamma_s+B_s,\Phi_s)\cong (\overline{X},\oG+\overline{\Delta}+\overline{D},\tau )$.
\end{enumerate}

\noindent \emph{Proof of Claim.}
Since $\mathfrak{F}^\nu$ is bounded, there exists a family of boundary polarized CY pairs $(Y,P+\Gamma+B) \to S$  over a smooth finite type  scheme $S$ such that if $(\overline{X},\oG+\overline{\Delta}+\overline{D})\in \mathfrak{F}^{\nu}$, then there is $s\in S$ such that $(Y_s,P_s+\Gamma_s+ B_s)\cong (\overline{X},\oG+\overline{\Delta}+D)$.
By replacing $S$ with a disjoint union of locally closed subsets, we may assume 
$(Y,\Gamma +P+B) \to S$ is a family of possibly disconnected  regularity 0 lc boundary polarized CY pairs and $P$ is a $\bZ$-divisor.

Next, consider the scheme ${\rm Isom}_S(P,P)$, which is a countable union of quasi-projective schemes.
Since $-K_{P/S}- {\rm Diff}_P(\Gamma) \sim_{\bQ} -(K_{Y/S}+P+\Gamma)\vert_P$ is ample over $S$, the subscheme of  ${\rm Isom}_S(P,P)$ fixing ${\rm Diff}_P(\Gamma)$ is  quasi-projective over $S$. 
Thus the locus parametrizing generically fixed point free involutions that fix ${\rm Diff}_{P}(\Gamma)$ and ${\rm Diff}_P(\Gamma+B)$ is a  quasi-projective scheme $S_1\subset {\rm Isom}_S(P,P)$.
Replacing $S$ with $S_1$ and $(Y,P+\Gamma+B) \to S$ with its base change by $S_1$ produces a triple satisfying (i),(iii), and (iv), except the smoothness of $S$.
By replacing $S$ with a disjoint union of locally closed subschemes, which are smooth varieties, we may assume $S$ is smooth.
Since $(Y_s,P_s+\Gamma_s+B_s)$ has regularity 0 for each $s\in S$, the pair $(Y_s,P_s+\Gamma_s)$ is plt by Proposition \ref{p:typeIIslc}. Hence $P_s$ is normal for each $s\in S$.
Thus, after replacing $S$ with a disjoint union of locally closed subsets, we may assume (ii) holds, which completes the proof of the claim.
\medskip

Consider a triple $(Y,P+\Gamma+B,\Phi)\to S$ satisfying the claim. 
Since $S$ is smooth, $(Y,\Gamma+P+B)$ is an lc pair  by Lemma \ref{l:slcadj}. 
Let $R(\phi)\subset X\times_S X$ denote the equivalence relation on $X$ induced by $\Phi$.
The involution $\Phi:P\to P$ induces a finite equivalence relation on $Y$.
By \cite[Cor. 5.51]{Kol13}, there is a demi-normal pair $(\tilde{Y},\tilde{\Gamma}+\tilde{B})$ over $S$ with normalization $(Y,P+\Gamma+B)$ and $\tilde{Y}$ is the geometric quotient of $Y$ by the equivalence relation.
Now, if $(X,\Delta+D) \in \mathfrak{F}$, then (iv) implies	$(Y_s,P_s+\Gamma_s+B_s,\Phi_s)\cong
(\overline{X},\oG+\overline{\Delta}+\overline{D},\tau)$.
Hence 	$(\tilde{Y}_s,\tilde{\Gamma}_s+\tilde{B}_s) \cong (X,\Delta+D)$. In conclusion, $\mathfrak{F}$ is  bounded.
\end{proof}

\begin{proof}[Proof of Proposition \ref{p:Hstabledegbound}]	
We first prove the proposition when each pair $(X,\Delta+D) \in \mathfrak{F}$ is normal. 
Since $\mathfrak{F}$ is bounded,  
there exists a family of boundary polarized CY pairs  $(Y,\Gamma + B)\to S$ over a smooth finite type scheme $S$ such that if $(X,\Delta+D)\in \mathfrak{F}$, then there is $s\in S$ such that $(Y_s,\Gamma_s+B_s) \cong (X,\Delta+D)$.
After replacing $S$ with a disjoint union of locally closed subsets, we may assume 
$(Y,\Gamma +B) \to S$ is a family of regularity 0 lc boundary polarized CY pairs that satisfy $(Y_s,\Gamma_s+(1+\vep) B_s)$ is lc for $0<\vep \ll1$ and $s\in S$.

We claim that after  replacing $S$ with a  disjoint union of locally closed sets, there exists a family of non-trivial test configurations $(\cY , \Gamma_{\cY}+\cB)\to \bA^1_S$ of $(Y,\Gamma+B)\to S$.
By Noetherian induction, it suffices to show that there exists a non-empty open set $U\subset S$ 
and a family of non-trivial test configurations $(\cY_U,\Gamma_{\cY_U}+\cB_U)\to \bA^1_U$ of $(Y_U,\Gamma_U+B_U)$.
To verify the latter, fix a generic point $g\in S$.
By Proposition \ref{p:tcHstableII}, $(Y_g,\Gamma_g+B_g)$  admits a non-trivial test configuration $(\cY_g,\Gamma_{\cY_g}+\cB_g)$.
Lemma \ref{l:tcextend} implies the test configuration extends to a family of non-trivial test configurations
$(\cY_U,\Gamma_{\cY_U}+\cB_U)\to \bA^1_U$ of $(Y_U,\Gamma_U+B_U)$ over an open set $g\in U\subset S$.
Therefore the claim holds.

To finish the proof in the normal case, it suffices to show that every element of $\overline{\mathfrak{F}}$ is a fiber of $(\cY , \Gamma_{\cY}+\cB)\to \bA^1_S$. 
This is clear, since if $(X,\Delta+D) \in \mathfrak{F}$, then there exists $s\in S$ such that  $(X,\Delta+D)\cong (Y_s,\Gamma_s+B_s)$
and a non-trivial test configuration $(\cY_s,\Gamma_{\cY_s}+\cB_s)$ of the latter pair. 
If $(X,\Delta+D)\rightsquigarrow (X_0,\Delta_0+D_0)$ is a weakly special degeneration, then Proposition \ref{p:tcHstableII} implies $(X_0,\Delta_0+D_0)$ must is isomorphic to a fiber of $(\cY_s,\Gamma_{\cY_s}+\cB_s) \to \bA^1$.

We will now deduce the full result from the normal case. 
By the normal case, we may assume the elements of $\mathfrak{F}$ are non-normal pairs.
Consider the set of pairs defined by normalization
\[
\mathfrak{F}^n:=\{   (\overline{X},\overline{\Delta}+\overline{G}+\overline{D}) \, \vert\, 
(X,\Delta+D) \in \mathfrak{F}\}
\quad
(\overline{\mathfrak{F}})^n:=\{   (\overline{X},\overline{\Delta}+\overline{G}+\overline{D}) \, \vert\, 
(X,\Delta+D) \in \overline{\mathfrak{F}}\}
.\]
Since $\mathfrak{F}$ is bounded, $\mathfrak{F}^n$ is bounded. 
By the normal case, the set  $\widetilde{\mathfrak{F}^n}$ consisting of weakly special degenerations of pairs in $\mathfrak{F}^n$, as well as their normalizations, is bounded.
Since $(\overline{\mathfrak{F}})^n \subset \widetilde{\mathfrak{F}^n}$, it follows that $(\overline{\mathfrak{F}})^n$ is bounded. 
Hence $\overline{\mathfrak{F}}$ is bounded by Lemma \ref{l:reducetonormalcase}.
\end{proof}

\subsubsection{Moduli of regularity 0 pairs}
In this section, we study  the substack 
\[
 \sM(\chi,N,{\bf r})^{\reg \leq 0}\subset \sM(\chi,N,{\bf r})
\]
consisting of families $(X,\Delta+D) \to B$ such that $\reg(X_b,\Delta_b+D_b)\leq 0$ for all $b\in B$. 
Note that the substack is open by Proposition \ref{p:regularityusc}. 
The following theorem shows that regularity $\leq 0$ pairs admit a good moduli theory.

\begin{thm}\label{t:modulireg0}
If $\sM \subset \sM(\chi,N,{\bf r},{\bf c})$ is an open, finite type substack, then 
\[
\overline{\sM}^{\reg \leq 0} : =\overline{\cM} \cap \sM(\chi,N,{\bf r})^{\reg \leq 0}
\]
\begin{enumerate}
\item is a finite type algebraic stack with affine diagonal and
\item admits a separated good moduli space.
\end{enumerate}
\end{thm}

Here $\overline{\cM}$ denotes the stack theoretic closure of $\cM$ and we moreover denote $\cM^{\reg \leq 0} := \cM \cap \sM(\chi,N,{\bf r})^{\reg \leq 0}$. To prove Theorem \ref{t:modulireg0}.1, we first verify the following boundedness result, which relies on   Proposition \ref{p:Hstabledegbound} and \cite{KX20}.

\begin{prop}\label{p:regbounded}
With the notation and assumptions in Theorem \ref{t:modulireg0}, 
the set of pairs $\overline{\sM}^{\reg \leq 0}(\bk)$ is bounded.
\end{prop}

\begin{proof} 
The proof uses the substack $\overline{\sM}^{\rm KSBA} \subset \overline{\sM}$
consisting of $(X,\Delta+D) \to B $ in $ \overline{\sM}$ such that, for each $b\in B$, $(X_b,\Delta_b+(1+\vep)D_b)\to B $ is slc for $0<\vep \ll1$. Note that the substack
\begin{enumerate}
	\item[(i)] is open by the openness of slc singularities \cite[Cor. 4.45]{KolNewBook},
	\item[(ii)] satisfies the the valuative criterion of properness \cite[Lem. 7]{KX20}, and 
	\item[(iii)] and is finite type by \cite[Thm. 2]{KX20}.
\end{enumerate}
Note that (iii) follows from \cite{KX20}, since $\overline{\sM}^{\rm KSBA}$ has finitely many irreducible components by the assumption that $\sM$ is finite type. (This  statements also follows from the stronger boundedness result in \cite[Cor. 1.8]{Bir22}).

We will first prove the proposition under the assumption that $\sM \subset \overline{\sM}^{\rm KSBA}$.
Since $\mathfrak{F}:=\cM^{\reg \leq 0}(\bk)$  and the above assumption holds, 
Proposition \ref{p:Hstabledegbound} implies
\[
\overline{\mathfrak{F}}:= \{(X_0,\Delta_0+D_0) \, \vert\, \exists (X,\Delta+D)\in \mathfrak{F}\, \text{ and a w.s. degeneration } (X,\Delta+D)\rightsquigarrow(X_0,\Delta_0+D_0)   \}
,\]
is bounded.
Thus, by Lemma \ref{p:boundedness}, there exist positive integers $k$ and $c$ such that
\begin{enumerate}
\item[($\dagger$)] $-k(K_X+\Delta)$ is very ample  Cartier divisor and $\deg( \Delta) \leq c$
\end{enumerate}
for every  if $(X,\Delta+D)$ in $ \overline{\mathfrak{F}}$.

We will show every pair  in  $\overline{\sM}^{\reg \leq 0} (\bk)$ satisfies ($\dagger$). 
Fix a pair $(X,\Delta+D)$ in $\overline{\sM}^{\reg \leq 0}(\bk)$.
Since the pair is in the closure of $\cM$, 
 there exists a family  
 \[
 (Y,\Gamma+G) \to C
 \] in $\overline{\cM}$
 over a pointed curve $0\in C$ such that the generic fiber $(Y_\eta, \Gamma_\eta+G_\eta)$ is in $\cM$, 
 and $(Y_0,\Gamma_0+G_0)\cong(X,\Delta+D)$.
 Using that  $\cM\subset \overline{\sM}^{\rm KSBA}$ and (iii), after possibly replacing $0\in C$ by a finite cover, there exists a family 
 \[
 (Y',\Gamma'+G')\to C\]
 in $\overline{\sM}^{\rm KSBA}$
 such that $(Y_\eta,\Gamma_\eta+G_\eta)\cong (Y'_\eta,\Gamma'_\eta+G'_\eta)$.
  Thus Corollary \ref{c:sequiv} implies that there exist weakly special degenerations to a boundary polarized CY pair:
  \[
  (X,\Delta+D)\rightsquigarrow (X_0,\Delta_0+D_0) \leftsquigarrow  (Y'_\kappa, \Gamma'_\kappa+G'_\kappa).
  \]
Since  $(X_0,\Delta_0+D_0)$ is in $\overline{\mathfrak{F}}$, $(X_0,\Delta_0+D_0)$  satisfies ($\dagger$)
and so $(X,\Delta+D)$ satisfies ($\dagger$) by Proposition \ref{p:anticanonicalsheaf}.
Therefore   $\overline{\cM}^{\reg \leq 0}(\bk)$ is bounded by Lemma \ref{p:boundedness}.

We will now deduce the full result from the above special case.
Since $\sM(\bk)$ is bounded,  Proposition \ref{p:boundedness} implies that there exist integers $k',c'>0$ such that
\begin{enumerate}
	\item[($\dagger'$)] $-k'(K_X+\Delta)$ is a very ample Cartier divisor and $\deg(\Delta) \leq c'.$
\end{enumerate}
holds for every $(X,\Delta+D)$ in $\cM(\bk)$.
Consider the morphism of stacks
\[
j:\cM(\chi,N, {\bf r})  \longrightarrow \cM(\chi,Nk',(r,s/k'))
\]
defined by modifying the markings of the boundary divisors by sending
\[
[(X,r\Delta^{\rm div},s D^{\rm div})\to B]\longmapsto [(X,r\Delta^{\rm div},\tfrac{s}{k'} (k'D^{\rm div}))\to B].
\] 
Let $\sM' \subset \cM(\chi, Nk',(r,s/k'))$ denote the finite type open substack consisting of families $(X,\Delta+D)\to B$ such that, for each $b\in B$, $(X_b,\Delta_b+D_b)$ satisfies  $(\dagger')$
and $(X_b,\Delta_b+(1+\vep)D_b)$ is slc for $0<\vep\ll1$.
Since $\sM' \subset \overline{\sM'}^{\rm KSBA}$, the previous  case implies $\overline{\sM'}^{\reg\leq 0}(\bk)$ is bounded. 

To finish the proof, note that:  if $(X,\Delta+D) $ is in $\cM(\bk)$, then 
\begin{itemize}
	\item $D \in \tfrac{1}{Nm}| -Nm(K_X+\Delta)|$ and 
	\item $(X,\Delta+H) $ is in $\sM'(\bk)$ for general $H \in \tfrac{1}{Nm}|-Nm(K_X+\Delta)|$.
\end{itemize}
Thus $j( \sM(\bk)) \subset \overline{\sM'}(\bk)$
and so
$j\big(\overline{\cM}^{\reg \leq 0}(\bk)\big)\subset \overline{\cM'}^{\reg \leq 0}(\bk)$.
Therefore $\overline{\cM}^{\reg \leq 0}(\bk)$ is bounded.
\end{proof}

\begin{proof}[Proof of Theorem \ref{t:modulireg0}]
Since $\overline{\sM}^{\reg  \leq 0}$ is a locally closed substack of $\sM(\chi,N,{\bf r})$, $\overline{\sM}^{\reg  \leq 0}$ is a locally finite type algebraic stack with affine diagonal by
Theorem \ref{t:stack}.
Using Proposition \ref{p:regbounded}, we see $\overline{\sM}^{\reg  \leq 0}$ is of finite type and so (1) holds.

To verify (2) and (3), let  $R$ be a DVR essentially of finite type over $\bk$.
Set
\[
S:=\Spec  \left( R[s,t]/(st-\pi) \right) \quad \text{ or } \quad   S:= \Spec(R[t]),\]
where $\bG_m$-acts on $S$ with weights $1$ and $-1$ on $s$ and $t$, respectively.
Let $0\in S$ denote the unique closed point fixed by $\bG_m$ and $S^\circ:= S \setminus 0$.
Given a $\bG_m$-equivariant family
$
(X^\circ, \Delta^\circ+D^\circ)\to S^\circ 
$
in $ \overline{\sM}^{\reg \leq 0}$,
Theorem \ref{t:Scomplete} implies the existance of  a unique $\bG_m$-equivariant extension 
$
(X,\Delta+D)\to S$
in $  \overline{\sM}$.
Since the restriction of the family to  $\{s= 0\} \hookrightarrow S$
is a test configuration and the fiber over $1\in \bA^1_\kappa$ has regularity at most $0$,
Propositin \ref{p:SrcSequiv} implies the fiber over $0 \in \bA^1_\kappa$ has  regularity at most $ 0$.
Thus the extension is in $ \overline{\sM}^{\reg \leq 0}$
and so $ \overline{\sM}^{\reg \leq 0}$ is S-complete and $\Theta$-reductive with respect to essentially finite type DVRs.
 Theorem \ref{t:AHLH} now implies (2) holds.
\end{proof}

\subsection{Coregularity 0 pairs}

We now prove a degeneration result for coregularity 0 pairs.
The result will play an important role in Section \ref{s:sequivcurves}
and Appendix \ref{a:coreg0}.

\begin{prop}\label{p:degenreducedboundary}
If $(X,\Delta+D)$ is  an lc boundary polarized CY pair of coregularity 0, then there exists a weakly special degeneration 
$(X,\Delta+D) 
\rightsquigarrow (X_0,\Delta_0+D_0)$
such that 
\[
(\overline{X}_0,\overline{G}_0+\overline{\Delta}_0+\overline{D}_0 ) \cbir (\bP^n_{x_i}, \{x_0 \cdots x_n =0\})
.\]
In particular, $K_{\overline{X}_0}+\overline{G}_0+\overline{\Delta}_0+\overline{D}_0 \sim 0$ 
and
$\Delta_0+D_0$ is a  $\bZ$-divisor.
\end{prop}

\begin{proof}
Fix a dlt modification $f:(Y,\Gamma_Y) \to (X,\Delta+D)$
and minimal lc center $Z\subset Y$ of $(Y,\Gamma_Y)$.
Since ${\rm coreg}(X,\Delta+D)=0$, $Z$ is a point.
Let $F\subset B_{Z}Y\to Z$ denote the blowup of $Y$ along $Z$ with exceptional divisor $F$. 
Since $F$ is an lc place of $(X,\Delta+D)$, 
there exists  a test configuration $(\cX,\Delta_{\cX}+\cD)\to \bA^1$ of $(X,\Delta+D)$ such that $v_{\cX_0}= \ord_F$ by Theorem \ref{thm:CZ-lcplace}.
Now consider
the sequence of birational maps 
\[
\cY
:=
B_{Z\times 0} (Y\times{\bA^1})
\overset{g}{ \longrightarrow }
Y\times{\bA^1} 
\overset{f_{\bA^1}}\longrightarrow 
X\times {\bA^1} 
\dasharrow
\cX.
\]
Let   $\cE\subset \cY$ denote  the exceptional divisor of $g$.
By Section \ref{ss:testconfigps},
$\cX_0$ is the birational transform of $\cE$ on $\cX$
and 
$\Gamma_{\cY}:= g_*^{-1}( \Gamma_{Y} \times \bA^1 + Y\times 0)$.
Since $Z\times 0$ is a 0-dimensionsal lc place of the dlt pair 
$(Y\times {\bA^1},\Gamma_{Y}\times \bA^1 + Y\times 0)$,
\[
K_{\cY}+ \Gamma_{\cY}+ \cE = g^*(K_{Y\times\bA^1}+ \Gamma_{Y} \times \bA^1 +Y\times 0 )
\]
and $(\cE,\Gamma_{\cY}\vert_{\cE}) \cong (\bP^n, \{x_0 \cdots x_n = 0\})$.
Thus the two maps
\[
(\cY, \Gamma_\cY + \cE) \to (X\times \bA^1, \Delta\times \bA^1 +D\times \bA^1+X\times 0) \dashrightarrow (\cX,\Delta_{\cX} +\cD+\cX_0)
\]
are  crepant birational; for the second map see 
Lemma \ref{l:CYbirational}.1.
Thus
\cite[Prop. 4.6]{Kol13} implies $(\cE , \Gamma_{\cY}\vert_\cE)\dashrightarrow  (\overline{X}_0,\overline{G}_0+\overline{\Delta}_0+\overline{D}_0 )$ is crepant birational, which proves the first statement. The second statement follows from the first.
\end{proof}

As a consequence of the proposition, we deduce the following sufficient condition for a boundary polarized pair to have positive coregularity.

\begin{prop}\label{p:coreg>0condition}
Let $(X,\Delta+D)$ be a boundary polarized CY pair
with   $\Delta$  a $\bZ$-divisor.

If there exists a rational number $r>0$ such that $D=r D'$ for some $\bZ$-divisor $D'$ and $1\notin r \bZ$, then  $\coreg(X,\Delta+D)>0$.
\end{prop}

\begin{proof}
By taking the normalization, we may assume $(X,\Delta+D)$ is lc. 
If $\coreg(X,\Delta+D)=0$, then 
Proposition \ref{p:degenreducedboundary} implies that there exists a weakly special degeneration 
\[
(X,\Delta+D) \rightsquigarrow(X_0,\Delta_0+D_0)
\]
such that $\Delta_0+D_0$ has coefficient 1 on each prime divisor in its support. 
Since $\Delta_0$ is a $\bZ$-divisor and $(X_0,\Delta_0+D_0)$ is slc,  $D_0$ must also be a reduced $\bZ$-divisor.
Since $D_0= r D'_0$ for some $\bZ$-divisor $D'_0$,  $ 1\in r\bZ$, which is a contradiction.
\end{proof}

\section{Moduli stacks of plane curves}\label{s:planecurves1}

In this section, we define the moduli stack $\cP_d^{\CY}$, as well as its various open substacks defined using KSBA stability, K-stability, regularity, and index of the canonical divisor. 
When $3\nmid d$, we show $\cP_{d}^{\CY}$ admits a proper good moduli space in Theorem \ref{t:moduliexists3notd}.

\subsection{Moduli stack}
To a smooth plane curve $C\subset \bP^2$ of degree $d \geq 3$, we associate  a boundary polarized CY pair
$(\bP^2, \Delta+D)$,
where $\Delta=1 \cdot \emptyset$ and  $D:= \tfrac{3}{d}C$.
Note that 
\[
N_d(K_{\bP^2}+\tfrac{3}{d} C)\sim 0, \quad \text{ where } N_d:=
\begin{cases}\tfrac{d}{3} & \text{ if } 3 \mid d \\
d & \text{ otherwise}
\end{cases}
\]
We will proceed  to define a stack parametrizing such pairs and their slc degenerations.

\begin{defn}[Stack of smooth plane curves]
Let $\chi(m):=
 \chi(\bP^2,\omega_{\bP^2}^{[-m]}) =\binom{3+m}{m}$.
For each integer $d\geq 3$, let
\[
\cP_d
\subset \cM(\chi,N_d,(1,\tfrac{3}{d}) )
\]
denote the open substack consisting of families $(X,\Delta+D) \to B$ in $\cM(\chi,N_d,(1,\tfrac{3}{d}) )$
with  $\Delta:= 1 \cdot \emptyset$ and $D: = \tfrac{3}{d} C$ such that   $X_b$ and $C_b$ are smooth for all $b\in B$.
\end{defn}

Since $\bP^2$ is the only smooth del Pezzo surface with Hilbert function $\chi$,  if $(X,D)\to B$ is  a family in $\cP_d$ and $b\in B$, then
$X_{\overline{b}} \cong \bP^2_{k(\overline{b})}$ and $C_{\overline{b}}$ is a smooth  plane curve of degree $d$.

\begin{defn}[Moduli of slc CY degenerations]\label{defn:stackCYplanecurves}
Let
${ \cP}_d^{\rm CY}
:= \overline{\cP_d}$, which denotes
the  stack theoretic closure of 
$\cP_d$ in $\sM(\chi,d,(1,\tfrac{3}{d}))$.
\end{defn}

For any $(X,\Delta+D) \to B$ in $\cP_{d}^{\CY}$, the boundary $\Delta:=1\cdot \emptyset$. 
Hence, we will from now on simply write $(X,D)\to B$ for the family.

\begin{rem}[Points of the stack]
A pair $(X, D)$
is in ${ \cP}_d^{\rm CY}(\bk)$ if 
there exists a family of boundary polarized CY pairs $(\cX,\cD) \to T$  over the germ of a pointed curve $0\in T$ such that  the restriction to $T^\circ:= T\setminus 0$ is a family in $\cP_{d}$ and
  $(\cX_0,\cD_0) \cong (X_0,D_0)$.
\end{rem}

\begin{example}[Boundedness fails]\label{e:unbounded}
When $3 \mid d$, the stack  $\cP_d^{\rm CY}$ is not of finite type.
This follows from either of the two examples, which shows $\cP_{d}^{\rm CY}(\bk)$ is not bounded.

\begin{enumerate}
\item For positive integers $a,b,c$ satisfying Markov's inequality, i.e. $a^2+b^2+c^2=3abc$, we claim that
the toric pair with reduced toric boundary
\[
(\bP_{x,y,z}(a^2,b^2, c^2), \tfrac{3}{d} \{(xyz)^{d/3}=0\})
\]
is a $\bk$-point of $\cP_{d}^{\rm CY}$. 
Since the canonical divisor of  $\bP(a^2,b^2, c^2)$ has Cartier index $abc$, which is unbounded, $\cP_{d}^{\CY}(\bk)$ is unbounded by \cite[Lem. 2.25]{Bir19}.

The claim is a consequence of \cite[Thm. 8.3]{Hac04}, 
which states that there is a projective family of klt pairs $X\to T$, over the germ of a smooth curve $0\in T$, such that 
$X_0 \cong \bP(a^2,b^2,c^2)$ and 
$X_t\cong \bP^2$
for $ t\neq 0$. 
Since
\[
(X_0, \tfrac{3}{d}C_0 ):= (\bP(a^2,b^2, c^2), \tfrac{3}{d} \{(xyz)^{d/3}=0\}),
\]
is a boundary polarized CY pair of index 1,
Lemma \ref{l:extenddivisor} (proven in a later section) implies that $C_0$ extends to a divisor $C\subset X$ such that 
$(X,D:=\tfrac{3}{d}C)\to T $ is a family of boundary polarized CY pairs with $K_{X/T}+D\sim 0$.
Since $(X_{K(T)}, D_{K(T)})$ is in $ \cP_{d}^{\rm CY}$, so is $(X_0,D_0)$.

\item 
When $3 \mid d$, methods from toric geometry produce weakly special degenerations
\[
(\bP^2, \tfrac{3}{d} \{(xyz)^{d/3}=0\}) \rightsquigarrow(X_0,D_0)
\]
such that $X_0$ is not irreducible and can have  arbitrarily many irreducible components. 
\end{enumerate}
\end{example}

Related moduli stacks of plane curves have been studied in the canonically polarized KSBA setting by Hacking \cite{Hac04} and in the log Fano K-moduli setting  in \cite{ADL19}.

\begin{defn}
The \emph{Hacking substack} $\cP_d^{\rm H} \subset \cP_d^{\CY}$ is the substack consisting of $(X,D) \to B $ in $\cP_{d}^{\CY}$ such that, for each $b\in B$, $(X_b,(1+\vep )D_b)$ is slc for   $0<\vep \ll1$.
\end{defn}

\begin{thm}[\cite{Hac04}]\label{t:Hacking}
For each integer $d\geq 4$,
\begin{enumerate}
\item $\cP_d^{\rm H}$ is an open substack of $\cP_d^{\CY}$ containing $\cP_d$, and 
\item $\cP_d^{\rm H}$ is a finite type, proper  Deligne--Mumford stack, and admits a projective coarse moduli space $P_{d}^{\rm H}$.
\end{enumerate}
\end{thm}

\begin{proof}
Statement (1) follows from the openness of slc singularities \cite[Cor. 4.45]{KolNewBook} and the definition of slc. 
Statement (2) follows from \cite{Hac04} and \cite{KP17} for the projectivity.
\end{proof}

In addition,  Hacking \cite[Thm. 7.1]{Hac04} proved that  
if $3 \nmid d$, then the stack $\cP_d^{\rm H}$ is smooth and the underlying surfaces $X$ are either Manetti surfaces, i.e. a klt surface that admits a $\bQ$-Gorenstein smoothing to $\bP^2$  (see \cite[Def. 8.1]{Hac04} or \cite[Cor. 1.2]{HP10}), or unions of two normal surfaces glued along a smooth rational curve.

\begin{defn}
The \emph{K-moduli substack} $\cP_d^{\rm K} \subset \cP_d^{\CY}$ is the substack consisting of $(X,D) \to B$ in $ \cP_{d}^{\CY}$ such that, for each $b\in B$,  $(X_b,(1-\vep )D_b)$ is a K-semistable log Fano pair for  $0<\vep \ll1$.
\end{defn}

 We will explain in Remark \ref{r:Kdef=ADL} that the K-moduli substack $\cP_d^{\rm K}$ agrees with the stack defined in \cite{ADL19}. 
See \cite{ADL19} for the definition of K-semistability.
While the exact definition will play no role in this paper, we will frequently use that K-semistable log Fano pairs are  klt by \cite{Oda13b}.

\begin{thm}[\cite{ADL19}]\label{t:ADL}
For each integer $d\geq 3$, 
\begin{enumerate}
\item $\cP_{d}^{\rm K}$ is a  open substack of $\cP_{d}^{\rm CY}$ containing $\cP_d$, and
\item $\cP_d^{\rm K}$ is finite type and admits a projective good moduli space $P_{d}^{\rm K}$.
\end{enumerate}
\end{thm}

\begin{proof}
Statement (1) follows from the openness of K-semistability in families of smoothable log Fano pairs in  \cite[Cor. 3.17 \& Prop. 4.6.3]{ADL19} (see also \cite{BLX19, Xu20}). Statement (2)  follows from \cite[Thm. 1.1 \& 1.2]{ADL19} and \cite{XZ20} for the projectivity party.
\end{proof}

\begin{rem}
For a parameter $0<c< \tfrac{3}{d}$, \cite{ADL19} also constructs compactifications $\cP_d\subset \cP_{d,c}^{\rm K}$ by viewing a smooth degree $d$ plane curve $C\subset \bP^2$ as a K-polystable log Fano pair $(\bP^2,cC)$ and using K-moduli theory to compactify the stack.
The authors show $\cP_{d,c}^{\rm K}$ satisfies a wall-crossing framework as $c$ varies and $\cP_{d,c}^{\rm K}$ parametrizes GIT semistable plane curves when $0<c\ll\tfrac{3}{d}$.
\end{rem}

\begin{lem}\label{l:SequivKmod}
If $(X,D)$ is in $\cP_d^{\CY}(\bk)$, then it is S-equivalent to a pair in
$\cP_{d}^{\rm K}(\bk)$.
\end{lem}

\begin{proof}
Fix a  family  $(\cX,\cD)\to T$ in $\cP_{d}^{\rm CY}$ over the germ of a curve $0 \in T$ such that $(X,D)\cong (\cX_0,\cD_0)$ and $(\cX_{K},\cD_K) \in \cP_d(K)$, 
where  $K:=K(T)$.
By \cite[Thm. 3.19]{ADL19}, after possibly replacing $0 \in T$ by a finite cover, 
there exists a family $(\cX',\cD') \to C$ in $\cP_d^{\rm K}$ 
such that $(\cX'_K,\cD'_K)\cong (\cX_K,\cD_K)$. 
Thus Corollary \ref{c:sequiv} implies $(X,D)\sim_S(\cX'_0,\cD'_0)$.
\end{proof}

\subsection{Stratification via type}

\begin{defn}[Type] A pair $(X,D)$ in $\cP_{d}^{\rm CY}(\bk)$ is called
\begin{enumerate}
\item \emph{Type I} if $(X,D)$ is klt,
\item \emph{Type II} if $\Src(X,D)$ is a curve, and
\item \emph{Type III} if $\Src(X,D)$ is a point.
\end{enumerate}
\end{defn}

\begin{rem}
A pair in $\cP_d^{\CY}(\bk)$ is Type I, II, or III if and only if the regularity is $-1$, $0$, or $1$, respectively.
When discussing degenerations of surface pairs, we use type, rather than regularity, to mirror terms used for degenerations of K3 surfaces \cite{Kul77}.
\end{rem}

\begin{rem}\label{r:Sequiv-type}
By Proposition \ref{p:SrcSequiv}, S-equivalent pairs in $\cP_{d}^{\CY}(\bk)$ are of the same type.
\end{rem}

Using Theorem \ref{t:modulireg0}, we will show that the union of the Type I and II loci of the stack admits a good moduli space. 
The Type III locus of the stack is related to the failure of boundedness in Example \ref{e:unbounded}
and will be analyzed in later sections.

\begin{defn}[Type I and II substack]
Let $\cP_{d}^{\rm CY,I+II}\subset  \cP_{d}^{\rm CY}$ be the substack consisting of  families
$(X,D) \to B$ in $\cP_{d}^{\rm CY,I+II}$ such that $(X_b,D_b)$ is Type I or II for all $b\in B$. 
\end{defn}

\begin{prop}\label{p:stackI+II}
The following hold:
\begin{enumerate}
\item $\cP_{d}^{\rm CY,I+II} \hookrightarrow \cP_d^{\rm CY}$ is an open embedding, 
\item $\cP_{d}^{\rm CY,I+II}$ is a finite type algebraic stack with affine diagonal, and
\item there exists a morphism $\cP_{d}^{\rm CY,I+II} \to P_{d}^{\rm CY,I+II}$ to a separated good moduli space.
\end{enumerate}
\end{prop}

\begin{proof}
Statement (1) follows from 	Proposition  \ref{p:regularityusc}, while  (2) and (3) 
from Theorem \ref{t:modulireg0}.
\end{proof}

In the case when the degree is not divisible by 3, we will show that Type III degenerations do not occur and deduce the existence of a good moduli space in that case. 

\begin{prop}\label{p:dnmid3plt}
If $(X,D)$ is in $\cP_{d}^{\CY}(\bk)$ and $3\nmid d$, then the pair is Type I or II and the normalization $(\overline{X},\overline{G})$ of $(X,0)$ is plt.
\end{prop}

\begin{proof}
Since $1\notin \tfrac{3}{d} \bZ$, Proposition \ref{p:coreg>0condition} implies $(X,D)$ must be Type I or II.
If  $(\overline{X},\overline{G})$ is not plt, then Proposition \ref{p:typeIIslc} implies $X$ is lc and not klt.
Thus  $X$ must be the cone over an elliptic curve by \cite[Thm. 8.5]{Hac04}.
This case cannot occur when $3\nmid d$ by either the proof of \cite[Proof of Thm. 7.1]{Hac04} or  Theorem \ref{thm:type2Sequiv}, which is proven in a later section.
\end{proof}

\begin{thm}\label{t:moduliexists3notd}
If $3 \nmid d$, then $ \cP_{d}^{\rm CY,I+II}=\cP_{d}^{\CY}$. Additionally,
\begin{enumerate}
\item $ \cP_{d}^{\CY}$ is a finite type algebraic stack with affine diagonal, and
\item there exists a good morphism  $\cP_{d}^{\CY} \to P_{d}^{\CY}$  to a proper good moduli space.
\end{enumerate}
\end{thm}

\begin{proof}
Propositions \ref{p:stackI+II} and \ref{p:dnmid3plt} imply statement (1) and  the existence of a separated good moduli space $P_{d}^{\rm CY}$.
Since $\cP_{d}^{\rm CY}$ satisfies the existence part of the valuative criterion of properness  by Theorem \ref{thm:properness}, $P_{d}^{\rm CY}$ is proper by Theorem \ref{t:AHLH}.
\end{proof}

\subsection{Stratification via index}

\begin{defn}	
For a boundary polarized CY pair $(X,D)$, the \emph{index} at $x\in X$ is  
\[
{\rm ind}_x(K_X) := \min \{ m\in \bZ_{>0} \, \vert\, mK_X \text{ is Cartier in a neighborhood of $x$}\}
.\]
For each $m\geq 1$, we define the  substack 
$\cP_{d,m}^{\rm CY} \subset \cP_{d}^{\rm CY}$ consisting of families $(X,D)\to B$ in $\cP_{d}^{\CY}$ such that ${\rm ind}_x  (K_{X_b}) \leq m$ at each $b\in B$ and $x\in X_b$.
\end{defn}

Observe that there is a chain of inclusions
\[
\cP_{d} \subset \cP_{d,1}^{\CY}\subset 
\cP_{d,2}^{\CY}
\subset \cdots 
\quad \text{ and } \quad  \cP_{d}^{\CY}= \cup_{m \geq 1}\cP_{d,m}^{\CY}
.\]
The advantage to working with $\cP_{d,m}^{\rm CY}$ is that the stack is finite type by the following lemma.

\begin{prop}\label{p:Pdmstack} The following hold:
\begin{enumerate}
\item $\cP_{d,m}^{\rm CY} \subset \cP_{d}^{\rm CY}$ is an open substack, 
\item $\cP_{d,m}^{\rm CY} $ is a finite type algebraic stack with affine diagonal.
\end{enumerate}
\end{prop}

\begin{proof}
To prove (1), it suffices to show that if $f:(X,D) \to B$ is a family in $\cP_d^{\rm CY}$, then
\[
B^\circ := \{ b\in B\, \vert\, {\rm ind}_{x} ({X_b}) \leq m \text { at each } x \in X_b\}
\]
is open in $B$.  By Definition \ref{d:nonNoeth}, we may assume $B$ is Noetherian. 
Now, for $1\leq j \leq m$,  let $U_j \subset X$ denote the open locus where $\omega_{X/B}^{[j]}$ is an invertible sheaf and $Z_j := X\setminus U_j$. 
For $b\in B$ and $x\in X_b$,  $jK_{X_b}$ is Cartier at $x$ if and only if $\omega_{X_b}^{[j]}$ is invertible at $x$. Since $\omega_{X/B}^{[j]}\vert_{X_b}\cong \omega_{X_b}^{[j]}$,
$jK_{X_b}$ is invertible at $x$  if and only if  $x\in U_j$.
Thus $B^\circ = B\setminus  f( \cap_{j=1}^m Z_j)$, which is open.

Since (1) holds and $\cP_{d}^{\rm CY}$ is a locally finite type algebraic stack with affine diagonal, $\cP_{d,m}^{\rm CY}$ satisfies the same properties.
To show the stack is of finite type, it suffices to verify  $\cP_{d,m}^{\CY}(\bk)$ is bounded. 
Now, if $(X,D:= \tfrac{3}{d}C)$ is in $\cP_{d,m}^{\CY}(\bk)$, then 
$-m!K_X$ is Cartier divisor with  Hilbert polynomial $\chi(m !\,\cdot\, )$. 
Thus \cite[Thm. 2.1.2]{Kol85} implies that there exists a fixed integer $k$ such that $-kK_X$ is a very ample Cartier divisor for every $(X,D)$ in $\cP_{d,m}^{\CY}(\bk)$. 
Thus $\cP_{d,m}^{\CY}$ is finite type by Proposition \ref{p:boundedness}.
\end{proof}

\section{Explicit S-equivalence classes}\label{s:sequivcurves}

In this section, we prove the following two theorems characterizing S-equivalence classes of Type II and III pairs in $\cP_{d}^{\rm CY}$.
Theorem \ref{thm:type2Sequiv} describes a representative element in each Type II S-equivalence class.
Theorem \ref{t:TypeIIISequiv} shows that all Type III pairs are S-equivalent.

\begin{thm}\label{thm:type2Sequiv}
Let $(X,D)\in \cP_d^{\CY}(\bk)$ be a Type II pair. Then $(X,D)$ is S-equivalent to one of the following pairs $(X_0,D_0)\in\cP_d^{\CY}(\bk)$.
\begin{enumerate}[label=(\roman*)]
    \item $X_0$ is a projective cone over a smooth elliptic curve polarized with a degree $9$ line bundle, and $D_0$ is the section at infinity;
    \item $X_0$ is a projective orbifold cone over $\bP^1$ whose orbifold divisor (see Definition \ref{def:orb-div}) is toric, and $D_0$ is the section at infinity plus a $\bQ$-linear combination of rulings;
    \item $X_0$ is a toric Manetti surface with a specific $\bG_m$-action, and $D_0$ is $\bQ$-linear combination of the closure of $\bG_m$-orbits, such that the broken $\bG_m$-orbit (see Definition \ref{def:brokenorbit}) on $X_0$ is non-reduced. 
    \item $X_0$ is the gluing of two orbifold cones over $\bP^1$ with the same toric orbifold divisor, and $D_0$ is a $\bQ$-linear combination of rulings on each irreducible component of $X_0$.
\end{enumerate}

Moreover, for a fixed source $(E,D_E)$, which is either a smooth elliptic curve or a CY pair on $\bP^1$, and a fixed integer $d\geq 3$, there exist at most finitely many $(X_0,D_0)\in\cP_d^{\CY}(\bk)$ up to isomorphism from the above list with $\Src(X_0,D_0)\cong (E,D_E)$.
\end{thm}

Since the proof of Theorem \ref{thm:type2Sequiv} relies on a lengthy and in depth analysis of the geometry of certain Type II pairs, we briefly  outline the argument.
First, by Proposition \ref{l:SequivKmod}, we may assume  $X$ is klt.
Next, if there is a non-exceptional lc place of $(X,D)$, we show in Section \ref{ss:nonexclcplaces} that the lc place induces a weakly special degeneration of $(X,D)$ to a pair satisfying (i) or (ii).
If all lc places of $(X,D)$ are exceptional, then we show that each lc place induces a  degeneration $(X,D) \rightsquigarrow (X_0,D_0)$ such that $X_0$ is klt and the pair admits an effective $\bG_m$-action.
By an in depth analysis of such pairs in Section \ref{s:Manetti}, we show in Section \ref{ss:exclcplaces} that $(X_0,D_0)$ satisfies (ii) or (iii) or degenerates to (iv).

\begin{thm}\label{t:TypeIIISequiv}
If $(X,D)\in \cP_d^{\CY}(\bk)$ is Type III, then $
(X,D) \sim_S
(\bP^2, \tfrac{3}{d}\{ (xyz)^{d/3}=0 \} )$
\end{thm}

The proof of Theorem  \ref{t:TypeIIISequiv} is relatively brief, in comparison to the proof of Theorem \ref{thm:type2Sequiv}. The argument relies on degenerating $(X,D)$ to an index 1 pair.

The above two theorems play key roles in the proof of Theorem \ref{t:main2}.
In particular, the theorems will be used to prove Proposition \ref{p:1-complementS} on the existence of 1-complements on Type II and III pairs. 
The latter is a key ingredient in the construction of $P_d^{\CY}$. 
In addition, the theorems will be used in the proof of Theorem \ref{t:main2}.4 to show the Hodge line bundle is big on subvarieties contained in the Type II and III loci of the moduli space.

\subsection{Manetti surfaces with $\bG_m$-actions}\label{s:Manetti}
Throughout this subsection, we assume the following condition on $(X,D)$. 
\begin{equation*}
    \begin{split}
    &(X,D)\in \cP_d^{\rm CY}(\bk) \textrm{ is a Type II pair such that $X$ is klt, $(X,D)$ admits an effective}\\& \textrm{\quad  $\bG_m$-action, and that every lc place of $(X,D)$ is exceptional over $X$}. 
    \end{split}\tag{$\dagger$}\label{eq:dagger}
\end{equation*}

There are precisely two lc places of $(X,D)$ denoted by $E_1$ and $E_2$. This follows because $(X,D)$ admits two test configurations of product type induced by the $\bG_m$-action $\lambda$ and its inverse $\lambda^{-1}$ which gives us at least two lc places by Lemma \ref{l:lctestconfig}.2. On the other hand, a Type II boundary polarized CY surface pair can have at most two log canonical places by Lemma \ref{l:typeIIlcplace}.1. 

Next, let $\mu:\tX\to X$ be the dlt modification of $(X,D)$. Then both $E_1$ and $E_2$ are $\mu$-exceptional prime divisors on $\tX$.  Let $\tD$ be the strict transform of $D$ in $\tX$. Then $(\tX, E_1+E_2+\tD)$ is a plt CY pair with a $\bG_m$-action. Denote by $D_{E_i}$ the different of $(\tX, \tD+ E_1+E_2)$ on $E_i$. Denote by $\Delta_{E_i}$ the different of $(\tX, E_1+E_2)$ on $E_i$.

The main result of this section is the following theorem.

\begin{thm}\label{t:Manetti-Gm}
Notation as above. Assume that $(X,D)$ satisfies \eqref{eq:dagger}. Denote by $\mu_i: \tX_i\to X$ the extraction of $E_i$.  
Then one of the following is true.
\begin{enumerate}[label=(\roman*)]
    \item $-K_{\tX_i} - E_i$ is ample for some $i\in \{1,2\}$, and  $\tX$, $\tX_i$, and $X$ are all toric.
    \item $-K_{\tX_i} - E_i$ is big and nef for all $i\in \{1,2\}$, and $((-K_{\tX_i} -  E_i)\cdot  C_i) = 0$ for a $\bG_m$-invariant curve $C_i$ on $\tX_i$.
\end{enumerate}
\end{thm}

We denote by $\phi_i:\tX\to \tX_i$ the birational morphism contracting $E_{3-i}$.
Then we have the following diagram:
\[
\begin{tikzcd}
\tX \arrow[r,"\phi_1"]\arrow[d,swap,"\phi_2"] \arrow[dr, "\mu"]& \tX_1\arrow[d,"\mu_1"]\\
\tX_2\arrow[r,swap,"\mu_2"] & X
\end{tikzcd}
\]
By abuse of notation, we  denote $\phi_{i,*} E_i$ by $E_i$ as the unique $\mu_i$-exceptional divisor on $\tX_i$.

\begin{lem}\label{l:typeIIKc}
Notation as above. Then $E_i$ is a Koll\'ar component over $\mu_i(E_i)\in X$ extracted by $\mu_i:\tX_i\to X$ for $i=1,2$.
\end{lem}

\begin{proof}
Since $(\tX, E_1+E_2+\tD)$ is plt, so is $(\tX, E_i)$ for each $i=1,2$. Thus $(\tX_i, E_i)$ is plt near $E_i$ and klt elsewhere as $\mu_i$ is isomorphic away from $E_i$. This implies that $(\tX_i, E_i)$ is plt. Since $X$ has only quotient singularities, it is $\bQ$-factorial. By \cite[Lem. 2.62]{KM98} we have that  $-E_i$ is $\mu_i$-ample as it is the only $\mu_i$-exceptional divisor.
\end{proof}

\begin{lem}\label{l:typeIIfanotype}
Notation as above. Then, $\tX$ is of Fano type and $\rho(\tX)=3$.
\end{lem}

Recall, $\widetilde{X}$ of \emph{Fano type} means that there exists a $\bQ$-divisor $\widetilde{B}$ on $\widetilde{X}$ such that $(\widetilde{X},\widetilde{B})$ is klt and $-K_{\widetilde{X}}- \widetilde{B}$ is big and nef. 

\begin{proof}
For $0< \varepsilon\ll 1$, we have that $(X, (1-\varepsilon)D)$ is a klt log Fano pair, and 
\[
A_{X, (1-\varepsilon)D}(E_i) = A_{X, D}(E_i) + \varepsilon \ord_{E_i}(D) = \varepsilon \ord_{E_i}(D) < 1.
\]
Denote by $a_i:=\varepsilon \ord_{E_i}(D)$. Then $(\tX,(1-\varepsilon)\tD + (1-a_1) E_1 + (1-a_2) E_2)$ is a klt weak log Fano pair as it is crepant birational to $(X, (1-\varepsilon) D)$. Thus $\tX$ is of Fano type.

By \cite[Prop. 6.3]{Hac04}, we know that $\rho(X)=1$. Since $\tX \to X$ extracts two exceptional curves $E_1$ and $E_2$, we have $\rho(\tX) = \rho(X) +2 =3$.
\end{proof}

The following result is useful for understanding the geometry of $\widetilde{X}$.

\begin{prop}\label{p:typeIIManetti}
Notation as above. Then, there exists a $\bG_m$-equivariant contraction $\pi:\tX \to B$ to a smooth rational curve $B$ with connected fibers such that the following hold.
\begin{enumerate}
    \item $E_1$ and $E_2$ are disjoint sections of $\pi$;
    \item every fiber of $\pi$ is $\bG_m$-invariant.
\end{enumerate}
\end{prop}

\begin{proof}
By \cite[Proof of Proposition 12.3.2]{FlipAbund}, running the $(K_{\tX}+ \tD + (1-\varepsilon) (E_1+E_2))$-MMP terminates with $ \tX\to Z \to B$ such that $\tX\to Z$ is birational, and $Z\to B$ is a Mori fiber space over a smooth curve with connected fibers. Since $\bG_m$ is a connected algebraic group, this MMP is automatically $\bG_m$-equivariant. Moreover, since $X$ is a rational surface, so is $\tX$. Thus $B$ is a smooth rational curve. Therefore (1) follows directly from \cite[Prop. 12.3.2]{FlipAbund}. 

For (2), notice that $(E_i, D_{E_i})$ is $\bG_m$-invariant. By adjunction, we know that $K_{E_i}+D_{E_i}\sim_{\bQ} 0$ and $(E_i, D_{E_i})$ is klt. Thus $\Supp(D_{E_i})$ contains at least three points on $E_i\cong \bP^1$, which implies that the $\bG_m$-action on $E_i$ has to be trivial. Since $\pi|_{E_i}: E_i\to B$ is a $\bG_m$-equivariant isomorphism, we know that the $\bG_m$-action on $B$ is also trivial. Thus every fiber of $\pi$ is $\bG_m$-invariant.
\end{proof}

\begin{defn}\label{def:brokenorbit}
    Notation as above. Assume that $(X,D)$ satisfies \eqref{eq:dagger}. A \emph{broken $\bG_m$-orbit} on $X$ is a divisor of the form $\mu_*\pi^*[b]$ for a point $b\in B$ such that $\pi^*[b]$ is reducible. 
\end{defn}

\begin{prop}\label{p:typeIIManetti-2exc}
Notation as above. Then there is precisely one broken $\bG_m$-orbit $\mu_*\pi^*[b]$ on $X$. Moreover,  $\pi^*[b]$ has precisely two irreducible components $\tC_1$ and $\tC_2$.

\end{prop}

\begin{proof}
Since $\mu$ only contracts $E_1$ and $E_2$, it suffices to show that there is precisely one reducible fiber $\pi^*[b]$, and such a fiber has two irreducible components. 
We first show that there is at most one reducible fiber. Assume to the contrary that fibers over two distinct points $b,b'\in B$ are reducible. Let $\tC_1$ (resp. $\tC_1'$) be an irreducible component of $\pi^*[b]$ (resp. $\pi^*[b']$). Let $m_1:=\ord_{\tC_1} \pi^*[b]>0$. Since $\pi$ has connected fibers, we have that $(\tC_1\cdot (F-m_1 \tC_1))>0$ and $(\tC_1\cdot F) =0$ for a general fiber $F$ of $\pi$. Thus, $(\tC_1^2)<0$, and similarly $(\tC_1'^2)<0$. Running the $(\tC_1+\tC_1')$-MMP over $B$ yields a birational morphism $\tX \to Z'$ that contracts both $\tC_1$ and $\tC_1'$. Note that we can run any MMP on $\tX$ since it is of Fano type by Lemma \ref{l:typeIIfanotype}. Thus we have $\rho(\tX) = \rho(Z') +2 \geq \rho(B) +3 = 4$, which is a contradiction to Lemma \ref{l:typeIIfanotype}. 

Next, we show that there is at least one reducible fiber. Assume to the contrary that every fiber is irreducible. Then all fibers define the same ray in $\overline{NE}(\tX)$, which implies that $\rho(\tX/B) = 1$. Thus $\rho(\tX) = \rho(B) + 1 = 2$, a contradiction to Lemma \ref{l:typeIIfanotype}. Thus (3) is proved.

For the last statement, similar arguments to those above show that if the reducible fiber $\pi^*[b]$ has at least three irreducible components, then we would have $\rho(\tX) \geq 4$ which contradicts Lemma \ref{l:typeIIfanotype}. Thus $\pi^*[b]$ has exactly two irreducible components $\tC_1$ and $\tC_2$.
\end{proof}

\begin{prop}\label{p:stP^1link}
Notation as above. There exists a $\bG_m$-equivariant birational morphism $\psi_i:\tX \to Z_i$ that contracts the curve $\tC_{3-i}$. Denote by $E_{i,j}$ (resp. $D_{Z_i}$) the pushforward of $E_j$ (resp. of $\tD$) to $Z_i$. Then $(Z_i, D_{Z_i}+E_{i,1}+E_{i,2})$ is a standard $\bP^1$-link over $B$. Moreover, after switching $\tC_1$ and $\tC_2$ if necessary, we have $(\tC_i\cdot E_i)>0$ and $(\tC_i\cdot E_j)=0$ for $j\neq i$.
\end{prop}

\begin{proof}
From the proof of Proposition \ref{p:typeIIManetti-2exc}, we have $(\tC_{3-i})^2<0$. Hence, we may run a $\tC_{3-i}$-MMP on $\tX$ over $B$ which yields a birational morphism $\psi_i: \tX\to Z_i$ that only contracts $\tC_{3-i}$. Since the MMP is automatically $\bG_m$-equivariant, $\lambda$ induces an effective $\bG_m$-action on $Z_i$.  Since $(\tX, \tD + E_1 + E_2)$ is a plt CY pair, so is $(Z_i, D_{Z_i} + E_{i,1} + E_{i,2})$. By Proposition \ref{p:typeIIManetti-2exc}, we know that every fiber of $Z_i\to B$ is irreducible. Thus $(Z_i, D_{Z_i})\to B$ is a log Mori fiber space. By \cite[Proof of Proposition 4.37]{Kol13}, we know that $(Z_i, D_{Z_i}+E_{i,1} + E_{i,2})$ is a standard $\bP^1$-link over $B$. 

To prove the last statement, we notice that each $E_i$ intersects at least one of $\tC_1$ and $\tC_2$, since the fiber $\pi^*[b]$ has support $\tC_1\cup \tC_2$ by Proposition \ref{p:typeIIManetti-2exc} and $E_i$ is a section by Proposition \ref{p:typeIIManetti}. If some $\tC_i$ intersects both $E_1$ and $E_2$, then after contracting $\tC_i$ we have that $E_{j,1}$ and $E_{j,2}$ intersect each other in $Z_j$ for $j\neq i$, which contradicts that $(Z_j, D_{Z_j}+E_{j,1} + E_{j,2})$ is a standard $\bP^1$-link. Thus the proof is finished.
\end{proof}

\begin{prop}\label{p:typeIItoric}
Notation as above. Let $m_i:=\ord_{\tC_i}(\pi^*[b])$ and $ a_i:=\ord_{\tC_i}(\tD)$. Then the following hold.
\begin{enumerate}
    \item $\frac{1-a_1}{m_1} = \frac{1-a_2}{m_2}$.
    \item If the broken $\bG_m$-orbit on $X$ is non-reduced, then $\tX$ is toric.
\end{enumerate}
\end{prop}

\begin{proof}
(1) Let $p_i$ be the point in $E_i$ such that $\pi(p_i) = b$. By contracting $\tC_j$ from $\tX$ where $j\neq i$, we see that the fiber of $Z_i \to B$ over $b$ is $m_i (\psi_{i,*}\tC_i)$. Thus by Propositions \ref{prop:orbcone-adjunction}, \ref{p:Seifert-P^1link}, and \ref{p:stP^1link}, and \eqref{eq:seifert-adj-2},  we get that $\ord_{p_i} \Delta_{E_i} = 1-\frac{1}{m_i}$ and $\ord_{p_i} \tD|_{E_i} = \frac{a_i}{m_i}$. 
Thus 
\[
\ord_{p_i} D_{E_i} = \ord_{p_i} \Delta_{E_i}+ \ord_{p_i} \tD|_{E_i} = 1-\frac{1-a_i}{m_i}.
\]
Since $(\tX, \tD+E_1+E_2)$ is a $\bP^1$-link over $B$, by \cite[Prop. 4.37]{Kol13} we know that $\ord_{p_1} D_{E_1}= \ord_{p_2} D_{E_2}$. This implies (1).

(2) The non-reducedness of the broken $\bG_m$-orbit is equivalent to $m_1>1$ or $m_2>1$ by Proposition \ref{p:typeIIManetti-2exc}. Let us assume $m_1>1$ without loss of generality. Consider the birational morphism $\psi_1: \tX \to Z_1$ by contracting $\tC_2$. Denote by $C_{1,1}:=\psi_{1, *} \tC_1$. We know that  $\pi_1: Z_1\to B$ is a compactified Seifert $\bG_m$-bundle by Proposition \ref{p:Seifert-P^1link}. Then we know that $\pi_1^{-1}(b) = \psi_{1,*} \pi^*[b] = \psi_{1,*} (m_1 \tC_1 + m_2 \tC_2) = m_1 C_{1,1}$  which is a non-reduced fiber.

We claim that the standard $\bP^1$-link $\pi_1: Z_1\to B$ has at most one non-reduced fiber other than $\pi_1^{-1}(b)$. Denote by $\Delta_{E_{1,1}}$ the different of $(Z_1, E_{1,1})$ on $E_{1,1}$. Then by Proposition \ref{p:Seifert-P^1link}, we know that $\Supp(\Delta_{E_{1,1}})$ equals the intersection of $E_{1,1}$ with the union of all non-reduced fibers of $\pi_1$. Thus it suffices to show that $\Supp(\Delta_{E_{1,1}})$ has at most two points. Since $\psi_1: \tX \to Z_1$ is isomorphic over a neighborhood of $E_{1,1}$, we know that $(E_{1,1}, \Delta_{E_{1,1}})\cong (E_1, \Delta_{E_1})$ which is toric by Lemma \ref{l:pltcyclicquot}. Thus $\Supp(\Delta_{E_{1,1}})$ has at most two points.

Next, we show that $Z_1$ is toric. Assume that $\pi_1$ is induced by the $\bQ$-divisor $L_1$ on $B$. Then $\Supp(\{L_1\}) = \Supp(\pi_{1,*}\Delta_{E_{1,1}})$, which contains at most two points. Thus $(B, \{L_1\})$ is toric. Since the Seifert $\bG_m$-bundle only depends on the $\bZ$-linear equivalence class of $L_1$, we may assume that $L_1$ is toric as well. This shows that $Z_1$ is toric. 

Finally, we show that $\tX$ is toric. From the paragraph above, we know that  $C_{1,1}$, $E_{1,1}$ and $E_{1,2}$ are all torus invariant divisors on $Z_1$. Thus $(Z_1, C_{1,1} + E_{1,1} + E_{1,2})$ is a toric log canonical pair. It suffices to show $\tC_2$ is an lc place of $(Z_1, C_{1,1} + E_{1,1}+ E_{1,2})$ which would imply that $\pi_1: \tX \to Z_1$ is toric. Since $(\tX, \tD + E_1 + E_2)$ and $(Z_1, D_{Z_1} + E_{1,1} + E_{1,2})$ are crepant birational, we have
\[
1-a_2 = A_{\tX, \tD + E_1 + E_2}(\tC_2) = A_{Z_1, D_{Z_1} + E_{1,1} + E_{1,2}}(\tC_2).
\]
By construction, we know that $D_{Z_1} -a_1 C_{1,1}$ is a $\bQ$-linear combination of fibers of $\pi_1$ other than $\pi_1^{-1}(b)$. Thus we have
\begin{align*}
    A_{Z_1, C_{1,1} + E_{1,1} + E_{1,2}}(\tC_2) & = A_{Z_1, D_{Z_1} + E_{1,1} + E_{1,2}}(\tC_2) - (1-a_1)\ord_{\tC_2} (C_{1,1})\\
    & = (1-a_2) - (1-a_1) \ord_{\tC_2} \left(\frac{\pi_1^{-1}(b)}{m_1}\right)\\
    & = (1-a_2) - (1-a_1)\frac{m_2}{m_1} =0,
\end{align*}
where the last equality follows from (1). Thus the proof is finished.
\end{proof}

\begin{lem}\label{l:typeIINEcone}
Notation as above. Then $\overline{NE}(\tX_i)$ is spanned by  $E_i$ and $C_i:=\phi_{i,*} \tC_i$.
\end{lem}

\begin{proof}
 Since $\phi_i$ is an isomorphism in a neighborhood of $E_i\cup \tC_i$, we know that $(E_i)_{\tX_i}^2 = (E_i^2)_{\tX}<0$ and $(C_i^2) = (\tC_i^2)<0$. Thus $[E_i]$ and $[C_i]$ are two extremal curve classes in $\overline{NE}(\tX_i)$. By Lemma \ref{l:typeIIfanotype}, we know that $\rho(\tX_i) = \rho(\tX) - 1 =2$. Thus the result follows.
\end{proof}

\begin{proof}[Proof of Theorem \ref{t:Manetti-Gm}]
We know that $\tX_i \to X$ extracts the Koll\'ar component $E_i$. Thus we have $((-K_{\tX_i}-E_i)\cdot E_i)>0$. By Lemma \ref{l:typeIINEcone}, the ampleness of $-K_{\tX_i}-E_i$ is equivalent to the positivity of its intersection with $C_i$. Since $\phi_i$ is isomorphic in a neighborhood of $E_i\cup \tC_i$, we have
\[
((-K_{\tX_i}-E_i)\cdot C_i) = ((-K_{\tX} - E_1 - E_2) \cdot \tC_i) = (\tD\cdot \tC_i).
\]
Since $\tD$ is supported in finitely many fibers of $\pi$, we know that $(\tD\cdot (m_1 \tC_1 + m_2 \tC_2)) = (\tD\cdot F) = 0$.

If $(\tD\cdot \tC_i)>0$ for some $i$, we claim that either $m_1$ or $m_2$ is larger than $1$. Assume to the contrary that $m_1=m_2 =1$. Then Proposition \ref{p:typeIItoric}.1 implies that $a_1=a_2$. Thus $\tD$ is a $\bQ$-linear combination of fibers of $\pi$, which implies that $(\tD\cdot \tC_1)=(\tD\cdot \tC_2) = 0$, a contradiction. Thus by Proposition \ref{p:typeIItoric}.2, we know that $\tX$ is toric. Since $\tC_j$ is torus invariant for $j\neq i$, we have that $\tX_i$ is also toric. This falls into case (i).

If $(\tD\cdot \tC_1) = (\tD\cdot \tC_2) =0$, then the above discussion implies that $-K_{\tX_i}-E_i$ is nef for any $i$. To show its bigness, consider the birational morphism $g_i: \tX_i\to \tX_i'$ by contracting $C_i$. Such a contraction exists as $(C_i)^2 = (\tC_i^2) <0$ and $\tX$ is of Fano type by Proposition \ref{l:typeIIfanotype}. Thus $-K_{\tX_i}-E_i = g_i^* L_i'$ for some $\bQ$-line bundle $L_i'$ on $\tX_i'$. Since $((-K_{\tX_i}-E_i)\cdot E_i) >0$, we know that $(L_i'\cdot g_{i,*} E_i) > 0$. Since $\rho(\tX_i') = \rho(\tX_i) -1 = 1$, this implies that $L_i'$ is ample on $\tX_i'$. This shows the bigness of $-K_{\tX_i}-E_i$. The proof is finished.
\end{proof}

\subsection{Type II pairs}

The goal of this subsection is to prove Theorem \ref{thm:type2Sequiv}.

By Lemma \ref{l:SequivKmod}, it suffices to show that every Type II pair $(X,D)\in \cP_d^{\rm K}(\bk)$ 
satisfies the statement. Because $X$ is a Manetti surface, $X$ has cyclic quotient singularities and $\rho(X) = 1$ by \cite[Thm. 8.3 \& Prop. 6.3]{Hac04}. Let $E$ be an lc place of $(X,D)$. Then either $E$ is a curve on $X$ or $E$ is an exceptional divisor over $X$. 

\subsubsection{Non-exceptional lc place}
\label{ss:nonexclcplaces}
If $E$ is a curve on $X$, we may degenerate $(X,D)$ to the normal cone over $(E,\Delta_E)$, where $\Delta_E$ is the different divisor of $(X,D)$ on $E$. This gives an S-equivalence between $(X,D)$ and a projective orbifold cone over $(E, \Delta_E)$:

\begin{lem} \label{l:degnormalcone}
Let $(X,D)$ be an lc boundary polarized CY pair and $E=\lfloor D\rfloor$  a prime divisor on $X$ such that $-K_X\sim_{\bQ} rE$ for some rational number $r\geq 1$.
If $(X,E)$ is plt, then $E$ induces a weakly special degeneration $(X,D) \rightsquigarrow (X_0,D_0)$, 
where $X_0$ is a projective orbifold cone over $E$, and $E$ degenerates to the section at infinity. Moreover, $X_0$ is klt if $r>1$.
\end{lem}

\begin{proof}
The proof is similar to \cite[Proof of Lemma 2.12]{LZ19}. Let $R= \oplus_{m=0}^\infty R_m$ where $R_m = H^0(X, \cO_X(mE))$ so $\Proj\, R \cong X$.  The divisor $E$ induces an $\bN$-filtration $\cF$ of $R$ as
\[
\cF^\lambda R_m : = \{s\in R_m \mid \ord_E(s) \geq \lambda \} \cong H^0(X, \cO_X(m-\lambda)E)\quad \textrm{ for }\lambda\in \bZ_{\geq 0}.
\]
Let $L_E:= E|_E$ be the $\bQ$-divisor on $E$ as defined in  \cite[Appendix A]{HLS19}, which is well-defined up to $\bZ$-linear equivalence. By similar arguments to \cite[Proof of Proposition 2.10]{LZ19}, we have the short exact sequence
\[
0\to \cO_X((m-\lambda -1 ) E) \to \cO_X((m-\lambda)E) \to \cO_E(\lfloor (m-\lambda)L_E\rfloor)\to 0.
\]
For $\lambda\in [0, m-1]$, we have $(m-\lambda-1) E -K_X$ is ample, hence Kawamata-Viehweg vanishing implies $H^1(X, \cO_X((m-\lambda-1) E)) = 0$. For $\lambda = m$, we know that $H^0(X, \cO_X)\cong \bk \cong H^0(E, \cO_E)$. As a result, we have that 
\[
\gr_{\cF} ^ \lambda R_m \cong H^0(E, \cO_E(\lfloor (m-\lambda)L_E\rfloor)) \quad \textrm{ for every }\lambda\in \bZ_{\geq 0}.
\]
Thus, by Theorem \ref{thm:CZ-lcplace}, we know that $E$ induces a weakly special degeneration $(X,D)\rightsquigarrow (X_0, D_0)$ such that the central fiber
\[
X_0 \cong \Proj \bigoplus_{m=0}^{+\infty} \bigoplus_{\lambda= 0}^{+\infty} H^0(E, \cO_E(\lfloor (m-\lambda)L_E\rfloor)).
\]
This is precisely $X_0\cong C_p(E, L_E)$.
Denote by $\Delta_E$ the different of $(X,0)$ on $E$. 
If $r>1$, then by adjunction we know that $-K_E-\Delta_E = (-K_X-E)|_E \sim_{\bQ} (r-1)L_E$ is ample. Thus $(E, \Delta_E)$ is a klt log Fano pair, which implies that $X_0$ is klt. The statement on the degeneration of $E$ follows from the same argument as \cite[Proof of Lemma 2.12]{LZ19}.
\end{proof}

\begin{prop}\label{prop:E-non-exc}
Let $(X,D)\in \cP_d^{\rm CY}(\bk)$ be a Type II pair such that $X$ is klt. Assume that there exists an lc place $E$ of $(X,D)$ as a prime divisor on $X$. Then $(X,D)$ admits a weakly special degeneration to $(X_0,D_0)$, which is one of the following:
\begin{enumerate}
 \item $X_0$ is a projective cone over an elliptic curve polarized with a degree $9$ line bundle, and $D_0$ is the section at infinity;
 \item $X_0$ is a projective orbifold cone over $\bP^1$, whose orbifold divisor is toric, and $D_0$ is the section of infinity plus a $\bQ$-linear combination of rulings.  
\end{enumerate}
\end{prop}

\begin{proof}
By Lemma \ref{l:degnormalcone}, $(X,D)$ admits a weakly special degeneration to $(X_0, D_0)$ where $X_0$ is a projective orbifold cone over $E$.

If $r=1$, then $E=D\sim_{\bQ} -K_X$. By adjunction, we know that $(E, \Delta_E)$ is klt and $K_E + \Delta_E \sim_{\bQ} 0$. Then, we know that $X_0$ has a strictly log canonical singularity at the cone point. By \cite[Thm. 8.5]{Hac04}, we get that $X_0$ is an elliptic cone of degree $9$, i.e. $E$ is an elliptic curve, $\Delta_E = 0$, and $\deg (L_E)=9$. By Lemma \ref{l:degnormalcone}, the degeneration of $E=D$ to $X_0$ is precisely the section at infinity.  Thus we get  degeneration (1).

If $r>1$, then $(E, \Delta_E)$ is a klt log Fano pair which implies that $E\cong \bP^1$. By Lemma \ref{l:degnormalcone} $X_0$ is a Manetti surface.
If $\Supp(\Delta_E)$ has at least $3$ points, then  the section at infinity $E_0$ in $X_0$ also satisfies $|\Supp(\Delta_{E_0})|\geq 3$. This implies that $X_0$ has at least $3$ singular points along $E_0$. Meanwhile, the cone point of $X_0$ cannot be smooth as every Koll\'ar component over a smooth point is toric by Lemma \ref{l:pltcyclicquot} (see also \cite[Lem. 2.11]{FJ04} or \cite{Kaw17}). Thus $X_0$ has at least $4$ singular points, which is a contradiction as every Manetti surface  has at most $3$ singular points by \cite[Thm. 8.3]{Hac04}. As a result, $\Supp(\Delta_E)$ has at most $2$ points which implies that  $(E,\Delta_E)$ is toric. Since $E$ is an irreducible component of $D$, by Lemma \ref{l:degnormalcone} we know that the degeneration $D_0$ of $D$ on $X_0$ is $\bG_m$-invariant which contains the section at infinity as an irreducible component. Thus we get degeneration (2), and the proof is finished.
\end{proof}

\subsubsection{Exceptional lc places}\label{ss:exclcplaces}

Next, we study the case where $E$ is exceptional over $X$.

\begin{lem}\label{lem:E-exc-deg}
Let $(X,D)\in \cP_d^{\CY}(\bk)$ be a Type II pair such that $X$ is a Manetti surface. Assume that $(X,D)$ has an lc place $E$ as an exceptional divisor over $X$. Then $E$ induces a non-trivial weakly special degeneration $(X_0, D_0)$ of $(X,D)$ such that $X_0$ is a Manetti surface equipped with an effective $\bG_m$-action and $D_0$ is $\bG_m$-invariant.
\end{lem}

\begin{proof}
Since $X$ has klt singularities, the exceptional divisor $E$ is a rational curve.
By Theorem \ref{thm:CZ-lcplace}, we have that $E$ induces a non-trivial weakly special degeneration $(X,D)\rightsquigarrow (X_0, D_0)$ with $X_0$ integral. If $X_0$ is not normal, then by \cite[Thm. 5.5 \& 6.5]{Hac04} we know that $(X_0, D_0)$ is of Type III, which is a contradiction. Hence, $X_0$ is normal. If $X_0$ is an elliptic cone, then its source is a smooth elliptic curve, which is a contradiction as $E$ is rational and the source is preserved under weakly special degenerations by Proposition \ref{p:SrcSequiv}. Thus, by \cite[Thm. 8.5]{Hac04}, we know that $X_0$ is a Manetti surface.
\end{proof}

From now on, we may assume that $(X,D)$ satisfies \eqref{eq:dagger}.
We will also use the notation from Section \ref{s:Manetti}. 

\begin{prop}\label{prop:S-equiv-finite-iii}
   For a fixed source $(E\cong \bP^1,D_E)$ and a fixed degree $d$ there exist at most finitely many pairs $(X,D)$ up to isomorphism satisfying \eqref{eq:dagger} such that the broken $\bG_m$-orbit on $X$ is non-reduced, and $\Src(X,D)\cong (E, D_E)$.
\end{prop}

\begin{proof}
First of all, we will show that, given a fixed source $(E, D_E)$, there exist finitely many  $\tX$, which are necessarily toric varieties by Theorem \ref{t:Manetti-Gm}.
By the boundedness of Type II pairs (Theorem \ref{t:modulireg0}), we know that such $X$ are bounded. Since $X$ is toric and boundedness implies that such $X$ are $\epsilon$-log terminal, this implies that $X$ has finitely many choices up to isomorphism by \cite{BB93}.  Next, we show that $A_X(E_1)$ and $A_X(E_2)$ are bounded from above. For simplicity let us assume $i=1$. Let $L_1 : = -E_1|_{E_1}$ be the ample $\bQ$-divisor from \cite{HLS19}*{Appx. A}. Recall that $\Delta_{E_1}$ is the different divisor of $(\tX, E_1)$. Then we have that $\Delta_{E_1}\leq D_E$ under the identification  $E_1\cong E$. Write $\{L_1\} = \sum_{j=1}^{l} \frac{a_j}{b_j} p_j$, where $p_j$ are distinct points on $E_1$, both $a_j, b_j\in \bZ_{>0}$, and  $\gcd(a_j,b_j)=1$. Then we have that $\Delta_{E_1} = (L_1)_{\orb} = \sum_{j=1}^l \frac{b_j-1}{b_j} p_j$ by \cite[Appendix A]{HLS19}. Hence there are only finitely many choices of $l$ and $(b_1,\cdots, b_l)$, which implies that there are only finitely many choices of $\{L_1\}$ and $\Delta_{E_1}$. We have 
\begin{align*}
    A_X(E_1)\cdot \deg L_1 & = A_X(E_1)\cdot(- E_1^2) = (-K_{\tX_1}-E_1)\cdot E_1 \\ &  = \deg(-K_{E_1}-\Delta_{E_1}) = 2- \deg \Delta_{E_1}.
\end{align*}
Since $X$ is bounded, its Gorenstein index is bounded from above, which implies that $A_X(E_1)$ has bounded denominators. On the other hand, since $\{L_1\}$ has finitely many choices, we know that $\deg L_1$ has bounded denominators as well. Thus their product having finitely many choices implies that both $A_X(E_1)$ and $\deg L_1$ have finitely many choices. Since $E_1$ and $E_2$ are toric divisors over $X$ with bounded log discrepancy, we know that $\tX$ has finitely many choices.

Let $m_1$ and $m_2$ be the multiplicities of the broken $\bG_m$-orbit on $X$ along its irreducible components as defined in Proposition \ref{p:typeIItoric}. 
Without loss of generality, we may assume $m_1>1$ as the broken $\bG_m$-orbit is non-reduced. Recall that $\psi_1: \tX \to Z_1$ is the contraction of the curve $\tC_2$ such that $\pi_1: Z_1 \to E$ is a compactified Seifert $\bG_m$-bundle. Since $\tX$ has finitely many choices and $\psi_1$ is toric, we know that $Z_1$ has finitely many choices. From the construction, we know that $(Z_1, D_{Z_1} + E_{1,1} + E_{1,2})$ is the compactified Seifert $\bG_m$-bundle over $(E,D_E)$. Thus $(Z_1, D_{Z_1} + E_{1,1} + E_{1,2})$ has finitely many choices. It suffices to show that the birational map $\psi_1: \tX \to Z_1$ has finitely many choices up to isomorphism. Since $C_{1,1}=(\psi_1)_* \tC_1$ has multiplicity $m_1>1$, we know that $p_j:=\pi_1(C_{1,1})$ lies in the support of $D_E$. Thus there are finitely many choices of $(Z_1, C_{1,1}+E_{1,1}+E_{1,2})$. From the proof of Proposition \ref{p:typeIItoric}, we know that $\tC_2$ is an lc place of  $(Z_1, C_{1,1}+E_{1,1}+E_{1,2})$. Moreover, finiteness of $\tX$ and $Z_1$ implies that $A_{Z_1}(\tC_2)$ is bounded from above. Thus there are only finitely many extractions $\psi_1: \tX\to (Z_1, C_{1,1}+E_{1,1}+E_{1,2})$ of $\tC_2$. This shows that $(\tX, \tD + E_1+E_2)$ has finitely many choices up to isomorphism, which implies that $(X,D)$ has finitely many choices up to isomorphism. 
\end{proof}

\begin{prop}\label{prop:E-exc}
Let $(X,D)\in \cP_d^{\rm CY}(\bk)$ be a Type II pair satisfying \eqref{eq:dagger}. If the broken $\bG_m$-orbit on $X$ is reduced, 
then $(X,D)$ admits a weakly special degeneration to $(X_0,D_0)$, where $X_0$ is the gluing of two orbifold cones over $\bP^1$ along sections at infinity and $D_0$ is a $\bQ$-linear combination of rulings on each irreducible component of $X_0$. Moreover, the conductor of $X_0$ with its different divisor is toric. 
\end{prop}

\begin{proof}
Let $x_1\in X$ be the center of $E_1$ on $X$. From Section \ref{s:Manetti}, we know that
there exists a proper birational morphism $\mu_1: \tX_1 \to X$ from a normal projective variety $\tX_1$ that extracts $E_1$. Since $(X,D)$ is of Type II, we know that $(\tX_1, (\mu_1)_{*}^{-1} D + E_1)$ is plt along $E_1$. By Lemma \ref{l:typeIIKc} we know that $\mu_1$ extracts $E_1$ as a Koll\'ar component over $x_1\in X$. Denote by $\Delta_{E_1}$ the different of $\tX_1$ on $E_1$. Denote by $L_1:= -E_1|_{E_1}$ as an ample $\bQ$-Cartier $\bQ$-divisor from \cite[Appendix A]{HLS19}.

Next, we perform the blowup on the trivial test configuration $(X_{\bA^1},D_{\bA^1}):= (X,D)\times\bA^1$ to extract a divisor $\cE$ induced by $E_1$. Let $\cE$ be the prime divisor over $X_{\bA^1}$ as the $(1,1)$ quasi-monomial combination of $(E_1)_{\bA^1}$ and $X\times\{0\}$. As a divisorial valuation, we know that $\ord_{\cE}$ is the $\bG_m$-equivariant extension of $\ord_{E_1}$ on $K(X_{\bA^1}) = K(X)(t)$ such that $\ord_{\cE}(t) = 1$, see e.g. \cite[(3.2)]{Liu18}. By the same argument as \cite[Proof of Lemma 33]{Liu18}, we know that $\oplus_{m=0}^\infty \fa_m(\ord_\cE)$ is finitely generated since $E_1$ is a Koll\'ar component over $x_1\in X$. Denote by 
\[
\cY := \Proj_{X_{\bA^1}} \bigoplus_{m=0}^\infty \fa_m(\ord_{\cE})
\]
Let $\sigma:\cY\to X_{\bA^1}$ be the projection map. By the same argument as \cite[Proof of Lemma 33]{Liu18}, we know that $\sigma$ extracts $\cE$ as a Koll\'ar component over $(x_1,0)\in X_{\bA^1}$, and \[
\cE \cong \Proj (\gr_{\ord_E} \cO_{X,x})[s],
\]
where $s$ is a free variable of degree $1$. Moreover, the central fiber $\cY_0 = \tX_1 + \cE$ where $\tX_1 \cap \cE = E_1$. By \cite[Prop. 2.10]{LZ19}, we know that $\gr_{\ord_{E_1}} \cO_{X,x}\cong \oplus_{m=0}^\infty\cO_{E_1}(\lfloor mL_1\rfloor)$ which implies that $\cE$ is isomorphic to the projective orbifold cone $C_p(E_1,L_1)$ over $(E_1,\Delta_{E_1})$. 

\medskip

\noindent \emph{Claim.} $\cY$ is $\bQ$-factorial and of Fano type over $\bA^1$.
\medskip

\noindent \emph{Proof of claim.} Since $X$ has only quotient singularities, it is $\bQ$-factorial, hence so is $X_{\bA^1}$. Since $\sigma: \cY\to X_{\bA^1}$ extracts a unique exceptional divisor $\cE$ which is $\bQ$-Cartier, we know that $\cY$ is $\bQ$-factorial. Denote by $\cD_{\cY}:=\sigma_*^{-1} D_{\bA^1}$. Since $\cE$ is an lc place of $(X_{\bA^1}, X\times\{0\} + D_{\bA^1})$, we know that $(X_{\bA^1}, (1-\varepsilon)(X\times\{0\} + D_{\bA^1}))$ is klt log Fano over $\bA^1$, and $a := A_{X_{\bA^1}, (1-\varepsilon)(X\times\{0\} + D_{\bA^1})}(\cE) \in (0,1)$ for $0<\varepsilon\ll 1$. Thus $(\cY, (1-\varepsilon)(\tX_1 + \cD_{\cY}) + (1-a) \cE)$ is klt weak log Fano over $\bA^1$, which implies that $\cY$ is of Fano type over $\bA^1$. The claim is proven. \qed

\medskip 

From the claim, by \cite[Cor. 1.3.2]{BCHM10} we may run the $\bG_m$-equivariant $\tX_1$-MMP on $\cY$ over $\bA^1$. 
Since $\tX_1|_{\tX_1} = -\cE|_{\tX_1} = -E_1$ is not nef while $\tX_1|_{\cE} = E_1$ is ample, the first step of the $\tX_1$-MMP yields a contraction $\phi_{\cY}:\cY\to \cX$ where an $E_1$-negative extremal ray in $\overline{NE}(\tX_1)$ gets contracted.  By Proposition \ref{p:stP^1link} and Lemma \ref{l:typeIINEcone} we know that $\overline{NE}(\tX_1)$ is spanned by $E_1$ and $C_1$ with $(C_1\cdot E_1)>0$ and $(C_1^2)<0$. Thus the contraction $\phi_{\cY}$ is a small flipping contraction that only contracts $C_1$ to a point. 

Denote by $\cE'$, $X'$, and $\cD$ the push-forward of $\cE$, $\tX_1$, $\cD_{\cY}$ under $\phi_{\cY}$ respectively, so that $\cX_0 = X' + \cE'$. Denote by $E':= \cE'\cap X'$.
Since $(\cY, \cD_{\cY}+\cY_0)$ is log canonical CY over $\bA^1$, so is $(\cX, \cD+\cX_0)$. 

Next, we show that $-K_{\cX}-\cX_0$ is $\bQ$-Cartier and ample over $\bA^1$.
By Theorem \ref{t:Manetti-Gm}, we know that $((K_{\tX_1}+E_1)\cdot C_1)=0$. By adjunction we have $(K_{\cY}+\cY_0)|_{\tX_1}=K_{\tX_1} +E_1$, which implies that $((K_{\cY}+\cY_0)\cdot C_1)= 0$. Thus by \cite[Cor. 3.17]{KM98} we know that $K_{\cX} + \cX_0 = (\phi_{\cY})_* (K_{\cY} + \cY_0)$ is $\bQ$-Cartier. Moreover, Theorem \ref{t:Manetti-Gm} implies that $-K_{\tX_1} - E_1$ is big and nef on $\tX_1$, which implies that $-K_{X'}-E'$ is big and nef on $X'$. Since $\rho(X')=\rho(\tX_1)-1=1$, we have that $-K_{X'}-E'$ is ample on $X'$. Similarly, we know that $(\cE',E')\cong (\cE, E_1)$ is a plt log Fano pair as $(E_1,\Delta_{E_1})$ is klt log Fano and $\rho(\cE)=1$. Thus by adjunction we know that $-K_{\cX} - \cX_0$ is ample over $\bA^1$. This implies that $(\cX, \cD)\to \bA^1$ is a weakly special test configuration of boundary polarized CY pairs. 

Finally, we show that $(\cX_0, \cD_0)$ is the gluing of two orbifold cones over $\bP^1$ along sections at infinity, and  the conductor with its different divisor is toric. It is clear that $(\cE', E' + \cD_0|_{\cE'})\cong (\cE, E_1 +\cD_{\cY,0}|_{\cE})$ is an orbifold cone over $(E_1, D_{E_1})$. On the other hand, we know that $X'$ is obtained by contracting both $E_2$ and $\tC_1$ from $\tX$. Thus $X'$ is also obtained by contracting $E_{2,2}$ on $Z_2$. By Proposition \ref{p:stP^1link} we know that $Z_2$ is a standard $\bP^1$-link with a negative section $E_{2,2}$. Hence from Section \ref{s:seifert} we know that $X'$ is isomorphic to an orbifold cone over $\bP^1$ where $E'$ is the section at infinity and $\cD_0|_{X'}$ is a $\bQ$-linear combination of rulings.  Denote  by $\Delta_{E'}$ the different divisor  of the conductor $E'$ of $\cX_0$. Since $(\cE', E')\cong (\cE, E_1)$, we have that $(E', \Delta_{E'})\cong (E_1, \Delta_{E_1})$ which is toric by Lemma  \ref{l:pltcyclicquot}. Thus the proof is finished.
\end{proof}

\begin{proof}[Proof of Theorem \ref{thm:type2Sequiv}]
By Lemma \ref{l:SequivKmod} and Remark \ref{r:Sequiv-type}, we may assume that $(X,D)\in \cP_d^{\rm K}(\bk)$. If there exists an lc place $E$ of $(X,D)$ as a prime divisor on $X$, then $(X,D)$ is S-equivalent to pairs in (i) or (ii) by Proposition \ref{prop:E-non-exc}. If an lc place $E$ of $(X,D)$ is exceptional over $X$, then by replacing $(X,D)$ with its degeneration induced by $E$ (see Lemma \ref{lem:E-exc-deg}) we may assume that $(X,D)$ is a Type II pair satisfying \eqref{eq:dagger}. If the broken $\bG_m$-orbit on $X$ is non-reduced, then we set $(X_0,D_0)=(X,D)$ which gives us case (iii). If the broken $\bG_m$-orbit on $X$ is reduced, then by Proposition \ref{prop:E-exc} we have a weakly special degeneration $(X,D)\rightsquigarrow (X_0,D_0)$ as in case (iv). The statement on S-equivalence is proved. 

Next we prove the finiteness of $(X_0,D_0)$ once its source $(E,D_E)$ is fixed in each case.

\medskip

\noindent (i) For a fixed elliptic curve $E$, the pair $(X_0,D_0)$ is uniquely determined by the polarization $L$ which is a degree $9$ line bundle. Since any two degree $9$ line bundles are isomorphic up to an automorphism of $E$, $(X_0,D_0)$ is uniquely determined by $E$ up to isomorphism.

\medskip

\noindent (ii) Let $L$ be the ample $\bQ$-divisor on $E$ such that $X_0=C_p(E, L)$. Since $(X_0,D_0)$ is a projective orbifold cone over $(E,D_E)$ with $\bQ$-polarization $L$, it suffices to show that there are at most finitely many $L$ up to $\bZ$-linear equivalence. First of all, we know that $L_{\orb} \leq D_E$ as $\lfloor D_E\rfloor=0$, which implies that there are at most finitely many choices of $L_{\orb}$ hence also $\{L\}$. Choose $r\in \bQ_{>0}$ such that $-K_E- L_{\orb} \sim_{\bQ} rL$ which implies $2-\deg (L_{\orb}) = r\deg (L)$. Then by Proposition \ref{prop:orbcone-vol} we have  
\[
9 = (-K_{X_0})^2 = (1+r)^2 \deg(L) =\frac{(1+r)^2}{r} (2 - \deg (L_{\orb})).
\]
Since there are finitely many choices of $L_{\orb}$, we know that $r$ and hence $\deg(L)$ have finitely many choices. This implies that $L$ has at most finitely many choices up to $\bZ$-linear equivalence. 

\medskip

\noindent  (iii) This follows directly from Proposition \ref{prop:S-equiv-finite-iii}.

\medskip

\noindent (iv) Let $(\oX_0, \oG_0 + \oD_0)$ be the normalization of $(X_0, D_0)$ where $\oG_0$ is the conductor. Then we know that $\oX_0$ has two connected components $\oX_{0,1}$ and $\oX_{0,2}$. Denote by $\oG_{0,i}$ and $\oD_{0,i}$ the restriction of $\oG_0$ and $\oD_0$ to $\oX_{0,i}$. Then we know that  $(\oX_{0,i}, \oG_{0,i} + \oD_{0,i})$ is a projective orbifold cone over $(E,D_E)$ with $\bQ$-polarization  $L_i$ on $E$. Moreover, since $X_0$ is an slc Fano variety, we know that the different divisors of $(\oX_{0,i}, \oG_{0,i})$ on $\oG_{0,i}$ are the same. By Proposition \ref{prop:orbcone-adjunction}.2, this implies that $(L_i)_{\orb} = \Delta_E$ is independent of the choice of $i$. Choose $r_i\in \bQ_{>0}$ such that $-K_E-\Delta_E \sim_{\bQ} r_i L_i$ which implies $r_i \deg (L_i) = 2-\deg \Delta_E$. Hence by Proposition \ref{prop:orbcone-vol} we have
\begin{align*}
9 & = (-K_{X_0})^2 = (-K_{\oX_{0}}-\oG_{0})^2 \\
&= r_1^2\deg(L_1) + r_2^2 \deg(L_2) = \left(\frac{1}{\deg(L_1)} + \frac{1}{\deg(L_2)}\right) (2-\deg\Delta_E)^2.
\end{align*}
Since $\Delta_E\leq D_E$ and $\lfloor D_E\rfloor = 0$, we know that  $\Delta_E$ and hence $\{L_i\}$ have finitely many choices. In particular, this shows that the possible values of $\frac{1}{\deg(L_i)}$ form an ACC set. Since $\frac{1}{\deg(L_1)} + \frac{1}{\deg(L_2)}$ has finitely many choices from the equation above, we know that both $\deg(L_1)$ and $\deg(L_2)$ have at most finitely many choices. This implies that both $L_1$ and $L_2$ have at most finitely many choices and so does the normalization $(\oX_0, \oG_0 + \oD_0)$. Moreover, since $(E,D_E)$ is a klt CY pair where $E\cong\bP^1$, we know that $\Supp(D_E)$ has at least three points which implies that $\Aut(E,D_E)$ is a finite group. Thus for each $(\oX_0, \oG_0 + \oD_0)$ there are at most finitely many ways to glue the two components of the conductor $\oG_0$. This implies that $(X_0, D_0)$ has at most finitely many choices up to isomorphism. The proof is finished.
\end{proof}

\subsection{Type III pairs}
We will now begin the proof of Theorem \ref{t:TypeIIISequiv}, which states that all Type III pairs are S-equivalent. The main step is  the following proposition.

\begin{prop}\label{p:typeIII-Sequiv}
If $(X,D) $ is a Type III lc pair in $\cP_{d}^{\rm CY}(\bk)$, then it is a S-equivalent to a pair $(X_0,D_0)$ with $K_{X_0}+D_0 \sim 0$. 
\end{prop}

\begin{proof}
By Proposition \ref{p:degenreducedboundary}, there exists a weakly special degeneration 
$(X,D) \rightsquigarrow (X_0,D_0)$
such that the normalization $(\overline{X}_0 , \overline{G}_0+\overline{D}_0 )$ is crepant birational to  $(\bP^2,\{xyz=0\})$ and $K_{\overline{X}_0}+\overline{G}_0+ \overline{D}_0 \sim 0$.
If $X_0$ is normal, then the proof is complete. 
Hence, we will now assume $X_0$ is not normal. 
\medskip

\noindent \emph{Claim 1}: The $\bZ$-divisor  $\overline{\Gamma}_0:= \overline{G}_0+ \overline{D}_0 $ is a cycle of $\bP^1$'s glued along nodes and ${\rm Diff}_{\overline{\Gamma}_0}(0)=0$. 

\noindent \emph{Proof of claim}: Fix a dlt modification $f:(Y,\Gamma_Y)\to (\overline{X}_0,\overline{\Gamma}_0)$.
Since $\{xyz=0\}$ is a cycle of smooth curves glued along nodes and
\[
(Y,\Gamma_Y) \sim_{\rm cbir} (\overline{X}_0,\overline{D}_0 +\overline{G}_0) \sim_{\rm cbir} (\bP^2,\{xyz=0\})
,\]
\cite[Prop. 11]{dFKX17} implies $\Gamma_Y$ a cycle of smooth curves glued along nodes. 
Using that $\Exc(f) \subset \Gamma_Y$ and $f$ have connected fibers, we see $\overline{\Gamma}_0$ is a cycle of curves as well. 
Since  $(\overline{\Gamma}_0, {\rm Diff}_{\overline{\Gamma}_0}(0))$ is an slc CY curve by adjunction,
we can conclude  $\overline{\Gamma}_0$ is a cycle of $\bP^1$'s  and $ {\rm Diff}_{\overline{\Gamma}_0}(0)=0$.
\qed

\medskip

\noindent \emph{Claim 2}: $\overline{G}_0 \cong \bP^1 \cup \bP^1$, where the components meet at a single node $p$.
\medskip

\noindent \emph{Proof of claim}: 
Since $(\overline{X}_0,\overline{G}_0)$ is an lc log Fano pair with reduced boundary,
\cite[Thm. 5.3 \& 6.5]{Hac04} implies one of the following hold:
\begin{enumerate}
	\item[(i)] $\overline{G}_0\cong \bP^1\cup \bP^1$, where the components meet at a single node $p\in \overline{G}_0$.
	\item[(ii)] $\overline{G}_0\cong \bP^1$ and $Y$ has a dihedral singularity \cite[3.35.3]{Kol13} at a point  $q \in \overline{G}_0$. 
\end{enumerate}
Case (ii)  cannot hold, since then ${\rm Diff}_{\overline{\Gamma}_0}(0)$ has coefficient 1 at $q$ by \cite[Cor. 3.45]{Kol13},
which  would contradict Claim 1. Thus (i) holds.
\qed
\medskip

Combining Claim 2 and Lemma \ref{l:1-compP1s} proven below, we conclude  $K_{X_0}+D_0\sim 0$.
\end{proof}

\begin{lem}\label{l:1-compP1s}
Let $(X,D)$ be an slc CY surface with normalization $(\overline{X},\overline{G}+\overline{D})$.
If $\overline{G} = \bP^1\cup \bP^1$, where the curves meet at  a single node $p$, and $K_{\overline{X}}+\overline{G}+ \overline{D} \sim 0$, then $K_{X}+D\sim 0$. 
\end{lem}

\begin{proof}
Since $K_{\overline{X}}+ \overline{G}+\overline{D} \sim 0$, there exists a non-vanishing section $\overline{s} \in H^0(\overline{X},\omega_{\overline{X}}(\overline{G}+\overline{D}))$. 
We aim to show $\overline{s}$ descends to a section $s\in H^0(X,\omega_{X}(D))$.

Let $\overline{G}^n=C_1\cup C_2$ and $p_i\in C_i$ denote the preimages of $p$.
Since $(\overline{G}^n, {\rm Diff}_{\overline{G}^n}(\overline{D}))$ is lc and 
$K_{\overline{G}^n}+ {\rm Diff}_{\overline{G}^n}(\overline{D})\sim 0$ by adjunction, 
there exists a point $ q_i\in C_i$  such that
${\rm Diff}_{\overline{G}^n}(\overline{D} )\vert_{C_i} = p_i + q_i $.
Note that the involution $\tau: \overline{G}^n \to \overline{G}^n$ acts by $\tau(p_1)=p_2$ and $\tau(q_1)=q_2$.

Now consider the canonically defined Poincar\'e residue maps:
\[
H^0(X,  \omega_{\overline{X}}(\overline{G}+ \overline{D}) ) 
\overset{R}{\longrightarrow}
H^0(C_1, \omega_{C_1}( p_1+q_1 ) )
\oplus H^0(C_2, \omega_{C_2}(p_2+q_2) )
\overset{R'}{\longrightarrow }
H^0(p_1, \omega_{p_1}) \oplus H^0(p_2, \omega_{p_2})
.\]
Note that 
$R'$ is an isomorphism, since
\[ 
H^0(C_i, \omega_{C_i}( p_i+q_i ) ) 
\to 
H^0(p_i, \omega_{p_i})
\]
is a nonzero linear map between 1-dimensional vector spaces.
By \cite[Prop. 5.8]{Kol13}, 
$\overline{s}$ descends to $X$ if and only if 
$\tau^*( R(\overline{s}) ) = - R(\overline{s})$.

To verify the equality holds, observe that
$R(\overline{s})$ descends to a section on $\overline{G}$, since $R$ factors through $H^0(\overline{G},  \omega_{\overline{G}}({\rm Diff}_{\overline{G}}(\overline{D})))$.
Thus  \cite[Prop. 5.8]{Kol13}  applied to $R(\overline{s})$ implies 
\[
\tau^*(R'(R(\overline{s}))) = - R'( R(\overline{s})).
\] 
Since $\tau^*$ and $R'$ commute (using that $\tau$ interchanges $p_1$ and $p_2$),
it follows
\[
R'( \tau^* (R(\overline{s}))) = \tau^* (R' ( R(\overline{s}))) = - R'(R(\overline{s}))
\]
Since $R'$ is an isomorphism, we further see
$\tau^* (R(\overline{s})) =- R(\overline{s})$.
Thus $\overline{s}$ descends to a non-zero section $s \in H^0(X,\omega_X(D))$,
which induces an effective divisor $B$ on $X$ such that $K_{X}+D\sim B$. 
Since $K_X+D\sim_{\bQ}0$, $B=0$. Therefore $K_X+D\sim0$.
\end{proof}

We will now deduce Theorem \ref{t:TypeIIISequiv} from  Proposition \ref{p:typeIII-Sequiv}.

\begin{proof}[Proof of Theorem \ref{t:TypeIIISequiv}]

By Lemma \ref{l:SequivKmod}, we may assume $X$ is normal.
By Proposition \ref{p:typeIII-Sequiv}, 
we may further assume  $K_{X}+D \sim 0$.
Thus we can write 
\[
D:= \tfrac{3}{d} (\tfrac{d}{3}C ),
\]
where we are viewing both $C$ and $\frac{d}{3}C$ as relative Mumford divisors and $\tfrac{3}{d}$ as the marking.

We claim that the marked pair $(X,1 C )$ is in $\cP_3^{\CY}$.
Indeed, since $(X,D)$ is in $\cP_d^{\CY}$, there is a family $(\cX,\cD) \to T$ in $\cP_d^{\CY}$ over the germ of a pointed curve $0\in T$  such that 
\[
(\cX_0,\cD_0)\cong (X,D) \quad \text{ and } \quad  \cX_{K(T)} \cong \bP^2_{K(T)}
\]
Since $K_X+C\sim 0$, Lemma  \ref{l:extenddivisor} (proven in a later section) implies that there exists a relative Mumford divisor  $\cC$ on $\cX$ such that $(\cX,\cC)\to T$ is a family of boundary polarized CY pairs with $K_{\cX/T}+\cC\sim 0$ and  $(\cX_0,\cC_0)\cong (X,C)$.
Thus  the marked pair $(X,1   C)$ is in $\cP_3^{\CY}$.

To finish the proof, we use the K-moduli space of cubic curves.
By  Lemma \ref{l:SequivKmod}, $(X,1 C)$ is S-equivalent to a Type III pair in $ \cP_{3}^{\rm K}$.
Note that the type III pairs in $\cP_{3}^{\rm K}$
are of the form $(\bP^2,\{ f_3(x,y,z)=0\})$ where  $f_3(x,y,z)=0$ is a nodal cubic by  \cite[Ex. 4.5]{ADL19}
and  any  nodal cubic degenerates to three lines in $\bP^2$ via a  map $\bG_m\to \Aut(\bP^2)$.
Therefore
$(X,1\, C) \sim_{S}  (\bP^2,1  \{xyz=0\})$, which implies $(X, D ) \sim_S   (\bP^2, \frac{3}{d}\{(xyz)^{d/3}=0\})$.
\end{proof}

\section{Twisted varieties}\label{s:twistedvarieties}

In this section, we associate to a  pair in $\cP_{d}^{\CY}$ its canonical covering stack, 
which is a  Gorenstein Deligne--Mumford stack.
This construction has previously been used to study moduli of certain classes of slc varieties and pairs  \cite{Hac04,AH11,BI}.
We will use the construction to understand the local structure of $\cP_{d}^{\CY}$ and prove a result that will be used to show $\cP_{3,m}^{\CY}$ admits a good moduli space.

\subsection{Canonical covering stack}\label{ss:canonicalcovering}

\begin{defn}
The \emph{covering stack} of a family $(X,D)\to B$ in $\cP_{d}^{\CY}$,  where $D:= \tfrac{3}{d}C$ and $B$ is a Noetherian scheme,
is the quotient stack
\[
\cX
:= \left[ \Spec_X \Big( \oplus_{m \in \bZ} \omega_{X/B}^{[m]} \Big)/ \bG_m \right] 
.\]
Note that the natural map $\pi:\cX\to X$ is an isomorphism on the open set $U\subset X$ where $\omega_{X/B}$ is a line bundle.  
We set
$\cC$ equal to the closure of $C\vert_{U}$ in $\cX$ and call
$
(\cX,\cD)\to B$, where $\cD:=\tfrac{3}{d} \cC$,
the \emph{canonical covering family} of $(X,D)\to B$.
\end{defn}

\begin{prop}\label{p:canonicalcover}
The following hold:
\begin{enumerate}
\item $\cX $ is a  Deligne--Mumford stack with coarse moduli space $\cX\to X$.
\item $\cX\to B$ is a proper flat morphism with $S_2$ fibers.
\item The canonical sheaf $\omega_{\cX/B}$ is a line bundle.
\item $\cC$ is a Cartier divisor on $\cX$ that is flat over $B$.
\item The pair $(\cX_b, \cD_b)$ is slc (see \cite[Def. 2.5]{BI}) for all $b\in B$.
\end{enumerate}
\end{prop}

In the above proposition,  $\omega_{\cX/B}:= j_* \omega_{\cU/B}$, where $j$ denotes the open immersion $\cU:= \pi^{-1}(U) \hookrightarrow\cX$. The sheaf is reflexive by \cite[Prop. 5.6]{LN18}. 

\begin{proof}
Since $\omega_{X/B}^{[m]}$ commutes with arbitrary base change  for all $m\in \bZ$, \cite[Prop. 5.3.2]{AH11} implies (1) and (2) hold and that there is a line bundle $\cL$ on $\cX$ that agrees with $\omega_{\cX/B}$ on the open locus  $\cU$.
Since ${\rm codim}_{\cX_b}(\cX_b\setminus \cU_b) \geq 2$ for all $b\in B$,  
\cite[Prop. 3.6.2]{HK} implies $\cL\cong \omega_{\cX/B}$ and so (3) holds.

To prove (4),
fix any $x\in X$.
We claim that there exists an integer $k$ and an open set $x\in V\subset X$ such that  $\omega_{X/B}^{[k]}(C)\vert_V$ is  a line bundle.
Indeed, if $3 \mid d$, then $\omega_{X/B}^{[d/3]} (C)\cong_{B} \cO_{X}$  by definition and, hence, $\omega_{X/B}^{[d/3]} (C)$ is a line bundle.
If $3\nmid d$, let $b$ denote the image of $x$ in $B$.  Proposition \ref{p:dnmid3plt}
and \cite[Proof of Lem. 3.13]{Hac04} imply $C_b$ is locally a multiple of $K_{X_b}$ at $x$. 
Hence, for some integer $k$,  $\omega_{X_b}^{[k]} (C_b)$ is a line bundle at $x\in X$.
Since
$\omega_{X/B}^{[k]} (C)$ commutes with base change, $\omega_{X/B}^{[k]}(C)\vert_{C_b}$ is also a line bundle at $x$.  This implies that there exists a neighborhood $V$ containing $x$ such that $\omega_{X/B}^{[k]}(C)\vert_{V}$ is a line bundle.

We now show that the claim implies (4).  Let $\cV:= \pi^{-1}(V) \hookrightarrow\cX$.
Since $\omega_{\cX/B}^{[k]}(\cC)\vert_{\cV}$ and 
$\pi^*\omega_{X/B}^{[k]}(C)\vert_{\cV}$
are isomorphic over $\cU\cap \cV$,
\cite[Prop. 3.6.2]{HK} implies the sheaves are isomorphic. 
Thus $\omega_{\cX/B}^{[k]}(\cC)\vert_{\cV}$ is a line bundle.
Using that $\omega_{\cX/B}$ is a line bundle, we see $\cO_{\cX}(\cC)\vert_{\cV}$ is a line bundle.
Thus $\cC$ is a Cartier divisor.

Since $\cC$ does not contain a fiber of $\cX\to B$ (as $C$ does not contain a fiber  $X\to B$), \cite[Lem. 4.20]{KolNewBook} implies $\cC\to B$ is flat. Finally, (5) follows from \cite[Obs. 2.6]{BI} as $\pi : \cX \to X$ is an isomorphism in codimension one. 
\end{proof}

\begin{rem}
In \cite[Sec. 3]{Hac04}, a different construction, which locally is a stack quotient of the index one cover of $\omega_{X/S}$, is used to construct a Deligne--Mumford stack with similar properties. These two constructions agree by \cite[Sec. 4.1]{BI}.
\end{rem}

\subsection{Local structure}\label{ss:localstructure}

In the section, we relate deformations of a family in $\cP_d^{\CY}$  to deformations of the canonical covering stack. 
Using this correspondence and results of \cite{Hac04}, we show the stack is smooth when $3 \nmid d$.

\subsubsection{Deformations}

\begin{prop}\label{p:deformations}
Let $A'$ be a local artinian ring with residue field $\bk$, $(X,D)$ in $\cP_{d}^{\CY}(\bk)$, and $(\cX,\cD)$ the canonical covering pair. Then there exists a bijection between:
\begin{enumerate}
\item   isomorphism classes of  deformations $(X',D')\to \Spec(A')$ of $(X, D)$ and
\item   isomorphism classes of  deformations $(\cX', \cD')\to \Spec(A')$ of  $(\cX,\cD)$.
\end{enumerate}
\end{prop}

In the above proposition, a deformation means the following:
\begin{itemize}
\item 
A  \emph{deformation}  of $(X,D)$ is a family $(X',D') \to \Spec(A') $ in $ \cM(\chi, N_d, (1,\tfrac{3}{d}))$ with an isomorphism $(X',D')_{\Spec(\bk)} \cong (X,D)$. 
\item  A \emph{deformation}  of $(\cX, \cD)$, where $\cD:= \tfrac{3}{d} \cC$, is the data of $(\cX',\cD')$, where  $\cX'$ is a Deligne--Mumford stack with a flat proper morphism $\cX'\to \Spec(A')$ and $\cD':= \tfrac{3}{d} \cC'$, where  $\cC'\subset \cX'$ is a flat Cartier divisor
with an isomorphism $(\cX',\cC' )\vert_{\Spec(\bk)} \cong (\cX,\cC)$.
\end{itemize}

\begin{proof}
The construction in Section \ref{ss:canonicalcovering} produces a map  from (1) to (2) given by taking the canonical covering family $(\cX',\cD') \to \Spec(A')$ of $(X',D')\to \Spec(A')$. 
To construct its inverse,  fix a deformation 
$(\cX',\cD')\to \Spec(A')$ of $(\cX,\cD)$
and let $X'$ denote the coarse moduli space of $\cX'$.
By \cite[Prop. 4.6.i \& 4.14.x]{Alp13}, there is a cartesian diagram
\[
\begin{tikzcd}
\cX\arrow[d,"\pi"] \arrow[r,hook]& \cX' \arrow[d,"\pi'"]\\
X \arrow[r,hook ]  &X'
\end{tikzcd}
,\]
where the vertical arrows are the coarse moduli space maps, and the horizontal arrows are closed embeddings, and the map $X'\to \Spec(A')$ is flat. 
Let $\cU\subset \cX$ and $\cU'\subset \cX'$ denote the loci where $\pi$ and $\pi'$ are isomorphisms. Since $\cU= \cU' \vert_{\cX}$, 
$
{\rm codim}_{\cX'}(\cX' \setminus \cU')\geq 2$.
Hence we may set  $D':= \tfrac{3}{d}C'$, where $C'$ is the closure of the (isomorphic) image of $\cC'\vert_{\cU'}$ in $X'$.

We now prove $(X',D')\to \Spec(A')$ is a family in  $ \cM(\chi, N_d, (1,\tfrac{3}{d}))$.
Since  $(X',D')_{\Spec(\bk)} \cong (X,D)$, which is an slc pair in $\cP_{d}^{\CY}$, it suffices to show
(i) $\omega_{X'/A'}^{[m]}(kC)$ commutes with base change for all $m,k \in \bZ$ and (ii) $\omega_{X'/A'}^{[N_d]}(N_d D')\cong  \cO_{X'}$. 
For (i), note that $\pi'_*\omega_{\cX'/A'}^{[m]}(k\cC')$ is flat over $B$ with pure $S_2$ fibers and taking  reflexive powers commutes with base change by the proof of \cite[Thm. 5.3.6]{AH11}.
Since $\pi'_*\omega_{\cX'/A'}^{[m]}(k\cC')$ and $\omega_{X'/A'}^{[m]}(kC')$
agree over $U$, \cite[Prop. 3.6.2]{HK} implies the sheaves are isomorphic. 
Thus (i) holds.

To show (ii) holds, note that $\omega_{X'/A'}^{[N_d]}(N_d D') \vert_{X} \cong  
\omega_{X}^{[N_d]}(N_d D)\cong \cO_{X}$, where the first isomorphism is by (i).
Hence $ \omega_{X'/A'}^{[N_d]}(N_d D') $ is a deformation of  $\cO_X$.
To show $\cO_{X}$ does not admit trivial deformations, factor $A'\to \bk$ by a sequence of maps of Artininian rings
\[
A'= A_0 \to A_1 \to \cdots A_r = \bk
\]
such that $J_i:= \ker (A_i \to A_{i+1})$ is a prinicipal ideal annihilated by the maximal ideal  $\fm_{i} \subset A_i$.
Set $X_i  : = X\times_{\Spec(A)} \Spec(A_i)$.
Since 
\[
H^1(X_{i+1}, J_i\otimes \cO_{X_{i+1}} )=H^1(X,\cO_X) = 0
,
\] where the first isomorphism uses that $J_i \cong A_{i}/\fm_{i}\cong \bk$ 
and the second is by  \cite[Thm. 11.34]{KolNewBook}, \cite[Thm. 6.4]{Har10} implies that there is at most one deformation  of $\cO_{X_i}$ to $X_{i+1}$. 
Hence the only deformation of $\cO_{X}$ to $X'$ is $\cO_{X'}$, and we can conclude that (ii) holds.

By the above discussion, $(X',D') \to \Spec(A')$ is a deformation of $(X,D)$.
Thus taking the coarse moduli space gives a map from (2) to (1).
To see that these constructions are inverses, note that the canonical covering stack of the coarse moduli space $X'\to \Spec(A')$ is isomorphic to $\cX'\to \Spec(A')$ by \cite[Prop. 3.7]{Hac04}.
Thus we get the desired bijection.
\end{proof}

\subsubsection{Smoothness and normality results}

\begin{prop}\label{p:smoothnesPCYatpt}
If $(X_0,D_0)$ is in $\cP_d^{\CY}(\bk)$ and the normalization $(\overline{X}_0,\overline{G}_0)$ of $(X_0,0)$ is plt,  
then $\cP_{d}^{\CY}$ is smooth at $(X_0, D_0)$.
\end{prop}

\begin{proof}
Since  $\cP_d^{\CY}$ is defined as the closure of the open substack $\cP_{d}$ in $\cM(\chi, N_d,(1,\tfrac{3}{d}))$, it suffices to prove $\cM(\chi, N_d,(1,\tfrac{3}{d}))$ is smooth at $(X_0,D_0)$.
By the  infinitesimal lifting criterion for smoothness
\cite[\href{https://stacks.math.columbia.edu/tag/02HY}{Tag 02HY} and \href{https://stacks.math.columbia.edu/tag/0DP0}{Tag 0DP0}]{stacks-project}, the latter  holds if and only if for every
\begin{itemize}
\item surjective morphism of Artinian local rings $A'\to A$ with residue field $\bk$ and
\item deformation $(X,D)\to \Spec(A)$ of $(X_0,D_0)$,
\end{itemize}
there exists a deformation $(X',D')\to \Spec(A')$ of $(X_0,D_0)$ extending $(X,D)\to \Spec(A)$.

To verify the existence of such a deformation, we  will first  deform the canonical covering family $(\cX,\cD)\to \Spec(A)$ of $(X,D)\to \Spec(A)$.
Since $(\overline{X}_0,\overline{G}_0)$ is plt by assumption, \cite[Thm. 8.2 \& 9.1]{Hac04} implies  the existence of a flat proper morphism  from a Deligne--Mumford stack $\cX'\to \Spec(A')$ and an isomorphism 
$\cX' \times_{\Spec(A')} \Spec(A)\cong \cX$.
Since $ H^i(X_0,\cO_{X_0}(C_0))=0= H^i(X_0,\cO_{X_0}) $
for all $i\geq 1$ by  \cite[Thm. 11.34]{KolNewBook} and
 $\cC$ is a Cartier divisor by  Proposition \ref{p:canonicalcover},
the proof of \cite[Thm. 3.12]{Hac04} implies $\cC\subset \cX$ deforms to a closed subscheme $\cC'\subset \cX$. 
Thus $(\cX',\cD' :=\tfrac{3}{d}\cC')\to \Spec(A')$ is a deformation of $(\cX_0,\cD_0)$. 
By Proposition \ref{p:deformations}, the coarse moduli space $(X',D')\to \Spec(A')$
is a deformation of $(X_0,D_0)$ extending $(X,D)\to \Spec(A)$.
Therefore $\cP_{d}^{\CY}$ is smooth at $(X_0,D_0)$.
\end{proof}

Using the previous proposition,  we deduce the following statements. 

\begin{prop}\label{p:normalPdCY}
If $3\nmid d$, then $\cP_{d}^{\CY}$ is smooth and  $P_{d}^{\CY}$ is normal.	
\end{prop}

\begin{proof}
By Propositions \ref{p:dnmid3plt} and \ref{p:smoothnesPCYatpt},
$\cP_{d}^{\CY}$ is smooth. 
Thus $P_d^{\CY}$ is normal by \cite[Thm. 4.16.viii]{Alp13}.
\end{proof}

\begin{prop}\label{p:Klocalstructure}
For each $d \geq 1$, $\cP_{d}^{\rm K}$
is smooth and $P_{d}^{\rm K}$ is normal.
\end{prop}

\begin{proof}
Since K-semistable log Fano pairs are klt,  Proposition \ref{p:smoothnesPCYatpt} implies $\cP_d^{\rm K}$ is smooth. Thus $P_d^{\rm K}$ is normal by \cite[Thm. 4.16.viii]{Alp13}.
\end{proof}

\begin{rem}\label{r:Kdef=ADL}
The definition of $\cP_d^{\rm K}$ in \cite{ADL19} differs from the definition in this paper. 
While it is straightforward to see the definitions agree over reduced bases using Lemma \ref{lem:flatdiv} and \cite[Rem. 3.25]{ADL19}, it is not immediately obvious that the two definitions agree over all base schemes. 
However, since  $\cP_d^{\rm K}$   is reduced by Proposition \ref{p:Klocalstructure} 
and so is the stack in \cite{ADL19} by definition, the two stacks agree.
\end{rem}

\subsection{Twisted curves}\label{ss:twistedcurves}
We now discuss \textit{twisted curves}, which arise naturally from families in $\cP^{\CY}_{3}$ and study some properties of their degenerations.

\begin{defn}[\cite{AV}]\label{def_tw_curve}
A  \emph{twisted curve} over a scheme $B$ is a proper flat Deligne--Mumford stack $\mathcal{C}\to B$ such that for each geometric point $\overline{b} \to B$, 
\begin{enumerate}
\item $\cC_{\overline{b}}$ is purely one-dimensional with at worst nodal singularities,
\item the coarse moduli map $ \cC_{\overline{b}} \to C_{\overline{b}}$ is an isomorphism away from the nodes of $\cC_{\overline{b}}$, and
\item  each node of $\cC_{\overline{b}}$ admits an \'etale neighborhood which is \'etale locally isomorphic to 
\[[(\operatorname{Spec} k(\overline{b})[z,w]/(zw))/\bmu_r],
\]
where the action of $\bmu_r$ is $\zeta \cdot z = \zeta z$ and $\zeta \cdot w = \zeta^{-1}w$ for $\zeta\in \bmu_r$.
\end{enumerate}
\end{defn}

\newcommand{\sC}{\mathscr{C}}
The following proposition produces a twisted curve from a family in $\cP_{3}^{\CY}$.

\begin{prop}\label{p:planecurvetotwisted}
If $(X,D := \frac{3}{3}C)\to B$ is a family in $\cP^{\CY}_3$ over a Noetherian scheme $B$ and $(\cX,\cD:=\tfrac{3}{3}\cC)$ is the canonical covering family, then $\cC\to B$ is a twisted curve.
\end{prop}

\begin{proof}
 By Proposition \ref{p:canonicalcover}.4, $\cC\to B$ is flat and its formation commutes with base change. Thus we need to check the geometric fibers are twisted curves. Since every geometric point of $\cP^{\CY}_3$ is in the closure of $\cP_3$, it suffices to prove that, when $B = \Spec\,R$ is the spectrum of a DVR with algebraically closed residue field and $(X,D) \to B$ has generic fiber in $\cP_3$, then the fiber over the closed point $0 \in B$ is a twisted curve.

We will now show $\cC_0$ satisfies (1)--(3) of Definition \ref{def_tw_curve} when $(X,D) \to B$ is as above. Part (1) holds, since  $\cC_0$ is pure 1-dimensional by definition and is at worst nodal by Proposition \ref{p:canonicalcover}.4 and adjunction.
For  (2), we will first show ${\rm Diff}_{C_0}(0)=0$. 
Since $\Supp  ({\rm Diff}_{C}(0))$ is flat over $B$ by  Lemma \ref{l:familyslcnormadj}.2
and  ${\rm Diff}_{C}(0)$ is trivial over the generic fiber, $\Diff_{C}(0)=0$.
Now we compute by adjunction that
\[
K_{C_0} + \Diff_{C_0}(0) \sim 0
\quad \text{ and } \quad
K_{C_0} \sim K_{C}\vert_{C_0} \sim  \left(K_{C}+\Diff_{C}(0)\right) \vert_{C_0} \sim 0.
\]
Thus $\Diff_{C_0}(0)\sim 0$, 
which implies $\Diff_{C_0}(0) =0$.
Therefore $X_0$ is smooth at the smooth points of $C_0$ by \cite[Prop. 2.35]{Kol13}.
Thus the canonical covering stack $\pi : \cX_0\to X_0$ is an isomorphism over smooth points of $C_0$ and so (2) holds.

Finally, we check (3). Let $p \in \cC_0$ be a nodal geometric point lying over the geometric point $q \to C_0$. By the previous argument, $\pi^{-1}(C_0^{sm}) = \cC_0^{sm}$ so $q$ is a node and by \cite[Prop. 2.2]{twisted_curves} and \cite[\href{https://stacks.math.columbia.edu/tag/04GH}{Tag 04GH}]{stacks-project} the pullback of $\cC_0$ to the strict Henselization of $C_0$ at $q$ is computed as
$$
\cC_0 \times_C \operatorname{Spec} \cO_{C, q}^{sh} \cong [\operatorname{Spec}(\cO_{B,0}[z,w]/(zw - t))^{sh}/\bmu_r]
$$
where $t \in \fm \subset R$ and a generator of $\bmu_r$ acts via $(z,w) \mapsto (\zeta_1z, \zeta_2w)$ for $\zeta_i$ primitive $r^{\rm th}$ roots of unity. Since $\cC \to B$ is a smoothing of $\cC_0$, $t \neq 0$. On the other hand, $\bmu_r$ acts trivially on $t$ and $(zw - t)$ must be a $\bmu_r$-invariant ideal so $\zeta_1\zeta_2 = 1$, concluding the proof of (3). 
\end{proof}

\begin{prop}\label{prop_from_plane_curves_to_twisted}
If $(X,D: =\tfrac{3}{3} C)$ is a pair in $\cP_3^{\CY}(\bk)$ with canonical cover $(\cX,\cD:= \tfrac{3}{3}\cC)$ and $\pi: \cX\to X$ is the coarse moduli space morphism, then
\[
{\rm ind}_{\pi(p)}(K_X)= |\Aut_{\cC}(p)| \quad \quad \text{ for all }p\in \cC .
\]
\end{prop}

\begin{proof}
Since $\cC \hookrightarrow \cX$ is a closed embedding, $\operatorname{Aut}_\cC(p)= \operatorname{Aut}_\cX(p)$.
By the last paragraph of \cite[Proof of Thm. 9.51]{Kol13}, the reduced fiber of $\cX\to X$ over $\pi(p)$ equals $[(\bA^1\setminus 0 )/\bG_m]$, where $\bG_m$-acts on $\bA^1\setminus 0$ by $t\cdot  x = t^m   x$ and $m$ is the index of $K_X$ at $\pi(p)$ so $|\Aut_\cX(p)|=m$.
\end{proof}

\begin{prop}\label{p:stabilizerjump} 
Let $\cC\to S$ be a family of twisted curves over a pure two dimensional scheme $S$. Let $0\in S$ be a closed point and $1 \in I \subset \bN$ a finite subset. If  $|\Aut_{\cC}(p)| \in I$ for all $p \in \cC\setminus \cC_0$, then  $|\Aut_{\cC}(p)| \in I$ for all $p \in \cC_0$.
\end{prop}

The result shows that, for  a family of twisted curves, the stabilizer does not jump up in codimension two.
An application to families in $\cP^{\operatorname{CY}}_{3}$ will be given in Proposition \ref{p:Pdmtwisted}. 
\begin{proof}

First note that the order of $\Aut_{\cC}(p)$ can be computed locally and after passing to the strict Henselization of the coarse space $C$ at $q = \pi(p)$. Indeed let $f : (\operatorname{Spec} A, q') \to (C,q)$ be any pointed \'etale map and consider the cartesian diagram

$$
\xymatrix{\cC_A \ar[r]^{f'} \ar[d]_{\pi'} & \cC \ar[d]^\pi \\ \operatorname{Spec} A \ar[r]_f & C}. 
$$
Then $\pi'$ is the coarse moduli map and $f'(p') = p$ where $p' \in \cC_A$ lies over $q' \in \operatorname{Spec} A$. Moreover, $f$ is stabilizer preserving so $f'$ is also stabilizer preserving. That is, 
$$
\Aut_{\cC_A}(p') \to \Aut_{\cC}(p)
$$
is an isomorphism. 

Thus we may assume that $(S,0)$ is a strictly Henselian local ring and replace $(C, q)$ with $\operatorname{Spec} \cO_{C,\overline{q}}^{sh}$. We may also assume that $p$ is a node of $\cC_0$ otherwise the statement is clear. Now the statement is a consequence of the description of the local structure of twisted curves \cite[Prop. 2.2]{twisted_curves}. 
Indeed by \emph{loc. cit.}, 
\[
\cC \times_C \operatorname{Spec} \cO_{C,\overline{q}}^{sh} \cong [\operatorname{Spec}(\cO_{S, 0}[z,w]/(zw - t))^{sh}/\bmu_r],
\]
where $t \in \fm \subset \cO_{S,0}$ and $\zeta \in \bmu_r$ acts by $(z,w) \mapsto (\zeta z, \zeta^{-1}w)$. This stack has stabilizer $\bmu_r$ along the nodal locus $\{t = 0\} \subset \cC$ and trivial stabilizer elsewhere. On the other hand, $V(t) \subset S$ has codimension at most $1$. Thus $\{t = 0\}$ intersects $\cC \setminus \cC_0$ nontrivially and so $r \in I$ as required. 
\end{proof}

\begin{prop}\label{p:Pdmtwisted}
Let $S$ be a Noetherian pure two dimensional scheme, $0\in S$ a closed point, and $S^\circ := S\setminus 0$.
If $(X,D)\to S$ is in $\cP_{3}^{\CY}$ and its restriction $(X^\circ, D^\circ) \to S^\circ$ is in $\cP_{3,m}^{\CY}$, 
then $(X,D)\to S$ is in $\cP_{3,m}^{\CY}$.
\end{prop}

\begin{proof}	
By Propositions \ref{p:planecurvetotwisted}, \ref{prop_from_plane_curves_to_twisted}, and \ref{p:stabilizerjump},   ${\rm ind}_{p}(K_{X_0})\leq m$ at all $p \in D_0$.
Since $K_{X_0}+D_0\sim 0$,
${\rm ind}_{p}(K_{X_0})=1$ at all $p\in X_0 \setminus D_0$. 
Therefore, $(X,D) \to S$ is in $\cP_{3,m}^{\CY}$.
\end{proof}

\section{Existence of good moduli spaces}\label{s:gmsforcurves}

In this section, we prove that $\cP_{d,m}^{\CY}$ admits a  good moduli space 
and study  properties of the moduli space.

\subsection{Index 1 CY pairs}\label{ss:1comp}
We now  prove  technical results that will allow us to reduce to the case when $d=3$ when proving that $\cP_{d,m}^{\CY}$ is S-complete and $\Theta$-reductive.

\begin{prop}\label{p:exists1comp}
Let $(X,D)$ be a pair in $\cP_{d}^{\rm CY}$ with an action by  $\bT:=\bG_m^r$ for some $r\geq 0$.

If $(X,D)$ is Type II or III, then there exists a $\bT$-invariant divisor $B$ on $X$ such that $(X,B)$ is a boundary polarized CY pair of index 1.
\end{prop}

In the language of complements, the divisor $B$ is a $\bT$-invariant 1-complement of $X$. 
To construct the divisor $B$, we first use results from Section \ref{s:sequivcurves} to degenerate
$(X,D)$ to a pair  $(X_0,D_0)$ admitting such a divisor $B_0$. 
The following lemma will allow us to deform $B_0$ to the desired divisor $B$ on $X$.

\begin{lem}\label{l:extenddivisor}
Let $\T:= \bG_m^r$  and $(X,\Delta+D)\to S$  be a  family of boundary polarized CY pairs over a reduced Noetherian affine scheme $S$ such that  $(X,\Delta) \to S$ is $\bT$-equivariant.

If $0 \in S$ i a $\bT$-invariant closed point and $B_0$ a $\bT$-invariant $\bQ$-divisor on $X_0$  satisfying
\begin{enumerate}
\item  $0\in  \overline{ \T \cdot s}$ for any $s\in S$,
\item  $(X_0,\Delta_0+B_0)$ is slc, and 
\item  ${0\sim N( K_{X_0}+\Delta_0+B_0)}$, 
\end{enumerate}
then $B_0$ extends to a $\bT$-equivariant $\bQ$-divisor $B$ on $X$ such that $(X,\Delta+B)\to S$ is a family of boundary polarized CY pairs with $N(K_{X/S}+\Delta+B) \sim 0$.
\end{lem}

\begin{proof}
Consider the restriction map
\[\phi:
H^0(X, \cO_X(-
N(K_{X/S}- \Delta ) ))  \twoheadrightarrow 
H^0(X_0, f_* \cO_{X_0}(-N(K_{X_0}-\Delta_0)))
,\]
which is surjective by  Proposition \ref{p:anticanonicalsheaf} and the assumption that $S$ is affine.
The $\T$-action on $(X,\Delta)\to S$ induces $\T$-actions on the source and target of $\phi$.
Thus the $\bk$-vector spaces  admit decompositions into weight spaces indexed by $\Z^r$ that respects $\phi$:
\[
\phi:
\bigoplus_{\vec{\la} \in \bZ^r} H^0(X, \cO_X(-N(K_{X/S}+\Delta) ) )_{\vec{\la}}
\twoheadrightarrow
\bigoplus_{\vec{\la} \in \bZ^r}
H^0(X_{0}, f_* \cO_{X_{0}}(-N(K_{X_0}+\Delta_0)))_{\vec{\la}}
\]
Since $NB_0 \sim -N(K_{X_0}-\Delta_0)$ is $\T$-equivariant,
there exists $\vec{\la}' \in \bZ^r$ and
\[
g_0\in  H^0(X_0, \cO_{X_0}(-N(K_{X_0}-\Delta_0)))_{\vec{\la}'}
\]
 such that $B_0= \tfrac{1}{N}\{g_0=0 \}$. 
Since $\phi$ is surjective, there exists 
\[
g\in H^0(X,  \cO_{X}(-N(K_{X/S}-\Delta)))_{\vec{\la}'}
\]
such that $\phi(g)= g_0$. 
Thus $B:= \tfrac{1}{N}\{ g=0 \} $ is a $\T$-equivariant  extension of $B_0$.
Note that
\[
0 \in S_{\rm slc} := \{ s\in S \, \vert\,  (X_s,\Delta_s + B_s) \text{ is slc} \}
\]
and $S_{\rm slc}\subset S$ is a $\bT$-equivariant open set by the openness of slc singularities \cite[Cor. 4.52]{KolNewBook}. 
Therefore assumptions (1) and (2) imply $S_{\rm slc}=S$ and so
$(X,\Delta+B) \to S$ is a $\T$-equivariant family of boundary polarized CY pairs with $N(K_{X/S}+\Delta+B)\sim 0$.
\end{proof}

\begin{proof}[Proof of Proposition \ref{p:exists1comp}]
By Theorems \ref{thm:type2Sequiv} and \ref{t:TypeIIISequiv}, $(X,D)$ is S-equivalent to a pair $(X',D')$, which is either (i)-(iv) listed in Theorem \ref{thm:type2Sequiv} or (v) $X'=\bP^2$ and $D':= \tfrac{3}{d} \{(xyz)^{d/3}=0\}$. 
Hence, the pairs admit weakly special degenerations
\[
(X,D) \rightsquigarrow (X_0,D_0) \leftsquigarrow (X',D')
\]
via test configurations, which we  denote by $(\cX,\cD)\to \bA^1$ and $(\cX',\cD')\to \bA^1$.
Since $(\cX,\cD)$ is $\bT$-equivariant by Proposition \ref{p:TequivariantTC}, the $\bT$-action on $(X,D)$ induces a $\bT\times \bG_m$-action on $(\cX,\cD)$  fixing $\cX_0$.

\medskip

\noindent \emph{Claim:} 
There is a $\bT\times \bG_m$-invariant divisor $B_0$ on $X_0$ such that $(X_0,B_0)$ is a CY pair of index $1$.

\medskip

\noindent \emph{Proof of claim}:
In each case, there exists a divisor $B'$ on $X'$ such that $(X',B')$ is a CY pair of index 1 
and  $\LC(X',D') \subset \LC(X',B')$. 
Indeed, in cases (i) and (v), set $B' :=D'$. In the remaining cases, 
\begin{enumerate}
\item[(ii)] $X'$ is a projective orbifold cone over $\bP^1$ whose orbifold divisor is toric. Let $B'$  be the reduced toric boundary divisor of $X'$.
\item[(iii)] $X'$ is a klt toric surface. Let
 $B'$ to be the reduced toric boundary divisor of $X'$.
\item[(iv)] $X'$ is the gluing of two projective orbifold cones $X'_1$ and $X'_2$ over $\bP^1$ with the same toric orbifold divisor. Denote the conductor divisor  on $X'_i$ by $G_i$. Let $B'|_{X_i'}$ be the reduced toric boundary divisor on $X'_i$ minus $G_i$.
\end{enumerate}
By Lemma \ref{l:lcDlcB},  there exists a divisor $B_0^0$ on $X_0$ and weakly special degeneration $(X',B')\rightsquigarrow (X_0,B_0^0)$ such that $\LC(X_0,D_0) \subset \LC(X_0,B_0^0)$.

While  $B_0^0$ is not necessarily $\bT\times\bG_m$-invariant, 
we can use a degeneration argument to remedy this.
Consider the inclusion of $\la : \bG_m \cong \bG_m \times 1^{\times r} \hookrightarrow \bT\times \bG_m$ and set 
\[
B_0^1 := \lim_{t\to 0}\la(t) \cdot B_0^0.
\]
By  Lemma \ref{l:lcDlcB} applied to the product test configuration of $(X_0,D_0)$ induced by $\la$, there is a weakly special degeneration $(X_0,B_0 ) \rightsquigarrow (X_0,B_0^1)$ such that $\LC(X_0,D_0) \subset \LC(X_0,B_0^1)$. Degenerating $r$-additional times using the remaining summands of $\bT\times \bG_m = \bG_{m}^{r+1}$  produces a sequence of weakly special degenerations 
\[
(X',D') \rightsquigarrow (X_0^0,B_0^0) \rightsquigarrow \cdots \rightsquigarrow (X^{r+1}_0,B^{r+1}_0)
\]
such that $B_0:= B_0^{r+1}$ is $\bT\times \bG_m$-invariant. 
Since $(X',B')$ has index 1,  Lemma \ref{l:Ncompspecialize} implies $(X_0,B_0)$ has index 1.
\qed

\medskip

By Lemma \ref{l:extenddivisor} applied to $B_0$,
there is a $\bT\times \bG_m$-invariant divisor $\cB$ on $\cX$ such that  
$(\cX,\cB) \to \bA^1$
is a family of boundary polarized CY pairs with $K_{\cX/\bA^1}+\cB\sim 0$.
Thus $(X,B) : = (\cX_1,\cB_1)$ is a $\bT$-equivariant boundary polarized CY pair of index 1.
\end{proof}

We now prove an analogue of Proposition   \ref{p:exists1comp} for families over the 
the surfaces appearing in the definition of S-completeness and $\Theta$-reductivity.

\begin{prop}\label{p:1-complementS}
Let $R$ be a DVR with uniformizer $\pi$.
Let
$S$ be the affine scheme
\[
\Spec  \left( R[s,t]/(st-\pi) \right) \quad \text{ or } \quad   \Spec(R[t]),\]
where $\bG_m$-acts on $S$ with weights 1 and -1 on $s$ and $t$, respectively,
and write  $0 \in S$ for the unique $\bG_m$-invariant closed point.

If $(X,D) \to S$ is a $\bG_m$-equivariant family in $\cP_{d}^{\rm CY}$ and $(X_0,D_0)$ is Type  II or III, then there exists a divisor $B$ on $X$ such that  $(X,B)\to S$ is a $\bG_m$-equivariant family in $\cP_{3}^{\rm CY}$.
\end{prop}

\begin{proof}
By Proposition \ref{p:exists1comp}, there exists a $\bG_m$-invariant divisor $B_0$  such that 
$(X_0,B_0)$ is a boundary polarized CY pair with index 1.
By Lemma \ref{l:extenddivisor}, $B_0$ extends to a divisor $B$ on $X$  such that 
$(X,B)\to S$
is a $\bG_m$-equivariant family of boundary polarized CY pairs with $K_{X/S}+B\sim 0$.
Since each fiber of $X\to S$ is a degeneration of $\bP^2$ and $K_{X/S}+B\sim 0$,
$(X,B) \to S$ is in $\cP_3^{\rm CY}$.
\end{proof}

\subsection{Existence of the moduli space}\label{sec:modulispaceofplanecurves}

\begin{thm}\label{t:existsPdm}
For each $d\geq 3$ and $m\geq 1$, there is a good moduli space
$
\phi_m:\cP_{d,m}^{\CY}\to P_{d,m}^{\CY}
$
and $P_{d,m}^{\CY}$ is a separated algebraic space.   
\end{thm}

The key step to proving the theorem is to verify the following statement.

\begin{prop}\label{p:PdmSTheta}
For each $d\geq 3$ and $m\geq 1$, $\cP_{d,m}^{\rm CY}$ is S-complete and $\Theta$-reductive.
\end{prop}

The proof of the proposition relies on reducing to the case when $d=3$ and the theory of twisted curves to analyze the case when $d=3$.

\begin{proof}
Let $R$ be a DVR essentially of finite over $\bk$ and $\pi\in R$ a uniformizer.
Set
\[
S:=\Spec  \left( R[s,t]/(st-\pi) \right) \quad \text{ or } \quad   S:= \Spec(R[t]),\]
where $\bG_m$-acts on $S$ with weights $1$ and $-1$ on $s$ and $t$, respectively.
Let $0\in S$ denote the unique closed point fixed by $\bG_m$ and $S^\circ:= S \setminus 0$.
Fix a $\bG_m$-equivariant family 
$(X^\circ,D^\circ )\to  S^\circ$ in $\cP_{d,m}^{\CY}$.
By Theorems \ref{t:Scomplete} and \ref{t:Thetared},
the family extends uniquely to a $\bG_m$-equivariant family 
 $(X,D)\to S$  in $\cP_{d}^{\rm CY}$.
We claim that $(X,D)\to S$ is in $\cP_{d,m}^{\CY}$.

\medskip

\noindent  \emph{Case 1: $(X_0,D_0)$ is Type I}. 
Let $(\cX,\cD)\to \bA^1_\kappa$ denote the test configuration given by
\[
(X,D)_{ \{s = 0 \}} \to \{s=0 \} 
\quad \quad \text{or } \quad \quad
(X,D)_{ \pi = 0 } \to \{\pi=0 \} 
,\]
depending on which scheme $S$ denotes.
Note that $(\cX_1,\cD_1)$ is in $\cP_{d,m}^{\CY}$, since it is a fiber of $(X^\circ ,D^\circ)\to S^\circ$.
Since the special fiber of the test configuration is $(X_0,D_0)$, which is klt, 
Proposition \ref{l:CYbirational}.2  implies $(\cX,\cD) \to \bA^1_\kappa$ is a trivial test configuration.
Therefore $(X_0,D_0) \cong (\cX_1,\cD_1)$ and so $(X,D) \to S$ is in $\cP_{d,m}^{\CY}$.

\medskip

\noindent \emph{Case 2: $(X_0,D_0)$ is Type II or III.}
This is the difficult case.
By Proposition \ref{p:1-complementS},  there exists a divisor $B$ on $X$ such that $(X,B) \to S$ is a $\bG_m$-equivariant family in $\cP_{3}^{\rm CY}$. 
Since $(X^\circ, D^\circ)\to S^\circ$ is in $\cP_{d,m}^{\CY}$, $(X^\circ, B^\circ)\to S^\circ$ is in $\cP_{3,m}^{\CY}$. Now  $(X,B) \to S$ is in $ \cP_{3,m}^{\CY}$ by Proposition  \ref{p:Pdmtwisted} and so  $(X,D)\to S$ is in $ \cP_{d,m}^{\CY}$.

\medskip

The previous two cases verify the claim. 
Therefore  $\cP_{d,m}^{\CY}$ is S-complete and $\Theta$-reductive with respect to essentially of finite type DVRs.
Using that $\cP_{d,m}^{\CY}$ is a finite type algebraic stack with affine diagonal by Proposition \ref{p:Pdmstack}, Theorem \ref{t:AHLH} (with the sentence directly after) implies $\cP_{d,m}^{\CY}$ is S-complete and $\Theta$-reductive with respect to all DVRs. 
\end{proof}

\begin{proof}[Proof of Theorem \ref{t:existsPdm}]
By Propositions \ref{p:Pdmstack} and \ref{p:PdmSTheta},  we may apply Theorem \ref{t:AHLH} to deduce the theorem.
\end{proof}

\subsection{Properties of the moduli space}

We now prove various properties of $P_{d,m}^{\rm CY}$
including the existence of wall crossing morphisms, properness, and stabilization.

\begin{prop}\label{p:normalawayfromIII+ell}(Local structure)
Fix $d\geq 3$ and $m\geq1$. 
\begin{enumerate}
	\item The moduli space $P_{d,m}^{\CY}$ is reduced and irreducible.
	\item If $p \in P_{d,m}^{\CY}$ is a point where the source is not a point or an elliptic curve, then $P_{d,m}^{\CY}$ is normal at $p$.
\end{enumerate}
\end{prop}

\begin{proof}
Note that $\cP_d$ is smooth by Proposition \ref{p:smoothnesPCYatpt} and clearly irreducible. 
Using that $\cP_{d}^{\CY}$ is the stack theoretic closure of $\cP_d$,
it follows that  $\cP_d^{\CY}$, as well as $\cP_{d,m}^{\CY}$, is  reduced and irreducible as well.
Thus \cite[Thm. 4.16.viii]{Alp13} implies (1) holds.

For (2), let  $\cZ\subset \cP_{d,m}^{\CY}$ denote  the non-smooth locus of the stack and 
\[
U :=  P_{d,m}^{\CY} \setminus \phi_m(\cZ) \subset P_{d,m}^{\CY},
\]
which is an open subset.
Since $\phi^{-1}(U) \to U$ is a good moduli space morphism by Proposition \ref{p:saturated} and $\phi^{-1}(U)$ is smooth, \cite[Thm. 4.16.viii]{Alp13} implies $U$ is normal.
By our assumption on $p$, \cite[Thm. 5.5 \& 6.5]{Hac04} implies: if $(X,D)$ is a pair 
in $\cP_{d}^{\CY}$ that maps to $p$, then the normalization of $(\oX,\oG)$ of $(X,0)$ is plt. 
Thus $p \in U$ by   Proposition \ref{p:smoothnesPCYatpt} and so  $P_{d,m}^{\CY}$ is normal at $p$.
\end{proof}

\begin{prop}[Wall crossing]\label{p:wallcrossing}
Fix $d \geq 3$. 	
If $m $ is sufficiently large,  then  $\cP_d^{\rm K}  \subset  \cP_{d,m}^{\CY}$,    $\cP_d^{\rm H}  \subset  \cP_{d,m}^{\CY}$,
and there is a commutative diagram
\[
\begin{tikzcd}
\cP_d^{ \rm K} \arrow[r,hook] \arrow[d] &
\cP_{d,m}^{\CY} \arrow[d] & 
\cP_{d}^{\rm H} \arrow[l,hook'] \arrow[d] \\
P_d^{\rm K} \arrow[r,"\psi_{ {\rm K},m}"]  & 
P_{d,m}^{\CY}  & 
P_{d}^{\rm H} \arrow[l,"\psi_{ {\rm H},m}"']
\end{tikzcd}
,\]
where the top arrows are open immersions and the vertical arrows are good moduli space morphisms. 
In addition, $\psi_{ {\rm K},m}$ and $\psi_{{\rm H},m}$ are proper surjective morphisms when $d\geq 3$ and $d\geq 4$, respectively.
\end{prop}

Note that third column of the diagram is trivial when $d=3$, since $\cP_3^{\rm H}=\emptyset = P_3^{\rm H}$.

\begin{proof}
Since $\cP_{d}^{\rm K}$ and $\cP_d^{\rm  H}$ are finite type,  Propostion \ref{p:boundedness} implies 
\[
\cP_d^{\rm K} \subset \cP_{d,m}^{\CY}\quad \text{ and } \quad \cP_d^{\rm H}\subset \cP_{d,m}^{\CY}
\]
 for $m$ sufficiently large.
By Proposition \ref{p:Alpergmprops},
there are unique morphisms $\psi_{{\rm K},m}$ and $\psi_{{\rm H},m}$ that make the diagram commute. 
Since $P_d^{\rm K}$ and $P_{d}^{\rm H}$ are proper (Theorem \ref{t:ADL} and \ref{t:Hacking}) and $P_{d,m}^{\rm CY}$ is separated  (Theorem \ref{t:existsPdm}), $\psi_{{\rm K},m}$ and $\psi_{{\rm H},m}$ are proper. 

To verify that $\psi_{{\rm K},m}$ is  surjective for $d\ge 3$,
note that $\cP_d \subset \cP_d^{\rm K}$ by Theorem \ref{t:ADL}.
Thus  $\cP_d^{\rm K}$ is dense in $\cP_{d}^{\rm CY}$
and so $ \psi_{\rm K,m}(P_{d}^{\rm K})$ is dense in $P_{d,m}^{\CY}$.
Since   $\psi_{{\rm K},m}$ is proper, 
it follows that   $\psi_{{\rm K},m}(P_{d}^{\rm K})= P_{d,m}^{\CY}$ as desired.
Similar argument shows that  
$\psi_{{\rm H},m}$ is surjective when $d\geq 4$. 
\end{proof}

\begin{prop}[Properness]\label{p:properness}
Fix $d\geq 3$. If $m$ is sufficiently large, then $P_{d,m}^{\rm CY}$ is proper.
\end{prop}

\begin{proof}	
Since $P_{d}^{\rm K}$ is proper by Theorem \ref{t:ADL} and there is a proper surjective morphism
$P_{d}^{\rm  K}\to P_{d,m}^{\rm CY}$ for   $m$ sufficiently large by Proposition \ref{p:wallcrossing}, $P_{d,m}^{\CY}$ is proper for $m$ sufficiently large.
\end{proof}

\begin{prop}[Compactification]\label{p:compactPd}
If $d \geq 4$ and $m\geq 1$, then the natural map
\[
\cP_d \to P_d:= \phi_m(\cP_d)
\]
is a coarse moduli space morphism and 
$P_d \subset P_{d,m}^{\CY}$ is a dense open subset. 
\end{prop}

\begin{proof}
If $d \geq 4$, then each pair  $(X,D)$ in $ \cP_{d}(\bk)$ is klt and is not S-equivalent to non-isomorphic pairs by Proposition \ref{p:sequivklt}.
Thus Proposition \ref{p:Alpergmprops}.3 and Lemma \ref{l:sequiv} imply
\[
\cP^{\CY}_d(\bk) \to P_{d,m}^{\CY}(\bk)
\]
is a bijection over $P_d(\bk)$ 
and so  $\cP_d$ is saturated with respect to $\phi_m$.
Therefore Proposition \ref{p:saturated}  implies   $P_d \subset P_{d,m}^{\CY}$ is an open subset and $\cP_d \to P_d$ is a good moduli space.
Since $\cP_d \to P_d$ is a bijection on $\bk$-points, it is a coarse moduli space.
\end{proof}

Next, we prove a stabilization result for  $P_{d,m}^{\CY}$.  If $m'\geq m$, then there is a diagram
\begin{equation}\label{e:Pdmdiag}
\begin{tikzcd}
\cP_{d,m}^{\CY} \arrow[d] \arrow[r,hook] & \cP_{d,m'}^{\CY} \arrow[d] \\
P_{d,m}^{\CY} \arrow[r,"\psi_{m,m'}"]  & P_{d,m'}^{\CY} 
\end{tikzcd}
\end{equation}
where the existence of a morphism $\phi_{m,m'}$, which makes the diagram commute, follows from Proposition \ref{p:Alpergmprops}.2.

\begin{prop}\label{p:stabilization}
Fix $d\geq 3$. If $m$ is sufficiently large and $m'\geq m$, then the  map 
\[
\psi_{m,m'}:P_{d,m}^{\CY} \to P_{d,m'}^{\CY}
\]
is a bijection on $\bk$-points.
\end{prop}

To prove the above proposition, we first analyze the stratification via type.

\begin{lem}\label{l:stabilization}
Fix $d\geq 3$. If  $m$ is sufficiently large, then  $\cP_{d}^{\rm CY,I+II}\subset \cP_{d,m}^{\CY}$
and the induced map
$P_{d}^{\rm CY,I+II} \to  P_{d,m}^{\CY}$
is an open immersion. 
Furthermore,
\begin{enumerate}
\item if $3 \nmid d$, then $P_{d}^{\rm CY,I+II} \to P_{d,m}^{\CY}$ is an isomorphism;
\item if $3 \, \, \vert \, \, d$, then $P_{d,m}^{\CY}  \setminus P_{d}^{\rm CY,I+II}$ is the $\bk$-point corresponding to the S-equivalence class of the pair  $(\bP^2,\tfrac{3}{d}\{(xyz)^{d/3}=0\})$.
\end{enumerate}
\end{lem}

\begin{proof}
Since $\cP_{d}^{\rm CY,I+II}$ is finite type by Theorem \ref{t:modulireg0},  Proposition \ref{p:boundedness} implies $\cP_{d}^{\rm CY,I+II}\subset \cP_{d,m}^{\CY}$ for $m$ sufficiently large.
To check $P_{d}^{\rm CY,I+II} \to  P_{d,m}^{\CY}$
is an open immersion, observe that if 
 $(\cX,\cD) \to \bA^1$ a test configuration of a pair $(X,D)$ in $ \cP_{d,m}^{\rm CY,I+II}(\bk)$, then $(\cX,\cD)\to \bA^1$ is a family in  $\cP_{d,m}^{\rm CY,I+II}$ by Lemma \ref{p:SrcSequiv}.
Therefore \cite[Lem. 3.32]{AHLH18} implies 
$\cP_{d}^{\rm CY,I+II}\subset \cP_{d,m}^{\CY}$ is saturated with respect to $\phi_m$. Thus
 $P_{d}^{\rm CY,I+II} \to P_{d,m}^{\CY}$ is an open immersion by Proposition \ref{p:saturated}.

 If $3 \nmid d$, then Theorem \ref{t:moduliexists3notd} implies (1) holds.
 We now assume $3 \mid d$.
We claim that if $m'\geq m\gg1$, then
\[
\psi_{m,m'}(P_{d,m}^{\CY} \setminus P_{d}^{\rm CY,I+II} ) = P_{d,m'}^{\CY}\setminus P_{d,m'}^{\rm CY,I+II}
.\]
To verify the claim,
note that $\cP_{d,m}^{\CY} $ is dense in  $\cP_{d,m'}^\CY$
and so  
$\psi_{m,m'}(P_{d,m}^{\CY})$ is dense in $P_{d,m'}^{\CY}$.
Since both $P_{d,m}^{\CY}$ and $P_{d,m'}^{\CY}$ are proper by Proposition \ref{p:properness}, 
$\psi_{m,m'}$ is proper and so  $\psi_{m,m'}$ is surjective. 
Since $\psi_{m,m'}$ is surjective and an isomorphism over the Type I and II loci,
the claim follows.

Now, to prove (2) it suffices to show that
$\psi_{m,m'}\big(P_{d}^{\CY} \setminus P_{d}^{\rm CY,I+II}\big) $ is the $\bk$-point 
\[
p_{m'}:=
\phi_{m'}( [(\bP^2,\tfrac{3}{d}\{(xyz)^{d/3}=0\})] ) \in P_{d,m'}^{\CY}
\]
for $m'\geq m\gg0$.
To verify the previous statement,
fix an irreducible component $Z\subset P_{d,m}^{\CY}\setminus P_{d,m}^{\rm CY,I+II}$ and  set $\bK:= \overline{K(Z)}$.  
Since $\cP_d^{\rm K}(\bK)\to P_d^{\rm K}(\bK)$ is surjective by Proposition \ref{p:Alpergmprops},
there exists a Type III pair $(X,D)$ in $\cP_{d,m}^{\CY}(\bK)$ mapping to the geometric generic point of $Z$.
By Theorem \ref{t:TypeIIISequiv}, there exist weakly special degenerations
\begin{equation}\label{e:sequiveq}
(X,D) 
\rightsquigarrow
(X_0,D_0)
\leftsquigarrow 
(\bP^2_{\bK}, \tfrac{3}{d}\{(xyz)^{d/3}=0 \}).
\end{equation}
 to a pair $(X_0,D_0)$ in $\cP_{d}^{\CY}(\bK)$.
If $m'$ is sufficiently large so that $(X_0,D_0)$ is in $\cP_{d,m'}^{\CY}(\bK)$, 
then the map $\cP_{d,m'}^{\CY} (\bK) \to P_{d}^{\CY}(\bK)$ identifies the pairs in \eqref{e:sequiveq}
and, hence, $\psi_{m,m'}(Z)= p_{m'}$.
Therefore $\psi_{m,m'}\big(P_{d}^{\CY} \setminus P_{d}^{\rm CY,I+II}\big) =p_{m'}$ for $m' \geq m\gg0$ as desired.
\end{proof}

\begin{proof}[Proof of Proposition \ref{p:stabilization}]
The statement follows immediately from Lemma \ref{l:stabilization}.
\end{proof}

\begin{prop}\label{p:PdmSequiv}
	Fix $d \geq 3$ and $m$ sufficiently large. 
	Two points $x:=[(X,D)]$ and $x':=[(X',D')]$ in $\cP_{d,m}^{\CY}(\bk)$ are identified in $P_{d,m}^{\CY}$ if and only if $(X,D)\sim_S (X',D')$. 
\end{prop}

\begin{proof}
To begin, fix $m\gg1$ so that the conclusion of Lemma \ref{l:stabilization} holds. 	
By Proposition \ref{p:Alpergmprops}.3, $x$ and $x'$ are identified in  $P_{d,m}^{\CY}$ if and only if 
$
\overline{\{x \}} \cap \overline{\{x'\}} =\emptyset$
in $\cP_{d,m}^{\CY}$.
Thus, if $x$ and $x'$ are identified in $P_{d,m}^{\CY}$, then $(X,D)\sim_{S}(X',D')$
by Lemma \ref{l:sequiv}. 

Conversely, assume  $(X,D)\sim_S(X',D')$. Then there exist weakly special degenerations
\[
(X,D)\rightsquigarrow(X_0,D_0)\leftsquigarrow(X',D')
\]
to a pair $(X_0,D_0 ) \in \cP_{d}^{\CY}(\bk)$. 
If $(X,D)$ is Type I or II, 
then  $(X_0,D_0)$ is also Type I or II by Proposition \ref{p:SrcSequiv}
and so  $(X_0,D_0)$ is in $ \cP_{d,m}^{\CY}$ by Lemma \ref{l:stabilization}.
Thus $x$ and $x'$ are identified in $P_{d,m}^{\CY}(\bk)$.
If $(X,D)$ is Type III, then
$(X',D')$ is also Type III by Proposition \ref{p:SrcSequiv} and, hence,  $x$ and $x'$ are identified in $P_{d,m}^{\CY}$ by Lemma \ref{l:stabilization}.2.
\end{proof}

\section{Asymptotically good moduli spaces}\label{s:agm}

In this section, we  prove Theorem \ref{t:main2}.1-.3, which produces a moduli space for the seminormalization of $\cP_{d}^{\CY}$. The moduli space will be an asymptotically good moduli space as defined below.

\subsection{Asymptotically good moduli spaces}\label{ss:asgm}

We first introduce the notion of an asymptotically good moduli space and describe its properties. 
The term is motivated by \emph{asymptotic GIT} in which one  views a polarized variety  $(X,L)$ as a point
\[
[(X,L^m) ] \in {\rm Chow }(\bP^{N_m})
\]
in the appropriate Chow scheme and analyzes the stability as $m\to \infty $.

\begin{defn}\label{d:asgm}
	A morphism $\phi:\cM\to M$ from an algebraic stack to an algebraic space is an \emph{asymptotically good moduli space} if there exists an ascending chain
	\[
	\cU_1 \subset \cU_2 \subset \cU_3 \subset \cdots 
	\]
	of open substacks of $\cM$ 
	such that $\cM=\cup_{i \geq 1} \cU_i$ and the composition
	$\cU_i \hookrightarrow \cM \to M$
	is a good moduli space  for all $i\geq1$.
\end{defn}

\begin{rem}[Construction]\label{r:constructionagm}
In practice, one might  construct an asymptotically good moduli space as follows. 
Fix a locally noetherian algebraic stack $\cM$ with an ascending chain
\[
\cU_1 \subset \cU_2 \subset \cU_3 \subset \cdots 
\] 
of open substacks of $\cM$  
such that $\cM = \bigcup_{i\geq 1} \cU_i$. Assume the following hold:
\begin{enumerate}
	\item[(i)] There exists a good moduli space $\cU_i \to U_i$ for each $i\geq 1$. 
	
	\item[(ii)]  In the commutative diagram 
	\[
	\begin{tikzcd}
		\cU_1 \arrow[r,hook] \arrow[d] & \cU_2 \arrow[r,hook] \arrow[d]  &  \cU_3 \arrow[d]  \arrow[r,hook] & \cdots \\
		U_1	\arrow[r,"\psi_1" ]  & U_2	\arrow[r,"\psi_2"]  &  U_3 \arrow[r,"\psi_3"]& \cdots
	\end{tikzcd}
	,\]
	which exists by Proposition \ref{p:Alpergmprops}.2,  there exists $k$ such that $\psi_i$ is an isomorphism for all $i\geq k$. 
\end{enumerate}
Since the above diagram commutes and  $\cM = \bigcup_{i\geq 1} \cU_i$, the morphisms $(\cU_i \to U_i)_{i \geq k}$ glue to give a morphism $\phi:\cM\to M$ with $M = U_k$ such that $\phi \vert_{\cU_i} = \phi_i$ for all $i\geq k$. Thus  $\cM\to M$ is an asymptotically good moduli space.
\end{rem}

Asymptotically good moduli spaces inherit many of the properties that good moduli spaces satisfy. We prove a few of these properties in the following proposition.

\begin{prop}\label{p:agmproperties}
If $\phi:\cM\to M$ is an asymptotically good moduli space, then: 
\begin{enumerate}
	\item The morphism $\phi$ is surjective.
	\item If $\cM$ is locally Noetherian,  then $\phi$ is universal for maps from $\cM$ to algebraic spaces.
	\item Two  points $x, x' \in \cM(\bk)$ are identified in $M$ if and only if $\overline{\{x \}} \cap \overline{ \{ x'\}} \neq  \emptyset $.
\end{enumerate}
\end{prop}

Statement (2) implies that an asymptotically good moduli space, if it exists, is unique up to unique isomorphism.

\begin{proof}
Since $\phi_i: \cU_i \to M$ is a good moduli space morphism, $\phi_i$ is surjective by Proposition \ref{p:Alpergmprops}.1. Therefore $\phi$ is surjective.

To prove (2), fix a morphism to an algebraic space $\rho: \cM \to Z$.
Let $\rho_i: \cU_i \to Z$ denote its restriction to $Z$. 
 Since $\phi_i$ is a good moduli space, 
  Proposition \ref{p:Alpergmprops}.2 implies that 
there exists a unique map $\pi_i: M\to Z$ such that the diagram 
\[
\begin{tikzcd}
\cU_i \arrow[r,"\phi_i"] \arrow[rr,bend left=30,"\rho_i"] &    M \arrow[r,"\pi_i"] & Z    
\end{tikzcd}
\]
commutes. 
Since $\pi_i$ is unique and satisfies 
\[
\rho_{i} = \rho_{i+1} \vert_{\cU_i}= (\pi_{i+1} \circ \phi_{i+1}) )\vert_{\cU_i} 
=
\pi_{i+1} \circ \phi_i
,\]
 $\pi_i$ is independent of $i$ and we can denote it without confusion by $\pi$. 
Using  that  $\cM = \bigcup_{i \geq 1} \cU_{i}$, we conclude   $\rho = \pi \circ \phi$. 
Thus (2) holds. 

For (3), 
note that the closure of $\{x \} $ and $\{x'\}$ in $\cM$ intersect if and only if the closure of $\{x\}$ and $\{x'\}$ in $\cU_i$ intersect for all $i\gg 1$. Thus  (3) follows from Proposition \ref{p:Alpergmprops}.2.
\end{proof}

\begin{prop}
Let $\cM$ be a locally finite type algebraic stack with affine diagonal. If there exists an asymptotically good moduli space $\cM \to M$
such that $M$ is finite type and separated, then $\cM\to M$ is $\Theta$-reductive and S-complete. 
\end{prop}

The converse to this statement is false. In a future paper, we will give an example of a moduli stack of certain polarized CY pairs that is S-complete and $\Theta$-reductive, but does not admit an asymptotically good moduli space.

\begin{proof}
Note that $\cU_i$ is quasi-compact, since  $\cU_i \to M$  and $M$ are quasi-compact. Since $\cU_i$ is locally of finite type as well, it follows that $\cU_i$ is finite type. Therefore we can apply Theorem \ref{t:AHLH} to deduce that $\cU_i$ is S-complete and $\Theta$-reductive. 
Since $\cM= \bigcup_{i \geq 1} \cU_i$,  $\cM$ is  S-complete and $\Theta$-reductive.
\end{proof}

\begin{rem}
    In general, an asymptotically good moduli space $\phi : \cM \to M$ need not be universally closed and $\phi_*$ need not preserve quasi-coherence.
\end{rem}

\subsection{Seminormalization of a stack}
In this section, we discuss the seminormalization of a stack and its relation to good moduli spaces.

For the definition of  a \emph{seminormal} scheme and the \emph{seminormalizaiton} of a scheme, see \cite[\S 10.8]{KolNewBook}. 
The definitions of these terms for stacks are built from the definitions for schemes.

\begin{defn}
An algebraic stack $\sM$ is \emph{seminormal} if there is a smooth morphism  from a seminormal scheme $U\to \sM$.
A \emph{seminormalization} of an algebraic stack $\sM$ is a representable morphism $\sM^{\rm sn }\to \sM$ such that  
$
\sM^{\rm sn}\times_{\sM}U \to U
$ is a seminormalizaiton for every smooth morphism from a scheme $U\to \cM$.
\end{defn}

The seminormalization of a locally Noetherian algebraic stack exists and is unique up to isomorphism. 
The proofs of these statements are similar to the corresponding proofs for the normalization in \cite[App. A]{AB21}, though the case of seminormalization is actually simpler, since the seminormalization of a scheme is functorial, while the normalization is not.

\begin{lem}\label{l:sngm}
Let $\sM$ be a finite type  algebraic stack with affine diagonal.
If there exists a good moduli space  $\sM\to M$,
 then the induced morphism of  seminormalizations
$\sM^{\rm sn}\to M^{\rm sn}$
is  a good moduli space.
\end{lem}

\begin{proof}
Since $\sM$ admits a good moduli space and  $\sM^{\rm sn}\to \sM$ is finite, \cite[Lem. 4.14 \& Thm. 4.16.x]{Alp13} implies
that there exists a good moduli space  $\sM^{\rm sn} \to P$ and, in the commutative diagram, which exists by Proposition \ref{p:Alpergmprops}.2,
\[
\begin{tikzcd}
\sM^{\rm sn} \arrow[r] \arrow[d] & \sM \arrow[d]\\
P \arrow[r] & M
\end{tikzcd}
\]
the bottom row is finite.
Note that $P\to M$ is a bijection on $\bk$-points, since good moduli spaces are bijections on closed $\bk$-points and seminormalizations are bijections on $\bk$-points. 
Thus $P\to M$ is a partial seminormalization. 

Next, observe that $P$ is seminormal.
Indeed, since $\sM^{\rm sn}$ is seminormal, $\sM^{\rm sn } \to P$ factors through $P^{\rm sn} \to P$ and so $P^{\rm sn} \to P$ is an isomorphism by Proposition \ref{p:Alpergmprops}.2. Therefore $P \to M$ is a seminormalization and the result follows.
\end{proof}

\subsection{Existence of the asymptotically good moduli space}
We now prove that the seminormalization of $\cP_d^{\rm CY}$ admits an asymptotically good moduli space.

\begin{prop}\label{p:existsagm}
There exists an asymptotically  good moduli space  
\[
\phi: (\cP_d^{\CY})^{\rm sn} \to P_{d}^{\CY}
\]
where $P_{d}^{\CY}$ is the seminormalization of $P_{d,m}^{\CY}$ for $m$ sufficiently large. 
\end{prop}

\begin{proof}
The existence of the asymptotically good moduli space follows from taking seminormalizations of the construction in Section \ref{s:gmsforcurves}.
Indeed, consider the commutative diagram
which arises from taking the seminormalization of the diagram in \eqref{e:Pdmdiag}.
\[
\begin{tikzcd}
	(\cP_{d,1}^{\CY})^{\rm sn} \arrow[r,hook ] \arrow[d] &  (\cP_{d,2}^{\CY})^{\rm sn} \arrow[r,hook] \arrow[d]  &  (\cP_{d,3}^{\CY})^{\rm sn} \arrow[d]  \arrow[r,hook] & \cdots \\
	(P_{d,1}^{\CY})^{\rm sn} 	\arrow[r, "\phi_1"]  &  (P_{d,2}^{\CY})^{\rm sn}    	\arrow[r,"\phi_2"]  &  (P_{d,3}^{\CY})^{\rm sn} \arrow[r, "\phi_3"]& \cdots
\end{tikzcd}
.\]
By Proposition \ref{p:stabilization}, there exists $k \gg 1$ such that  $\phi_m$ is a bijection on $\bk$-points for $m\geq k$. 
Using that $(P_{d,m+1}^{\CY})^{\rm sn}$ is seminormal and proper by Proposition \ref{p:properness}, we see $\phi_m$ is an isomorphism for $m\geq k$.
Thus Remark \ref{r:constructionagm} implies
 that there exists  an asymptotically good moduli space $\phi: (\cP_{d}^{\CY})^{\rm sn}\to P_{d}^{\CY}$ such that the restriction of  $\phi$ to $(\cP_{d,m}^{\CY})^{\rm sn}$ is the morphism $(\cP_{d,m}^{\CY})^{\rm sn}\to (P_{d,m}^{\CY})^{\rm sn}$.
\end{proof}

\begin{rem}\label{r:notgms}
The morphism $\phi:(\cP_d^{\CY})^{\rm sn} \to P_{d}^{\CY}$ is not  a good moduli space when $3 \mid d$.  
Indeed, since $(\cP_{d}^{\CY})^{\rm sn}$ is  locally of finite type, but not of finite type by Example \ref{e:unbounded}, $(\cP_{d}^{\CY})^{\rm sn}$ is not quasi-compact.
On the other hand, $P_d^{\CY}$ is quasi-compact since it is proper.
Therefore $\phi$ cannot be quasi-compact, which violates Definition \ref{d:gm}.
\end{rem}

\section{Ampleness of the Hodge line bundle}\label{s:hodgelinebundle}

In this section, we prove that the Hodge line bundle is ample on $P_{d,m}^{\CY}$ for $m\gg0$ (Theorem \ref{t:HodgeAmple}).
As a consequence of the result, we prove the b-semiampleness conjecture in relative dimension 2 (Theorem \ref{t:bsemiample}).

\subsection{Canonical bundle formula}\label{s:canonicalbundle}

We first recall the canonical bundle formula of an lc-trivial fibration and  relevant positivity results. 
See \cite{Amb05,Kol07,FG14moduli} for further details on the topic.

\begin{defn}\label{def:lctrivial}
An \emph{lc-trivial fibration} $f:(X,D)\to Y$ is a surjective morphism  $f:X\to Y$ between normal proper varieties with connected fibers and a $\bQ$-divisor $D$ on $X$ satisfying
\begin{enumerate}
\item $(X,D)$ is sub-lc over the generic point of $Y$,
\item ${\rm rank} (f \circ g)_* \cO_{Z}( \lceil  A_Z^*(X,D) \rceil) =1$, 
where $g:Z\to X$ is a log resolution of $(X,D)$, $ A_Z (X,D):= K_Z - f^*(K_X+D) = \sum_i a_i D_i$, and $ A_Z^*(X,D) := \sum_{a_i >-1} a_i D_i$, and
\item there exists a $\bQ$-Cartier $\bQ$-divisor $B$ on $Y$ such that $K_{X}+D\sim_{\Q} f^*B$.
\end{enumerate}
Note that (2) is automatically satisfied if $D$ is effective on the general fiber, 
since then $\lceil A^*_{Z}(X,D)\rceil$ is effective and $g$-exceptional.
\end{defn}

Given an lc-trivial fibration $(X,D)\to Y$ and $Y'\to Y$ a generically finite morphism with $Y'$ normal, 
there is an induced lc-trivial fibration $(X',D')\to Y'$. Here, $\varphi':X'\to X$ is the main component of the normalization of $X\times_Y Y'$ and $D'$ is defined by $K_{X'}+D'= \varphi'^*(K_{X}+D)$.

\begin{defn}[Canonical bundle formula]\label{def:cbformula}
Let $f:(X,D)\to Y$ be an lc-trivial fibration.
There is a $\bQ$-linear equivalence
\[
K_{X}+D \sim_{\bQ} f^*(K_{Y} + D_Y + M_Y),
\]
where $D_Y$ and $M_{Y}$ are $\bQ$-divisors on $Y$ defined as follows:
\begin{enumerate}
\item $D_Y$ is the \emph{discriminant divisor}, which is defined by $D_Y:= \sum_{P} (1-b_P)P$, 
where the sum runs through prime divisors $P\subset Y$ and 
\[
b_P := \max\{ \la \in \bQ\, \vert\, (X,D+ \la f^* P) \text{ is sub-lc over the generic point of $P$}\}.
\]
\item $M_Y :=  G-K_Y-D_Y$ is the \emph{moduli divisor},
where $G$ is a $\bQ$-Cartier  $\bQ$-divisor on $Y$ such that $K_X+D\sim_{\bQ} f^*B$.
Note that $M_Y$  is only defined  up to $\bQ$-linear equivalence, 
since it depends on the choice of  $B$.\footnote{If $N(K_{X_\eta} +D_\eta) \sim 0$ for some integer $N>0$, then we can choose $B$ such that $N(K_{X}+ D - f^*B) \sim 0$ and, hence, $NM_Y$ will be defined up to  $\bZ$-linear equivalence; see \cite[\S 7.5]{PS09}.}
\end{enumerate}
\end{defn}

\begin{rem}\label{rmk:bandm}
If $\varphi:Y'\to Y$ is a proper birational morphism with $Y$ normal and
$(X',D') \to Y'$ is the induced lc-trivial fibration, then
the  divisors in Definition \ref{def:cbformula} satisfy
$\varphi_* M_{Y'} = M_{Y}$ and $\varphi_* D_{Y'}=D_Y$.
Hence there exist b-$\bQ$-divisors ${\bf M}$ and ${\bf D}$ on $Y$ such that ${\bf M}_{Y'} := M_{Y'}$ and ${\bf D}_{Y'}:= D_{Y'}$ for each proper birational morphism $Y'\to Y$ with $Y'$ normal.
\end{rem}

The following positivity statements for the moduli divisor, which follow from \cite{Amb05,Kol07,FG14moduli}, will be used to prove the ampleness
of the Hodge line bundle on $P_{d,m}^{\rm CY}$.

\begin{thm}\label{t:positivitymoduli}
If $(X,D) \to Y$ is an lc-trivial fibration, then
\begin{enumerate}
\item[(1)]   ${\bf M}$ is b-$\bQ$-Cartier and b-nef;
\item[(2)]  if the  generic fiber $(X_{\eta},D_{\eta})$ is a projective klt pair and there is a non-empty open set $U\subset Y$ satisfying that for every $y\in U(\bk)$ 
\[
\#\{ z \in U(\bk)\, \vert\, (X_y,D_y) \cong (X_{z},D_{z}) \}<+\infty,
\]
then ${\bf  M}$ is b-big. 
\end{enumerate}
\end{thm}

The statement that ${\bf M}$  is \emph{b-$\bQ$-Cartier} means that there exists a proper birational morphism $Y'\to Y$ such that ${\bf M}_{Y'}$ is $\bQ$-Catier and ${\bf M}= \overline{{\bf M}_{Y'}}$ (i.e. ${\bf M}_{Y''} = \psi^* {\bf M}_{Y'}$ for any proper birational morphism of normal varieties $\psi : Y'' \to Y'$).
The terms \emph{b-nef} (resp.,  \emph{b-big}, \emph{b-semiample}, \emph{b-free}) mean we can additionally choose $Y'$ such that ${\bf M}_{Y'}$ is nef (resp., big, semiample, basepoint free).

\begin{proof}
Statement (1) is \cite[Thm. 8.3.7]{Kol07} and \cite[Thm. 3.6]{FG14moduli} (see also \cite{Amb05}).
Under the assumptions of (2), \cite[Thm. 3.3]{Amb05} states that there  exists a diagram 
\[
\begin{tikzcd}
(X,D)  \arrow[d, "f"] & & (X^!,D^!) \arrow[d,"f^!"']\\
Y  & \overline{Y} \arrow[l, "\tau"'] \arrow[r, "\rho"]& Y^!
\end{tikzcd}
\]
such that
(i) $(X^!,D^!)\to Y^!$ is an lc-trivial fibration, (ii) $\overline{Y}$ a smooth projective variety, (iii)
$\tau$ is generically finite and surjective, and $\rho$ is surjective, 
(iv)  there exists a non-empty open set $V\subset \overline{Y}$  such that $(X,D) \times_Y V \cong (X^!,D^!)\times_{Y^!} V$ over $V$, and
(v) ${\bf M}^!$ is b-nef and big with $\tau^* {\bf M} = \rho^* {\bf M}^!$, where ${\bf M}$ and ${\bf M}^!$ denote the moduli b-${\bf Q}$-divisors of $f$ and $f^!$.

We aim to show  that $\rho$ is generically finite. 
By shrinking $V$, we may assume $\tau(V) \subset U$ and $\tau\vert_V$ has finite fibers. 
For a closed point $y \in V$,  observe
\begin{align*}
\rho^{-1}(\rho(y)) \cap V &\subset \{ z\in  V \, \vert\, (X^!_{\rho(y)}, D^{!}_{\rho(y)}) \cong (X^!_{\rho(z)},D^!_{\rho(z)})  \} \\
& = \{ z\in  V \, \vert\, (X_{\tau(y)}, D_{\tau(y)}) \cong (X_{\tau(z)},D_{\tau(z)})  \}.
\end{align*}
where the equality is by (iv).  
The last set is finite since $\tau(V) \subset U$ and $\tau\vert_V$ has finite fibers.
Therefore, $\rho^{-1}(\rho(y)) \cap V$ is finite for each $y \in V$, which implies $\rho$ is generically finite.
Since $\rho$ and $\tau$ are both generically finite, (v) implies  ${\bf M}$ is b-big. 
\end{proof}

\subsection{Hodge line bundle}

\begin{defn}[Hodge line bundle]
If $f:(X,\Delta+D)\to S$ is  a family of boundary polarized CY pairs with index dividing $N$, we set 
\[
\la_{{\rm Hodge},f,N}
:= 
f_* \omega_{X/S}^{[N]}(N(\Delta+D)).
\]
\end{defn}

\begin{rem}
There is also a natural line bundle on the base of a family of boundary polarized CY pairs  called the CM line bundle, which has relations to K-stability.
The CM line bundle of $(X,\Delta+D)\to S$  with respect to the polarization  $L:=-K_{X}-\Delta$ agrees with the Hodge line bundle up to a positive constant; see e.g. \cite[pg. 14]{ADL19}.
\end{rem}

We state a few  basic properties of the Hodge line bundle.

\begin{prop}\label{p:Hodgebasics}
Let $f:(X,\Delta+D)\to S$ be a family of boundary polarized CY pairs with index dividing $N$ over an arbitrary scheme $S$. The following hold:
\begin{enumerate}
\item[(1)]  $\la_{{\rm Hodge},f,N}$ is the unique line bundle (up to isomorphism)
satisfying 
\[
\omega_{X/S}^{[N]}(N(\Delta+D)) \cong f^* \la_{\Hodge,f,N}.
\]
\item[(2)] If $\varphi:S'\to S$ is a morphism and $f':(X',\Delta'+ D')\to S'$ denotes the pullback of $(X,\Delta+S)\to S$ by $\varphi$, then there is a canonical isomorphism\
\[
 \varphi^* \la_{\Hodge,f,N} 
 \overset{\sim}{\longrightarrow} \la_{ \Hodge,f',N} 
.\]
\item[(3)]  If $S$ is smooth and the generic fiber of $X\to S$ is normal, then $(X,\Delta+D)\to S$ is an lc-trivial fibration with $\cO_S(NM_{S}) \cong \la_{\Hodge,f,N}$ and ${\bf M}:= \overline{ M_S}$.
\end{enumerate}
\end{prop}

\begin{proof}
First note that $\pi_* \cO_X\cong \cO_S$. 
Indeed, since $H^0(X_s,\cO_{X_s})$ for all $s\in S$, Proposition  \ref{p:anticanonicalsheaf} implies $f_* \cO_X$ is a line bundle and the natural map $f_*\cO_{X} \otimes k(s) \to H^0(X_s, \cO_{X_s})$ is an isomorphism for all $s\in S$. 
 Thus the natural map $\cO_S \to f_*\cO_X$ restricts to an isomorphism  $\cO_S\otimes k(s) \to f_* \cO_X \otimes k(s)$ for all $s\in S$
 and so $\cO_S \to f_* \cO_X$ is an isomorphism.

To prove (1),  note that $\omega_{X/S}^{[N]}(N(\Delta+D))$ and $\cO_X$ are locally isomorphic over $S$. 
Since $f_*\cO_X$ is a line bundle and the natural map $ f^*f_* \cO_X\to \cO_X$ is an isomorphism, $\la_{\Hodge, f,N}$ is a line bundle and the natural map  $f^*\la_{\Hodge,f,N} \to  \omega_{X/S}^{[N]}(N(\Delta+D))$ is an isomorphism.
If $\cL_S$ is a line bundle satisfying  $\omega_{X/S}^{[N]}(N(\Delta+D))\cong f^*\cL_S$, then
\[\la_{\Hodge,f,N} \cong f_* f^*\cL_S \cong f_* \cO_X\otimes  \cL_S \cong \cL_S,
\]
which completes (1).
Next, observe that there are natural morphisms
\[
\varphi^*f_*\omega_{X/S}^{[N]} (m(\Delta+D)) \longrightarrow f'_* \varphi'^* \omega_{X/S}^{[N]}(N(\Delta+D))
\overset{\cong}{\longrightarrow}
f'_* \omega_{X'/S'}^{[N]}(N(\Delta'+D')).
\]
Since the first map is locally of the form
$\varphi^* f_* \cO_X \to f'_*  \varphi'^*\cO_{X}$, 
it is an isomorphism.
Thus (2) holds.

Finally, we prove (3). Lemma \ref{l:slcadj} implies  $(X,\Delta+D)$ is an slc pair and the conductor divisor does not contain any fibers.
Since the generic fiber of $X\to S$ is normal,  the conductor divisor on $X$ must be 0 and, hence, $(X,\Delta+D)$ is lc. 
Next, observe that 
\[
0 \sim_{\bQ,S} K_{X/S}+\Delta+D  \sim_{\bQ,S} K_{X}+\Delta+D
,\]
where the second equivalence uses that $K_{S}$ is Cartier.
Therefore $(X,\Delta+D)\to S$ is an lc-trivial fibration. 

We claim that the discriminant b-divisor equal to 0. 
To see this, fix a proper birational morphism $\varphi:S'\to S$ with $S'$ smooth. 
Since the pullback $f':(X',\Delta'+D')\to S'$ is a family of  boundary polarized CY pairs, $f'$ is  an lc-trivial fibration. 
By Lemma \ref{l:slcadj}, 
 $(X',\Delta'+D'+f'^*P')$ is lc  for any prime divisor $P'$ on $X'$.
 Therefore  the discriminant b-divisor is 0, which implies 
$N(K_{X/S}+ \Delta+D )\sim N f^* M_S$ and so
 $\cO_S(N{\bf M}_{S}) \cong \la_{\Hodge,f,N}$.
Using (2), we additionally see ${\bf M} =\overline{{\bf M}_S}$
\end{proof}

The following  statement is a well-known    consequence of \cite{FG14brep}  
and can also be deduced from the K-semistability of   polarized slc CY pairs proved in \cite{Oda13b} and Theorem \ref{t:AutTorus}.

\begin{prop}\label{p:actstrivHodge}
If $(X,\Delta+D)$ is a boundary polarized CY pair, 
then the identity component of $\Aut(X,\Delta+D)$ acts trivially on 
$H^0(X, \omega_{X}^{[N]}(N(\Delta+D)) )$.
\end{prop}

\begin{proof}
The action of $\Aut^0(X,\Delta+D)$ on $(X,\Delta+D)$ induces an action on  $(\overline{X}, \overline{G}+\overline{\Delta} +\overline{D}):=\sqcup_{i=1}^r 
(\overline{X}_i,\overline{G}_i+ \overline{\Delta}_i +\overline{D}_i)$ and there is an  $\Aut^0(X,\Delta+D)$ equivariant isomorphism 
\[
H^0 (X, \omega_{X}^{[N]}(N(\Delta+D)))  \cong H^0 (\overline{X}_i, \omega_{\overline{X}_i}^{[N]}(N( \overline{G}_i + \overline{\Delta}_i +\overline{D}_i))) 
\]
induced by pulling back forms. Hence it suffices to prove the result in the case when $X$ is normal.
If $(X,\Delta+D)$ is normal, then the finiteness of B-representations \cite[Thm. 3.15]{FG14brep} implies that the image of the natural map
\[
\Aut^0 (X,\Delta+D)\to {\rm GL}  ( H^0(X,\omega_{X}^{[N]}(N(\Delta+D)) ))
\] 
is finite. 
Since $\Aut^0 (X,\Delta+D)$ is connected, the image is trivial.
\end{proof}

Let $\sM\subset \cM(\chi, N,{\bf r})$ be a locally closed substack
and $k$ be a positive multiple of $N$.
Since any family $f: (X,\Delta + D) \to S$  in $\cM(\chi, N,{\bf r})$ has index dividing $k$,
Proposition \ref{p:Hodgebasics} implies the collection of line bundles $\lambda_{\Hodge,f,k}$ induces  a line bundle on $\cM$
and we denote it by $\la_{\Hodge,k}$. Note that $\la_{\Hodge,k}^{\otimes l} = \la_{\Hodge,kl}$ by \ref{p:Hodgebasics}.1.

\begin{prop}\label{p:Hodgedescends}
Let $\cM\subset  \cM(\chi, N,{\bf r})$ be a finite type locally closed substack that admits a good moduli space $\sM\to M$.
If $k$ is a sufficiently divisible positive integer, then $\la_{\Hodge,k}$ descends to a unique line bundle $L_{\Hodge,k}$ on $M$.
\end{prop}

\begin{proof}
By \cite[Thm. 10.3]{Alp13}, $\lambda_{\Hodge,k}$  descends uniquely to a line bundle on $M$ if and only if, for each closed point $z \in \cM(\bk)$, the stabilizer of $z$ acts trivially on $\lambda_{\Hodge,k} \vert_{z}$. 
The latter is equivalent to the condition  that $\Aut(X,\Delta+D)$ acts trivially on $H^0(X, \omega_{X}^{[k]}(k(\Delta+D)) )$
 for each  each closed point $[(X,\Delta+D)]$ in $\cM(\bk)$.
Using  Proposition \ref{p:actstrivHodge} and the fact that $\cM$ is finite type with affine diagonal, the previous condition holds for $k\gg0$; see 
\cite[Thm. 5.2]{LWX18} or \cite[Lem. 10.2]{CP21} for a similar argument.
\end{proof}

\begin{defn}
We define the \emph{Hodge $\bQ$-line bundle} on various moduli spaces as follows. 
\begin{enumerate}
	\item[(1)] In the setup of Proposition \ref{p:Hodgedescends}, we define $L_{\rm Hodge}:= L_{\rm Hodge,k}^{1/k}$. 
 Using the uniqueness statement of the proposition and that $\la_{\Hodge,kl}= \la_{\Hodge,k}^{\otimes l}$ for a positive integer $l$, we see $L_{\Hodge,kl} = L_{\Hodge,k}^{\otimes l}$.
Thus $L_{\Hodge}$ is independent of the choice of $k$ when viewed as a $\bQ$-line bundle.

\item[(2)] In the case when $\cM:=\cP_{d,m}^{\CY}$ and $M:= P_{d,m}^{\CY}$, we define the Hodge $\bQ$-line bundle $L_{\rm Hodge}$ as in (1).
Using the commutative diagram \eqref{e:Pdmdiag} and the uniquenes part of Proposition \ref{p:Hodgedescends}, it follows that the Hodge $\bQ$-line bundle pulls back to the Hodge $\bQ$-line bundle via the map $P_{d,m}^{\CY}\to P_{d,m+1}^{\CY}$.

\item[(3)]
Let $P_d^{\CY}$ be the asymptotically good moduli space in Proposition \ref{p:existsagm}.
By construction,  the natural map 
$
P_d^{\CY}\to P_{d,m}^{\CY}
$
is a seminormalization for $m\gg0$. 
We define the Hodge $\bQ$-line bundle on $P_d^{\CY}$ as the pullback of the Hodge $\bQ$-line bundle on $P_{d,m}^{\CY}$ when $m\gg0$. 
This $\bQ$-line bundle is independent of the choice of $m\gg0$ by (2).
\end{enumerate} 
\end{defn}

\subsection{Existence of covers by schemes}

We now  prove a stack theoretic result on the existence of generically finite covers of certain stacks by schemes. 
The result will be used in the proof of the ampleness of the Hodge line bundle on $P_{d,m}^{\CY}$ in order to cook up a generically finite cover of a closed subspace $Z\subset P_{d,m}^{\CY}$ such that the cover supports a universal family of boundary polarized CY pairs.

\begin{thm}\label{p:finitecoverbyscheme}
Suppose that $\cX$ is a finite type algebraic stack with a separated good moduli space $\pi : \mathcal{X} \to X$ and let $Z \subseteq X$ be an irreducible closed subspace. Suppose that the connected component of the identity of the automorphism group $\Aut^0_{\cX}(x) \subset \Aut_{\cX}(x)$ is an algebraic torus for each $x \in \mathcal{X}$. Then there exists a commutative diagram
\[ 
\begin{tikzcd}
S  \arrow[d] \arrow[r] &  \mathcal{X} \arrow[d]\\
Z \arrow[r,hook]  & X
\end{tikzcd}
\] 
where $S$ is a smooth quasi-projective variety and $S\to Z$ is proper, surjective, and generically finite. Moreover, there exists a dense open subset $U \subset S$ such that closed points of $U$ map to closed points of $\cX$.
\end{thm}

The proof relies on  \cite{ER21} to reduce to the case when $\cX$ is a gerbe over a Deligne--Mumford stack and then rigidification to further reduce to when $\cX$  is a Deligne--Mumford stack.

\begin{proof} Since the property of being a good moduli space is preserved under base change, it suffices to prove the statement when $Z = X$. Moreover, by \cite[Thm. 4.16.viii]{Alp13}, we may assume $\cX$ and $X$ are reduced. We use throughout that for any closed substack $\cZ \subset \cX$, $\pi(\cZ) \subset X$ is closed and $\cZ \to \pi(\cZ)$ is a good moduli space. Let $\cX^{\rm ps}$ be the locus of closed points in $\cX$. 

First we claim that $\cX^{\rm ps}$ is the set of closed points of a constructible subset, which we also denote by $\cX^{\rm ps}$. Indeed by the Luna \'etale slice theorem \cite[Thm. 4.12]{luna_slice} and condition $\mathrm{I}$ in the proof of \cite[Thm. 5.3]{luna_slice}, for any $x \in \cX^{\rm ps}$, there exists an affine scheme $\operatorname{Spec} A$ with an $\Aut_{\cX}(x)$ action, an $\Aut_{\cX}(x)$ fixed closed point $w \in \operatorname{Spec}A$, and a Cartesian diagram of pointed stacks
$$
\xymatrix{(\cW,w) \ar[d]_{\pi'} \ar[r]^f & (\cX, x) \ar[d]^\pi \\ (W, \pi'(w)) \ar[r] & (X, \pi(x))}
$$
where $\cW = [\operatorname{Spec} A/\Aut_{\cX}(x)]$, $\pi'$ is a good moduli space map, the horizontal maps are \'etale, and for every closed point $w' \in \cW$, $f$ is stabilizer preserving at $w'$, and $f(w')$ is closed in $\cX$. Then $f^{-1}(\cX^{\rm ps}) = \cW^{\rm ps}$. Since the property of being constructible is \'etale local, it suffices to show that $\cW^{\rm ps}$ is constructible in $\cW$. As $\Aut_{\cX}(x)$ is reductive and $\cW^{\rm ps}$ is the polystable locus of an affine GIT quotient, the claim now follows as in the proof of \cite[Thm. 4.5]{ABHLX20}. 

For each closed point $y \in X$, there exists a unique closed point $x \in \cX^{\rm ps}$ lying over it so $\pi$ induces a bijection between $\cX^{\rm ps}$ and closed points of $X$. Let $\cZ \subset \cX$ be the closure of $\cX^{\rm ps}$. Then $\pi(\cZ) = X$, $\pi|_{\cZ} : \cZ \to X$ is a good moduli space, and $\cX^{\rm ps} = \cZ^{\rm ps}$. Thus it suffices to replace $\cX$ with $\cZ$ and prove the statement when $\cX^{\rm ps}$ is dense in $\cX$. Moreover, $\cX$ has finitely many irreducible components $\cX_1, \ldots, \cX_n$ so $\pi(\cX_i) = X$ for some $i$. Then $\cX_i \to X$ is a good moduli space and up to replacing $\cX$ with $\cX_i$ we may assume that $\cX$ is irreducible. 

Following \cite[Def. 2.5]{ER21}, we let $\cX^{\rm s}$ denote the set of $x \in \cX$ such that $\pi^{-1}(\pi(x)) = \{x\}$. By \cite[Prop. 2.6 \& 2.7]{ER21}, $\cX^{\rm s}$ is the largest saturated open substack of $\cX$ such that $\pi$ induces a homeomorphism $|\cX^{\rm s}| \to \pi(\cX^{\rm s})$. Now $\cX^{\rm ps}$ is constructible and dense so there exists an open $\cU \subset \cX$ contained in $\cX^{\rm ps}$. The restriction of $\pi$ to $\cX^{\rm ps}$ is a bijection onto $X$ so the restriction to $\cU$ is a bijection onto its image. Thus up to shrinking $\cU$, $\pi|_{\cU}$ is a homeomorphism onto its image.
In particular, $\cU \subset \cX^{\rm s}$ so $\cX^{\rm s} \neq \emptyset$. Since $\cX$ is irreducible, it follows that $\cX^{\rm s} \subset \cX$ is dense. 

Now we may apply \cite[Thm. 2.11]{ER21} to obtain a stack $\cY$ and a morphism $f : \cY \to \cX$ satisfying the following properties: 
\begin{enumerate}
    \item $\cY$ admits a good moduli space $Y$, 
    \item $f : \cY \to \cX$ factors as a sequence of saturated blowups, 
    \item the induced map $g: Y \to X$ factors as a sequence of blowups, 
    \item $f$ is an isomorphism over $\cX^{\rm s}$, and
    \item $\cY \to Y$ factors as a gerbe over a Deligne--Mumford stack $\cY \to \cY_{DM}$ and a coarse moduli space $\cY_{DM} \to Y$. 
\end{enumerate}

Since $\cX$ is reduced and irreducible, so are $\cY$ and $\cY_{DM}$. By \cite[Thm. 6.6]{LMB00}, there exists a quasi-projective scheme $Y'$ and a proper, surjective and generically \'etale morphism $Y' \to \cY_{DM}$. Up to resolving singularities, $Y'$ is smooth and irreducible.

Consider the pullback
$$
\xymatrix{\cY' \ar[r] \ar[d] & \cY \ar[d] \\ Y' \ar[r] & \cY_{DM}}.
$$
Then $\cY' \to Y'$ is a gerbe by \cite[\href{https://stacks.math.columbia.edu/tag/06QE}{Tag 06QE}]{stacks-project}. Moreover, $\cY' \to Y'$ is a good moduli space since it is a finite type gerbe with affine diagonal and reductive stabilizer groups. 

Consider the inertia stack $I\cY' \to \cY'$ which is flat by \cite[\href{https://stacks.math.columbia.edu/tag/06QJ}{Tag 06QJ}]{stacks-project} and affine. In particular, the connected component of the identity $A^0 \subset \cY'$ is an open and closed subgroup flat over $\cY'$. Thus we can take the rigidification $\rho : \cY' \to \cY' \thickslash A^0$. By \cite[Thm. A.1]{AOV}. $\rho$ is a gerbe and for all $x \in \cY'$, $\Aut_{\cY'\thickslash A^0}(\rho(x)) = \Aut_{\cY'}(x)/\Aut^0_{\cY'}(x)$. In particular, $\cY' \thickslash A^0$ is Deligne--Mumford so we can apply \cite[Thm. 6.6]{LMB00} again to find a smooth $Y''$ with a proper, surjective, generically \'etale morphism $Y'' \to \cY \thickslash A^0$ and consider the pullback 
$$
\xymatrix{\cY'' \ar[r] \ar[d] & \cY' \ar[d] \\ Y'' \ar[r] & \cY' \thickslash A^0}.
$$

Then $\cY'' \to Y''$ is a gerbe such that for all $x \in \cY''$, $\Aut_{\cY''}(x) = \Aut_{\cY'}^0(x)$ is connected. Moreover, the composition $\tau : \cY'' \to \cX$ is representable as saturated blowups and pullbacks of representable morphisms are representable. Thus $\Aut_{\cY''}(x) \subset \Aut_{\cX}^0(\tau(x))$ is a torus for each $x \in \cY''$. By Proposition \ref{p:torusgerbe} below, there exists a proper, surjective and generically finite map $S \to Y''$ with $S$ smooth and quasi-projective which trivializes the gerbe, that is, which admits a section $s : S \to \cY''$. The composition $S \to X$ is then proper, surjective and generically finite and admits a map to $\cX$ via the composition $\tau \circ s$. 
\end{proof}

\begin{prop}\label{p:torusgerbe}
    Let $\cX \to X$ be an fppf gerbe with affine diagonal over an algebraic space. Suppose that $X$ is separated and finite type and that $\Aut_{\cX}(x)$ is an algebraic torus for all $x \in \cX$. Then there exists a proper, surjective, and generically finite morphism $S \to X$ with $S$ smooth and quasi-projective and a section $S \to \cX$. 
\end{prop}

\begin{proof} By Chow's lemma for algebraic spaces \cite[\href{https://stacks.math.columbia.edu/tag/088U}{Tag 088U}]{stacks-project}, we may assume that $X$ is quasi-projective. Moreover, we may normalize and work one connected component at time and assume $X$ is normal and irreducible. Since $\pi: \cX \to X$ is an abelian gerbe, by \cite[\href{https://stacks.math.columbia.edu/tag/0CJY}{Tag 0CJY}]{stacks-project} there exists a band, that is, an fppf abelian group scheme $G \to X$, such that for all points $x \in \cX$, we have $G_{\pi(x)} \cong \Aut_\cX(x)$.

By assumption, $G_{x}$ is a torus for all $x \to \cX$ so  $G \to X$ is a torus by \cite[Cor. 3.1.8]{Conrad}. There exists some finite extension $F/k(X)$ such that $G_F \cong \mathbb{G}_{m,F}^r$. Up to replacing $X$ with the normalization $X' \to X$ of $X$ in $F$, we may assume that $G_{k(X)} \cong \mathbb{G}_{m,k(X)}^r$. Since $\mathrm{Isom}_{X\text{-Grp}}(\mathbb{G}_{m,X}^r, G)$ is representable by a smooth and separated scheme over $X$, by \cite[XI 4.1]{SGA3}, the $k(X)$ point of $\mathrm{Isom}_{X\text{-Grp}}(\mathbb{G}_{m,X}^r, G)$ spreads out to a section over some dense open set $U \subset X$. By \cite[Cor. B.3.6]{Conrad}, algebraic tori over $X$ are classified by $\pi_1^{\et}(X, \overline{x})$ representations on a lattice. By \cite[\href{https://stacks.math.columbia.edu/tag/0BN6}{Tag 0BN6}]{stacks-project}, the map $\pi_1^{\et}(U,\overline{x}) \to \pi_1^{\et}(X,\overline{x})$ is surjective for any base point $\overline{x} \to U \subset X$. Thus, the triviality of $G|_U$ implies that $G \cong \mathbb{G}_{m,X}^r$ and we are reduced to trivializing the class of the gerbe $\cX \to X$ in 
$$
H^2_{\et}(X, \mathbb{G}_{m}^r) = \bigoplus_{i = 1}^r H^2_{\et}(X, \mathbb{G}_{m}) = \bigoplus_{i = 1}^r {\rm Br}(X). 
$$

Without loss of generality, we may assume that $r = 1$. Thus we need to show that a cohomological Brauer class $\alpha \in {\rm Br}(X)$ can be trivialized after a proper, surjective and generically finite morphism $S \to X$ with $S$ smooth. Up to resolving singularities and pulling back $\alpha$, we may assume $X$ is smooth. By \cite[Thm. 3.5.5 \& 4.2.1]{CTS}, $\alpha$ is the class of an Azumaya algebra so by \cite[Thm. 3.6]{EHKV}, there exists a flat projective surjection $\rho : Y \to X$ such that $\rho^*\alpha = 0$. Slicing by $\rho$-ample divisors, we can find $S_0 \subset Y \to X$ such that the composition is projective, surjective and generically finite. Let $S \to S_0$ be a resolution of singularities. Then the composition $S \to X$ trivializes $\alpha$, thus completing the proof.  \end{proof}

\subsection{Ampleness of the Hodge line bundle}

In this section, we prove the Hodge line bundle is ample on $P_{d,m}^{\CY}$ for $m\gg0$.  We will deduce the result from the following more general theorem 
and the description of the S-equivalence classes of the pairs at the boundary in Section \ref{s:sequivcurves}.

\begin{thm}\label{p:ftHodgeampleonsubspace}
Let $\cM\subset  \cM(\chi, N,{\bf r})$ be a finite type locally closed substack that admits a good moduli space $\sM\to M$.	If $Z\subset M$ is an irreducible closed subspace which is proper, then $(L_{\Hodge}\vert_Z)^{\dim Z}\geq 0$. Furthermore, the inequality is strict if one of the following holds:
    \begin{enumerate}
		\item   The  fiber of $\cM(\bk)\to M(\bk)$ over a general point of $Z$ contains only  klt pairs.

		\item  There exists an open set $U\subset Z$ such that for every $y \in U(\bk)$,
		\[
		\#\{z\in   U(\bk) \, \vert\, \Src(z) = \Src(y) \} < +\infty.
		\]
	\end{enumerate}
\end{thm}

In condition (2), $\Src(z)$ denotes the source of an element in the S-equivalence class of a pair in the preimage of $\cM(\bk) \to M(\bk)$ over $z$ 
and is independent of the choice of pair in the preimage by Lemma \ref{l:sequiv} and Proposition \ref{p:SrcSequiv}.

\begin{proof}
Let $k$ be a positive integer multiple of $N$ such that $\la_{\Hodge,k}$ on $\sM$ descends to a line bundle $L_{\Hodge,k}$ on $M$.
By Theorem \ref{p:finitecoverbyscheme} and Theorem \ref{t:AutTorus}, there is a diagram 
\[
\begin{tikzcd}
S\arrow[r] \arrow[d] & \sM \arrow[d] \\
Z \arrow[r,hook] & M
\end{tikzcd}
,\]
where $S$ is a smooth projective variety and $S\to Z$ is proper, surjective, and a generically finite morphism. 
Let $f:(X,\Delta+D)\to S$ be the family of boundary polarized CY pairs induced by the morphism $S\to \sM$.
Since
\[
(L_{\Hodge}\vert_Z)^{\dim Z}
=
k^{-\dim Z} (L_{\Hodge,k}\vert_Z)^{\dim Z}
=
(\deg(S/Z) k)^{-\dim Z}
({\la_{\Hodge,f,k}})^{\dim Z}
\]
it suffices to show that $\la_{\Hodge,f,k}$
is nef and, when (1) or (2) holds,  big.
		
To verify these statements, first assume the generic fiber of $(X,\Delta+D)\to S$ is klt. 
Fix a dense open set $U\subset S$ such that the fibers of $(X_U,\Delta_U+D_U)\to U$ are klt and  $U\to M$ has finite fibers.
By Proposition \ref{p:Hodgebasics}, $(X,\Delta+D)\to S$ is an lc-trivial fibration  $\cO_S(k{\bf M}_S) \cong \la_{\Hodge,f,k}$ and ${\bf M} = \overline{{\bf M}_S}$.
Thus Theorem \ref{t:positivitymoduli} implies
$\la_{\Hodge,f,k}$ is big and nef.
	
Now assume the generic fiber of $(X,\Delta+D)\to S$ is not klt. In a number of steps, we will proceed to construct a new family where a general fiber is replaced by its source.

First,  let $(\overline{X},\overline{G}+\overline{\Delta}+\overline{D})$ be a connected component of the normalization of $(X,\Delta+D)$
and $\overline{f}$ denote the composition $\overline{X}\overset{\pi}{\to} X\overset{f}{\to} S$.
Note that
\[
\overline{f}: (\overline{X},\overline{G}+\overline{\Delta}+\overline{D})\to S
\]
 is a family of possibly disconnected boundary polarized CY pairs by Lemma \ref{l:familyslcnormadj}.1 and adjunction.

Second, fix a dlt modification $(Y,\Gamma)\to (\overline{X},\overline{G}+\overline{\Delta}+\overline{D})$.
Let $g$ denote the composition $Y\to \overline{X}\to S$
and 
$\mu$ the composition $Y\to \overline{X} \to X$. We claim that 
$
g: (Y,\Gamma) \to S
$
is a projective family of slc pairs.
Indeed, since $(\overline{X},\overline{G}+\overline{\Delta}+\overline{D})\to S$ is a projective family of slc pairs, for each point $s\in S$ and regular system of parameters $x_1, \ldots, x_r \in \cO_{S,s}$, the pair $(\overline{X},\overline{G}+\overline{\Delta}+\overline{D}	+ \overline{f}^*\{x_1 \cdots x_r = 0\})$ is slc over a neighborhood of $s$. 
	Thus its crepant pullback $(Y, \Gamma+ g^* \{ x_1\cdots x_r=0\})$ is slc  over a neighborhood of $s$. 
	Therefore  the claim holds by Lemma \ref{l:slcadj}. 
 
Third, fix a minimal lc center $T\subset (Y,\Gamma)$, and write $h:T\to S$ for the induced morphism.
	Since $(Y,\Gamma)\to S$ is a family of slc pairs, $T$ must dominate the base $S$ as a consequence of Lemma \ref{l:slcadj}.

After replacing $S$ by a generically finite cover by a smooth projective variety, 
we may assume that the geometric generic fiber of $h$ is connected. Note that
\[
h:(T, \Gamma_T:  = \Diff_T^*(\Gamma)) \to S
\]
is a family of projective slc pairs by  Lemma \ref{l:familyslcnormadj}.2  applied multiple times.
Additionally, there exists a dense open subset $V\subset S$ such that $\overline{X}_s\to X_s$ is a normalization and $(Y_s,\Gamma_s)\to (\overline{X}_s, \overline{G}_s+\overline{\Delta}_s+\overline{D}_s)$ is a dlt modification. Therefore, after possibly shrinking $V$,  $(T_s, (\Gamma_T)_s) $ is a source of $(X_s, \Delta_s+D_s)$ for all $s\in V$.

Since $(T, \Gamma) \to S$ is a projective family of slc pairs with klt generic fiber  and 
\[
N(K_{T/S}+\Gamma)
 \sim  N(K_{Y/S}+\Gamma) \vert_{T} \sim N(K_{X/S}+\Delta+D)\vert_T \sim_{S} 0,
\] 
the proof of Proposition \ref{p:Hodgebasics}.4 shows $	(T,\Gamma ) \to S$
is an lc-trivial fibration with moduli divisor satisfying ${\bf M}= \overline{{\bf M}_S}$ and discriminant $b$-divisor equal to zero.
By Theorem \ref{t:positivitymoduli}.1, ${\bf M}_S$ is nef.
Additionally, if (2) holds, then, by shrinking $V$ further, we may assume $V\to Z$ has finite fibers and factors through $U$.  
Thus Theorem \ref{t:positivitymoduli}.2 implies ${\bf M}_S$ is also big. 
Since
\[
h^* \cO_S(k{\bf M}_S)
\cong 
\omega_{T/S}^{[k]}(k\Gamma)
\cong 
(\mu^*
\omega_{X/S}^{[k]}(k(\Delta+D)))\vert_{T}
\cong 
 (\mu^* f^*\la_{\Hodge,f,k})\vert_T
\cong 
h^*\la_{\Hodge,f,k}
,
\]
it follows that  $\cO_S(k{\bf M}_S)\cong \la_{\Hodge,f,k}$.
Thus $ \la_{\Hodge,f,k}$ is nef and when (2) holds is big.
\end{proof}

\begin{thm}\label{t:HodgeAmple}
Fix $d\geq 3$. If $m\gg0$, then the Hodge $\bQ$-line bundle $L_{\rm Hodge}$ on $P_{d,m}^{\CY}$ is ample.
As a consequence, the Hodge $\bQ$-line bundle is ample on $P_d^{\CY}$. 
\end{thm}

\begin{proof}[Proof of Theorem \ref{t:HodgeAmple}]
Fix $m\gg0$ so that the conclusions of Propositions \ref{p:wallcrossing} and \ref{p:properness} and  Lemma \ref{l:stabilization} hold.
By the Nakai--Moishezon criterion for ampleness on proper algebraic spaces \cite[Thm. 3.11]{Kol90}, it suffices to show $(L_{\Hodge}\vert_Z)^{\dim Z}>0$ for each irreducible closed subspace $Z\subset P_{d,m}^\CY$ of positive dimension.
Note that for a general element $z\in Z(\bk)$, $\Src(z)$ is either  (i) a surface pair, (ii) a curve pair or (iii) a point. 
We analyze each case separately. 
\begin{enumerate}
\item[(i)] In this case, the preimage  of a general point in $Z(\bk)$ under $\cP_{d,m}^{\CY}(\bk)\to P_{d,m}^{\CY}(\bk)$ contains only klt pairs.
Thus Theorem \ref{p:ftHodgeampleonsubspace}.1 implies  $(L_{\Hodge}\vert_Z)^{\dim Z}>0$.
	
\item[(ii)] In this case, there exists a nonempty open set $U\subset Z$ such that $\Src(z)$ is a curve pair for all $z\in U$. 
Since $\cP_{d,m}^{\CY}(\bk)\to P_{d,m}^{\CY}(\bk)$ identifies S-equivalent pairs by Theorem  \ref{p:PdmSequiv} and each curve pair is the source of finitely many different  S-equivalence classes  by Theorem \ref{thm:type2Sequiv}, 
\[
\#\{ z\in U\, \vert\, \Src(z)\cong \Src(y) \}< \infty
\]
for each $y \in Z$. 
Thus Proposition  \ref{p:ftHodgeampleonsubspace}.2 implies  $(L_{\Hodge}\vert_Z)^{\dim Z}>0$.

\item[(iii)]
This case does not  occur, since the Type III locus of $P_{d,m}^{\CY}$ is a point  by Lemma \ref{l:stabilization} and $\dim(Z)>0$ by assumption.
\end{enumerate}
By the above three cases, $(L_{\Hodge}\vert_Z)^{\dim Z}>0$ as desired.
Therefore $L_{\Hodge}$ is ample on $P_{d,m}^{\CY}$. 
Since $P_d^{\CY}\to P_{d,m}^{\CY}$ is a seminormalization and in particular finite, $L_{\Hodge}$ pulls back to an ample $\bQ$-line bundle on $P_d^{\CY}$.
\end{proof}

\subsection{B-semiampleness}

We  now use Theorem \ref{t:HodgeAmple} to prove cases of the  b-semiampleness conjecture
and its effective variant due to Prokhorov and Shokurov \cite[Conj. 7.13]{PS09}.

\begin{prop}\label{p:effectivenessP2}
If $d\geq 4$, then there exists an integer $I:= I(d)$ satisfying the following:

If $(X,D)\to Y$ is an lc-trivial fibration such that the  geometric generic fiber  $(X_{\overline{\eta}}, D_{\overline{\eta}})$ is a klt pair with $X_{\overline{\eta}}\cong \bP^2_{k(\overline{\eta})}$  and $\tfrac{d}{3}D_{\overline{\eta}}$ is a Weil divisor, then $I {\bf M}$ is  b-free.  
\end{prop}

The result follows from the ampleness of the Hodge line bundle on $P_{d,m}^{\CY}$ and the next lemma.

\begin{lem}\label{l:genfinitecover}
For each $d\geq 4$, there exists a constant $c:=c(d)$ satisfying the following:

If $(X_U,D_U) \to U$ is a family in $\cP_d^{\rm H}$ and there is a compactification  $U\subset Y$ with $Y$ a normal projective variety, then there is a diagram
\[
\begin{tikzcd}
(X_U,D_U) \arrow[d] & & (X',D')\arrow[d]\\
U \arrow[r , hook] & Y  & Y' \arrow[l, "\tau" ']
\end{tikzcd}
\]
such that 
(i) $Y'$ is a smooth projective variety,  
(ii) $\tau$ is  generically finite with $\deg(\tau) \leq  c$, and 
(iii) $(X',D')\to Y'$ is a family in $\cP_{d}^{\rm H}$ with $(X',D')\times_{Y} U \cong (X_U,D_U)$ over $U$.
\end{lem}

\begin{proof}
Since $\cP_{d}^{\rm H}$ is a finite type Deligne--Mumford stack by Theorem \ref{t:Hacking}, 
\cite[Thm. 16.6]{LMB00} implies that there exists a scheme $S$ with  a finite surjective morphism $f:S\to\cP_d^{\rm H}$.  
Let $\widetilde{U}$ denote a component of the normalization of $U \times_{\cP_{d}^{\rm H}} S$ dominating $U$. 
Since  $S \to \cP_d^{\rm H}$ is representable and finite, $U \times_{\cP_{d}^{\rm H}} S$ is a scheme and finite over $U$.
Thus $\widetilde{U}$ is a scheme and finite over $U$. 
Choose a compactification  $\widetilde{U} \subset \widetilde{Y}$ such that $\widetilde{Y}$ is a proper normal variety and the morphisms  $\widetilde{U}\to U$ and $\widetilde{U}\to Y$ extend to morphisms $\widetilde{Y}\to Y$ and $\widetilde{Y}\to S$. Thus we have a diagram:
\[
\begin{tikzcd}
\widetilde{U} \arrow[r, hook] \arrow[d] & \widetilde{Y} \arrow[r]  \arrow[d] &  S \arrow[d] \\
U \arrow[r, hook]  \arrow[rr, bend right=6mm] & Y & \cP_{d}^{\rm H} .
\end{tikzcd}
\]
Let $Y' \to \widetilde{Y}$ denote a projective resolution of singularities. 
Let $(X',D') \to Y'$ denote the family induced by  the morphism $Y' \to \cP_d^{\rm H}$. 
Thus  (i) and (iii) are satisfied.

For (ii), note that
\[
c: = \max \{ \deg(S_z) \, \vert\, z\in
\cP_{d}^{\rm H}(\bk) 
\}
\]
 is finite, since $S\to \cP_{d}^{\rm H}$ is finite and $\deg(S_z)$ is constructible.
Now
$\deg (\tau)
=
\deg(\widetilde{Y}/Y) \leq c$ as desired.
\end{proof}

\begin{proof}[Proof of Proposition \ref{p:effectivenessP2}]
By Proposition \ref{p:Hodgedescends} and Theorem \ref{t:HodgeAmple}, there exist $m,N>0$ such that the line bundle $\la_{\Hodge,N}$ on $\cP_{d,m}^{\CY}$
descends to an ample line bundle $L_{\Hodge,N}$ on $P_{d,m}^{\CY}$.
Fix a positive integer $b$ such that $bL_{\Hodge,N}$ is very ample.
Fix $c>0$ satisfying the conclusion of Lemma \ref{l:genfinitecover}. 
Now set $I:= N\cdot b\cdot c!$.

By replacing $(X,D)\to Y$ by a pullback via a proper birational morphism, we may assume ${\bf M}:= \overline{{\bf M}_Y}$.
Since 
$(X_{\overline{\eta}},D_{\overline{\eta}})$ is a projective klt pair with $X_{\oeta} \cong \bP^2_{k(\overline{\eta})}$ and $\tfrac{d}{3} D_{\overline{\eta}}$  a Weil divisor,
 there exists a non-empty open set $U\subset S$ such that $(X_U,D_U) \to U$ is a family of klt boundary polarized CY pairs and $\tfrac{d}{3} D_U$ is  Weil divisor. 
Hence $(X_U,D_U)\to U$ is a family in $\cP_d^{\rm H}$
and we can find  a diagram
\[
\begin{tikzcd}
(X_U,D_U) \arrow[d,"f"] & & (X',D')\arrow[d,"f'"]\\
U \arrow[r , hook] & Y  & Y' \arrow[l, "\tau" ']
\end{tikzcd}
\]
satisfying the conclusion of Lemma \ref{l:genfinitecover}.
Let $(X'',D'')\to Y'$ denote the lc-trivial fibration induced via  base changing $(X,D)\to Y$ by $Y'\to Y$. 
Since $(X',D')\to Y'$ and $(X'',D'')\to Y'$ have isomorphic generic fiber, their moduli divisors on $Y'$ are equal by \cite[Prop. 8.4.9]{Kol07}.
Therefore we may write ${\bf M}_{Y'}$ without confusion.
By Propositions \ref{p:Hodgebasics}.3 and \ref{p:Hodgedescends},
\[
\cO_{Y'}(N{\bf M}_{Y'})
\cong 
{\la_{\Hodge, f', N}} =\rho^* L_{\Hodge, N},
\]  
where $\rho:Y' \to P_{d,m}^{\rm CY}$ is the  moduli map.
Since $bL_{\Hodge, N}$ is very ample, 
$N\cdot b{\bf M}$ is b-free. 
Since ${\bf M}= \tau^* {\bf M}$ by \cite[Prop. 3.1]{Amb05}, \cite[Proof of Lem. 2.1]{FL19} implies $N\cdot b\cdot  \deg (\tau) {\bf M}$ is b-free.
Using that $N\cdot b\cdot \deg(\tau)$ divides $ I$,  we conclude $I {\bf M}$ is b-free.
\end{proof}

\begin{thm}\label{t:bsemiample}
If $(X,D)\to Y$ is an lc-trivial fibration
such that the  generic fiber $(X_{\eta},D_{\eta})$ is a projective lc pair and $\dim(X)-\dim(Y)\leq 2$, then ${\bf M}$ is b-semiample.
\end{thm}

\begin{proof}
If $(X_{\eta},D_{\eta})$ is not klt, then the statement holds by \cite[Rem. 4.1]{FG14moduli}, 
which reduces the problem to the case when $\dim(X)-\dim(Y)\leq 1$ and then applies \cite[Thm. 8.1]{PS09}.
If $(X_{\eta},D_{\eta})$ is klt, then the statements holds   when $X_{\oeta}\not\cong \bP^2_{k(\oeta)}$  by \cite[Thm. 1.7]{Fil18}, which reduces to cases proven in \cite{Fuj03,PS09},  and when $X_{\oeta}\cong \bP^2_{k(\oeta)}$ by Proposition \ref{p:effectivenessP2}.
\end{proof}

Proposition \ref{p:effectivenessP2} and arguments similar to \cite{Fil18}
also imply an effective version of Theorem \ref{t:bsemiample} stated in the introduction (Theorem \ref{t:effectiveb}).

Before proving the statement, we verify the following two lemmas. 

\begin{lem}\label{l:effectivebsemi}
	For each positive integer $r$, there exists an integer $I_r>0$ satisfying:

If $(X,D)\to Y$ is an lc-trivial fibration such that the generic fiber $(X_{{\eta}},D_{{\eta}})$ is an lc pair, $r(K_{X_\eta}+D_\eta)\sim 0$, and either
	(i) $\dim(X)-\dim(Y)=1$, (ii) $X_{\oeta}\cong \bP^2_{k(\oeta)}$, or
	(iii) there exists a proper normal variety $Z$ and  commutative diagram
\[
		\begin{tikzcd}
			X\arrow[rd,"f",swap] \arrow[r,"g"]& Z\arrow[d,"h"]\\
			& Y	,
		\end{tikzcd}
\]
	of surjective morphisms with connected fibers such that $g$ and $h$ are of relative dimension 1 and the general fiber of $g$ is a  genus 0 curve,
	then $I_r {\bf M}$ is b-free.
\end{lem}

\begin{proof}
	Case (i) holds by  \cite[Thm. 8.1]{PS09}. 
	If additionally $(X_\eta,D_\eta)$ is klt, then case (ii) holds by  Proposition \ref{p:effectivenessP2}
	and case (iii) holds by \cite[Proofs of Thm. 1.7 and 1.8]{Fil18}, which reduces the problem to  \cite[Cor. 7.4]{Fil18}.
	
	It remains to consider the case when $(X_\eta,D_\eta)$ is lc, but not klt.
	In this case, the problem can be reduced to (i) by using  \cite[Thm. 5.2]{FL19} as follows.
 By replacing $(X,D)\to Y$ with a pullback, we may assume ${\bf M}:= \overline{{\bf M}_Y}$.
	Fix a sequence of crepant proper birational morphisms
	\[
	(X^{\rm res},D^{\rm res})\overset{\pi}{\longrightarrow} (X^{\rm dlt},D^{\rm dlt}) \overset{\rho}{\longrightarrow} (X,D)
	\]
	such that $\rho$ restricts to a dlt modification over the generic point of $Y$ and $\pi$ is a log resolution.
	By \cite[Thm. 10.45]{Kol13}, we may additionally choose $\pi$ to be an isomorphism over the snc locus.
	Now fix a minimal lc center $T$ of the geometric generic fiber  $(X^{\rm dlt}_{\oeta},D^{\rm dlt}_{\oeta})$.
	Write $S$ for the closure of the image of $T$ on $X^{\rm dlt}$.
Let $S'$ be the birational transform of $S$ on $X^{\rm res}$ and
$	S'\to Y^1 \overset{\tau}{\to} Y$
	 the Stein factorization. 
	By the proof of \cite[Thm. 5.2]{FL19}, there exists 
	an lc-trivial fibration $(S^1,D^1)\to Y^1$ such that
	\begin{itemize}
		\item $\dim(Y_1)-\dim(S_1) \leq 1$
		\item ${\bf M}^1 = \tau^* {\bf M}$, where ${\bf M}$ and ${\bf M}^1$ are the moduli b-$\bQ$-divisors of $(X,D)\to Y$ and $(S^1,D^1)\to Y^1$, respectively, and 
		\item  $r( K_{S^1}+D^1)$ is linearly equivalent to 0 over the generic fiber.
	\end{itemize} 
	Thus case (i) implies there exists an integer $I'_{r}>0$ depending only on $r$ such that $ I'_r {\bf M}^1$ is b-free.
	Since ${\bf M}^1=\tau^* {\bf M}$, \cite[Proof of Lem. 2.1]{FL19} implies $\deg(\tau) I'_r {\bf M}$ is b-free.
	Hence it suffices to show that $T$ can be chosen such that $\deg(\tau)$ is bounded.
 One approach is to reduce to the case when $T$ is a curve and then apply \cite[Thm. 1.15]{Bir22}. 
 We give a more explicit approach below.

Note that $\deg(\tau)$ equals the number of connected components of $S'_{\oeta}$, which equals the number of connected components of $S_{\oeta}$.
	Since $S_{\eta}$ is the image of $T$ under the map $X_{\oeta}\to X_\eta$, 
	\cite[\href{https://stacks.math.columbia.edu/tag/04KZ}{Lem. 04KZ}]{stacks-project}
	implies $S_{\oeta}$ is the orbit of $T$ under the Galois group of $k(\oeta)/k(\eta)$.
	Hence, to complete the proof, it suffices to choose $T$ such that its Galois orbit has a bounded number of components.
	
\medskip

\noindent \emph{Case 1: $(X_{\overline{\eta}},D_{\overline{\eta}})$ has regularity 0.}

Since  $D^{\rm dlt,=1}_{\overline{\eta}}$ has at most two connected components
by \cite[Prop. 4.37.3]{Kol13}, 
$D^{\rm dlt,=1}_{\overline{\eta}}$ is either a prime divisor or a union of two disjoint divisors.
 Hence the Galois orbit of any choice of minimal lc center $T$ of $(X^{\rm dlt}_{\oeta}, D^{\rm dlt}_{\oeta})$ has  at most two components.
\medskip

\noindent \emph{Case 2: $(X_{\overline{\eta}},D_{\overline{\eta}})$ has regularity 1 and (ii) holds.}

First assume $D_{\overline{\eta}}$ has a coefficient 1 along a curve $C\subset X_{\overline{\eta}}$. 
Since $D_{\overline{\eta}} \sim_{\bQ} - K_{X_{\overline{\eta}}}$ and (ii) holds,
 there are  at most 3 such curves.
Let $T$ be a minimal lc center of $X^{\rm dlt}_{\overline{\eta}}$  that is contained in the strict transform of $\widetilde{C}$ of $C$.
Since  the restriction of $(X^{\rm dlt}, D^{\rm dlt})$ to $\widetilde{C}$ is an lc  CY curve pair, 
there are at most two such lc centers. 
Thus the Galois orbit of $T$ has at most six components.

Next, assume  $D_{\overline{\eta}}$ does not have coefficient 1 along a curve.
By \cite[Prop. 4.37]{Kol13},
$D_{\overline{\eta}}^{\dlt,=1}$  is  connected. 
Denote by $x\in X_{\oeta}$ the image of $D_{\oeta}^{\dlt, =1}$.  
Since the exceptional divisor of every log resolution of the smooth point $x\in X_{\oeta}$ is a tree of rational curves, $D_{\oeta}^{\dlt, =1}$ is a chain of rational curves. 
Hence, if we let $T$ be a minimal lc center of $(X^{\dlt}_{\overline{\eta}},D_{\overline{\eta}}^{\dlt,=1})$, which  is furthest to an endpoint of the chain, then the Galois orbit of $T$ has at most two components.

\medskip

\noindent \emph{Case 3: $(X_{\overline{\eta}},D_{\overline{\eta}})$ has regularity 1 and  (iii) holds.}

First assume there is an lc center $T'$ of  $(X^{\rm dlt}_{\overline{\eta}},D^{\rm dlt}_{\overline{\eta}})$ dominating $Z_{\overline{\eta}}$.
Since the general fiber of 
$X^{\rm dlt}_{\overline{\eta}} \to Z_{\overline{\eta}}$ 
is an lc CY curve pair,
there are at most two such lc centers.
Since $(T', D^{\rm dlt}_{T'})$ is again an lc CY curve, there are at most 2 lc centers contained in $T\subset T'$.
Hence, the  Galois orbit of such a $T$ has at most four components.

Next, assume no  lc center of  $(X^{\rm dlt}_{\overline{\eta}},D^{\rm dlt}_{\overline{\eta}})$ dominates $Z_{\overline{\eta}}$. By \cite[Prop. 4.37]{Kol13}, $D_{\oeta}^{\dlt, =1}$ is connected, and, hence, is contained in a fiber of $ X_{\oeta}^{\dlt} \to Z_{\oeta}$.  Since $X'_{\oeta}\to Z_{\oeta}$ is birational to a  smooth ruled surface, we know that every fiber is a tree of rational curves. Thus $D_{\oeta}^{\dlt, =1}$ is a chain of rational curves.
Hence, if $T$ is a minimal lc center chosen furthest along the chain, the Galois orbit has at most two components.
\medskip

By the previous cases, $T$ can be chosen so that the Galois orbit has at most 6 components and, hence, $\deg(\tau)\leq 6$. Therefore $60 I_r {\bf M}$ is b-free.
\end{proof}

\begin{lem}\label{l:delpezzoK}
There exists a constant $d>0$ such that the following holds:

If $X$ is a smooth del Pezzo surface over a field $K$ of characteristic zero and $X\times_K \overline{K}\not\cong \bP^2_{\overline{K}}$, then there exists a field extension $K'/K$ of  degree $\leq d$ such that $X_{K'}$ is either isomorphic to a blowup of $\bP^2_{K'}$ at distinct $K'$-rational points or isomorphic to $\bP^1_{K'}\times \bP^1_{K'}$.
\end{lem}

\begin{proof}
Let $\oX:=X\times_K \oK$ and $\Lambda: = [K_X]^{\perp} \subset \Pic(\oX)$.
Then $(\Lambda, \cdot)$ is a negative definite lattice by the Hodge index theorem. Hence the action of Galois group $\Gal(\oK/K)$ on  $\Pic(\oX)$ factors through the finite group $W$, which is the automorphism group of the lattice $(\Lambda, \cdot)$. Thus there exists a finite   extension of fields $K_1/K$ with $[K_1:K]\leq |W|$ such that $\Gal(\oK/K_1)$ acts trivially on $\Pic(\oX)$. 

If $\oX$ contains $(-1)$-curves,  we set $K'=K_1$. Then every $(-1)$-curve is fixed by the  $\Gal(\oK/K')$-action. We know that $\oX$ is isomorphic to the blowup of $\bP^2_{\oK}$ at $m=9-(K_{\oX}^2)$ distinct points $\{p_i\}_{i=1}^{m}$. Denote by $\oE_i$ the $(-1)$-curve on $\oX$ as the exceptional divisor over $p_i$. Since each $\oE_i$ is fixed by the $\Gal(\oK/K')$-action, we know that it is defined over $K'$ as a $(-1)$-curve $E_{i}'\subset X_{K'}$. Thus, the birational morphism $\oX \to \bP^2_{\oK}$ is also defined over $K'$, which implies that we have a birational morphism $\phi:X_{K'}\to S$ over $K'$ where $S$ is a Severi--Brauer surface over $K'$. Since $\phi(E_i')$ is a $K'$-rational point of $S$, we know that $S\cong \bP^2_{K'}$. 

If $\oX \cong \bP^1_{\oK}\times \bP^1_{\oK}$, then both classes of curves $[\oF_1]$ and $[\oF_2]$ of bidegree $(1,0)$ and $(0,1)$ respectively are invariant under the $\Gal(\oK/K_1)$-action. Since each projection map $\mathrm{pr}_j: \oX \to \bP^1_{\oK}$ for $j=1,2$ is an extremal contraction of $[\oF_{3-j}]$, it is defined over $K_1$ as $\pi_j: X_{K_1} \to Q_j$ where $Q_j$ is a smooth conic curve over $K_1$. Thus we get an isomorphism $(\pi_1,\pi_2): X_{K_1} \xrightarrow{\cong} Q_1 \times Q_2$. Then we let $K'/K_1$ be a finite field extension such that both $Q_1$ and $Q_2$ admit $K'$-rational points, which can be chosen such that $[K':K_1]\leq 4$. This implies that $X_{K'}\cong \bP^1_{K'}\times \bP^1_{K'}$. In the previous two cases, $\deg(K'/K)$ is bounded by a constant $d$ depending only on the lattice $\Lambda$ of $\overline{X}$, which only depends on the deformation family of $\oX$. Since there are only $9$ deformation families of $\oX$, the proof is complete.
\end{proof}

\begin{thm}\label{t:effectivebprecise}
For each finite subset $\Lambda \subset \mathbb{Q}\cap [0,1]$, there exists an integer $I:= I(\Lambda)$ satisfying:

If $(X,D)\to Y$ is an lc-trivial fibration such that $\dim(X) - \dim(Y)\leq 2$,  the generic fiber $(X_\eta,D_\eta)$ is an lc pair, and $D_\eta$ is a nonzero divisor with coefficients in $\Lambda$, then $I {\bf M}$ is b-free.
\end{thm}

\begin{proof}	
By \cite[Cor. 1.11]{PS09}, there exists a positive integer $r$ that depends  only on the coefficient set $\Lambda$ such that $r(K_{X_\eta}+D_\eta)\sim 0$. 
Thus $r {\bf M}$ is a ${\bf b}$-divisor and the coefficients of $D_\eta$ are contained in $r^{-1}\bZ \cap[0,1]$.
We will proceed to show that the theorem holds with $I:= d! I_r $, where $I_r$ and $d$ are the constants in Lemmas \ref{l:effectivebsemi} and \ref{l:delpezzoK}.

We first  reduce to the case when $X_\eta$ is a smooth del Pezzo surface.
By replacing $(X,D)\to Y$ with a base change,
we may assume $Y$ is projective. 
By \cite[Cor. 1.2]{HX13}, we may assume $X$ is projective. 
Fix a minimal resolution $\widetilde{X}_\eta \to X_\eta$, which exists 
by \cite[Thm. 2.25]{Kol13},
and a resolution $f:\widetilde{X}\to X$ that restricts to the minimal resolution on the general fiber. 
Write $\widetilde{D}$ for the divisor such that $K_{\widetilde{X}}+\widetilde{D}=f^*(K_X+D)$.
Note that $\widetilde{D}_\eta>0$  and  the coefficients of $\widetilde{D}_\eta$ are in $r^{-1}\bZ$, since $D_\eta>0$, $K_{\widetilde{X}_{\eta}/ X_\eta}\leq0$, and $r(K_{X_\eta}+{D}_\eta)\sim0$.
Using that $K_{\widetilde{X}} +\widetilde{D}\sim_{\bQ}0$, we see  $K_{X'}$ is not pseudo-effective over $Y$.
Thus \cite[Cor. 1.3.3]{BCHM10} implies that  a $K_{\widetilde{X}}$-MMP over $Y$ terminates with a Mori fibre space:
\[
\widetilde{X}\dashrightarrow \cdots \dashrightarrow X^{
	\rm m} \to Z^{\rm m}
.\]
Since $\widetilde{X}$ is terminal,  $X^{\rm m}$ is terminal and, hence, $X^{\rm m}_\eta$ is smooth. 
Let $D^{\rm m}$ denote the pushforward of $\widetilde{D}$ to $X^{\rm m}$.
Since $K_{\widetilde{X}}+\widetilde{D}\sim_{\bQ,Y}0$, $K_{X^{\rm m}}+D^{\rm m}\sim_{Y,\bQ}0$
and the pairs  are crepant birational over $Y$. 
Thus the moduli b-divisors of 
$(\widetilde{X},\widetilde{D})\to Y$ and $(X^{\rm m},D^{\rm m})\to Y$ are equal by \cite[Rem. 3.5]{FL19}.
If $\dim Z^{\rm m}>\dim Y$, then we are in the setting of Lemma \ref{l:effectivebsemi}.iii
and, hence, $I_r {\bf M}$ is b-free.
If not, then $X^{\rm m}_\eta$ is a smooth del Pezzo surface.
Thus, by  replacing $(X,D)$ with $(X^{\rm m},D^{\rm m})$, we may assume $X_\eta$ is a smooth del Pezzo surface. 

If $X_{\overline{\eta}} \cong \bP^2_{k(\overline{\eta})}$, then  $I_r {\bf M}$ is b-free  by Lemma \ref{l:effectivebsemi}.ii.
If not, then Lemma \ref{l:delpezzoK} implies there  exists  a field extension  $K'$ of $K(\eta)$ of degree $\leq d$ such that $X_{K'}$ is  isomorphic to $\bP^1_{K'}\times \bP^1_{K'}$ or $\bP^2_{K'}$ blown up  at  $K'$-points. 
In both cases, there is a surjective morphism
$X_{K'}\to \bP^1_{K'}$.
Thus there exists of finite morphism of normal varieties $\rho:Y'\to Y$ such that $\deg(Y'/Y)\leq d$ and  the generic fiber of $X\times_{Y} Y' \to Y'$
admits a surjective  morphism to $\bP^1_{K'}$. 
Now let $(X',D')\to Y'$ denote the lc-trivial fibration induced by pullback.
The moduli divisors of the two lc-trivial fibrations are related by that  ${\bf M}'= \rho^* {\bf M}$ by \cite[Prop. 3.1]{Amb05}.
There exists  a proper normal variety $Z'$ and a commutative diagram of dominant  maps
\[
\begin{tikzcd}
X' \arrow[r,dashed,"g'"] \arrow[rd,"f'",swap] & Z'\arrow[d,dashed,"h'"] \\
& Y',
\end{tikzcd}
\]
such that the generic fiber of $h'$ is $\bP^1_{K'}$.
By taking birational modifications of $X'$ and $Z'$ that are isomorphisms over the generic  point of $Y'$, we may assume that $g'$ and $h'$ are morphisms. 
This does not change the moduli divisor  by \cite[Prop. 8.4.9.1]{Kol07}. 
Since $f'$ has connected fibers, so does $g'$ and $h'$. 
Now Lemma \ref{l:effectivebsemi} implies $I_r {\bf M}'$ is b-free. 
Therefore $\deg(Y'/Y)I_r {\bf M}$ is b-free by \cite[Proof of Lem. 2.1]{FL19} and so  $I{\bf M}$ is b-free.
\end{proof}

\section{Proofs of main theorems}\label{s:proofofmainthms}

In this section, we prove   Theorems \ref{t:main1}, \ref{t:main2}, and \ref{t:effectiveb}, which were stated in the introduction.
Note that Corollary \ref{c:sequiv} was previously proven in Section \ref{ss:sequiv}.

\begin{proof}[Proof of Theorem \ref{t:main1}]
The first part of the theorem follows from Theorem \ref{t:stack}.
Theorems \ref{t:Scomplete}, \ref{t:Thetared}, and   \ref{thm:properness} combined with Lemma \ref{l:Ncompspecialize} imply (1)-(3) hold.
\end{proof}

\begin{proof}[Proof of Theorem \ref{t:main2}]
Consider the asymptotically good moduli space 
\[
\phi: (\cP_d^{\CY})^{\rm sn} \to P_{d}^{\CY}
\]
in Proposition \ref{p:existsagm}. Note that $P_d^{\CY}$ is the seminormalization of $P_{d,m}^{\CY}$ for $m$ sufficiently large.
The morphism $\phi$ is surjective by Proposition \ref{p:agmproperties}.
The algebraic space $P_{d,m}^{\CY}$ is reduced, irreducible,  and proper by Propositions \ref{p:normalawayfromIII+ell} and \ref{p:properness}.
We  now  verify  (1)-(4) of 
Theorem \ref{t:main2} hold.
\begin{enumerate}
\item  Proposition \ref{p:agmproperties}.2 implies $\phi$ is universal among maps to algebaic spaces.
\item   Lemma \ref{l:sequiv} and Proposition \ref{p:agmproperties} imply the map 
\[
(\cP_{d}^{\CY})^{\rm sn}(\bk) \to  (P_{d}^{\CY})^{\rm sn}(\bk)
\]
is surjective and identifies two pairs $(X,D) $ and $(X',D')$ if and only if they are S-equivalent. 
		
\item 	
Note that $	\cP_{d}^{\rm K}$ and $	P_{d}^{\rm K}$ are  seminormal by Proposition \ref{p:Klocalstructure}.
Thus there exists a commutative diagram as in (3) by Proposition \ref{p:agmproperties}.2.
Since $P_d^{\rm K}$ and $(P_d^{\rm H})^{\rm sn}$ are projective by Theorems \ref{t:ADL} and \ref{t:Hacking} and $P_d^{\CY}$ is separated, 
the maps in the bottom row are projective morphisms.
When $d\geq 4$,  $\cP_d \subset \cP_d^{\rm K}$ and $\cP_d \subset \cP_d^{\rm H}$.
Using Proposition \ref{p:compactPd}, we see the maps in the bottom row are bijections on $\bk$-points  over $\phi(\cP_d^{\rm sn})$. 
Thus the maps are birational.
		
\item Theorem \ref{t:HodgeAmple} implies  the Hodge line bundle on $P_{d,m}^{\CY}$ is ample.
Thus the pullback of the line bundle by $P_d^{\CY}\to P_{d,m}^{\CY}$ is also ample.

	\end{enumerate}
We have now shown that $P_d^{\CY}$ is a proper algebraic space that is reduced, irreducible, and admits  an ample line bundle. Thus $P_d^{\CY}$ is a projective variety.
\end{proof}

\begin{proof}[Proof of Theorem \ref{t:effectiveb}]
The results follows from Theorems \ref{t:bsemiample} and \ref{t:effectivebprecise}.
\end{proof}

\section{Examples}\label{sec:K3}

In this section we give concrete examples of the moduli space $P_d^{\CY}$ and apply Theorem \ref{t:main1} to additional moduli problems.

\subsection{Low degree plane curves}
We now describe the moduli space $P_{d}^{\CY}$   when $3 \leq d\leq  6$. When  $d=4$ and $d=6$, we assume the base field $\bk=\mathbb{C}$ in order to relate $P_d^{\CY}$ to certain Baily-Borel compactifications of moduli of K3 surfaces that are defined over the complex numbers.

\subsubsection{Degree 3}
By \cite[Ex. 4.5.3]{ADL19}, 
$\cP_3^{\rm K}$ parametrizes GIT semistable cubics and $P_3^{\rm K}$ is isomorphic to the GIT compactification of plane cubics $P_3^{\rm GIT}$.
Thus, a pair $(X,D)$ is in $\cP_3^{\rm K}(\bk)$ if and only if $(X,D)\cong (\bP^2, C)$, where
either
\begin{enumerate}
	\item[(a)] $C$ is a smooth cubic curve, or 
	\item[(b)] $C$ is a nodal cubic curve.
\end{enumerate}
Additionally,  in case (a), each pair is a closed point of $\cP_d^{\rm K}$ with finite stabilizer. In case (b), all pairs are S-equivalent, and $(\bP^2, \{xyz=0\})$ is the unique closed point in $\cP_d^{\rm K}$ from this S-equivalence class. 

\begin{prop}
The natural map $P_3^{\rm K} \to P_{3}^{\CY}$ is an isomorphism. 
\end{prop}

\begin{proof}
Since $P_3^{\rm K}$ and $P_{3}^{\CY}$ are projective 
seminormal varieties, it suffices to show that $P_3^{\rm K} \to P_{3}^{\CY}$ is a bijection on $\bk$-points.
Proposition \ref{p:wallcrossing} states that the map is surjective.
Additionally, if $(X,D)$ and $(X',D')$ are pairs in $\cP_{3}^{\rm K}(\bk)$ that are identified in $P_{3}^{\CY}(\bk)$, then $(X,D)$ and $(X',D')$ are S-equivalent by Theorem \ref{t:main2} and, hence, have isomorphic source by Proposition \ref{p:sequiv}. 
The above classification of pairs in $\cP_3^{\rm K}$ implies $(X,D)$ and $(X',D')$ are identified in $P_3^{\rm K}$. 
Therefore $P_3^{\rm K} \to P_{3}^{\CY}$ is bijective on $\bk$-points  as desired.
\end{proof}

The inclusion $\cP^{\rm K}_3\subset \cP_3^{\CY}$ is not an equality.
Indeed,  a pair $(X,D)$ in case (a) admits a  degeneration to $(X,D) \rightsquigarrow(X_0,D_0)$, where $X_0$ is the cone over $D$ and $D_0$ is the divisor at infinity. In addition, the S-equivalence class of Type III pairs in $\cP_3^{\CY}$ is unbounded by Example \ref{e:unbounded}.

\subsubsection{Degree 4}

In \cite{Kondo}, Kond\={o}  uses the observation 
that a degree 4 cyclic cover of $\bP^2$ along a smooth quartic curve is a degree K3 surface to construct a compactification $P_4\subset P_4^*$, which is the Baily-Borel compactification of a period domain parametrizing 
degree 4 K3 surfaces with a $\mathbb{Z}/4\mathbb{Z}$ symmetry.

\begin{prop}\label{prop:quartics}
The birational map $P_4^{\CY} \dashrightarrow P_4^{*}$ is an isomorphism.
\end{prop}

\begin{proof}
Note that $P_4^{\CY}$ and $P_4^{\rm K}$ are both normal by 
Propositions \ref{p:normalPdCY}.
By \cite[Thm. 6.5]{ADL19}, $P_4^{*}$ the ample model of the Hodge line bundle  $L_{\rm Hodge}^{\rm K}$ on $P_4^{\rm K}$, which means
\[
P_4^{*} \cong 
\Proj \left( \oplus_{m \in \bN} H^0\left(P_d^{\rm K}, m L_{\rm Hodge}^{\rm K}\right)\right)
.\]
Since $L_{\rm Hodge}$ is ample on $P_d^{\CY}$ by  Theorem \ref{t:main2}.4 and $L_{\Hodge}^{\rm K}$ is the pullback of $L_{\rm Hodge}$,
$P_4^{\CY}$ is also the ample model of $L_{\rm Hodge}^{\rm K}$. Thus $P_4^{\CY} \dashrightarrow P_4^{*}$ extends to an isomorphism.
\end{proof}

\begin{thm}
	The polystable pairs parametrized by $P_4^{\rm CY}(\bk)$ are the following:
	\begin{enumerate}
		\item $(\bP^2,\frac{3}{4}C)$ where $C$ is a plane quartic curve with at worst cuspidal singularities;
		\item $(\bP(1,1,4), \frac{3}{4}C)$ where $C$ is a degree $8$ curve not passing through the cone point with at worst cuspidal singularities;
		\item $(\bP(1,1,2)\cup \bP(1,1,2), \frac{3}{4}C)$ where $C$ is of degree $4$ on each component, and has a tacnodal singularity on each component, illustrated in Figure \ref{f:polystablequartics}. 
	\end{enumerate}
\end{thm}

\begin{proof}
By \cite[Section 6.2.1]{ADL19} and the isomorphism $P_4^{\CY}\cong P_4^*$ from Proposition \ref{prop:quartics}, we know that the birational morphism $P_4^{\rm K} \to P_4^{\CY}$ contracts an irreducible curve $\Sigma_{A_3}$ to a point $p$, and is an isomorphism elsewhere. Here $\Sigma_{A_3}$ is the closed locus of $P_d^{\rm K}$ parametrizing curves with tacnodal singularities.  Moreover, $P_4^{\rm K}\setminus \Sigma_{A_3}$ precisely parametrizes klt boundary polarized CY pairs $(X,D)$ which are the pairs from (1) or (2).  On the other hand, $\Sigma_{A_3}$ parametrizes pairs $(\bP^2, \frac{3}{4}C)$ where $C$ is a cat-eye or an ox, together with $(\bP(1,1,4)_{[y_0: y_1: y_2]}, \frac{3}{4}\{y_2^2 = y_0^4 y_1^4\})$. Thus it suffices to show that the point $p\in P_4^{\CY}$ is precisely the pair from (3). Indeed, consider the test configuration $(\cX,\cD) \to \bA^1$ where 
\[
\cX = \{x_0x_2 = tx_1^2\}\subset \bP(1,1,1,2)_{[x_0:x_1:x_2:x_3]}\times \bA^1_t, \quad \cD = \tfrac{3}{4}\{x_3^2 = x_1^4\}|_{\cX},
\]
and the $\bG_m$-action is given by $s\cdot ([x_0:x_1:x_2:x_3],t) = ([sx_0:x_1:x_2:x_3], st)$. Then the fiber over $t=1$ is isomorphic to $(\bP(1,1,4), \frac{3}{4}\{y_2^2 = y_0y_1^4\})$ via the embedding $\bP(1,1,4)\hookrightarrow \bP(1,1,1,2)$ given by $[y_0:y_1:y_2]\mapsto [y_0^2: y_0y_1 : y_1^2 : y_2]$. It is clear that the central fiber over $t=0$ is isomorphic to the pair in (3). Since the central fiber is of Type II and  admits an effective $\bG_m^2$-action by $(s_1,s_2)\cdot [x_0:x_1:x_2:x_3]=[s_1x_0:x_1: s_2 x_2:x_3]$, we know that it is polystable by Proposition \ref{prop:reg0polystable}. Thus the central fiber gives the polystable point $p$.
\end{proof}

\begin{figure}[h]
\caption{The common polystable degeneration of the cat-eye and ox.}
\label{f:polystablequartics}
\begin{tabular}{ccc}
{\resizebox{.2\textwidth}{!}{\input{Pics/cateye.tex}}} & \rotatebox[origin=c]{-30}{\Large{$\rightsquigarrow$}} 
& \multirow{2}{*}{\resizebox{.26\textwidth}{!}{\input{Pics/oxdegeneration.tex}}} \\ {\resizebox{.2\textwidth}{!}{\input{Pics/ox.tex}}} & \rotatebox[origin=c]{30}{\Large{$\rightsquigarrow$}}  &
\end{tabular}
\end{figure}
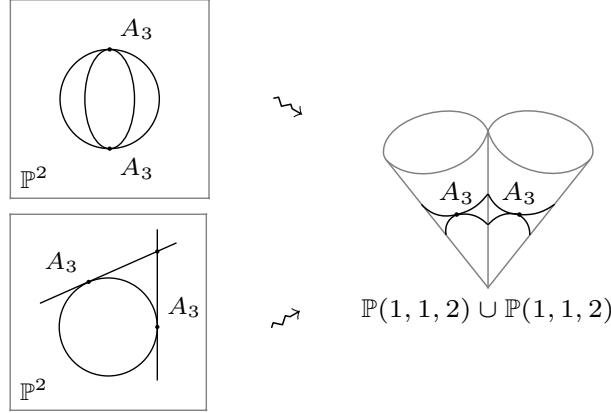

\subsubsection{Degree 5}
As in the degree 4 case above, we can describe the polystable representatives in $P_5^{\rm CY}$.

\begin{thm}
	The polystable pairs $(X,D)$ parametrized by $P_5^{\rm CY}(\bk)$ satisfy that either $(X, D)$ is klt, or it is isomorphic to one of the following two pairs, illustrated in Figure \ref{f:polystablequintics}:
 \begin{enumerate}
     \item $(\bP(1,1,5)\cup \bP(1,4,5) = (\{z_2z_0 = 0\}\subset \bP(1,1,4,5)_{[z_0:z_1:z_2:z_3]}), \frac{3}{5}\{z_3^2 = z_1^{10}\})$;
     \item $(\bP(1,1,2)\cup \bP(1,1,2) = (\{x_0x_2=0\}\subset \bP(1,1,1,2)_{[x_0:x_1:x_2:x_3]}), \frac{3}{5}\{x_1(x_3^2 - x_1^4)=0\})$.
 \end{enumerate}
 Moreover, the birational map $P_5^{\rm K} \to P_5^{\CY}$ contracts the $A_9$-locus and $D_6$-locus  to (1) and (2) respectively, and is isomorphic elsewhere. 
\end{thm}

\begin{proof}
By \cite[Section 9.4]{ADL19}, we know that there are two disjoint irreducible closed subvarieties $\Sigma_{A_9}$ and $\Sigma_{D_6}$ of $P_5^{\rm K}$ parametrizing curves with $A_9$ or $D_6$ singularities respectively. Moreover, $P_5^{\rm K}\setminus (\Sigma_{A_9}\cup\Sigma_{D_6})$ is precisely the locus parametrizing klt boundary polarized CY pairs. By \cite[Prop. 9.21 and 9.22]{ADL19}, we know that every boundary polarized CY pair in $\Sigma_{A_9}$ (resp. in $\Sigma_{D_6}$) admits a weakly special degeneration to the pair from (1) (resp. from (2)). Since the pair in (1) or (2) is of Type II and admits an effective $\bG_m^2$-action, we know that it is polystable by Proposition \ref{prop:reg0polystable}. Thus the birational map $P_5^{\rm K}\to P_5^{\CY}$ contracts $\Sigma_{A_9}$ and $\Sigma_{D_6}$ to (1) and (2) respectively, and is an isomorphism on  $P_5^{\rm K}\setminus (\Sigma_{A_9}\cup\Sigma_{D_6})$.
\end{proof}

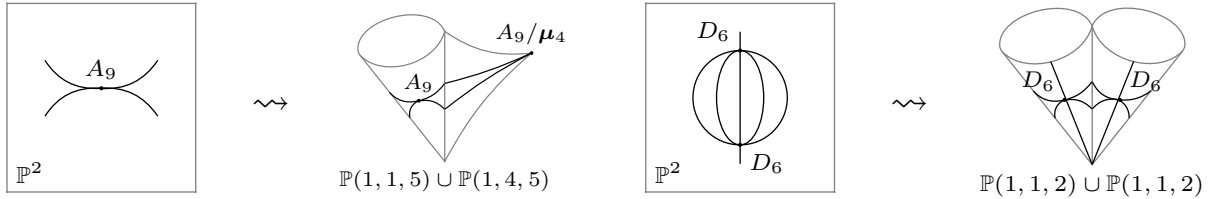
\begin{figure}[h]
\caption{The polystable degenerations of the $A_9$ and $D_6$ quintic curves.}
\label{f:polystablequintics}
\begin{tabular}{cccccc}
{\resizebox{.19\textwidth}{!}{\input{Pics/A9.tex}}} & {\Large{$\rightsquigarrow$}} &{\resizebox{.23\textwidth}{!}{\input{Pics/A9rep.tex}}} & 
{\resizebox{.19\textwidth}{!}{\input{Pics/D6.tex}}} & {\Large{$\rightsquigarrow$}} &{\resizebox{.23\textwidth}{!}{\input{Pics/D6rep.tex}}}
\end{tabular}
\end{figure}

\subsubsection{Degree 6}
The double cover of $\bP^2$ branched along a smooth sextic plane curve is a degree 2 K3 surface and this construction induces a birational map
\[
P_6
\to \sF_2
,\]
where $\sF_2$ is the coarse moduli space of primitively polarized  degree 2 K3 surfaces with ADE singularities. 
Using that $\sF_{2}$ can be endowed with the structure of a Hermitian symmetric domain by the global Torelli theorem, $\sF_2$ admits a Baily-Borel compactification $\sF_2^{\rm BB}$, which is a normal projective variety.
The compactifications $P_6^{\rm K}$ and $P_6^{\rm H}$ and their relation to $\sF_2^{\rm BB}$ are well-understood. 
\begin{itemize}
	\item The moduli space $ P_6^{\rm H}$ was studied in detail by Alexeev, Engel, and Thompson \cite{AET}. 
	In particular, the authors analyzed the boundary strata and proved that the normalization of $P_6^{\rm H}$ is a semitoroidal compactification of $\sF_2$.
	
\item By \cite[\S 6.2.2]{ADL19}, the moduli space $ P_6^{\rm K}$ agrees with  a compactification of Shah \cite{Sha80} that was constructed as a partial Kirwan desingularization of the moduli space $P_6^{\rm GIT}$ of GIT semistable plane sextics. 
	In addition, Shah constructed a set theoretic map $P_6^{\rm K}\to \sF_2^{\rm BB}$ extending the map $P_6 \to \sF_2$. The map was later shown to be algebraic by work of Looijenga \cite{Loo03} (see  \cite[Thm. 1.9]{Laz16}).
\end{itemize}
Using that the Hodge line bundle is ample on $P_6^{\rm CY}$, we deduce the following.

\begin{prop}
The  map $P_6 \to \sF_2$ extends to an isomorphism  $P_6^{\CY} \to  \sF_2^{\rm BB}$. 
\end{prop}

\begin{proof}
Note that $P_6^{\rm K}$ is normal by Proposition \ref{p:Klocalstructure} and  $P_6^{\CY}$ is seminormal by Proposition \ref{p:existsagm}.
By \cite[Thm. 6.6]{ADL19},
 $\sF_2^{\rm BB}$ is the ample normal model of  $L_{\rm Hodge}^{\rm K}$ on $P_6^{\rm K}$.
Since $L_{\rm Hodge}$ is ample on $P_6^{\CY}$  by Theorem \ref{t:main2}.4 and $L_{\rm Hodge}^{\rm K}$ is the pullback of $L_{\rm Hodge}$ to $P_6^{\rm K}$, the normalization of $P_6^{\CY}$ is the ample normal model of $L_{\Hodge}^{\rm K}$ as well. Thus  $P_6^{\rm K} \to P_6^{\CY} $ factors as
\[
\begin{tikzcd}
P_6^{\rm K} \arrow[r] & \sF_2^{\rm BB}\arrow[r,"\nu"] & P_{6}^{\rm CY}
\end{tikzcd}
,\]
where $\nu$ is the normalization of $P_6^{\CY}$.

We claim that $\nu$ is a bijection on $\bk$-points. 
Since $\nu$ is a normalization, surjectivity holds.
To show injectivity, it suffices to show: if $(X,D)$ and $(X',D')$  are pairs in $P_6^{\rm K}(\bk)$ that are identified in $P_d^{\CY}$, then  $(X,D)$ and $(X',D')$ are identified in $\sF_2^{\rm BB}$.
Note that as $(X,D)$ and $(X',D')$ are identified in $P_6^{\CY}$, $\Src(X,D)= \Src(X',D')$ by Proposition \ref{p:SrcSequiv}.  We consider three cases: 
\vspace{.1 cm}

\noindent \emph{Case 1: The source is not a smooth elliptic curve or a point.} By Proposition \ref{p:normalawayfromIII+ell}, 
$P_6^{\CY}$ is normal at the image of $(X,D)$ and $(X',D')$. Therefore, 
$\nu$ is a bijection at those points and the pairs are  identified in $\sF_2^{\rm BB}$.

\noindent \emph{Case 2: The source is a smooth elliptic curve.} The classification of the boundary of $P_6^{\rm K}$ in  \cite[pg. 233]{Laz16} implies $(X,D)\cong (\bP^2, C)$ and $(X',D')\cong (\bP^2, C')$ for some smooth cubic curves $C$ and $C'$ in $\bP^2$. Since the sources  are isomorphic, we must have $(X,D)\cong (X',D')$. 

\noindent \emph{Case 3: The source is a point.} 
Since the map $P_6^{\rm K}\to \sF_2^{\rm BB}$ contracts the Type III locus to a single point \cite[pg. 234]{Laz16}, 
$(X,D)$ and $(X',D')$ are identified in $\sF_2^{\rm BB}$.
\vspace{.1 cm}

By the above cases, $\nu$ is a bijection on $\bk$-points. Using that $P_6^{\rm CY}$ seminormal and  $\sF_2^{\rm BB}$ is reduced,  we conclude that $\nu$ is an isomorphism.
 \end{proof}

\begin{thm}\label{t:deg6ps}
The polystable pairs of Type II in $P_6^{\CY}(\bk)$, as illustrated in Figure \ref{f:polystablesextics}, are the following:
\begin{enumerate}
    \item $(C_p(E,L), E_\infty)$ where $E$ is a smooth elliptic curve, $L$ is a degree $9$ line bundle on $E$, and $E_\infty$ is the section at infinity;
    \item $(\bP(1,1,2)\cup \bP(1,1,2), \frac{1}{2}C)$ where on each component of the surface, $C$ is the sum of three sections tangent at a common point away from the non-normal locus;
    \item $(\bP^2\cup \bF_1, \frac{1}{2}C)$ where the surface is glued along a line $l$ in $\bP^2$ and the negative section in $\bF_1$,  $C|_{\bP^2}$ is the sum of four distinct lines passing through a common point away from $l$, and $C|_{\bF_1}$ is the sum of four distinct fibers together with  a doubled positive section;
    \item $(\bP(1,1,4)\cup \bF_4, \frac{1}{2}C)$ where the surface is glued along the section at infinity of $\bP(1,1,4)$ and the negative section of  $\bF_4$, $C|_{\bP(1,1,4)}$ is the sum of four distinct rulings, and $C|_{\bF_4}$ is the sum of four distinct fibers together with a doubled positive section.
\end{enumerate}
\end{thm}

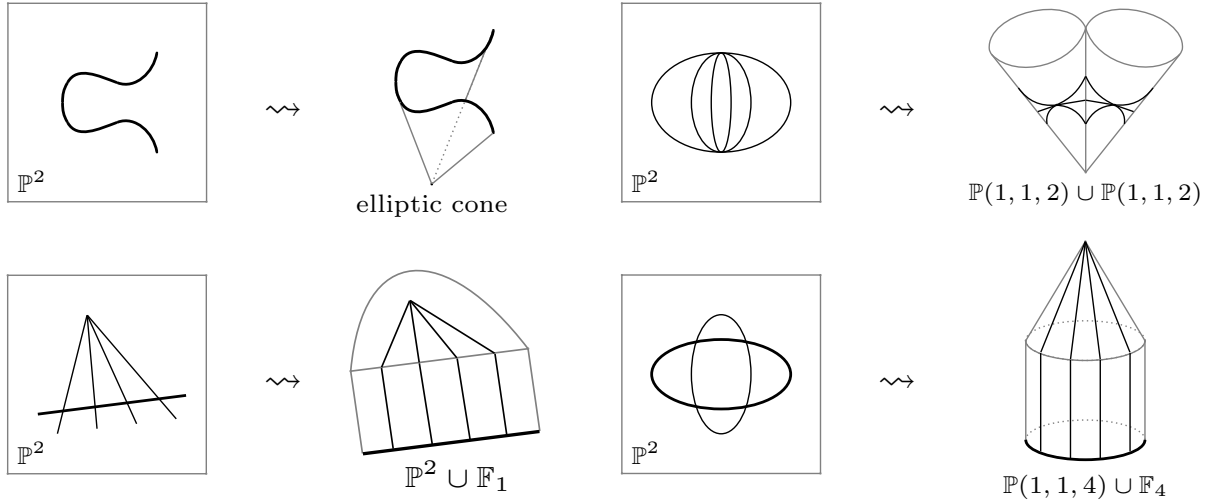
\begin{figure}[h]
\caption{The polystable degenerations of Type II sextic curves.}
\label{f:polystablesextics}
\begin{tabular}{cccccc}
{\resizebox{.2\textwidth}{!}{\input{Pics/doublecubic.tex}}} & {\Large{$\rightsquigarrow$}} & {\resizebox{.2\textwidth}{!}{\input{Pics/doublecubicrep.tex}}} & {\resizebox{.2\textwidth}{!}{\input{Pics/triplecat.tex}}} & {\Large{$\rightsquigarrow$}} & {\resizebox{.24\textwidth}{!}{\input{Pics/triplecatrep.tex}}} \\
{\resizebox{.2\textwidth}{!}{\input{Pics/lines.tex}}} & {\Large{$\rightsquigarrow$}} & {\resizebox{.2\textwidth}{!}{\input{Pics/linesrep.tex}}} & {\resizebox{.2\textwidth}{!}{\input{Pics/12conic.tex}}} & {\Large{$\rightsquigarrow$}} & {\resizebox{.2\textwidth}{!}{\input{Pics/12conicrep.tex}}} \\
\end{tabular}
\end{figure}

\begin{proof}
The polystability of these pairs in (2), (3), and (4)  follows from Proposition \ref{prop:reg0polystable}. Let $(X,D)$ be a pair from (1) whose polystable degeneration is $(X_0, D_0)$. Assume to the contrary that $(X,D)$ is not polystable so by Propositions \ref{p:torusrk+1} and \ref{prop:reg0polystable} we know that $\Aut^0(X_0,D_0)\cong \bG_m^2$. By \cite[Thm. 5.5 \& 6.5]{Hac04} we know that  $X_0$ is either normal or obtained by gluing two normal surfaces along a smooth rational curve. The latter is not possible as $\Src(X_0,D_0)\cong \Src(X,D)$ is a smooth elliptic curve.  Thus $X_0$ is normal, and $\Aut^0(X_0,D_0)\cong \bG_m^2$ implies that $X_0$ is toric and $D_0$ is torus invariant. Let $D_0'$ denote the reduced toric boundary of $X_0$. Then we have $K_{X_0} + D_0' \sim_{\bQ} 0$ and $D_0\leq D_0'$, which implies that $D_0 = D_0'$. Thus the source of $(X_0, D_0)$ is a point which cannot be a smooth elliptic curve, a contradiction.  Thus we conclude that pairs from (1) are polystable.

One can work out the polystable degenerations from the Type II pairs in $P_6^{\rm K}$ (see e.g. \cite[pg. 233]{Laz16})  to these four cases. For (1), one simply take the degeneration to normal cone. For (2), the construction is similar to  cases from degree $4$ and $5$ where the degeneration of the underlying surface is also $\bP(1,1,2)\cup\bP(1,1,2)$. For (3) and (4), one takes the blowup of the trivial test configuration along the double curve in the central fiber.
\end{proof}

\subsection{Points on $\bP^1$}
We now use the machinery of the paper to compactify the moduli space of unordered points on $\bP^1$. For other approaches using KSBA or K-stability, see e.g. \cite{Has03, Fuj17}.

\begin{defn}
Fix an integer $d \ge 2$.  Let $\chi(m) := \chi(\bP^1, \omega_{\bP^1}^{[-m]})$ and 
\[
\sC_d^{\CY} \subset \cM(\chi, d, (1, \tfrac{2}{d})))
\]
denote the open and closed substack of 
$\cM(\chi, d, (1, \frac{2}{d}))$ consisting of  families $[(X,\Delta +D)\to B]$ with $\Delta = 1 \cdot \emptyset$ (see \ref{r:Delta=0}).
Let 
\[
\sC_d \subset \sC_d^{\GIT}\subset \sC_d^{\CY}
\]
denote open substacks such that  $\sC_d^{\GIT}(\bk)$ parametrizes pairs $(X,D)$, where $X\cong \bP^1$, and $\sC_d(\bk)$ parametrizes pairs $(X,D=\frac{2}{d}D^{\Div})$, where $X\cong \bP^1$ and  $D^{\Div}$ is reduced on $X$. 
\end{defn}

\begin{lem}\label{lem:P^1-stack-sm}
Let $d\geq 2$ be an integer.
The following properties of $\sC_d^{\CY}$ hold.
\begin{enumerate}
    \item If $(X,D)$ is a pair in $\sC_d^{\CY}(\bk)$,  then  $X\cong \bP^1$ or $X\cong \bP^1\vee \bP^1$, where  $\bP^1\vee \bP^1$ denotes the nodal curve obtained by gluing two copies of $\bP^1$ at a point.
    \item $\sC_d^{\CY}$ is smooth.
    \item $\sC_d^{\GIT}$ is dense in $\sC_d^{\CY}$.
\end{enumerate}
\end{lem}

\begin{proof}
    (1) Since $(X,D)$ is a boundary polarized CY pair of dimenison $1$, we know that $X$ is a nodal curve with $-K_X$ ample. If $X$ is normal, then it is isomorphic to $\bP^1$. If $X$ is non-normal, then its normalization $(\oX, \oG)$ is a log Fano pair. Thus every component of $\oX$ is $\bP^1$ containing a unique point in the conductor divisor. Thus $X\cong \bP^1\vee \bP^1$. 

    (2) To show the stack is smooth, it suffices to verify the infinitesimal lifting criterion for smoothness as in the proof of Proposition \ref{p:smoothnesPCYatpt}. Here the situation is much simpler as every pair $(X,D)$ in $\sC_d^{\CY}(\bk)$ satisfies that $X$ is a reduced Gorenstein curve with at worst nodal singularities, and $D^{\Div}$ is a Cartier divisor not passing through the singular locus of $X$. 
    By \cite[\href{https://stacks.math.columbia.edu/tag/0DZX}{Tag 0DZX}]{stacks-project} we know that $X$ has unobstructed deformations. By \cite[Thm. 6.2.b]{Har10} we know that the Cartier divisor $D^{\Div}$ has unobstructed deformations as $H^1(D^{\Div}, \cO_X(D^{\Div})|_{D^{\Div}}) = 0$.
    Thus the statement follows by combining the unobstructedness of deformations of both $X$ and $D^{\Div}$.

    (3) We know that $\bP^1\vee \bP^1$ is an  isotrivial degeneration of $\bP^1$, as one can just take the blowup of $\bP^1\times \bA^1$ at a point in the central fiber. Thus by the unobstructedness of deformations of $D$ from (2) we know that every pair $(X, D)$ in $\sC_d^{\CY}(\bk)$ with $X\cong \bP^1\vee\bP^1$ deforms to a family in $\sC_d^{\GIT}$. 
\end{proof}

\begin{thm}\label{thm:pts-P^1}
    For any $d \ge 2$, the stack $\sC_d^{\rm CY}$ is of finite type and admits a projective good moduli space ${\rm C}_d^{\rm CY}$. Moreover,
    \begin{enumerate}
        \item The stack $\sC_d^{\GIT}$ is isomorphic to the GIT quotient stack $[\bP(H^0(\bP^1,\cO(d)))^{\rm ss}/\PGL_2]$ and hence admits a projective good moduli space ${\rm C}_d^{\GIT}$ isomorphic to the GIT quotient space $\bP(H^0(\bP^1,\cO(d)))^{\rm ss}\sslash \PGL_2$. 
        \item There is a commutative diagram
    \[
	\begin{tikzcd}
\sC_{d}^{\GIT} \arrow[r,hook]\arrow[d, "\phi_{\GIT}",swap] & \sC_{d}^{\CY}\arrow[d,"\phi"]\\
		{\rm C}_{d}^{\GIT} \arrow[r,"\cong"] & {\rm C}_{d}^{\CY} 
	\end{tikzcd}
	\]
 where the bottom row is an isomorphism.
        \item The Hodge line bundle on ${\rm C}_d^{\rm CY}$ is ample.
    \end{enumerate}
\end{thm}

A version of the moduli space ${\rm C}_d^{\rm GIT}$ for marked points on $\bP^1$ notably appears in the work of Deligne and Mostow \cite{DM86}.

\begin{proof}
    Every pair  $(X,D)$ in $\sC_d^{\rm CY}(\bk)$ has regularity at most 0 as $\dim X = 1$.  Therefore, Theorem \ref{t:modulireg0} implies that $\sC_d^{\CY}$ is of finite type and admits a separated good moduli space ${\rm C}_d^{\CY}$. Since $\sC_d^{\CY}$ satisfies the existence part of the valuative criterion for properness by Theorem \ref{thm:properness}, we conclude that ${\rm C}_d^{\CY}$ is a proper algebraic space by Theorem \ref{t:AHLH}.  The projectivity will be verified in (3) by proving the ampleness of the Hodge line bundle.

    (1) By definition we know that $(X,D)$ is in $\sC_d^{\GIT}(\bk)$ if and only if $X\cong \bP^1$ and $\mult_p D^{\Div}\leq \frac{d}{2}$ for every $p\in X$. By \cite[\S 4.1]{GIT} this is equivalent to $D^{\Div}$ being GIT semistable as a binary form in $H^0(\bP^1, \cO(d))$. 
    To show that there is an isomorphism between these two stacks, we look at the forgetful map $\psi:\sC_d^{\GIT} \to B\PGL_2$ sending $[(X,D) \to S]$ to $[X \to S]$ where $B\PGL_2$ is the classifying stack of $\mathbb{P}^1$-bundles. Clearly the fiber of $\psi$ is the scheme $\bP(H^0(\bP^1,\cO(d)))^{\rm ss}$ equipped with the natural $\PGL_2$-action. Thus we obtain the isomorphism between stacks.

    The isomorphism of good moduli spaces follows from Proposition \ref{p:Alpergmprops}.3. 

    (2) The commutative diagram exists by the universality of good moduli spaces from Proposition \ref{p:Alpergmprops}.3. Therefore it remains to show that $\sigma:{\rm C}_{d}^{\GIT} \to {\rm C}_{d}^{\CY}$ is an isomorphism. By Lemma \ref{lem:P^1-stack-sm}, we know that $\sigma$ is dominant and ${\rm C}_d^{\CY}$ is normal. If $d=2$ then ${\rm C}_{d}^{\GIT}$ is a point which implies the isomorphism. Assume $d\geq 3$ from now on. Then, there is a dense saturated open substack $\sC_d^{\rm s}\subset \sC_d^{\GIT}$ parametrizing $(X,D)$ where $X\cong \bP^1$ and $D^{\Div}$ is GIT stable. Moreover, $\sC_d^{\rm s}$ admits a coarse moduli space ${\rm C}_d^{\rm s}$ as an open subscheme of ${\rm C}_d^{\rm GIT}$. Since every pair $(X,D)$ in $\sC_d^{\rm s}(\bk)$ is klt by \cite[\S 4.1]{GIT}, by Proposition \ref{p:sequivklt} we know that the map $\sigma|_{{\rm C}_d^{\rm s}}: {\rm C}_d^{\rm s}\to {\rm C}_d^{\CY}$ is injective.     
    If $d$ is odd, then by \cite[\S 4.1]{GIT} we know that ${\rm C}_d^{\rm s}= {\rm C}_d^{\GIT}$ and hence $\sigma$ is dominant and injective between normal proper varieties, which implies that $\sigma$ is an isomorphism by Zariski's main theorem. If $d$ is even, then  by \cite[\S 4.1]{GIT} we know that $ {\rm C}_d^{\GIT}\setminus {\rm C}_d^{\rm s}$ is a single point. Thus $\sigma$ is dominant, birational and quasi-finite between normal proper varieties, which implies that $\sigma$ is an isomorphism by Zariski's main theorem.

    (3)  As $\sC_d^{\rm CY}$ is finite type, Proposition \ref{p:Hodgedescends} gives the existence of the Hodge line bundle on ${\rm C}_d^{\rm CY}$.  If $d = 2$, ${\rm C}_d^{\rm CY}$ is a point, so the result follows.  Assume then $d \ge 3$. By the Nakai--Moishezon criterion for ampleness \cite[Thm. 3.11]{Kol90}, it suffices to show $L_{\Hodge}\vert_Z^{\dim Z}>0$ for each irreducible closed subspace $Z\subset {\rm C}_{d}^\CY$ of positive dimension.  Since $\dim Z > 0$ and the non-klt locus is at most a single point of $ {\rm C}_{d}^\CY$ by \cite[\S 4.1]{GIT}, the  fiber of $\sC_d^{\rm CY} \to {\rm C}_d^{\rm CY}$ over a general point of $Z$  contains only klt pairs.  Therefore, Theorem \ref{p:ftHodgeampleonsubspace}.1 implies $L_{\Hodge}\vert_Z^{\dim Z}>0$ and we conclude that the Hodge line bundle is ample. 
\end{proof}

In fact, we can say more: if $d$ is odd, then 
by Proposition \ref{p:coreg>0condition}, for any $(X,D) \in \sC_d^{\rm CY}(\bk)$, $\coreg(X,D) >0$ so $\reg(X,D) < 0$, i.e. $(X,D)$ is klt. Thus $X\cong \bP^1$ which shows that $\sC_d^{\CY} = \sC_d^{\GIT} =\sC_d^{\rm s}$ and hence the good moduli space $ \sC_d^{\rm CY} \to  {\rm C}_d^{\rm CY}$ is a coarse moduli space.  
If $d$ is even, we can have lc pairs $(X,D)$ of coregularity 0 and hence regularity 0. Since ${\rm C}_d^{\GIT}\setminus {\rm C}_d^{\rm s}$ consists of a single point $(\bP^1, \frac{2}{d}(\frac{d}{2}[0]+\frac{d}{2}[\infty]))$ by \cite[\S 4.1]{GIT}, we know that all non-klt pairs in $\sC_d^{\CY}$ are S-equivalent to each other. By blowing up the point $([0],0)$ in the family $(\bP^1\times \bA^1, \frac{2}{d}(\frac{d}{2}[0]\times\bA^1+\frac{d}{2}[\infty]\times\bA^1))\to \bA^1$, we obtain a test configuration of $(\bP^1, \frac{2}{d}(\frac{d}{2}[0]+\frac{d}{2}[\infty]))$ whose central fiber is isomorphic to $(\bP^1 \vee \bP^1, \frac{2}{d} (\frac{d}{2} p_1+ \frac{d}{2} p_2))$ where $p_1$ is a point on the first copy of $\bP^1$ and $p_2$ is a point on the second copy. Since the identity component of the automorphism group of $(\bP^1 \vee \bP^1, \frac{2}{d} (\frac{d}{2} p_1+ \frac{d}{2} p_2))$ is $\bG_m^2$, this pair is the unique non-klt polystable point in ${\rm C}_d^{\CY}$ by Proposition \ref{prop:reg0polystable}.

\subsection{Curves of odd degree in $\bP^1\times\bP^1$}
We now study the moduli stack $\sK_{d}^{\CY}$ of boundary polarized CY pairs $(\bP^1\times\bP^1, \frac{2}{d}C)$ and their degenerations where $d\geq 3$ is odd and $C\in |\cO(d,d)|$. 

\begin{defn}
Fix an integer $d\geq 2$.
Let $\chi'(m):=\chi(\bP^1\times\bP^1, \omega_{\bP^1\times\bP^1}^{[-m]}) = (2m+1)^2$ and
\[
\sK_d \subset \cM(\chi', d, (1, \tfrac{2}{d}))
\]
denote the open substack whose $\bk$-points are pairs $(X,\Delta+D)$ 
such that $X\cong \bP^1 \times \bP^1$, $\Delta : = 1\cdot \emptyset$, and $D:=\frac{2}{d}C$, where $C$ is a smooth curve of bidegree $(d,d)$.

Let $\sK_d^{\CY}:=\overline{\sK_d}$ be the stack theoretic closure of $\sK_d$ in $\cM(\chi', d, (1, \frac{2}{d}))$. 
\end{defn}

\begin{thm}\label{thm:p1xp1}
    Assume that $d\geq 3$ is an odd integer. Then 
    the stack $\sK_{d}^{\CY}$ is of finite type  and admits a proper good moduli space ${\rm K}_d^{\CY}$. 
\end{thm}

\begin{proof} Since $d$ is odd, we know that $1\not\in \frac{2}{d}\bZ$. By Proposition \ref{p:coreg>0condition}, every pair $(X,D)$ in $\sK_{d}^{\CY}(\bk)$ has $\coreg(X,D)>0$, and so 
\[
\reg(X,D)= \dim X -1 -\coreg(X,D)= 1-\coreg(X,D)\leq 0.
 \]
Thus Theorem \ref{t:modulireg0} implies that $\sK_d^{\CY}$ is of finite type and admits a separated good moduli space ${\rm K}_d^{\CY}$. 
Since $\sK_d^{\CY}$ satisfies the existence part of valuative criterion for properness by Theorem \ref{thm:properness}, ${\rm K}_d^{\CY}$ is proper by Theorem \ref{t:AHLH}.
\end{proof}

As a consequence, we obtain a wall crossing diagram including birational morphisms to ${\rm K}_3^{\CY}$ from both the K-moduli space studied in \cite{ADL20} and the KSBA moduli space studied in \cite{DH21} for $d=3$.

\subsection{Quintic surfaces in $\bP^3$}

Let $d\geq 4$ be an integer, $\chi''(m) := \chi(\bP^3, \omega_{\bP^3}^{[-m]})$, and 
\[
\sS_d\subset \cM(\chi'', d , (1 , 4/d))
\]
denote the open substack whose $\bk$-points are the pairs   $(X,\Delta+D)$ such that $X\cong \bP^3$, $\Delta= 1 \cdot \emptyset$, and $D:= \tfrac{4}{d} D^{\rm Div}$, where $D^{\rm Div} \subset \bP^3$ is a smooth surface of degree $d$.
Let
$
\sS_d^{\rm CY}: = \overline{\sS_d}$
denote the stack theoretic closure of $\sS_d$ in $\cM(\chi'', d, (1, 4/d))$.

\begin{thm}
The stack $\sS_5^{\CY}$ is a finite type algebraic stack and admits a proper good moduli space ${\rm S}_5^{\CY}$.
\end{thm}

The proof relies on \cite[Thm. 5.3]{DeV19}, which will be used to show that all pairs in the moduli space have regularity at most 0.

\begin{proof}
We will first show that every pair $(X,D)$ in $\sS_5^{\CY}(\bk)$ has $\reg(X,D) \leq 0$. 
To proceed, fix a family $(Y,G) \to C$ in $\sS_5^{\CY}$ over the germ of a curve $0 \in C$ such that $(X,D)\cong (Y_0,G_0)$ and $(Y_K,D_K)$ is in $\sS_5$, where $K:= K(C)$. 
Since $(Y_K, (1+\vep)D_K)$ is slc for $0<\vep \ll1$, \cite[Thm. 3.15]{DeV19} implies that, after possibly replacing $0 \in C$ by a finite cover, there exists a family $(Y',G') \to C$ in $\sS_5^{\CY}$ such that $(Y'_0,(1+\vep)G'_0) $ is slc. 
Now,
\[
\reg(X,D) = \reg(Y'_0,G'_0) \leq 0
,\]
where the first equality is by Propositions \ref{p:sequiv} and
\ref{p:SrcSequiv} and the second is \cite[Thm. 5.3]{DeV19}.

Since all pairs in $\sS_5^{\CY}$ have regularity at most 0,  Theorem \ref{t:modulireg0} 
implies $\sS_5^{\CY}$ admits a separated good moduli space ${\rm S}_5^{\CY}$. 
Since $\sS_5^{\CY}$ satisfies the existence part of the valuative criterion for properness by Theorem \ref{thm:properness},  ${\rm S}_5^{\CY}$ is proper by Theorem \ref{t:AHLH}.
\end{proof}

The above proof does not generalize to the case when degree $d \ne 5$, since there exist pairs appearing with higher regularity  \cite[Rem. 5.4]{DeV19}.

\appendix

\section{Special degenerations of $\bP^2$}\label{a:specialdegenerations}
In this section, we use a technical result from Section \ref{ss:1comp} to  characterize all special degenerations of $\bP^2$.
This gives a complete answer to \cite[Prob. 3.4]{AIM20}. 
 To begin, we recall the following definition.

\begin{defn}
A \emph{special test configuration} $\cX$ of a klt Fano variety $X$ is the data of 
\begin{enumerate}
	\item a $\bG_m$-equivariant proper flat family of klt varieties $\cX\to \bA^1$ such that $-K_{\cX}$ is relatively ample, and
	\item a $\bG_m$-equivariant isomorphism
	\[
	\cX\times_{\bA^1} (\bA^1\setminus 0) \cong  X\times (\bA^1\setminus 0)
	,\]
	where $\bG_m$ acts on the right side as the product of the trivial and standard actions.	
\end{enumerate}
We say $X_0$ is a \emph{special degeneration} of $X$ if there exists a special test configuration $\cX$ of $X$ with $\cX_0\cong X_0$. 
\end{defn}

\begin{thm}\label{thm:P^2-1comp}
Any non-trivial special test configuration $\cX$ of $\bP^2$ satisfies that $r(\ord_{\cX_0})$ is an lc place of a $1$-complement of $\bP^2$. 
\end{thm}

\begin{proof}
Let $\cX \to \bA^1$ be a special test configuration of $\bP^2$. By \cite[Thm. A.2]{BLX19}, there exists an $N$-complement $D$ on $\bP^2$ and an lc place $E$ of $(\bP^2, D)$ such that $r(\ord_{\cX_0}) = b\cdot \ord_E$ for some $b\in \bZ_{>0}$. Let $\cD$ be the closure of $D\times (\bA^1\setminus\{0\})$ in $\cX$. By Lemma \ref{l:tclcplace}, we know that $(\cX,\cD)$ is a weakly special test configuration of $(\bP^2, D)$. Thus $(\cX_0, \cD_0)\in \cP_{3N}^{\rm CY}(\bk)$. By Proposition \ref{p:exists1comp}, there exists a $\bG_m$-equivariant $1$-complement $B_0$ on $X_0$. By Lemma \ref{l:extenddivisor}, $B_0$ extends to a $\bG_m$-equivariant divisor $\cB$ on $\cX$ such that $(\cX, \cB)\to \bA^1$ is a $\bG_m$-equivariant family of boundary polarized CY pairs of index $1$. Thus there is a $1$-complement $B$ of $\bP^2$ such that $(\cX,\cB)\to \bA^1$ is a weakly special test configuration of $(\bP^2, B)$. By Lemma \ref{l:lctestconfig}.2, we know that $E$ is also an lc place of $(\bP^2,B)$. The proof is finished. 
\end{proof}

It is clear that a $1$-complement of $\bP^2$ is one of the following: a smooth cubic curve, an irreducible nodal cubic curve, a smooth conic union a transversal line, or a triangle of lines. With the help of \cite[Section 6]{LXZ21}, we can completely classify all special degenerations of $\bP^2$.

\begin{thm}
Let $X_0$ be a special degeneration of $\bP^2$. Then $X_0$ is one of the following:  $\bP^2$, $\bP(1,1,4)$, $\bP(1, d_n^2, d_{n+1}^2) $, or the weighted hypersurfaces $\{x_0 x_3 = x_1^{d_{n+1}} + x_2^{d_{n-1}}\}\subset \bP(1, d_{n-1}, d_{n+1}, d_n^2)$. Here $(d_n)_{n\geq 0}$ is the sequence of integers such that $d_0=d_1 =1$ and $d_{n+1} = 3d_n - d_{n-1}$.
\end{thm}

\begin{proof}
Let $B$ be the $1$-complement on $\bP^2$ from Theorem \ref{thm:P^2-1comp}.
If $B$ is a smooth cubic curve, then $E=B$ which implies that $X_0$ is the projective cone over $B$ by Lemma \ref{l:degnormalcone}, which is not klt. If $B$ is an irreducible nodal cubic curve, then the special degenerations were classified in \cite[Thm. 6.1 \& Rem. 6.6]{LXZ21} which gives the list from the statement. If $B$ is a triangle, then $(\bP^2, B)$ is toric which implies that any special degeneration is of product type, hence $X_0\cong\bP^2$. 

If $B$ is a smooth conic $Q$ union a transversal line $L$, then we claim that $X_0$ is isomorphic to either $\bP^2$ or $\bP(1,1,4)$. If $E$ is a divisor on $\bP^2$, then $X_0$ is isomorphic to either $\bP^2$ (when $E=L$) or $\bP(1,1,4)$ (when $E=Q$) by Lemma \ref{l:degnormalcone}. 
Thus we may assume that the center of $E$ on $\bP^2$ is a point $p\in Q\cap L$. In particular, $\ord_E$ is proportional to the monomial valuation $v_t$ at $p$ of weight $(1,t)$ in the analytic coordinate defined by $(Q,L)$ for some $t\in \bQ_{>0}$. If $t\geq \frac{1}{2}$, then Proposition \ref{prop:P2-S-inv} implies that $v_t$ and hence $E$ are toric, which indicates $X_0\cong \bP^2$. We may assume that $t\in (0, \frac{1}{2})$. Since $S_{\bP^2}(v_t)$ is affine for $t\in [0,\frac{1}{2}]$ by Proposition \ref{prop:P2-S-inv}, \cite[Proof of Lemma 4.6]{LXZ21} implies that $X_0 \cong \Proj\,\gr_{v_t} R \cong \Proj\,\gr_{v_0}(\gr_{v_\frac{1}{2}} R)$ where $R$ is the anti-canonical ring of $\bP^2$. Since $v_{\frac{1}{2}}$ is toric whose induced $\bG_m$-action preserves $Q$, we know that $\gr_{v_\frac{1}{2}} R \cong R$, and the induced filtration of $v_0$ on $\gr_{v_\frac{1}{2}} R $ is the same as that on $R$. Thus we have $X_0 \cong \Proj\,\gr_{v_0} R \cong \bP(1,1,4)$. The proof is finished.
\end{proof}

\begin{prop}\label{prop:P2-S-inv}
Let $Q$ and $L$ be a smooth conic and a transversal line in $\bP^2$. Let $p\in Q\cap L$ be an intersection point. Denote by $v_t$ the monomial valuation at $p$ of weights $(1,t)$ in the analytic coordinates defined by $(Q,L)$. Then $v_t$ is a toric valuation for $t\geq \frac{1}{2}$. Moreover, 
\[
S_{\bP^2}(v_t) = \begin{cases}
\frac{1}{2} + 2t  & \textrm{ if }t\in [0, \frac{1}{2}];\\
1 + t & \textrm{ if }t\in [\frac{1}{2},+\infty).
\end{cases}
\]
Here we follow the notation from \cite{BJ20} on the $S$-invariant of valuations.
\end{prop}

\begin{proof}
We may pick affine coordinates $(x,y)$  at $p$ such that $Q= \{x'=0\}$  and $L=\{y=0\}$ where $x' = x-y^2$.  Then we know that $v_t(x) = \min \{1, 2t\}$. By similar arguments as \cite[Proof of Proposition 6.4]{LXZ21}, we have that $v_t$ is a toric monomial valuation in the coordinates $(x,y)$ of weight $(1,t)$ whenever $t\geq \frac{1}{2}$. Thus a standard toric computation shows that $S_{\bP^2}(v_t) = 1+t$ for $t\geq \frac{1}{2}$. Meanwhile, $S_{\bP^2} (v_0) = S_{\bP^2}(Q) = \frac{1}{2}$ by a standard volume computation. We know that $S_{\bP^2}(v_t)$ is a concave function by \cite[Lem. 4.6.1]{LXZ21}. Thus it suffices to show that $S_{\bP^2}(v_{\frac{1}{4}}) = \frac{1}{2}(S_{\bP^2}(v_0) +S_{\bP^2}(v_{\frac{1}{2}}))=1 $.

Let $\mu:\tX \to \bP^2$ be the $(4,1)$-weighted blowup at $p$ in the analytic coordinates $(x',y)$ with exceptional divisor $E$. Denote by $\tQ: = \mu_*^{-1}Q$. Then we have 
\[
(\tQ^2) = (\mu^* Q - 4E)^2 =  (Q^2) + 16(E^2) = 4 - 16\cdot \frac{1}{4} = 0.
\]
Thus by \cite[Lem. 1.22]{KM98} we know that $[\tQ]$ is at the boundary of $\overline{NE}(\tX)$. Thus we have $T_{\bP^2} (\ord_E) = \ord_E(\frac{3}{2}Q) = 6$. By \cite[Lem. 6.2]{LXZ21} we know that 
\[
S_{\bP^2}(\ord_E) = \frac{T_{\bP^2}(\ord_E) }{3} + \frac{12}{T_{\bP^2}(\ord_E) }= 4.
\]
Since $v_{\frac{1}{4}} = \frac{1}{4}\ord_E$, we have that $S_{\bP^2}(v_{\frac{1}{4}}) = \frac{1}{4}S_{\bP^2}(\ord_E) = 1$. The proof is finished. 
\end{proof}

\section{Coregularity 0 pairs}\label{a:coreg0}
In this appendix, we prove two results concerning coregularity 0 CY pairs.

\begin{thm}\label{t:coreg0index}
If $(X,\Delta)$ is a projective slc CY pair of coregularity $0$ with Weil index\footnote{The \emph{Weil index} of $\Delta$ is the smallest positive integer $\la$ such that $\la \Delta$ is an integral Weil divisor.}  $\la$, then  $\la'(K_{X}+\Delta)\sim 0$, where $\la'= {\rm lcm}(\la,2)$.
\end{thm}

\begin{thm}\label{t:coreg0Fano}
Let $X$ be an lc  Fano variety.
If $X$  admits a complement of coregularity 0, then $X$ admits a 2-complement of coregularity 0.
\end{thm}

Recall, a \emph{complement} of an lc Fano variety $X$ is a $\bQ$-divisor $D$ such that $(X,D)$ is an lc CY pair. A complement has \emph{coregularity 0} if $\coreg(X,D)=0$ and is a \emph{2-complement} if $2(K_X+D)\sim0$.

The above theorems were previously proven in \cite[Cor. 9]{FMM22} and \cite[Thm. 4]{FFMP22} using different techniques from the proofs given below. 
For example, the proof of Theorem \ref{t:coreg0Fano} in \cite{FFMP22} uses the strategy of Birkar's proof of the boundedness of complements \cite{Bir19} with significant upgrades to utilize the coregularity assumption. 
Note that \cite{FMM22,FFMP22} also prove versions of these theorems for generalized pairs.
We thank Stefano Filipazzi and Joaqu\'{i}n Moraga for explaining their results in \cite{FFMP22,FMM22}, which led us to the following arguments.

We prove Theorems \ref{t:coreg0index} and \ref{t:coreg0Fano} using a degeneration argument.
The proof of Theorem \ref{t:coreg0Fano} takes the following strategy.
\begin{enumerate}
\item Using that $X$ admits a complement $D$   of  coregularity 0, we can construct  a test configuration degenerating
$
(X,D)\rightsquigarrow (X_0,D_0),
$
where the normalization  $(\overline{X}_0,\overline{G}_0+\overline{D}_0)$ is an lc 
CY pair with $K_{\overline{X}_0}+\overline{G}_0+\overline{D}_0\sim 0$.
 \item By descending a non-vanishing section of 
$\omega_{\oX_0}^{[2]}(2(\overline{G}_0+\overline{D}_0))$ to $X_0$, we show  $\overline{D}_0$ is a 2-complement. 
This step uses higher Poincar\'e maps \cite[4.18]{Kol13} and is similar to \cite[Sec. 5.3]{FMM22}

\item Finally, we  extend $D_0$ to a 2-complement on $X$.
\end{enumerate}
The proof of Theorem \ref{t:coreg0index} takes a similar approach.

\subsection{Poincar\'e residue map}
Let $(Y,\Delta_Y)$ be a dlt sub-pair and $p \in Y$ a minimal lc center of dimension $0$.
 If $m$ is an even integer such that $m(K_Y+\Delta_Y)$ is Cartier, then there is a Poincar\'e residue map 
\[
\cR^m_{(Y,\Delta_Y)\to p}: 
H^0(Y,\omega_{Y}^{[m]}(m\Delta_Y)) \to \bk,
\]
which is defined by writing an element $s\in H^0(Y,\omega_{Y}^{[m]}(m\Delta_Y))$ in the form 
\[
s =_{\rm loc} f \cdot  ( \tfrac{dx_1}{x_1} \wedge \cdot \cdots \cdot \wedge  \tfrac{dx_n}{x_n})^{\otimes m},
\]
where $x_1,\ldots, x_n \in \cO_{Y,p}$ is a regular system of parameters such that  $\Supp(\Delta_Y) =_{\rm loc}\{ x_1 \cdots x_n=0\}$ and $f \in \cO_{Y,p}$, and setting $\cR^{m}_{(Y,\Delta_Y)\to p}(s):=f(p)$. 
Note that the map is independent of the choice of parameters, as well as their ordering, which uses the assumption that $m$ is even.
See \cite[\S 4.18]{Kol13} for a detailed study of such residue maps and their generalizations.

Next let $(X,\Delta)$ be a projective lc  CY pair of coregularity 0. 
Fix a dlt modification $g:(Y,\Delta_Y)\to (X,\Delta)$ and a  minimal lc center $p\in Y$.
For each even integer $m$ such that $m(K_X+\Delta)\sim 0$, 
we define a map 
\[
R_{(X,\Delta)}^m:
	H^0(X,\omega_{X}^{[m]}(m\Delta))
	\to \bk
\]
by the composition of the pullback map 
\[
g^*:H^0(X,\omega_{X}^{[m]}(m\Delta))
\to 
H^0(Y, \omega_Y^{[m]}(m\Delta_Y))
\]
followed by the Poincar\'e residue map $\cR_{(Y,\Delta_Y)\to p}^m$.
Note that  $R_{(X,\Delta)}^m$ is an isomorphism since it is a nonzero linear map between 1-dimensional vector spaces.

\begin{prop}\label{p:Poincareind}
The map $R_{(X,\Delta)}^m$ is independent of the choice of dlt modification $Y$  and minimal lc center $p$.
\end{prop}

\begin{proof}
The independence of the choice of minimal lc center $p \in Y$ was shown in \cite[Prop. 14]{Kol16}.
To see the independence of the choice of dlt modification,  we follow similar arguments in \cite{Kol16}. Fix two dlt modifications
$(Y,\Delta_{Y})$ and $(Y',\Delta_{Y'})$  of $(X,\Delta)$. 
By \cite[Thm. 10.45]{Kol13}, there exist log resolutions 
$(Z,\Delta_{Z}) \to (Y,\Delta_{Y})$
and 
$(Z',\Delta_{Z'}) \to (Y',\Delta_{Y'})$,
 which are isomorphisms at all minimal lc centers. 
By the weak factorization theorem  \cite{AKMW02}, 
there is a sequence of smooth blowups and blowdowns
\[
Z:=Z_1 \overset{\pi_1}{\dashrightarrow} Z_2 \dashrightarrow \cdots  \dashrightarrow Z_{r-1}\overset{\pi_{i-1}}{\dashrightarrow} Z_r =:Z'
\]
and an integer $1 \leq s<r$
such that the  composition $\pi_i^{-1}\circ \cdots \circ \pi_1^{-1}:Z_i\dashrightarrow  Z_1$ for $i\leq s$ and $\pi_{r-1} \circ \cdots  \circ \pi_i: Z_i \dashrightarrow Z_r$ for $i\geq s$ are morphisms. 
We define $\Delta_{Z_i}$ to be the crepant pullback of $\Delta_{Z_1}$ for $i \leq s$ and the crepant pullback of $\Delta_{Z_r}$ for $i\geq s$ via the above morphisms. Additionally,  we may assume each $\pi_i$ is a blowup or blowdown of a smooth subvariety that has simple normal crossing with $\Supp(\Delta_{Z_{i+1}})$ or $\Supp(\Delta_{Z_i})$, respectively.

By the above discussion, the proposition reduces to  understanding the residue map under a single blowup. In particular, it suffices to verify: 
if $g:(W,\Delta_W) \to (Z,\Delta_Z)$  is the blowup along  a subvariety that has simple normal crossing with $\Supp(\Delta_Z)$, then

\[
\cR^m_{(W,\Delta_{W})\to w}\circ g^* = \cR^m_{(Z,\Delta_{Z})\to z} ,
\]
where $z\in Z$ and $w \in W$ are minimal lc centers satisfying $g(w)=z$. 
This follows  from a local computation. 
\end{proof}

\subsection{Descending complements}

\begin{prop}\label{p:descendcomplement}
Let $(X,\Delta)$ be an slc CY pair of coregularity 0 and $m$ an even integer.
If the normalization $(\oX,\oG+\oDe)$ satisfies $m(K_{\oX}+\oG+\oDe)\sim 0$, then $m(K_{X}+\Delta)\sim 0$.
\end{prop}

The  proposition is proven in \cite[Thm. 4.10]{FFMP22} using the theory of admissible sections.
We give a different proof using residue maps and \cite[Prop. 5.8]{Kol13}.

\begin{proof}[Proof of Proposition \ref{p:descendcomplement}]
Let $(\overline{X},\overline{G}+\overline{\Delta}):=\sqcup_{i=1}^r (\overline{X}_i,\overline{G}_i+\overline{\Delta}_i)$
 denote the normalization of $(X,\Delta)$ and $\oG_{i}^n= \sqcup_j \oG_{i,j}$ the conductor divisor.
By assumption, there is a generically non-vanishing section 
$ \overline{s}\in H^0(\overline{X},\omega_{\overline{X}}^{[m]}(m(\overline{\Delta}+\overline{G})))$.
By rescaling each $\overline{s}\vert_{\oX_i}$, we may assume  $R^m_{(\overline{X}_i,\overline{\Delta}_i+\overline{G}_i)}(\overline{s} \vert_{\oX_i})= 1$.
Let $\overline{t}$ denote the image of $\overline{s}$ under the residue map
\[
H^0(\oX,\omega_{\oX}^{[m]}(m(\oG+\oDe)))
\to 
H^0(\oG^n,\omega_{\oG^n}^{[m]}(m {\rm Diff}_{\oG^n}(\oDe)))
.\]
Since $m$ is even,  $\overline{s}$ descends to a section $s\in H^0(X,\omega_{X}^{[m]}(m\Delta))$ if and only if  
$\overline{t}=\tau^*\overline{t}$, 
where $\tau: \overline{G}^n\to \overline{G}^n$ is the induced involution; see \cite[Prop. 5.8]{Kol13}.

To verify the previous condition,  note that 
$R^m_{(\overline{X}_i,\overline{\Delta}_i+\overline{G}_i)}$ equals the composition
\[
H^0(\overline{X}_i,  \omega_{\oX_i}(m(\overline{G}_i+\overline{\Delta}_i)) ) 
\longrightarrow
H^0( \oG_{i,j},\omega_{\oG_{i,j}}^{[m]} (m {\rm Diff}_{\oG_{i,j}} (\Delta)) )
\overset{R_{i,j}}{\longrightarrow}
\bk
,\]
where $\overline{G}_{i,j}$ is an irreducible component of $\overline{G}_i$
and
$R_{i,j}:=R^m_{(\oG_{i,j}, {\rm Diff}_{\oG_{ij}}\overline{\Delta})}$ by \cite[Thm. 17]{Kol16}.
Thus $R_{i,j} ( \overline{t}  \vert_{\oG_{i,j}})=1$ for all $i$ and $j$
and so
$R_{i,j} ((\tau^* \overline{t}) \vert_{ \oG_{i,j}})
=1$ as well.
Using that  $R_{i,j}$ is an isomorphism, it follows that 
$(\tau^*\overline{t})\vert_{ \oG_{i,j}}
= 
\overline{t}\vert_{\oG_{i,j}}$,
and hence $\overline{t}= \tau^*\overline{t}$.
Therefore  $\overline{s}$ descends to a section of $ s\in H^0(X,\omega_{X}^{[m]}(m\Delta))$, which is generically non-zero.
Thus there is an effective divisor $D$ with $D\sim m(K_{X}+\Delta)$.
Since $K_X+\Delta\sim_{\bQ} 0$,  we conclude $0 =D\sim m(K_X+\Delta)$.
\end{proof}

\subsection{Proof of theorems}

\begin{proof}[Proof of Theorem \ref{t:coreg0index}]
We may assume $X$ is normal by Proposition \ref{p:descendcomplement}. 
Now choose an ample Cartier divisor  $L$ on $X$ and consider the  pair $(Y,\Gamma+D)$, where $Y=C_p(X,L)$ is the projective cone over $X$ with respect to $L$, $\Gamma:= \Delta_{L,p}$ and $D:=X_\infty$ in the notation of Definition \ref{def:orb-div}. 
By Proposition \ref{prop:orbcone-adjunction}.3,  $(Y,\Gamma+D)$ is an lc boundary polarized CY pair 
and  $(D, {\rm Diff}_\Gamma(D)) \cong   (X,\Delta)$. Note that $(Y,\Gamma+D)$ has coregularity 0.
By Propositions \ref{p:degenreducedboundary} and \ref{p:descendcomplement}, there exists a test configuration 
$(\cY, \Gamma_{\cY}+\cD)
\to \bA^1$
of $(Y,\Gamma+D)$
such that 
$2(K_{\cY_0}+\Gamma_{\cY_0}+\cD_0)\sim 0$. 

To construct a degeneration of $(X,\Delta)$, set  
\[\cX:= \cD\quad \text{ and }\quad \Delta_{\cX}:= {\rm Diff}_{\cD}(\Gamma_{\cY}).
\]
Note that $(\cX,\Delta_{\cX})\to \bA^1$ is a projective family of slc pairs  by Lemma \ref{l:slcadj} and $(\cX,\cD)\vert_{\bA^1 \setminus 0} \cong (X,D)\times(\bA^1 \setminus 0)$. 
By adjunction, 
$K_{\cX/\bA^1} + \Delta_{\cX}\sim_{\bQ}0$ and
$2(K_{\cX_0}+ \Delta_{\cX_0}) \sim 0$.
Now set $
\cL:= \omega_{\cX/\bA^1}^{[\la']}(\la' \Delta_{\cX})$.
Since $\la' (K_{\cX/\bA^1}+\Delta_{\cX})$ is a $\bQ$-Cartier $\Z$-divisor, 
$\cL_t
\cong   
\omega_{\cX_t}^{[\la']}(\la'\Delta_{\cX_t}) $ for all $t\in \bA^1$
 by \cite[Prop. 2.79]{KolNewBook}. 
Thus
$
\cL_0
$
is a line bundle and  $h^0(\cX_0,\cL_0)\neq 0$.
Thus $\cL$ is a line bundle in a neighborhood of $\cX_0$ and hence a line bundle on $\cX$.
Since $\cL$ is a line bundle and $K_{\cX/\bA^1}+\Delta_{\cX} \sim_{\bQ}0$, $h^0(\cX_t,\cL_t)$ is independent of $t\in U$ by \cite[Prop. 2.65]{KolNewBook}. 
Therefore
\[
h^0(X,\omega_{X}^{[\la']}(\la'\Delta))
=
h^0(\cX_1,\cL_1)
 =
 h^0(\cX_0,\cL_0)> 0.
\]
Thus there is a divisor $0 \leq  B\sim \la'(K_{X}+\Delta)$. 
Since $K_{X}+\Delta \sim_{\bQ}0$,  $ B=0$ as desired.
\end{proof}

\begin{proof}[Proof of Theorem \ref{t:coreg0Fano}]
Fix  a complement $D$ of $X$ such that $(X,D)$ has  coregularity 0. 
By Propositions \ref{p:degenreducedboundary} and  \ref{p:descendcomplement}, there exists a test configuration $(\cX,\cD) \to \bA^1$ of $(X,D)$ such that 
$2(K_{\cX_0} +\cD_0) \sim 0$.
By Lemma \ref{l:extenddivisor},
there exists a  $\Q$-divisor $\cB$ on $\cX$ such that $\cB_0= \cD_0$ and $(\cX,\cB)\to \bA^1$ is a $\bG_m$-equivariant family of boundary polarized CY pairs with  $2(K_{\cX/\bA^1} +B)\sim 0$. 
Therefore $B=\cB_1$ is a 2-complement of $X$
and
\[
{\rm coreg}(X,B)= {\rm coreg}(\cX_0,\cB_0)
= {\rm coreg}(X, D)=0
,\]
where the first two equalities hold by Proposition \ref{p:SrcSequiv}.
\end{proof}

\bibliographystyle{alpha}
\bibliography{bpcy}

\end{document}

%% file: Pics/cateye.tex
\begin{tabular}{c}
\begin{tikzpicture}[gren0/.style = {draw, circle,fill=greener!80,scale=.7},gren/.style ={draw, circle, fill=greener!80,scale=.4},blk/.style ={draw, circle, fill=black!,scale=.08},plc/.style ={draw, circle, color=white!100,fill=white!100,scale=0.02},smt/.style ={draw, circle, color=gray!100,fill=gray!100,scale=0.02},lbl/.style ={scale=.2}] 
\node[smt] at (0,0) (1){};
\node[smt] at (0,2) (2){};
\node[smt] at (2,2) (3){};
\node[smt] at (2,0) (4){};
\draw [-,color=gray] (1) to (2) to (3) to (4) to (1);
\draw (1,1) circle (0.5);
\draw (1,1) ellipse (0.25 and 0.5);
\node[blk] at (1,0.5) {};
\node[blk] at (1,1.5) {};

\node[above right, node font=\tiny] at (0,0) {$\bP^2$};
\node[below right, node font=\tiny] at (1,0.5) {$A_3$};
\node[above right, node font=\tiny] at (1,1.5) {$A_3$};

\end{tikzpicture}
\end{tabular}

%% file: Pics/oxdegeneration.tex
\begin{tabular}{c}
\begin{tikzpicture}[gren0/.style = {draw, circle,fill=greener!80,scale=.7},gren/.style ={draw, circle, fill=greener!80,scale=.4},blk/.style ={draw, circle, fill=black!,scale=.08},plc/.style ={draw, circle, color=white!100,fill=white!100,scale=0.02},smt/.style ={draw, circle, color=gray!100,fill=gray!100,scale=0.02},lbl/.style ={scale=.2}] 
\node[smt] at (0,0) (1){};
\node[smt] at (0,1.5) (2){};
\node[plc] at (1,1.25) (3){};
\node[plc] at (0.98,1.23) (3r){};
\node[plc] at (-1,1.25) (4){};
\node[plc] at (-0.98,1.23) (4r){};
\node[blk] at (0.3,0.7) (5){};
\node[plc] at (0,0.9) (6){};
\node[plc] at (0,0.6) (7){};
\node[plc] at (0.64,0.8) (8){};
\node[plc] at (0.4,0.5) (9){};
\node[blk] at (-0.3,0.7) (10){};
\node[plc] at (-0.64,0.8) (11){};
\node[plc] at (-0.4,0.5) (12){};
\draw [-,color=gray] (1) to (2);
\draw [-,color=gray] (1) to (3r); 
\draw [-,color=gray] (1) to (4r);
\draw [-,color=gray,rounded corners=1pt] (2) to [bend left=80] (3) to [bend left = 60] (2);
\draw [-,color=gray,rounded corners=1pt] (2) to [bend left=60] (4) to [bend left = 80] (2);
\draw [-] (6) to [bend right=30] (5) to [bend right=20] (8);
\draw [-] (7) to [bend left=30] (5) to [bend left=30] (9);
\draw [-] (11) to [bend right=30] (10) to [bend right=20] (6);
\draw [-] (12) to [bend left=37] (10) to [bend left=30] (7);

\node[above, node font=\tiny] at (5) {$A_3$};
\node[above, node font=\tiny] at (10) {$A_3$};
\node[below, node font=\tiny] at (1) {$\bP(1,1,2) \cup \bP(1,1,2)$};

\end{tikzpicture}
\end{tabular}

%% file: Pics/ox.tex
\begin{tabular}{c}
\begin{tikzpicture}[gren0/.style = {draw, circle,fill=greener!80,scale=.7},gren/.style ={draw, circle, fill=greener!80,scale=.4},blk/.style ={draw, circle, fill=black!,scale=.08},plc/.style ={draw, circle, color=white!100,fill=white!100,scale=0.02},smt/.style ={draw, circle, color=gray!100,fill=gray!100,scale=0.02},lbl/.style ={scale=.2}] 
\node[smt] at (-.3,-.15) (1){};
\node[smt] at (-.3,1.85) (2){};
\node[smt] at (1.7,1.85) (3){};
\node[smt] at (1.7,-.15) (4){};
\node[blk] at (1.2,.7) (5){};
\node[blk] at (0.5,1.16) (6){};
\node[blk] at (1.2,1.47) (7){};
\node[plc] at (1.2, 1.7) (8){};
\node[plc] at (1.4,1.56) (9){};
\node[plc] at (1.2, 0.15) (10){};
\node[plc] at (0,0.94) (11){};

\draw [-,color=gray] (1) to (2) to (3) to (4) to (1);
\draw (0.7,0.7) circle (0.5);
\draw [-] (10) to (8);
\draw [-] (11) to (9);

\node[above right, node font=\tiny] at (5) {$A_3$};
\node[above left, node font=\tiny] at (6) {$A_3$};
\node[above right, node font=\tiny] at (1) {$\bP^2$};

\end{tikzpicture}
\end{tabular}

%% file: Pics/A9.tex
\begin{tabular}{c}
\begin{tikzpicture}[gren0/.style = {draw, circle,fill=greener!80,scale=.7},gren/.style ={draw, circle, fill=greener!80,scale=.4},blk/.style ={draw, circle, fill=black!,scale=.08},plc/.style ={draw, circle, color=white!100,fill=white!100,scale=0.02},smt/.style ={draw, circle, color=gray!100,fill=gray!100,scale=0.02},lbl/.style ={scale=.2}] 
\node[smt] at (0,0) (1){};
\node[smt] at (0,2) (2){};
\node[smt] at (2,2) (3){};
\node[smt] at (2,0) (4){};
\node[plc] at (0.4,1.4) (5){};
\node[plc] at (1.6,1.4) (6){};
\node[blk] at (1,1.1) (7){};
\node[plc] at (0.4,0.8) (8){};
\node[plc] at (1.6,0.8) (9){};
\draw [-,color=gray] (1) to (2) to (3) to (4) to (1);
\draw [-] (5) to [bend right=30] (7) to [bend right=30] (6);
\draw [-] (8) to [bend left=30] (7) to [bend left=30] (9);

\node[above right, node font=\tiny] at (0,0) {$\bP^2$};
\node[above, node font=\tiny] at (7) {$A_9$};
\end{tikzpicture}
\end{tabular}

%% file: Pics/A9rep.tex
\begin{tabular}{c}
\begin{tikzpicture}[gren0/.style = {draw, circle,fill=greener!80,scale=.7},gren/.style ={draw, circle, fill=greener!80,scale=.4},blk/.style ={draw, circle, fill=black!,scale=.08},plc/.style ={draw, circle, color=white!100,fill=white!100,scale=0.02},smt/.style ={draw, circle, color=gray!100,fill=gray!100,scale=0.02},lbl/.style ={scale=.2}] 
\node[smt] at (0,0) (1){};
\node[smt] at (0,1.5) (2){};
\node[blk] at (1,1.25) (3){};
\node[plc] at (0.98,1.23) (3r){};
\node[plc] at (-1,1.25) (4){};
\node[plc] at (-0.98,1.23) (4r){};
\node[plc] at (0,0.9) (6){};
\node[plc] at (0,0.6) (7){};
\node[plc] at (0.54,0.8) (8){};
\node[plc] at (0.39,0.6) (9){};
\node[blk] at (-0.3,0.7) (10){};
\node[plc] at (-0.64,0.8) (11){};
\node[plc] at (-0.4,0.5) (12){};
\draw [-,color=gray] (1) to (2);
\draw [-,color=gray] (1) to [bend left=10] (3); 
\draw [-,color=gray] (1) to (4r);
\draw [-,color=gray,rounded corners=1pt] (2) to [bend right=20] (3);
\draw [-,color=gray,rounded corners=1pt] (2) to [bend left=60] (4) to [bend left = 80] (2);
\draw [-] (6) to [bend right=5] (3);
\draw [-] (7) to [bend left=5] (3);
\draw [-] (11) to [bend right=30] (10) to [bend right=20] (6);
\draw [-] (12) to [bend left=37] (10) to [bend left=30] (7);

\node[above, node font=\tiny] at (3) {$A_9/\bmu_4$};
\node[above, node font=\tiny] at (10) {$A_9$};
\node[below, node font=\tiny] at (1) {$\bP(1,1,5) \cup \bP(1,4,5)$};

\end{tikzpicture}
\end{tabular}

%% file: Pics/D6.tex
\begin{tabular}{c}
\begin{tikzpicture}[gren0/.style = {draw, circle,fill=greener!80,scale=.7},gren/.style ={draw, circle, fill=greener!80,scale=.4},blk/.style ={draw, circle, fill=black!,scale=.08},plc/.style ={draw, circle, color=white!100,fill=white!100,scale=0.02},smt/.style ={draw, circle, color=gray!100,fill=gray!100,scale=0.02},lbl/.style ={scale=.2}] 
\node[smt] at (0,0) (1){};
\node[smt] at (0,2) (2){};
\node[smt] at (2,2) (3){};
\node[smt] at (2,0) (4){};
\draw [-,color=gray] (1) to (2) to (3) to (4) to (1);
\draw (1,1) circle (0.5);
\draw (1,1) ellipse (0.25 and 0.5);
\draw [-] (1,0.3) to (1,1.7);
\node[blk] at (1,0.5) (5){};
\node[blk] at (1,1.5) (6){};

\node[below right, node font=\tiny] at (5) {$D_6$};
\node[above left, node font=\tiny] at (6) {$D_6$};
\node[above right, node font=\tiny] at (0,0) {$\bP^2$};
\end{tikzpicture}
\end{tabular}

%% file: Pics/D6rep.tex
\begin{tabular}{c}
\begin{tikzpicture}[gren0/.style = {draw, circle,fill=greener!80,scale=.7},gren/.style ={draw, circle, fill=greener!80,scale=.4},blk/.style ={draw, circle, fill=black!,scale=.08},plc/.style ={draw, circle, color=white!100,fill=white!100,scale=0.02},smt/.style ={draw, circle, color=gray!100,fill=gray!100,scale=0.02},lbl/.style ={scale=.2}] 
\node[smt] at (0,0) (1){};
\node[smt] at (0,1.5) (2){};
\node[plc] at (1,1.25) (3){};
\node[plc] at (0.98,1.23) (3r){};
\node[plc] at (-1,1.25) (4){};
\node[plc] at (-0.98,1.23) (4r){};
\node[blk] at (0.3,0.7) (5){};
\node[plc] at (0,0.9) (6){};
\node[plc] at (0,0.6) (7){};
\node[plc] at (0.64,0.8) (8){};
\node[plc] at (0.4,0.5) (9){};
\node[blk] at (-0.3,0.7) (10){};
\node[plc] at (-0.64,0.8) (11){};
\node[plc] at (-0.4,0.5) (12){};
\node[plc] at (0.45,1.12) (13){};
\node[plc] at (-0.45,1.12) (14){};
\draw [-,color=gray] (1) to (2);
\draw [-,color=gray] (1) to (3r); 
\draw [-,color=gray] (1) to (4r);
\draw [-,color=gray,rounded corners=1pt] (2) to [bend left=80] (3) to [bend left = 60] (2);
\draw [-,color=gray,rounded corners=1pt] (2) to [bend left=60] (4) to [bend left = 80] (2);
\draw [-] (6) to [bend right=30] (5) to [bend right=20] (8);
\draw [-] (7) to [bend left=30] (5) to [bend left=30] (9);
\draw [-] (11) to [bend right=30] (10) to [bend right=20] (6);
\draw [-] (12) to [bend left=37] (10) to [bend left=30] (7);
\draw [-] (1) to (13);
\draw [-] (1) to (14);

\node[above right, node font=\tiny] at (5) {$D_6$};
\node[above left, node font=\tiny] at (10) {$D_6$};
\node[below, node font=\tiny] at (1) {$\bP(1,1,2) \cup \bP(1,1,2)$};

\end{tikzpicture}
\end{tabular}

%% file: Pics/doublecubic.tex
\begin{tabular}{c}
\begin{tikzpicture}[gren0/.style = {draw, circle,fill=greener!80,scale=.7},gren/.style ={draw, circle, fill=greener!80,scale=.4},blk/.style ={draw, circle, fill=black!,scale=.03},plc/.style ={draw, circle, color=white!100,fill=white!100,scale=0.02},smt/.style ={draw, circle, color=gray!100,fill=gray!100,scale=0.02},lbl/.style ={scale=.2}] 
\node[smt] at (0,0) (1){};
\node[smt] at (0,2) (2){};
\node[smt] at (2,2) (3){};
\node[smt] at (2,0) (4){};
\node[blk] at (1.5,0.5) (5){};
\node[blk] at (1.1,0.8) (6){};
\node[blk] at (0.6,0.8) (7){};
\node[blk] at (.55,1) (8){};
\node[blk] at (0.6,1.2) (9){};
\node[blk] at (1.1,1.2) (10){};
\node[blk] at (1.5,1.5) (11){};
\draw [-,color=gray] (1) to (2) to (3) to (4) to (1);
\draw [-,thick] (5) to [in=20,out=100] (6) to [in=300,out=200] (7) to [in=270,out=120] (8) to [in=240,out=90] (9) to [in=160,out=60] (10) to [in=260,out=340] (11);

\node[above right, node font=\tiny] at (0,0) {$\bP^2$};

\end{tikzpicture}
\end{tabular}

%% file: Pics/doublecubicrep.tex
\begin{tabular}{c}
\begin{tikzpicture}[gren0/.style = {draw, circle,fill=greener!80,scale=.7},gren/.style ={draw, circle, fill=greener!80,scale=.4},blk/.style ={draw, circle, fill=black!,scale=.03},plc/.style ={draw, circle, color=white!100,fill=white!100,scale=0.02},smt/.style ={draw, circle, color=gray!100,fill=gray!100,scale=0.02},lbl/.style ={scale=.2}] 
\node[plc] at (.3,-.2) (A){};
\node[plc] at (.3,2) (B){};
\node[plc] at (2,2) (C){};
\node[plc] at (2,0) (D){};
\node[blk] at (0.9,0) (1){};
\node[blk] at (1.5,0.5) (5){};
\node[blk] at (1.1,0.8) (6){};
\node[blk] at (0.6,0.8) (7){};
\node[blk] at (.55,1) (8){};
\node[blk] at (0.6,1.2) (9){};
\node[blk] at (1.1,1.2) (10){};
\node[blk] at (1.5,1.5) (11){};
\draw[-,color=gray] (1) to (5);
\draw[-,color=gray] (1) to (.55,.9);
\draw[color=gray,densely dotted] (1) to (11);
\draw[-,color=gray] (1.22,0.8) to (11);
\draw [-,thick] (5) to [in=20,out=100] (6) to [in=300,out=200] (7) to [in=270,out=120] (8) to [in=240,out=90] (9) to [in=160,out=60] (10) to [in=260,out=340] (11);

\node[below, node font=\tiny] at (1) {elliptic cone};

\end{tikzpicture}
\end{tabular}

%% file: Pics/triplecat.tex
\begin{tabular}{c}
\begin{tikzpicture}[gren0/.style = {draw, circle,fill=greener!80,scale=.7},gren/.style ={draw, circle, fill=greener!80,scale=.4},blk/.style ={draw, circle, fill=black!,scale=.03},plc/.style ={draw, circle, color=white!100,fill=white!100,scale=0.02},smt/.style ={draw, circle, color=gray!100,fill=gray!100,scale=0.02},lbl/.style ={scale=.2}] 
\node[smt] at (0,0) (1){};
\node[smt] at (0,2) (2){};
\node[smt] at (2,2) (3){};
\node[smt] at (2,0) (4){};
\draw [-,color=gray] (1) to (2) to (3) to (4) to (1);
\draw (1,1) ellipse (0.7 and 0.5);
\draw (1,1) ellipse (0.3 and 0.5);
\draw (1,1) ellipse (0.1 and 0.5);

\node[above right, node font=\tiny] at (0,0) {$\bP^2$};
\end{tikzpicture}
\end{tabular}

%% file: Pics/triplecatrep.tex
\begin{tabular}{c}
\begin{tikzpicture}[gren0/.style = {draw, circle,fill=greener!80,scale=.7},gren/.style ={draw, circle, fill=greener!80,scale=.4},blk/.style ={draw, circle, fill=black!,scale=.03},plc/.style ={draw, circle, color=white!100,fill=white!100,scale=0.02},smt/.style ={draw, circle, color=gray!100,fill=gray!100,scale=0.02},lbl/.style ={scale=.2}] 
\node[smt] at (0,0) (1){};
\node[smt] at (0,1.5) (2){};
\node[plc] at (1,1.25) (3){};
\node[plc] at (0.98,1.23) (3r){};
\node[plc] at (-1,1.25) (4){};
\node[plc] at (-0.98,1.23) (4r){};
\node[blk] at (0.3,0.7) (5){};
\node[plc] at (0,1) (6){};
\node[plc] at (0,0.5) (7){};
\node[plc] at (0.7,0.88) (8){};
\node[plc] at (0.4,0.5) (9){};
\node[blk] at (-0.3,0.7) (10){};
\node[plc] at (-0.7,0.88) (11){};
\node[plc] at (-0.4,0.5) (12){};
\node[plc] at (.5,0.63) (13){};
\node[plc] at (0,0.75) (14){};
\node[plc] at (-.5,0.63) (15){};
\draw [-,color=gray] (1) to (2);
\draw [-,color=gray] (1) to (3r); 
\draw [-,color=gray] (1) to (4r);
\draw [-,color=gray,rounded corners=1pt] (2) to [bend left=80] (3) to [bend left = 60] (2);
\draw [-,color=gray,rounded corners=1pt] (2) to [bend left=60] (4) to [bend left = 80] (2);
\draw [-] (6) to [bend right=20] (5) to [bend right=30] (8);
\draw [-] (7) to [bend left=30] (5) to [bend left=37] (9);
\draw [-] (14) to [bend left=0] (5) to [bend left=0] (13);
\draw [-] (11) to [bend right=30] (10) to [bend right=20] (6);
\draw [-] (12) to [bend left=37] (10) to [bend left=30] (7);
\draw [-] (14) to [bend left=0] (10) to [bend left=0] (15);

\node[below, node font=\tiny] at (1) {$\bP(1,1,2) \cup \bP(1,1,2)$};

\end{tikzpicture}
\end{tabular}

%% file: Pics/lines.tex
\begin{tabular}{c}
\begin{tikzpicture}[gren0/.style = {draw, circle,fill=greener!80,scale=.7},gren/.style ={draw, circle, fill=greener!80,scale=.4},blk/.style ={draw, circle, fill=black!,scale=.03},plc/.style ={draw, circle, color=white!100,fill=white!100,scale=0.02},smt/.style ={draw, circle, color=gray!100,fill=gray!100,scale=0.02},lbl/.style ={scale=.2}] 
\node[smt] at (-.1,-.1) (1){};
\node[smt] at (-.1,1.9) (2){};
\node[smt] at (1.9,1.9) (3){};
\node[smt] at (1.9,-.1) (4){};
\node[plc] at (0.2,0.5) (5){};
\node[plc] at (1.7,0.69) (6){};
\node[plc] at (0.7,1.5) (7){};
\node[plc] at (0.4,0.3) (8){};
\node[plc] at (0.8,0.35) (9){};
\node[plc] at (1.2,0.4) (10){};
\node[plc] at (1.6,0.45) (11){};
\draw [-,color=gray] (1) to (2) to (3) to (4) to (1);
\draw [-, thick] (5) to (6);
\draw [-] (7) to (8);
\draw [-] (7) to (9);
\draw [-] (7) to (10);
\draw [-] (7) to (11);

\node[above right, node font=\tiny] at (1) {$\bP^2$};
\end{tikzpicture}
\end{tabular}

%% file: Pics/linesrep.tex
\begin{tabular}{c}
\begin{tikzpicture}[gren0/.style = {draw, circle,fill=greener!80,scale=.7},gren/.style ={draw, circle, fill=greener!80,scale=.4},blk/.style ={draw, circle, fill=black!,scale=.03},plc/.style ={draw, circle, color=white!100,fill=white!100,scale=0.02},smt/.style ={draw, circle, color=gray!100,fill=gray!100,scale=0.02},lbl/.style ={scale=.2}] 

\node[smt] at (0.2,0.5) (5){};
\node[smt] at (1.7,0.69) (6){};
\node[plc] at (0.7,1.1) (7){};
\node[plc] at (0.46,0.53) (8){};
\node[plc] at (0.78,0.57) (9){};
\node[plc] at (1.1,0.61) (10){};
\node[plc] at (1.42,0.65) (11){};
\node[plc] at (0.3,-.2) (12){};
\node[plc] at (1.8,-0.01) (13){};
\node[plc] at (0.56,-0.16) (14){};
\node[plc] at (0.89,-0.12) (15){};
\node[plc] at (1.21,-0.08) (16){};
\node[plc] at (1.52,-0.04) (17){};
\draw [-,color=gray] (5) .. controls (0.3,1.6)and(1,1.6) .. (6);
\draw [-,color=gray] (5) to (12) to (13) to (6);
\draw [-,color=gray] (5) to (6);
\draw [-,thick] (12) to (13);
\draw [-] (7) to (8);
\draw [-] (7) to (9);
\draw [-] (7) to (10);
\draw [-] (7) to (11);
\draw [-] (8) to (14);
\draw [-] (9) to (15);
\draw [-] (10) to (16);
\draw [-] (11) to (17);

\node[below right, node font=\tiny] at (14) {$\bP^2 \cup \mathbb{F}_1$};

\end{tikzpicture}
\end{tabular}

%% file: Pics/12conic.tex
\begin{tabular}{c}
\begin{tikzpicture}[gren0/.style = {draw, circle,fill=greener!80,scale=.7},gren/.style ={draw, circle, fill=greener!80,scale=.4},blk/.style ={draw, circle, fill=black!,scale=.03},plc/.style ={draw, circle, color=white!100,fill=white!100,scale=0.02},smt/.style ={draw, circle, color=gray!100,fill=gray!100,scale=0.02},lbl/.style ={scale=.2}] 
\node[smt] at (0,0) (1){};
\node[smt] at (0,2) (2){};
\node[smt] at (2,2) (3){};
\node[smt] at (2,0) (4){};
\draw [-,color=gray] (1) to (2) to (3) to (4) to (1);
\draw[thick] (1,1) ellipse (0.7 and 0.35);
\draw (1,1) ellipse (0.3 and 0.6);

\node[above right, node font=\tiny] at (1) {$\bP^2$};
\end{tikzpicture}
\end{tabular}

%% file: Pics/12conicrep.tex
\begin{tabular}{c}
\begin{tikzpicture}[gren0/.style = {draw, circle,fill=greener!80,scale=.7},gren/.style ={draw, circle, fill=greener!80,scale=.4},blk/.style ={draw, circle, fill=black!,scale=.03},plc/.style ={draw, circle, color=white!100,fill=white!100,scale=0.02},smt/.style ={draw, circle, color=gray!100,fill=gray!100,scale=0.02},lbl/.style ={scale=.2}] 
\node[smt] at (0,2) (1){};
\node[smt] at (-.6,1) (2){};
\node[smt] at (.6,1) (3){};
\node[smt] at (-.6,0) (4){};
\node[smt] at (.6,0) (5){};
\node[plc] at (-1,1) (A){};
\node[plc] at (1,1) (B){};
\node[plc] at (-.45,.88) (6){};
\node[plc] at (-.15,.81) (7){};
\node[plc] at (.15,.81) (8){};
\node[plc] at (.45,.88) (9){};
\node[plc] at (-.45,-.12) (10){};
\node[plc] at (-.15,-.19) (11){};
\node[plc] at (.15,-.19) (12){};
\node[plc] at (.45,-.12) (13){};
\draw[gray,densely dotted] (0,0) ellipse (0.6 and 0.2);
\draw[gray,densely dotted] (0,1) ellipse (0.6 and 0.2);
\draw[gray] (-0.6,0) arc(180:360:0.6 and 0.2);
\draw[gray] (-0.6,1) arc(180:360:0.6 and 0.2);
\draw[-,gray] (1) to (2);
\draw[-,gray] (1) to (3);
\draw[-,gray] (2) to (4);
\draw[-,gray] (3) to (5);
\draw[thick] (-0.6,0) arc(180:360:0.6 and 0.2);
\draw[-] (1) to (6) to (10);
\draw[-] (1) to (7) to (11);
\draw[-] (1) to (8) to (12);
\draw[-] (1) to (9) to (13);

\node[below, node font=\tiny] at (0,-.25) {$\bP(1,1,4) \cup \mathbb{F}_4$}; 
\end{tikzpicture}
\end{tabular}